\documentclass[a4paper,11pt]{book}

\usepackage{amsmath}
\usepackage{amsthm}
\usepackage{bbm}
\usepackage{mathtools}
\usepackage{amssymb}
\usepackage{graphicx}
\usepackage{color}
\usepackage{latexsym}
\usepackage[latin1]{inputenc}
\usepackage[english,francais]{babel} 
\usepackage[T1]{fontenc}
\usepackage{comment}
\usepackage{stmaryrd}
\usepackage{amssymb} 
\usepackage[mathcal]{eucal}
\usepackage{fancybox}
\usepackage{algorithm,algorithmicx,algpseudocode}
\usepackage{xcolor}
\usepackage{rotating}
\usepackage{pifont}
\usepackage{palatino}

\algrenewcommand\algorithmicindent{2.5em}
\algrenewcommand{\algorithmiccomment}[1]{{\color{brown} \hfill$\blacktriangleright$ #1 }}

\DeclareMathOperator{\pgcd}{pgcd}
\DeclareMathOperator{\Act}{Act}
\DeclareMathOperator{\Arg}{Arg}
\DeclareMathOperator{\cc}{cc}
\DeclareMathOperator{\cycl}{cycl}

\DeclareMathOperator{\face}{f}
\DeclareMathOperator{\arete}{a}
\DeclareMathOperator{\som}{s}

\DeclareMathOperator{\inte}{intact}
\DeclareMathOperator{\ext}{extact}

\DeclareMathOperator{\Cat}{Cat}
\DeclareMathOperator{\Ima}{Im}
\DeclareMathOperator{\Ree}{Re}

\newcommand{\R}{\mathbb R}
\newcommand{\Q}{\mathbb Q}
\newcommand{\C}{\mathbb C}
\newcommand{\N}{\mathbb N}

\newcommand{\K}{\mathbb K}
\newcommand{\F}{\mathbb F}

\newcommand{\pd}[2]{\frac{\partial#1}{\partial#2}}
\newcommand{\pdd}[3]{\frac{\partial^{#1} #2}{\partial#3^#1}}
\newcommand{\pdtrip}[3]{\frac{\partial^{#1} #2}{\partial#3}}
\newcommand{\diffs}[2]{#1 \, \ominus \, #2}

\usepackage[a4paper,linktocpage,bookmarksnumbered,pagebackref]{hyperref}

\newcounter{nota}

\usepackage{multido}

\newcommand{\tut}{\mathcal T_{G,\Delta}}
\newcommand{\sing}[1]{\{#1\}}
\newcommand{\ens}[1]{\left\{#1\right\}}
\newcommand{\pare}[1]{\left(#1\right)}
\newcommand{\module}[1]{\left|#1\right|}
\newcommand{\enstq}[2]{\left\{#1\ \left|\ #2\right.\right\}}
\newcommand{\ar}[1]{\{#1,#1'\}}
\newcommand{\contract}[2]{\cccc(#1,#2)}
\newcommand{\delete}[2]{\dddd(#1,#2)}

\DeclareMathOperator{\cccc}{contract}
\DeclareMathOperator{\dddd}{delete}
\newcommand{\ent}[2]{\ens{#1,\dots,#2}}

\newcommand{\Som}[1]{\mathcal {S} \left(#1 \right)}
\newcommand{\Ar}[1]{\mathcal A \hspace{-2pt} \left(#1 \right)}

\newcommand{\bSi}{\textbf{S}$_i$}
\newcommand{\mSi}{\textbf{S}_i}
\newcommand{\bSe}{\textbf{S}$_e$}
\newcommand{\mSe}{\textbf{S}_e}
\newcommand{\bI}{\textbf{I}}
\newcommand{\bL}{\textbf{L}}
\newcommand{\ua}[1]{\underset{#1}{\rightarrow}}

\newcommand{\fig}[3]{
\begin{figure}[h!]
\begin{center}
 \includegraphics #1
 \end{center}
\vspace{-15pt}
\caption{ #2}
\label{#3}
\end{figure}
}

\newtheorem{prop}{Proposition}[section]

\newtheorem{lem}[prop]{Lemme}
\newtheorem{defi}[prop]{Définition}

\newtheorem{theo}[prop]{Théorème}

\newtheorem{cor}[prop]{Corollaire}
\newtheorem{conjecture}{Conjecture}

\newtheorem{leme}[prop]{Lemma}
\newtheorem{defie}[prop]{Définition}
\newtheorem{core}[prop]{Corollary}
\newtheorem{theoe}[prop]{Theorem}

\definecolor{darkgray}{gray}{0.50}
\definecolor{Navy}{RGB}{0,0,180}

\def\emm#1,{{\em #1}}
\newcommand{\beq}{\begin{equation}}
\newcommand{\eeq}{\end{equation}}

\catcode`\@=11
\def\section{\@startsection{section}{1}%
 \z@{10pt}{5pt}%
 {\normalfont\bfseries\scshape\centering\Large}
}

\def\subsubsection{\@startsection{subsubsection}{3}%
 \z@{.5\linespacing\@plus\linespacing}{-.5em}%{.5\linespacing}%
  {\normalfont\bfseries\itshape}}
\catcode`\@=12

	{	\begin{floexample}
		\begin{center}
		\begin{minipage}{.9\textwidth}
		\caption{\sffamily #1} \sffamily \small \medskip	
	}%à l'ouverture \begin
	{	\end{minipage}
		\end{center}    
		\end{floexample}
	}%à la fermeture \end

	{
		\end{quote}%\par				
	}%à la fermeture \end

	\usepackage[a4paper, hmarginratio=1:1, vmarginratio=1:1]{geometry} %\geometry{ hmargin=3.5cm, vmargin=2cm }
	\vfuzz \hfuzz
 	\usepackage{setspace}
	\usepackage{lettrine}
	
	\usepackage[helvetica]{quotchap} %helvetica en \arabic, newcentury ou times en \Roman, utopia
	\usepackage{fancyhdr}
	\renewcommand{\chaptermark}[1]{\markboth{\thechapter.\ #1}{}}
	
	\let\origdoublepage\cleardoublepage
	\newcommand{\clearemptydoublepage}{%
	\clearpage
	{\pagestyle{empty}\origdoublepage}%
	}
	\let\cleardoublepage\clearemptydoublepage 
\title{} 
\author{}
\date{}

\begin{document}

\newlength{\plarg}
\setlength{\plarg}{\textwidth}

\begin{titlepage}
\thispagestyle{empty}

%\begin{titlepage}
\thispagestyle{empty}

\noindent
\vspace{-20pt}
\begin{figure}[!h]
\begin{minipage}[t]{0.35\textwidth}
\vspace{-80pt}
 \hspace{-0.4 \textwidth}\includegraphics[width= 1.1\textwidth]{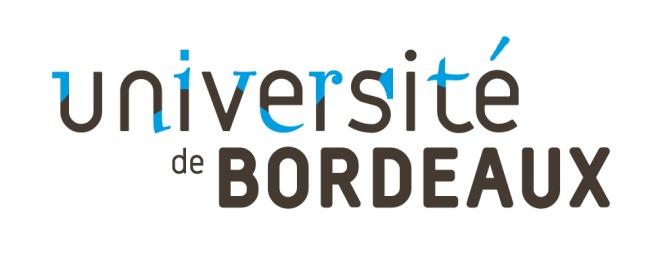} 
\end{minipage}
\begin{minipage}{0.3\textwidth}
\vspace{0.4cm} \centering
{\Large\bfseries Thèse}\\ \vspace{0.4cm}
présentée à\\ \vspace{1cm}
\end{minipage}
\hfill
\end{figure}
\vspace{0.4cm}
\noindent
\vspace{-80pt}
\begin{center}
\begin{minipage}{\plarg}
\vspace{0.6cm} \centering
%{\Large\bfseries Thèse} \hfill \vspace{10pt}\includegraphics[width=0.1\textwidth]{logoBx1}\\ \vspace{10mm}
%présentée à\\ \vspace{1cm}
{\Large\bfseries L'UNIVERSIT\'E DE BORDEAUX}\\ \vspace{0.5cm}
\'ECOLE DOCTORALE DE MATH\'EMATIQUES ET INFORMATIQUE
\vspace{1cm}

\newlength{\trz}
\settowidth{\trz}{Mme Mireille Bousquet-Mélou}
\newlength{\tyy}
\settowidth{\tyy}{Examinatrice}
\newlength{\tqq}
\settowidth{\tqq}{Directrice de recherche CNRS}

{par \Large\bfseries \hspace{10pt}  Julien Courtiel\hspace{10pt} }\\ \vspace{1cm}
{POUR OBTENIR LE GRADE DE}\\\vspace{0.2cm}
{\Large \bfseries DOCTEUR}\\\vspace{0.5cm}
{\large SP\'ECIALIT\'E: \bfseries Informatique}
\centerline{\line(1,0){500}} \vskip 0.5cm
{\Huge\bfseries Combinatoire du polynôme de Tutte et des cartes planaires}\\ \vspace{0.2cm}
\centerline{\line(1,0){500}} \vskip 1cm
Soutenue le 3 octobre 2014 au Laboratoire Bordelais de Recherche en Informatique\\
\textbf{après avis des rapporteurs :}
\begin{tabular}{p{\trz}p{\tqq}p{\tyy}}
\vfill M. Marc Noy & \vfill Professeur & \   \\
\vfill M. Gilles Schaeffer  & \vfill Directeur de recherche CNRS  \\
\end{tabular}
\vspace{0.5cm}

\textbf{devant la commission d'examen composée de :}
\begin{tabular}{p{\trz}p{\tqq}p{\tyy}}
\vfill Mme Frédérique Bassino & \vfill Professeure & \vfill Examinatrice \\
\vfill Mme Mireille Bousquet-Mélou & \vfill Directrice de recherche CNRS & \vfill Directrice \\
\vfill M. Jérémie Bouttier & \vfill Chercheur CEA & \vfill Examinateur \\
\vfill M. Robert Cori & \vfill Professeur émérite & \vfill Examinateur \\
\vfill M. Marc Noy & \vfill Professeur & \vfill Rapporteur \\
\vfill M. Bruno Salvy & \vfill Directeur de recherche INRIA & \vfill Examinateur \\
\end{tabular}
\end{minipage}
\end{center}

\end{titlepage}

\maketitle

\tableofcontents

\selectlanguage{french}

\chapter*{Remerciements}
\chaptermark{Remerciements}

Les premiers remerciements s'adressent tout naturellement à ma chère directrice de thèse, Mireille. Je m'estime très chanceux d'avoir travaillé sous sa gouverne. J'ai beaucoup appris auprès d'elle, que ce soit d'un point de vue scientifique ou extra-scientifique. Je ne la remercierai jamais assez pour la patience qu'elle a usée lors de la lecture et relecture de ce manuscrit, pour tous les conseils avisés qu'elle m'a prodigués, et pour tous les moments de complicité que nous avons pu partager. Je lui décerne sans hésiter la médaille d'or des directrices de thèse. 

Un grand merci à Marc Noy et Gilles Schaeffer pour avoir bien voulu rapporter ma thèse. La lecture d'un manuscrit de deux cents pages et l'écriture d'un rapport ne constituent pas un programme estival très enthousiasmant ; je suis très reconnaissant du sacrifice. Je les remercie également pour leurs appréciations élogieuses qui sauront m'encourager et m'inspirer pour la suite.

Je suis très flatté que Frédérique Bassino, Jérémie Bouttier, Robert Cori et Bruno Salvy aient également accepté de faire partie du jury de cette thèse. C'est au final un grand honneur de voir les noms de toutes ces personnes, dont j'apprécie les qualités humaines et scientifiques, sur la première page de ce manuscrit.

Je remercie également tous les collègues combinatoriciens que j'ai pu côtoyer dans ma petite vie de doctorant. Je mesure la chance, d'un point de vue professionnel et humain, de faire partie de ce qui porte maintenant le nom de la "communauté Aléa". Je salue tout particulièrement Adeline, Alice, Axel, Basile, Cécile, \'Elie, \'Eric, Gwendal, Jérémie, Julien, Kerstin, Kilian, Philippe, Vincent.
\`A une échelle plus locale, je remercie tous les membres du LaBRI, et plus particulièrement tous ceux qui ont fait/font partie du groupe combinatoire bordelais : Adrien, Bétréma, Jean-Christophe, Jean-François, Jérôme, Marie-Line, Marni, Mathilde, Nicolas, Olivier, Philippe, Robert, Romaric, Sacha, Valentin, Vincent, Yvan...
Leur accueil fut si chaleureux, leur bonne humeur si communicative, les déjeuners du vendredi midi si conviviaux \footnote{bien qu'on me forçât régulièrement à remplir la carafe} qu'ils firent oublier à un Toulousain qu'il habitait Bordeaux.

Je ne saurais oublier tous mes amis qui ont été présents dans cette thèse,  particulièrement dans les moments un peu difficiles : Fif, Grand Ju, Guigui, Lolo, Mama, Ninou, Olaf, Toto, Tutu, Vass, Yoyo, Zez\footnote{Oui, on aime bien les surnoms un peu bizarres.}... Leur amitié m'est précieuse, je ne les oublierai pas malgré la distance future.

Les derniers remerciements sont adressés à l'ensemble de ma famille, pour leur proximité et leur soutien. Le prochain Noël va être difficile\footnote{Et je ne dis pas ça parce que je n'aurai pas de cadeaux...} ! Je remercie tout particulièrement mes parents pour m'avoir aidé dans le grand rush final ; que ce soit pour le déménagement qui a malheureusement coïncidé avec la fin de rédaction de cette thèse, ou pour la préparation de cette soutenance. Un remerciement final à ma s\oe{}ur et à mes grands-parents qui ont pris la peine de parcourir des kilomètres pour voir ma modeste pomme raconter des choses qu'ils ne comprendront pas forcément.

\newpage
\strut
\newpage

\chapter{Introduction}

\section{La combinatoire énumérative}

\subsection{Principe}

Que ce soit en informatique, en mathématiques ou en physique théorique, les structures combinatoires telles que les arbres ou les graphes  abondent. La\textit{ combinatoire énumérative} est la science qui cherche à compter ces objets, que ce soit de manière exacte ou asymptotique. Les motivations sont diverses : analyse d'algorithmes, étude de probabilités, lien avec les fonctions de partition en physique statistique, génération aléatoire...

Illustrons notre propos par un exemple classique  : celui des arbres plans binaires. Un \textit{arbre plan binaire} (ou \textit{arbre binaire}) est une structure de données arborescente constituée de n\oe{}uds de deux sortes :  les \textit{n\oe{}uds internes}, qui possèdent chacun deux n\oe{}uds enfants, un enfant gauche 
 et un enfant droit, et les \textit{feuilles},  qui ne possèdent pas d'enfants. Il existe un unique n\oe{}ud sans parent, qui est appelé \textit{n\oe{}ud racine}. La figure \ref{arbresbinaires} montre des exemples d'arbres  binaires.

\fig{[width = 0.9 \textwidth]{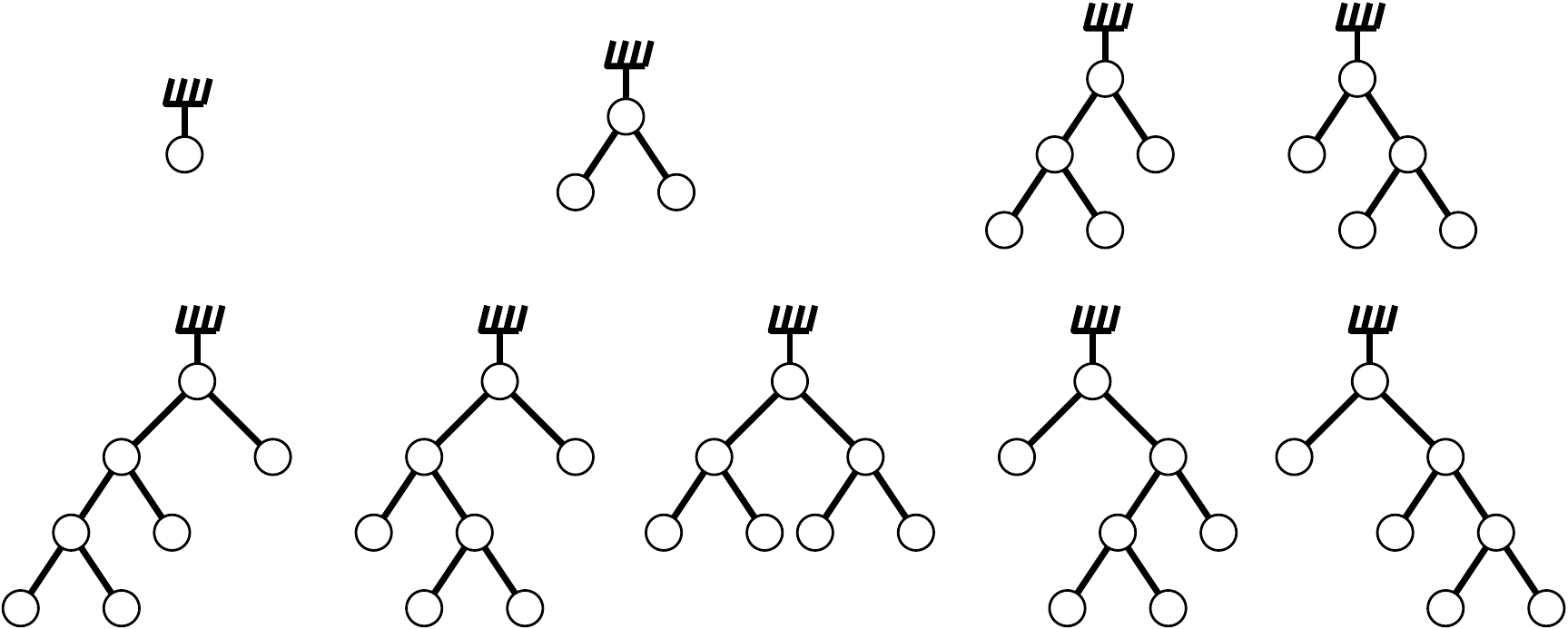}}{Les arbres binaires de taille au plus $4$. Le sommet racine est celui indiqué par le symbole en forme de râteau.}{arbresbinaires}

%Lorsqu'un ensemble d'objets combinatoires est muni d'une fonction de \textit{taille} à valeurs dans $\N$ telle que le nombre d'objets à taille fixée est fini, nous parlons de \textit{classe combinatoire}. Munis de la fonction qui compte le nombre de feuilles, les arbres binaires forment une classe combinatoire.

Appelons $c_n$ le nombre d'arbres binaires à $n$ feuilles, pour $n \geq 0$. Ce nombre est le célèbre \textit{nombre de Catalan}
\begin{equation} \label{catalan}
c_{n-1} = \frac 1 {n+1} {2n \choose n},
\end{equation}
 intervenant dans de nombreux problèmes de dénombrement : citons par exemple les arbres plans comptés selon le nombre d'arêtes ou les chemins de Dyck. Deux questions naturelles se posent alors : comment ces objets sont-ils reliés entre eux ? Y a-t-il un argument qui expliquerait la simplicité du cardinal de ces objets ?
Des réponses peuvent être apportées sous forme de \textit{bijections} : contour de l'arbre, mots de \L{}ukasiewicz, lemme cyclique, etc. (Nous ne donnerons aucun détail, car cela serait légèrement hors propos ; nous renvoyons le lecteur aux livres de Richard Stanley \cite{stanley-vol1} \cite{stanley-vol-2}.)

De manière générale, trouver des bijections entre objets combinatoires de différents types  est une des vocations de la combinatoire énumérative. L'exemple qui nous intéressera particulièrement pour ce mémoire est la bijection entre ce qu'on appellera les \textit{cartes planaires} et certains arbres plans dits \textit{bourgeonnants} \cite{bdg2002} (voir Chapitre \ref{c:cartesarbres}).

\subsection{Séries génératrices et approche récursive}

Beaucoup de problèmes énumératifs ne présentent pas de formules closes aussi simples que \eqref{catalan}. Il est souvent plus pratique d'obtenir des informations sous forme de séries génératrices. La \textit{série génératrice} (ici \textit{ordinaire}) 
associée à une classe d'objets combinatoires dont le cardinal à taille $n$ fixée est donné par $a_n$ est définie comme la série formelle
$$A(z) = \sum_{n \geq 0} a_n z^n.$$
Les séries génératrices constituent un outil fondamental de la combinatoire énumérative pour la simple et bonne raison qu'une décomposition récursive des objets se traduit souvent en une équation fonctionnelle satisfaite par la série génératrice.

Pour revenir à l'exemple précédent, si un arbre binaire n'est pas réduit à une feuille, alors il se décompose au niveau du sommet racine en un couple (ordonné) de deux arbres binaires de taille inférieure (voir figure \ref{arbrebinairedecompo}). Dans ce cas-là, la somme des tailles des deux sous-arbres donne la taille totale de l'arbre. Nous obtenons alors la relation de récurrence 
\begin{equation}
c_1 = 1, \quad c_n = \sum_{k = 1}^{n-1} c_k \, c_{n-k} \quad \textrm{ pour }n \geq 1,
\label{reccn}
\end{equation}
ce qui se traduit pour la série génératrice $C(z)$ des arbres binaires par
\begin{equation}
C(z) = z + C(z)^2.
\label{ef:catalan}
\end{equation}
La résolution de cette équation du second degré nous donne immédiatement la valeur de $C(z)$, à savoir
\begin{equation}
C(z) = \frac 1 2 - \frac 1 2 \sqrt{1 - 4z}.
\label{sg:catalan}
\end{equation}
Nous retrouvons alors la formule close \eqref{catalan} des nombres $c_n$ par un développement de Taylor de cette série (ou alors en appliquant le \textit{théorème d'inversion de Lagrange} \cite[Théo. 5.4.2]{stanley-vol1} à l'équation \eqref{ef:catalan}, ce qui a l'avantage d'être une méthode plus générale).

\fig{[width = 0.6 \textwidth]{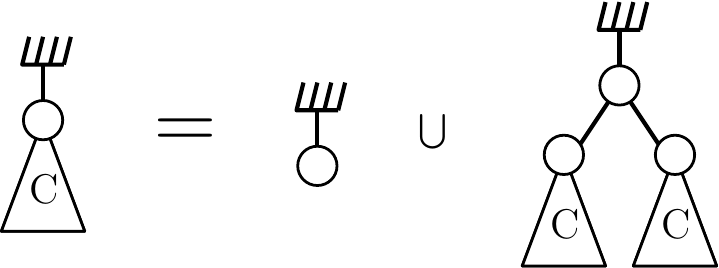}}{Décomposition d'un arbre plan binaire au sommet racine.}{arbrebinairedecompo}

Dans ce mémoire, les objets seront généralement énumérés selon plusieurs paramètres. En d'autres termes, il n'y aura pas qu'une seule fonction de taille. La définition de série génératrice pour des objets énumérés par une suite multivariée $a_{n_1,n_2,\dots,n_m}$ est donnée par la formule plus générale 
$$A(z_1,\dots,z_m) = \sum_{n_1,n_2,\dots,n_m \geq 0} a_{n_1,n_2,\dots,n_m} \, {z_1}^{n_1} \, {z_2}^{n_2} \, \dots \, {z_m}^{n_m}.$$
Par exemple, nous pouvons nous intéresser au nombre $c_{n,p}$ d'arbres binaires à $n$ feuilles et $p$ sommets internes \footnote{ce qui est plutôt idiot puisque le nombre de n\oe{}uds internes est égal au nombre de feuilles moins un, mais gardons cet exemple pour l'illustration de la méthode.} et appeler $C(z,x)$ la nouvelle série génératrice $\sum_{n,p\geq 0} c_{n,p} \,z^n\, x^p$.

Obtenir une relation de récurrence comme \eqref{reccn} devient de plus en plus laborieux lorsque le nombre de paramètres grandit. La  \textit{méthode symbolique} de Philippe Flajolet et Robert Sedgewick \cite{flajolet-sedgewick} permet de contourner ce problème en donnant directement une équation fonctionnelle comme \eqref{ef:catalan}. Pour expliquer en quelques mots le principe de cette méthode, nous voyons  les objets combinatoires comme des agrégats d'\textit{atomes}, des sortes de "briques élémentaires" que nous pondérons individuellement par une variable. Par exemple pour les arbres binaires, nous pondérons chaque feuille par $z$ et chaque n\oe{}ud interne par $x$. Puis les opérations ensemblistes classiques sont traduites par des opérations basiques sur les séries génératrices : l'union disjointe correspond à l'addition, le produit cartésien à la multiplication... Cela permet de transformer beaucoup de décompositions récursives en des équations fonctionnelles. Par exemple, quand nous disons qu'un arbre binaire est soit une feuille (poids $z$), soit un n\oe{}ud interne (poids $x$) sur lequel est attaché deux arbres binaires (poids $C(z,x)$ chacun), cela se traduit par
$$C(z,x) = z + x  \, C(z,x)^2.$$
Nours verrons une application plus poussée de la méthode symbolique quand nous décrirons la série génératrice des arbres \textit{bourgeonnants} (arbres liés aux cartes planaires, on le rappelle).

\subsection{\'Enumération asymptotique et combinatoire analytique}
\label{ss:asympt}

La question du comportement asymptotique est très présente en combinatoire énumérative. En effet, les propriétés qui souvent nous intéressent  concernent des objets de grande taille (complexité asymptotique en informatique, modèle discret qui tend vers le continu en physique statistique, etc.).

Une simple application de la formule de Stirling à \eqref{catalan} permet d'obtenir l'équivalent asymptotique
$$c_n \sim \frac {4^{n-1}}{\pi n^{3/2}}.$$
Comme souligné plus haut, tous les problèmes énumératifs ne mènent pas à des formules closes comme \eqref{catalan}. Toutefois, de nombreux théorèmes d'analyse complexe permettent de déduire les régimes asymptotiques des nombres étudiés à partir du comportement singulier de leur série génératrice. Ces théorèmes forment ce qu'on appelle la \textit{combinatoire analytique}. Ils sont en grande majorité répertoriés dans le célèbre livre de Philippe Flajolet et Robert Sedgewick \cite{flajolet-sedgewick} (citons également Andrew Odlyzko \cite{flajolet-odlyzko} qui a grandement contribué à cette théorie) et seront fréquemment utilisés au cours de ce mémoire\footnote{Le lecteur motivé sera donc avisé de se munir de ce livre !}.

Le théorème qui est peut-être le plus représentatif de la combinatoire analytique s'appelle \textit{théorème de transfert} et sera utilisé à plusieurs reprises dans cette thèse. Voici son énoncé (sous une forme qui nous sera utile ici).

\begin{theo} \emph{\textbf{(théorème de transfert)}}  \label{theotransfert}
Appelons $\Delta$-domaine de rayon $\rho$ tout domaine complexe de la forme
$$ \enstq z {\module z < \rho + \varepsilon} \ \bigcap \ \enstq z {\module{\Arg(z-\rho)} > \alpha   },$$
avec $\varepsilon > 0$ et $\alpha \in ]0,\pi/2[$ (voir figure \ref{deltadomaine}).

\fig{[scale = 1.2]{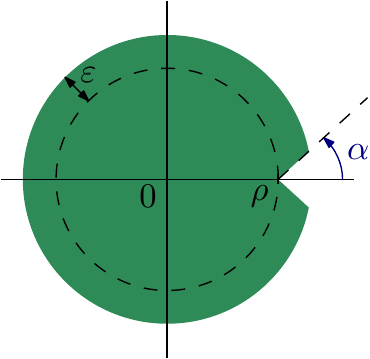}}{Un $\Delta$-domaine de rayon $\rho$.}{deltadomaine}

Soit $f(z) = \sum_{n \geq 0} f_n z^n$ une fonction analytique sur un $\Delta$-domaine de rayon  $\rho$. Si le comportement singulier de $f$ au voisinage de $\rho$ est de la forme
$$f(z) \underset{z \rightarrow \rho}\sim c \, (1-z/\rho)^{-\alpha},$$
où $\alpha$ est un nombre complexe qui n'est pas un entier négatif et où $c$ est une constante non nulle, alors
$$f_n \sim \frac{c}{\Gamma(\alpha)} \rho^{-n} n^{\alpha - 1}.$$
\end{theo}

Citons également quelques petites variantes du théorème de transfert que nous utiliserons par la suite, valables pour tout entier $k \geq 0$ et pour toute constante $c \neq 0$ (nous avons toujours besoin de l'analyticité de $f$ sur un $\Delta$-domaine de rayon $\rho$) :
\begin{align*} f(z) \underset{z \rightarrow \rho}\sim  c \pare{1 - z/\rho}^k \ln\pare{1 - z/\rho} \quad & \Longrightarrow \quad f_n \sim  (-1)^{k+1} c \,k! \,\rho^{-n} \, n^{-(k+1)},\\
 f(z) \underset{z \rightarrow \rho}\sim  c \pare{1 - z/\rho}^k \ln^{-1}\pare{1 - z/\rho} \quad & \Longrightarrow \quad f_n \sim  (-1)^{k} c \,k! \,\rho^{-n} \, n^{-(k+1)} \ln^{-2}(n),\\
 f(z) \underset{z \rightarrow \rho}\sim  c \pare{1 - z/\rho}^{-(k+1)} \ln^{-2}\pare{1 - z/\rho} \quad & \Longrightarrow \quad f_n \sim   \frac{c} {k!} \,\rho^{-n} \,  n^{k} \, \ln^{-2}(n).
\end{align*}

Ainsi  l'équivalent $c_n \sim 4^{n-1}/\pare{\pi n^{3/2}}$ provient d'une simple application du théorème de transfert, une fois  \eqref{sg:catalan} obtenu. Notons que l'exposant en $-3/2$ est typique des arbres plans.

\subsection{Nature des séries}
\label{ss:nature}

Il existe une hiérarchie naturelle entre les séries formelles (potentiellement multivariées) que nous pouvons présenter de la manière suivante (les définitions vont suivre) :
$$\textrm{rationnelle }\rightarrow\textrm{ algébrique }\rightarrow\textrm{ holonome }\rightarrow\textrm{ différentiellement algébrique.}$$
Toutes ces classes possèdent des propriétés de clôture (clôture par composition, dérivation, addition, multiplication, etc. -- on se référera notamment aux articles de Lipshitz pour les séries holonomes \cite{lipshitz-diag,lipshitz-df})  mais aussi des propriétés donnant des informations sur la série elle-même, comme par exemple le comportement asymptotique possible de ses coefficients (nous reviendrons sur ce point). En quelque sorte, cette hiérarchie  mesure la complexité d'une série formelle, les séries les plus simples étant les séries rationnelles, les plus complexes celles qui ne sont pas différentiellement algébriques. 

Passons maintenant aux définitions. Soit $A(z_1,\dots,z_m)$ une série à $m$ variables et à coefficients dans $\Q$. Cette série est dite \textit{rationnelle} si elle est de la forme $A=P/Q$ où $P$ et $Q$ sont deux polynômes à coefficients dans $\Q$. Elle est dite \textit{algébrique} s'il existe un polynôme $P$ à coefficients rationnels tel que $P(A(z_1,\dots,z_m),z_1,\dots,z_m) = 0$. Elle est dite \textit{holonome} (ou \textit{différentiellement finie}) si l'espace vectoriel sur $\mathbb Q(z_1,\dots,z_m)$  engendré par toutes ses dérivées partielles, autrement dit
$$\textrm{Vect}_{\mathbb Q(z_1,\dots,z_m)} \ens{ \frac {\partial^{i_1+ \dots + i_m} A} {\partial {z_1}^{i_1} \dots \partial {z_m}^{i_m}} }_{i_1,\dots,i_m \geq 0},$$
est de dimension finie. Dans le cas particulier des séries formelles d'une seule variable $z$, être holonome signifie satisfaire une équation différentielle linéaire (non triviale) avec des coefficients polynômiaux en $z$.  La série $A$ est dite \textit{différentiellement algébrique} par rapport à la variable $z_i$ s'il existe un entier $n \geq 0$ et un polynôme non nul $P \in \mathbb Q[x_0,x_1,\dots,x_n,z_1,\dots,z_m]$  tel que
$$P\pare{A,\pd A {{z_i}},\dots, \pdd n  A {{z_i}}  ,z_1,\dots,z_m} = 0.$$

Comme nous l'indiquions plus haut, ces classes sont imbriquées les unes dans les autres. Par exemple, la série génératrice \eqref{sg:catalan} des arbres binaires est algébrique (comme l'indique l'équation polynomiale \eqref{ef:catalan}), donc holonome et différentiellement algébrique. Toutefois, la présence de la racine carrée l'empêche d'être rationnelle.

\begin{figure}[h!]
\begin{center}
\begin{tabular}{|c|c|} \hline Nature de la série & Comportement asymptotique des coefficients  \\
\hline
Rationnelle & $\sim \,C  \rho^{-n} n^\ell$ \\ 
Algébrique & $\sim \,C \rho^{-n} n^{p/q}$ \\
Holonome &  $\sim \,C \rho^{-n} e^{P(n^{1/r})}n^\theta \ln^\ell(n)$ \\ \hline
\end{tabular}
\end{center}
\vspace{-15pt}
\caption{Comportement  asymptotique des coefficients selon la nature de la série. (Ici $C$ est une constante, $\ell$ et $r$ des entiers positifs, $p/q$ un rationnel, $\rho$ et $\theta$ sont des nombres algébriques et $P$ un polynôme.)}
\label{hierar}
\end{figure}

En outre, le comportement asymptotique possible des coefficients d'une série dépend de la nature de cette série. La figure \ref{hierar}, qui est une copie de \cite[Fig. VII.1 p. 445]{flajolet-sedgewick}, montre la forme d'un tel comportement en cas de rationa\-lité/algé\-bricité/holo\-nomie. 
Concernant les séries différentiellement algébriques, les propriétés  asymptotiques semblent moins claires en tout généralité\footnote{du moins à ma connaissance}. Néanmoins, il arrive que l'on puisse déduire des informations asymptotiques des équations différentielles polynomiales. Par exemple, le papier d'Odlyzko et Richmond \cite{odlyzko-richmond} montre comment obtenir le comportement asymptotique des \textit{triangulations bien coloriées} à partir d'une équation différentielle satisfaite par la série génératrice.

 \section{Les cartes planaires}
 
 Objets combinatoires sur lesquels notre attention va se porter pendant la première partie de ce mémoire, les cartes planaires  constituent un domaine de recherche particulièrement actif. Initiée par Tutte dans les années soixante à des fins combinatoires \cite{tutte-triangulations,tutte-census-maps}, l'énumération des cartes planaires s'est avérée intéressante dans beaucoup d'autres domaines mathématiques, comme la physique statistique, la théorie des probabilités ou la géométrie algorithmique. 
 
 Cette section a pour objectif d'introduire la notion de carte planaire et de présenter les différents concepts qui s'y rattachent.

 \label{s:introcartes}

 \subsection{Graphes et cartes}

 Un \textit{graphe} est la donnée d'un ensemble de \textit{sommets}, d'un ensemble d'\textit{arêtes} et d'une relation d'\textit{incidence} entre sommets et arêtes. Chaque arête est incidente à un ou deux sommets (ses \textit{extrémités}). Pour un graphe $G$, l'ensemble des sommets est noté $\Som G$ et l'ensemble des arêtes $\Ar G$. 
Les arêtes qui sont incidentes à une même paire de sommets (ce qu'on appelle \textit{arêtes multiples}) sont distinguables les unes des autres.
%Une \textit{boucle} est une arête incidente à un seul sommet. Des arêtes peuvent être incidentes à une même paire de sommets, on les appellera alors \texttt{arêtes multiples}. De telles arêtes sont distinguables l'une de l'autre. 
Tout graphe peut être dessiné en représentant chaque sommet par un point et chaque arête par une ligne qui relie ses deux extrémités. La figure \ref{troisgraphes} montre un même graphe dessiné de quatre manières différentes.

\fig{[width=\textwidth]{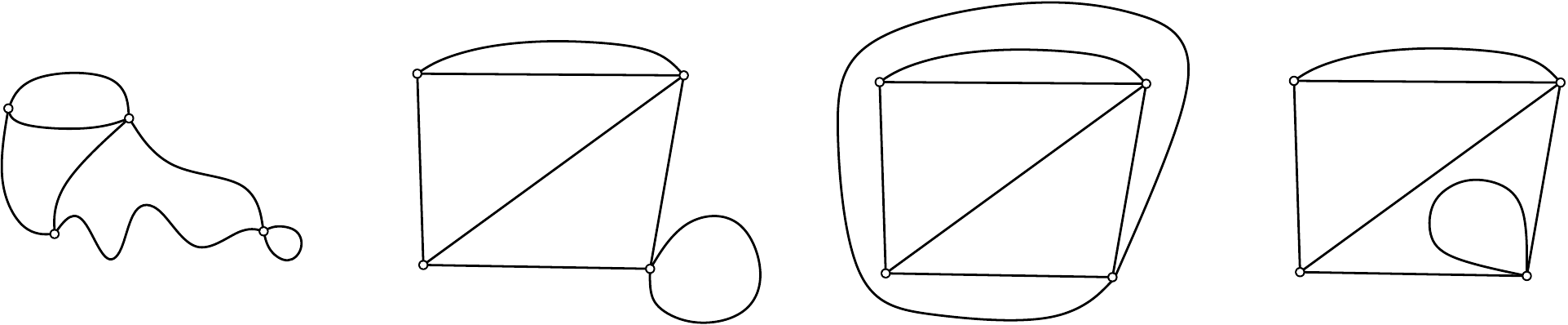}}{Quatre plongements d'un même graphe. Les trois premiers plongements représentent la même carte, le quatrième est différent. }{troisgraphes}

Le fait de dessiner un graphe sans croisement d'arêtes (sauf éventuellement en leur extrémités) définit un \textit{plongement} de ce graphe. Seuls certains graphes admettent un plongement dans le plan, on les appelle les \textit{graphes planaires}. Remarquons que tout graphe planaire admet également un plongement dans la sphère, ce dernier pouvant se déduire d'un plongement planaire par inverse de la projection stéréographique (voir figure \ref{stereo}).

\fig{[scale=0.85]{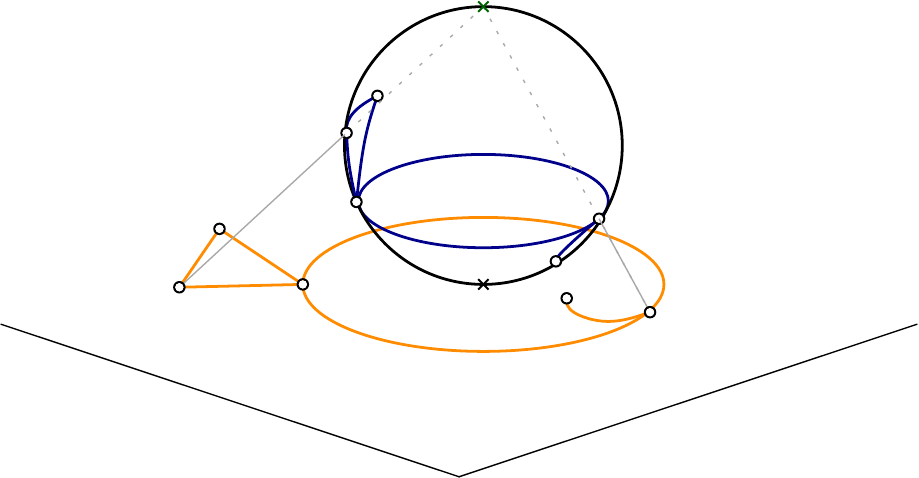}}{Image d'un plongement par projection stéréographique.}{stereo}

Une \textit{carte planaire} est un plongement d'un graphe connexe et planaire sur la sphère, considéré à \textit{déformation continue} près. Bien que définies sur la sphère, on préfèrera dessiner nos cartes dans le plan. Les feuilles de papier étant planes, il est en effet plus commode de représenter nos cartes dans le plan que sur une sphère. Il ne faudra toutefois pas oublier que ces objets sont en réalité plongés dans la sphère et qu'une arête n'est pas "bloquée" par le point à l'infini.
%En effet, les arêtes ne sont pas bloquées par le point à l'infini. 
Par exemple, les trois premiers plongements de la figure \ref{troisgraphes} donnent la même carte planaire \footnote{même s'il n'existe pas de déformation continue \textit{du plan} entre le deuxième et le troisième}. Par contre, même si le graphe sous-jacent est identique, le quatrième plongement est une carte planaire différente des autres. (Une explication sera donnée par la suite.)  
%En effet, aucune \textit{face} n'est un pentagone.

En coupant une arête en son milieu, on obtient deux \textit{demi-arêtes}, chacune étant incidente à une des extrémités de l'arête. Un \textit{coin} désigne un couple de demi-arêtes incidentes à un même sommet qui se suivent dans le sens trigonométrique. Pour représenter un coin sur un carte, on insèrera une flèche entre les deux demi-arêtes correspondantes, pointant vers le sommet incident (voir figure \ref{enracinement}). On peut d'ailleurs facilement voir que les demi-arêtes sont en bijection naturelle avec les coins. 

Dans le contexte de la combinatoire énumérative, la présence de symétries rend l'étude des cartes planaires délicate. Pour contourner ce problème, on distingue dans chacune de nos cartes un coin baptisé \textit{racine}. Une carte avec une racine est appelée \textit{carte enracinée}.  La figure \ref{enracinement} illustre cette définition. Le \textit{sommet racine} est l'unique sommet incident à la racine. L'\textit{arête racine} (resp. \textit{demi-arête racine}) est l'arête (resp. demi-arête) qui suit la racine dans le sens trigonométrique. Dans tout ce mémoire, \textbf{toutes nos cartes planaires seront supposées enracinées.} 

\fig{[scale=0.85]{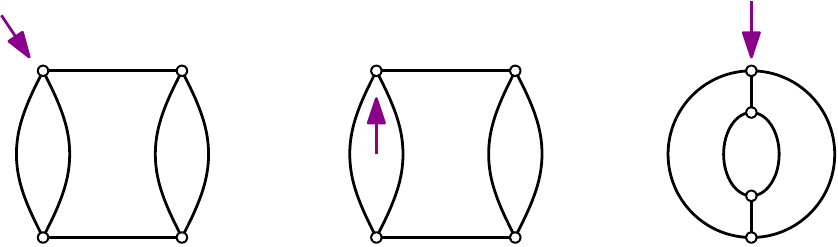}}{Trois cartes enracinées. Les deux dernières sont égales.}{enracinement}

\subsection{Faces et dualité}

Données topologiques du plongement, les \textit{faces} d'une carte planaire correspondent aux composantes connexes de la sphère privée du plongement. Les relations d'incidence peuvent être naturellement  étendues aux faces : une face $f$ est \textit{incidente} à une arête $a$ (resp. un coin $c$, un sommet $s$) si $a$ (resp. $c$, $s$) est inclus dans la frontière de $f$. Le \textit{degré} d'une face est alors le nombre de coins incidents à cette face\footnote{Notons que le degré d'un sommet peut se définir exactement de la même manière.}. La \textit{face racine} est l'unique face incidente à la racine. Pour revenir aux exemples de la figure \ref{troisgraphes}, on compte une face de degré $5$ et aucune face de degré $4$ pour chacune des trois premières cartes. Par contre, la dernière a deux faces de degré $4$ mais aucune de degré $5$, ce qui justifie \textit{a posteriori} la non-égalité avec les trois autres cartes.

La \textit{relation d'Euler} est une identité qui lie le nombre de sommets, le nombre d'arêtes et le nombre de faces d'une carte planaire $C$. Plus précisément, si ces derniers nombres sont respectivement notés $\som(C)$, $\arete(C)$, $\face(C)$, alors 
$$\som(C) - \arete(C) + \face(C) = 2.$$

Dans le monde des cartes, les faces et les sommets jouent des rôles symétriques. La notion qui permet d'échanger faces et sommets s'appelle \textit{dualité}. \'Etant donnée une carte $C$, on construit la \textit{carte duale} $C^*$ en insérant un sommet de $C^*$ dans chaque face de $C$ et en traçant une arête de $C^*$ de manière transversale pour chaque arête de $C$ (voir la figure \ref{dualite}). L'enracinement de la carte duale est canonique : en superposant les deux cartes, on obtient une quadrangulation où chaque face est composée de deux demi-arêtes de la carte originale et de deux demi-arêtes de la carte duale. La racine de la carte duale est alors l'unique coin de la carte duale qui appartient à la même face que la racine de la carte originale. Comme toute autre notion de dualité en mathématiques,  $\left(C^* \right)^* = C$. De plus, les faces sont en bijection directe avec les sommets de la carte duale, et cette bijection préserve les degrés. Ainsi, comme le montre la figure \ref{dualite}, une carte dont chaque sommet a  degré $4$ (également appelée \textit{carte tétravalente}) admet comme duale une carte dont chaque face a  degré $4$ (autrement dit une \textit{quadrangulation}).

\fig{[width=\textwidth]{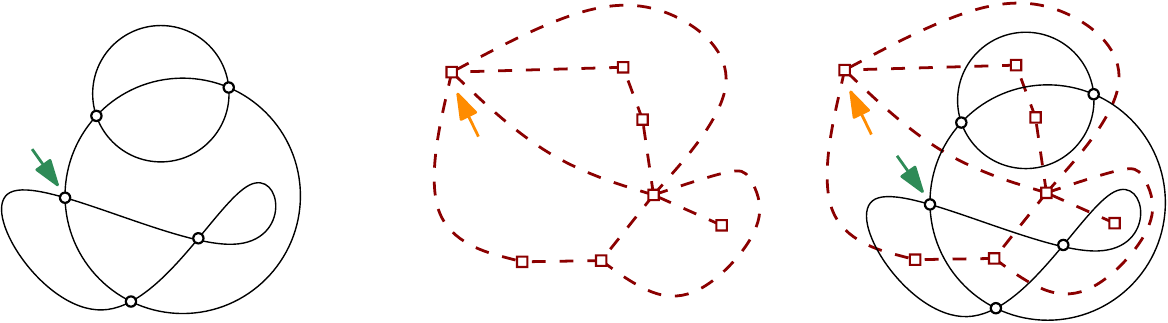}}{\`A gauche : une carte tétravalente. Au milieu : la quadrangulation duale. \`A droite : la superposition des deux cartes. }{dualite}

\subsection{\'Enumération des cartes planaires}

%Tutte peut être considéré comme le précurseur de la théorie énumérative des cartes. \`A travers sa série de papiers \textit{A census of} \cite{tutte-census-maps,tutte-triangulations,...}, il a mis au point de nombreuses techniques d'énumération nouvelles, comme l

Le comptage des cartes est une science dans laquelle les méthodes employées n'ont cessé de se diversifier de décennie en décennie. En voici une rapide chronologie. Tutte fut dans les années soixante le précurseur  du domaine. Sa \textit{méthode récursive} \cite{tutte-triangulations,tutte-census-maps}, très robuste, a permis l'énumération de nombreuses familles de cartes. Quelques années plus tard, des physiciens théoriciens  ont mis au point une redoutable technique (mais pas toujours rigoureuse) pour les problèmes de comptage de cartes : les  \textit{intégrales de matrices} \cite{brezin}. Puis vient l'\textit{approche bijective} initiée par Robert Cori et Bernard Vauquelin \cite{cori-vauquelin}, et approfondie par Gilles Schaeffer \cite{schaeffer-these}, puis Jérémie Bouttier, Philippe Di Francesco et Emmanuel Guitter \cite{bdg2002,BDG-blocked}. Elle consiste à transformer les cartes planaires en des arbres, qui sont plus facilement énumérables. La dernière application en date de cette théorie concerne les \textit{limites d'échelle}. Il a été récemment prouvé \cite{legall-brownian,miermont-brownian} que les cartes, en tant qu'espaces métriques aléatoires, convergent vers une carte continue et universelle, appelée \textit{carte brownienne}. 

% La théorie énumérative des carteient de décennis est un domaine très actif où les techniques employées se diversife en décennie : d'abord l'approche récursive inventée par Tutte , puis l'approche par intégrales de matrices avec les travaux  des physiciens t'Hooft \na{ref}, et Brézin, Itzykson, Parisi et Zyber \na{ref}, puis l'approche bijective initiée par Robert Cori et Bernard Vauquelin \cite{cori-vauquelin}, approfondie par Gilles Schaeffer \cite{schaeffer-these}, et enfin, les travaux sur les limites d'échelle de cartes.

Afin de mettre en relief certains points de cette thèse, présentons   quelques résultats énumératifs sur les cartes planaires tétravalentes. Le nombre de cartes planaires tétravalentes à $n$ sommets est 
$$m_n = 2 \, \frac{3^n} {(n+1)(n+2)} {2n \choose n}.$$
Cette très belle formule et sa proximité avec \eqref{catalan} suggèrent l'existence d'une bijection entre les cartes tétravalentes et des objets plus simples, comme des arbres.  C'est le cas ;  nous décrivons d'ailleurs dans le chapitre \ref{c:cartesarbres} de telles bijections. Nous démontrons ainsi de nombreuses formules énumératives. Par exemple, ces bijections peuvent expliquer pourquoi la série génératrice $M(t) = \sum_n m_n t^n$ des cartes tétravalentes satisfait
$$M(t) = T(t) - t \, T(t)^3,$$
où 
$$T(t) = 1 + 3\,t\,T(t)^2.$$
Nous constatons  que la série génératrice des cartes tétravalentes est algébrique. Ceci n'est pas un résultat isolé : les travaux de Bender et Canfield \cite{bender-canfield} montrent que sous des hypothèses raisonnables sur la distribution des degrés, la série génératrice des cartes planaires est algébrique. Dans cette thèse, nous montrerons (chapitre \ref{c:ed}) avec des hypothèses similaires que la série génératrice des cartes planaires munies d'une forêt couvrante est différentiellement algébrique.

Intéressons-nous au comportement asymptotique de $m_n$. D'après la formule de Stirling, il se comporte comme
$$m_n \sim 2.\, 12^n \, n^{-5/2} / \sqrt{\pi}.$$
Il est à noter que cet exposant en $-5/2$ est typique du comportement asymptotique des cartes planaires. Nous le retrouverons à plusieurs reprises au cours de cette thèse. Toutefois nous observerons (et prouverons) des comportements plus étranges, qui à notre connaissance n'ont jamais été vus auparavant dans le monde des cartes. Par exemple, nous montrerons (chapitre \ref{c:asympt}) que le nombre des quadrangulations munies d'une configuration récurrente du modèle du sable avec un nombre minimal de grains se comporte comme $c \, \rho^{-n} n^{-3} (\ln n)^{-2}$.

\section{Polynôme de Tutte et modèles associés}
\label{s:poltutte}

Le polynôme de Tutte constitue la clef de voûte de cette thèse. Dans une première partie, nous allons étudier les cartes planaires munies d'une forêt couvrante, ce qui s'interprète comme une spécialisation du polynôme de Tutte. Dans une seconde partie, nous donnerons de nouvelles descriptions de ce polynôme, plus précisément une nouvelle manière de voir les \textit{activités} qui s'y rattachent.

Cette section introduit le polynôme de Tutte, et deux modèles de physique statistique apparentés : le modèle de Potts et le modèle du tas de sable.

% Dans un premier temps, nous énumérerons les cartes planaires pondérées par une certaine spécialisation du polynôme de Tutte.  Cette pondération peut s'interpréter en terme de forêts couvrantes mais elle a également un interprétation physique en terme du modèle de Potts, que nous présenterons dans cette section.

%Le polynôme de Tutte a été étudié sous deux angles différents durant cette thèse. La première étude concerne l'énumération des cartes planaires pondérées par une certaine spécialisation du polynôme de Tutte.  La seconde étude porte sur une nouvelle description d

\subsection{Polynôme de Tutte}

Voulu dans un premier temps comme une généralisation du polynôme chromatique \cite{whitney,tutte54}, le polynôme de Tutte est un célèbre invariant de graphe jouant un rôle fondamental en combinatoire. Parce qu'il a été étudié de manière intensive et que sa propriété de contraction/suppression est universelle \cite{broxley}, ce polynôme est souvent utilisé comme référence quand il s'agit de relier entre eux les polynômes de graphe à travers divers domaines de recherches. Parmi les exemples que nous ne détaillerons pas dans ce mémoire, citons le polynôme énumérateur des poids en théorie des codes \cite{curtis}, le polynôme de fiabilité en théorie des réseaux \cite{oxley-welsh} et le polynôme de Jones d'un n\oe{}ud alternant en théorie des n\oe{}uds \cite{jonespoly}. Tous peuvent être exprimés comme des spécialisations du polynôme de Tutte.

Passons à la définition de ce polynôme. \'Etant donné un graphe $G$, un \textit{sous-graphe couvrant} de $G$ désigne un graphe ayant le même ensemble de sommets que $G$ et dont l'ensemble des arêtes est inclus dans celui de $G$ (voir figure \ref{tutte1}). Le \textit{nombre cyclomatique} d'un graphe $G$ est le nombre minimal d'arêtes qu'on doit supprimer dans $G$ pour que $G$ devienne acyclique. Il se définit alternativement comme $\cc(G) + \arete(G) -  \som(G)$, où $\cc(G)$ est le nombre de composantes connexes de $G$, $\arete(G)$ son nombre d'arêtes et $\som(G)$ son nombre de sommets. Par exemple, le nombre cyclomatique du graphe de la figure \ref{tutte1} est $2$. Le \textit{polynôme de Tutte} d'un graphe $G$, traditionnellement noté $T_G(x,y)$, est le polynôme bivarié  : 
$$T_G(x,y) = \sum_{\substack{S\textrm{ sous-graphe} \\\textrm{couvrant de }G}} (x-1)^{\cc(S)-\cc(G)} (y-1)^{\cycl(S)}.$$
%où $\cc(.)$ désigne le nombre de composantes connexes et $\cycl(.)$ le nombre cyclomatique. 
Si $G$ comporte des arêtes multiples, alors les sous-graphes couvrants $S$ dans la somme ci-dessus sont comptés avec multiplicité.

\fig{[scale =1.2]{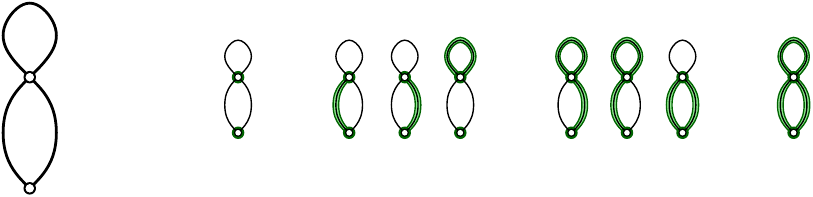}}{Un graphe et ses neuf sous-graphes couvrants, triés selon le nombre d'arêtes.}{tutte1}

Si nous reprenons l'exemple du graphe $G$ de la figure \ref{tutte1},  le polynôme de Tutte vaut
$$T_G(x,y) = (x-1) + 2 + (x-1)(y-1) + 3 (y-1) + (y-1)^2 = xy + y^2.$$  

Propriété nullement manifeste au vu de la première définition, le polynôme de Tutte a des coefficients positifs en $x$ et en $y$. La raison de ceci réside dans la propriété suivante.

\begin{prop} \label{p:eqderec}
Soit $G$ un  graphe ayant au moins une arête $e$. Le polynôme satisfait l'identité de récurrence
$$T_G(x,y) = \left\{ \begin{array}{ll} x \, T_{\contract G e}(x,y) & \textrm{si }e\textrm{ est un isthme}, \\
y \, T_{\delete G e}(x,y) & \textrm{si }e\textrm{ est une boucle}, 
\\
T_{\contract G e}(x,y) + T_{\delete G e}(x,y) & \textrm{si }e\textrm{ n'est ni une boucle, ni une arête}.  \end{array} \right. $$
où $\contract G e$ (resp. $\delete G e$) désigne  le graphe $G$ dans lequel on a contracté (resp. supprimé) l'arête $e$. (Un isthme est une arête telle que la suppression de cette arête augmente de un le nombre de composantes connexes dans le graphe.)
\end{prop}

Ainsi, par une induction quasi-immédiate (dans le cas où $G$ n'a aucune arête,  $T_G = 1$ par définition), nous en déduisons la positivité des coefficients.

\begin{cor} Quel que soit le graphe $G$, les coefficients du polynôme de Tutte de $G$ sont entiers et positifs.
\end{cor}

Naturellement nous pouvons nous demander s'il existe une interprétation combinatoire expliquant ces coefficients entiers positifs. La réponse à cette question est positive et se base sur  la notion d'\textit{activité} d'arbre couvrant\footnote{Historiquement, Tutte s'était servi de cette notion d'activité la première fois qu'il a défini le polynôme qui porte son nom. La définition en termes de séries génératrices de sous-graphes couvrants, qui est plus pratique par bien des aspects, arrivera plus tard.}.  Le polynôme de Tutte d'un graphe $G$ peut alors s'écrire 
\begin{equation} \label{tuttepremiereecriture}
T_G(x,y) = \sum_{T\textrm{ arbre couvrant de }G}\, x^{\inte(T)}\, y^{\ext(T)},
\end{equation}
où $\inte(T)$ et $\ext(T)$ désignent respectivement le nombre d'arêtes internes actives et externes actives.
 Le lecteur se réfèrera à la section \ref{sec:activity} p. \pageref{sec:activity} pour la définition des arêtes actives (un peu longue pour être introduite ici).

Remarquons que les variables $x$ et $y$ jouent un rôle symétrique dans le cas planaire. En effet, pour deux cartes duales $C$ et $C^*$, nous pouvons montrer que
\begin{equation}
T_{C^*}(x,y) = T_C(y,x).
\label{tuttedual}
\end{equation}

De nombreuses informations quantitatives sur le graphe émanent du polynôme de Tutte et la plupart se déduisent directement de sa définition. Par exemple, quand $G$ est connexe, le nombre d'arbres couvrants de $G$ est compté par $T_G(1,1)$, le nombre de sous-graphes couvrants connexes  par $T_G(1,2)$ et -- ce qui nous intéressera particulièrement dans le cadre de cette thèse -- le nombre de forêts couvrantes par $T_G(2,1)$. Le \textit{polynôme chromatique} (i.e. le nombre de coloriages propres en $q$ couleurs) est également une spécialisation du polynôme de Tutte\footnote{Tutte a inventé ce polynôme pour être une généralisation à deux variables du polynôme chromatique. D'ailleurs, la terminologie originelle était "polynôme dichromatique".} :
$$\chi_G(q) = (-1)^{\som(G)} q^{\cc(G)} T_G(1-q,0),$$
$\cc(G)$ désignant le nombre de composantes de $G$ et $\som(G)$ son nombre de sommets.
Cette identité se prouve par induction en utilisant l'équation de récurrence de la proposition \ref{p:eqderec}, comme la plupart des identités impliquant le polynôme de Tutte.

\subsection{Le modèle de Potts}
\label{ss:potts}

Le lien établi ci-dessus avec le polynôme chromatique peut être généralisé pour couvrir le polynôme de Tutte dans sa globalité. En effet, à un changement affine de variables près \cite{fk}, le polynôme de Tutte compte tous les coloriages de $G$ en $q$ couleurs, non nécessairement propres, avec un poids $\nu$ par arête monochrome. Plus précisément, pour tout entier naturel $q$ de la forme $q = (\mu -1)(\nu - 1)$, 
$$\pare{\mu - 1}^{\cc(G)} \, \pare{\nu - 1}^{\som(G)} \, T_G(\mu,\nu) = P_G(q,\nu),$$
avec 
$$P_G(q,\nu) = \sum_{\substack{c \textrm{ coloriage}\\\textrm{à }q\textrm{ couleurs}}}  \nu^{m(c)},$$
où $m(c)$ désigne le nombre d'arêtes monochromes de $G$.

Ce polynôme $P_G(q,\nu)$ est très étudié en physique statistique \cite{wu} : il s'agit de la \textit{fonction de partition du modèle de Potts à $q$ états} sur le graphe $G$. Dans ce modèle, des particules occupant les sommets possèdent des états  indexés par $\ens{1,2,\dots,q}$. La probabilité d'une configuration $c$ dépend du nombre d'arêtes monochromes (ici une arête monochrome est une arête reliant deux particules de même état) : elle est proportionnelle à 
$ \exp(K \, m(c))$
où $K$ est une paramètre et $m(c)$ est le nombre d'arêtes monochromes. Lorsque $q = 2$, nous retrouvons le célèbre \textit{modèle d'Ising}.

Le modèle de Potts s'étudie plus naturellement sur des réseaux réguliers et sur des cartes planaires que sur des graphes sans structure, notamment pour des raisons physiques. Les phénomènes qu'on cherche à comprendre se produisent sur des surfaces (ou des volumes) avec un très grand nombre de particules. Dans le cadre de la gravitation quantique, un  bon modèle pour étudier ces comportements macroscopiques est de considérer une carte planaire aléatoire munie d'une coloration de ses sommets, de faire tendre la taille de cette carte vers l'infini et d'observer son comportement asymptotique. Il s'agit notamment de comprendre les phénomènes de \textit{transitions de phase} qui apparaissent, c'est-à-dire des changements abrupts (non analytiques) pour certaines valeurs des paramètres $q$ et $\nu$. 

Dans ce contexte, nous aimerions étudier la série génératrice des cartes planaires pondérées par $P_G(q,\nu)$. Beaucoup de travaux ont été accomplis  dans ce sens  \cite{baxter82,daul,DF-Eynard-Guitter,eynard-bonnet-potts,zinn-justin-dilute-potts} -- entre autres, Olivier Bernardi et Mireille Bousquet-Mélou ont montré  l'algébricité différentielle de cette série \cite{bernardi-mbm-de} -- mais l'étude de cette série est très difficile de manière générique. Néanmoins de nombreux cas particuliers ont été résolus. Citons par exemple l'étude des cartes munies d'un arbre couvrant réalisée par Mullin \cite{mullin-boisees} et celle du modèle d'Ising, d'abord résolu en utilisant des  intégrales de matrice \cite{boulatov,BDG-hard-bip-matrix}, puis de manière bijective \cite{mbm-schaeffer}. Dans  cette thèse, nous travaillerons  avec une spécialisation de la fonction de partition du modèle de Potts qui revient à dénombrer les cartes planaires munies d'une forêt couvrante.

\subsection{Modèle du tas de sable}
\label{ss:sable}

Un autre modèle de physique statistique lié au polynôme de Tutte est celui du tas de sable \cite{bak,dhar}. Il est également connu en combinatoire sous la forme d'un jeu à base de jetons dont le nom anglais est \textit{chip firing game} \cite{bjorner}. Dans ce modèle, des grains de sable sont disposés sur les sommets d'un graphe.  Lorsque le nombre de grains sur un même sommet excède le degré de ce sommet, nous assistons à \textit{l'éboulement} de ce sommet, c'est-à-dire que ce sommet perd autant de grains que son degré et chacun de ses voisins en gagne un.

Passons à des définitions plus précises. Nous considérons $G$ un graphe connexe avec un sommet distingué $s_0$ appelé \textit{puits} (si le graphe considéré est une carte enracinée, le sommet racine est naturellement désigné comme le puits). Une \textit{configuration} du tas de sable est une fonction $\gamma$ de l'ensemble des sommets de $G$  vers $\N$ : pour un sommet $s$, l'entier $\gamma(s)$ représente le nombre de grains sur $s$. Un sommet $s$ est \textit{instable} pour une configuration $\gamma$ si $\gamma(s)$ est supérieur ou égal au degré de ce sommet. Quand un sommet $s$ est instable, nous pouvons procéder à son \textit{éboulement} : cela mène à une nouvelle configuration $\gamma'$ où $\gamma'(t) = \gamma(t) + \deg(s,t)$ pour $s \neq t$ et $\gamma'(s) = \gamma(s) - \deg_*(s)$, où $\deg(s,t)$ désigne le nombre d'arêtes d'extrémités $s$ et $t$, et $\deg_*(s)$ le nombre d'arêtes non boucle incidentes à $s$. La figure \ref{tasdesable} montre des exemples d'éboulements.

\fig{[width=\textwidth]{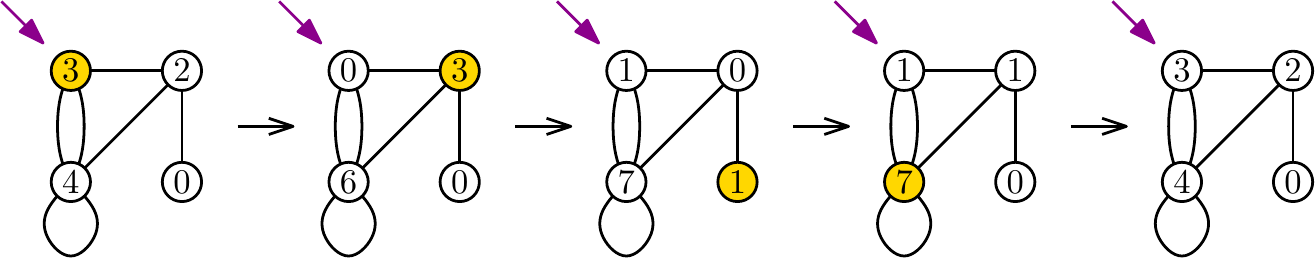}}{Une configuration récurrente avec une suite d'éboulements valide.}{tasdesable}

Une configuration est \textit{stable} si aucun sommet à part éventuellement le puits n'est instable. Une configuration stable est \textit{récurrente}\footnote{Comme ce nom le suggère, se cache derrière le modèle du tas de sable une chaîne de Markov. Nous ne la détaillerons pas ici.} si le nombre de grains sur le puits est égal à son degré et s'il existe une permutation $s_0, s_1, \dots, s_{n -1}$ de l'ensemble des sommets de $G$  telle qu'il est possible d'ébouler successivement $s_0, s_1, \dots, s_{n -1}$ dans cet ordre. Remarquons que si on éboule chaque sommet d'un graphe, on revient forcément à la configuration initiale. Par exemple, la première configuration  de la figure \ref{tasdesable} est récurrente, comme le prouvent les éboulements qui suivent. 

On peut montrer que le nombre total  de grains dans une configuration récurrente ne peut être inférieur au nombre d'arêtes dans le graphe. Le \textit{niveau} $\ell(\gamma)$ d'une configuration récurrente $\gamma$ se définit alors comme le nombre de grains normalisé :
$$\ell(\gamma) = \pare{\sum_{s\textrm{ sommet de }G} \gamma(s) } - \arete(G),$$
avec $\arete(G)$ le nombre d'arêtes dans le graphe $G$. Pour revenir à notre exemple, le niveau de la configuration récurrente de la figure \ref{tasdesable} vaut $3$.

Merino-L\'opez a prouvé \cite{merino} que le polynôme de Tutte d'un graphe $G$ en $x = 1$ compte les configurations récurrentes de $G$ selon leur niveau :
$$T_G(1,y) = \sum_{\substack{\gamma \textrm{ configuration}\\\textrm{récurrente de }G}} y^{\ell(\gamma)}.$$
Il existe plusieurs preuves bijectives de cette identité, comme celles apportées par Robert Cori et Yvan Le Borgne \cite{cori-borgne} ou Olivier Bernardi \cite{bernardisandpile}.

\section{Organisation de ce mémoire}

%Avant de , nous tenions à commencer par un chapitre qui introduit les bijections entre cartes planaires et certains arbres plans.

Avant d'introduire les résultats trouvés pendant cette thèse, nous commençons par un chapitre introductif (\textbf{chapitre \ref{c:cartesarbres}}) qui présente quelques bijections connues entre cartes planaires et arbres. D'une part, cela permet d'introduire la notion d'\textit{arbre bourgeonnant} \cite{Sch97,bdg2002} qui sera récurrente dans ce mémoire. D'autre part, nous donnons quelques résultats énumératifs sur les cartes planaires, que nous utiliserons tout au long de la première partie.

\subsection*{Partie I : cartes forestières.}

Nous étudions dans cette partie les  \textit{cartes forestières}, c'est-à-dire les cartes planaires munies d'une forêt couvrante, et plus particulièrement leur série génératrice $F(z,u,t)$ (avec prescription sur les degrés des sommets) où $z$ compte les faces, $u$ les composantes connexes de la forêt, et $t$ les arêtes.  De manière équivalente, nous énumérons les cartes planaires munies d'un arbre couvrant avec un poids $z$ par face, un poids $u+1$ par arête interne active (au sens de Tutte) et un poids $t$ par arête. En termes du modèle de Potts, ce problème énumératif revient à faire tendre le nombre de couleurs $q$ vers $0$ (d'une manière bien précise). Ce problème a été déjà étudié dans \cite{sportiello} par une approche matricielle, mais sans aucune explicite.

Par une approche purement combinatoire directement inspiré de \cite{BDG-blocked}, nous décrivons un système fonctionnel satisfait par la série génératrice $F$. Nous en déduisons que cette série est différentiellement algébrique  selon chacune de ses trois variables $z$, $u$ et $t$ (sous quelques hypothèses sur les degrés prescrits, pas trop contraignantes). Puis, pour $ u \geq -1$ fixé et $t = 1$, nous étudions les singularités de $F$ et le comportement asymptotique de ses coefficients. Pour $u >0$, nous retrouvons le régime asymptotique standard des cartes planaires avec 
un facteur sous-exponentiel en $n^{-5/2}$. En $u=0$, nous constatons  une transition  de phase avec un facteur $n^{-3}$. Lorsque $u \in [-1,0[$, nous obtenons un comportement extrêmement inhabituel en $n^{-3} \pare{\ln n}^{-2}$ : c'est la première fois (à notre connaissance) que nous observons une telle "classe d'universalité" pour les cartes planaires. En particulier, cela montre que $F(z,u,t)$ n'est pas holonome.

Précisons plus en détail comment sont articulés les différents chapitres de cette partie. Le premier chapitre  (\textbf{chapitre \ref{c:cartesforestieres}}) introduit la série génératrice $F(z,u,t)$ des cartes forestières. Nous présentons dans un premier temps les nombreuses interprétations combinatoires qui y sont liées (polynôme de Tutte, modèle de Potts, modèle du tas de sable...). Puis par une approche empruntée à \cite{BDG-blocked}, nous exprimons $\pd F z$ en fonction de deux séries $R$ et $S$ qui sont solutions d'un système fonctionnel à deux inconnues. Dans le \textbf{chapitre \ref{c:ed}}, nous montrons grâce à ce système fonctionnel que $F$ est différentiellement algébrique. En réalité, le résultat que nous prouverons est plus général que cela ; nous démontrerons l'algébricité différentielle pour des cartes dites \textit{décorées} par des objets combinatoires qui admettent une série génératrice holonome. 

Le reste de la première partie est essentiellement consacré à l'étude du comportement asymptotique des coefficients de $F$. Comme l'analyse du système est plutôt difficile, notamment quand $u \in [-1,0[$, nous aurons besoin de nombreux lemmes techniques (mais intéressants en soi) pour arriver au bout de certaines preuves. Dans le \textbf{chapitre \ref{c:enrichis}}, nous expliquons comment  interpréter combinatoirement les séries $R$ et $S$ susmentionnées (et autres séries apparentées) lorsque $u \in [-1,+\infty[$, notamment en termes d'arbres bourgeonnants. Ces interprétations combinatoires font intervenir des séries à coefficients positifs, m\^eme pour $u \in [-1,0[$.
% Des résultats qui semblaient difficilement accessibles d'un point de vue purement analytique seront alors montrés. 
 Mentionnons également que ce chapitre comporte une description d'une bijection entre cartes forestières et arbres bourgeonnants munis d'une forêt couvrante qui permet de faire le lien avec la seconde partie. Dans le  \textbf{chapitre \ref{c:implicites}}, nous établissons quelques résultats d'analyse complexe concernant les fonctions implicites. 

Dans les chapitres qui suivent, nous déterminons le comportement asymptotique du nombre (pondéré) de cartes forestières, pour plusieurs classes de cartes.  Le \textbf{chapitre \ref{c:asympt}} traite des cartes forestières \textit{4-eulériennes et apériodiques} (toujours avec prescription des degrés), c'est-à-dire où le degré de chaque sommet est pair, différent de $2$, et qu'il existe deux cartes dont le nombre de faces ne diffère que d'un.  Ces conditions, qui sont satisfaites par exemple par les cartes tétravalentes, facilitent l'analyse du comportement asymptotique, notamment grâce à des équations plus simples (on n'a qu'une seule inconnue : la série $S$ vaut $0$). Dans le \textbf{chapitre \ref{c:2qreg}}, nous nous intéressons aux cartes forestières $2q$-régulières (tous les sommets ont pour degré $2q$). La nouvelle difficulté  est la présence d'une période supérieure à $1$, ce que nous avions exclu par hypothèse dans le chapitre précédent. Enfin, nous étudions dans le \textbf{chapitre \ref{c:cubique}} les cartes forestières cubiques  (tous les sommets ont pour degré $3$). Ici le système fait bien intervenir deux inconnues $R$ et $S$. Son analyse est  délicate, même dans le cas où $u$ est positif. Tous ces résultats se basent sur les théorèmes de combinatoire analytique  de \cite{flajolet-sedgewick} (voir sous-section \ref{ss:asympt}).

Nous terminons cette partie par le \textbf{chapitre \ref{c:perspectives}} qui ouvre sur de nouvelles perspectives. Nous y introduisons la notion de \textit{cartes animalières}, c'est-à-dire des cartes munies d'un sous-graphe connexe non nécessairement couvrant.  Ces objets sont en lien avec la \textit{percolation}, modèle de physique statistique décrivant  le passage d'un fluide à travers un milieu plus ou moins perméable. D'un point de vue combinatoire, les cartes animalières sont l'extension naturelle des cartes forestières. Nous prouvons quelques résultats susceptibles d'être prolongés dans des travaux futurs.

\subsection*{Partie II : une notion générale d'activité pour le polynôme de Tutte.}

La littérature recense plusieurs descriptions possibles pour le polynôme de Tutte d'un graphe. La première définition due à Tutte se base sur une notion d'\textit{activité} qui nécessite d'ordonner les arêtes du graphe. Quelques décennies plus tard, Olivier Bernardi a décrit une notion différente de l'activité en plongeant le graphe sur une surface. Dans cette partie, nous expliquons comment d'autres notions d'activité peuvent être définies et comment elles peuvent être regroupées au sein d'une même famille d'activités, appelées $\Delta$-activités. Nous développons ainsi une modeste théorie qui éclaire les relations entre les différentes expressions du polynôme de Tutte.

Voici comment est organisée cette seconde partie. Le \textbf{chapitre   \ref{c:activities}} introduit de nombreuses définitions qui nous seront utiles par la suite, dont celle de l'\textit{activité} : une \textit{activité} désigne ici une application qui associe à tout arbre couvrant un sous-ensemble d'arêtes (alors appelées \textit{arêtes actives}). Pour un graphe $G$ connexe, on dit qu'une activité $\psi$ \textit{décrit le polynôme de Tutte} si elle satisfait
$$T_G(x,y) = \sum_{T\textrm{ arbre couvrant de }G} \, x^{\module{\psi(T)\cap T}} \, y^{\module{\psi(T) \backslash T}},$$
où $T_G$ est le polynôme de Tutte de $G$ et  $\module{\psi(T)\cap T}$ (resp. $\module{\psi(T) \backslash T}$) est le nombre d'arêtes actives qui appartiennent (resp. qui n'appartient pas) à l'arbre couvrant $T$.
Toujours dans le même chapitre, nous définissons trois familles d'activités qui décrivent le polynôme de Tutte, parmi lesquelles se trouvent celles de Tutte et Bernardi.

Dans le \textbf{chapitre \ref{c:delta}}, nous introduisons la notion de $\Delta$-activité. Dans un premier temps, nous définissons cette notion via un algorithme. Notons que ce dernier comporte de nombreuses similarités avec l'algorithme décrit dans \cite{gordon-traldi}\footnote{Dans cet article, Gordon et Traldi indiquent comment retrouver la notion d'activité au sens de Tutte, pour un matroïde $M$, par une "réduction" arête par arête de $M$. Notre approche ajoute un degré de liberté dans cette réduction.}. Nous prouvons que ces $\Delta$-activités décrivent  le polynôme de Tutte. Puis, nous donnons une nouvelle description des $\Delta$-activités en termes d'ordres totaux sur les arêtes. En quelques mots, cette définition dit qu'une arête est active si elle est maximale (pour un certain ordre dépendant de l'arbre couvrant) dans son \textit{cycle/cocycle fondamental}.
Dans le \textbf{chapitre \ref{c:partition}}, nous expliquons comment  définir une partition de l'ensemble des sous-graphes à partir d'une $\Delta$-activité. Nous étendons ainsi les résultats de Crapo \cite{crapo} aux \mbox{$\Delta$-activités}. Nous donnons à travers cette partition une interprétation combinatoire des différents liens entre les descriptions du polynôme de Tutte.

Dans le \textbf{chapitre \ref{c:spec}}, nous montrons que les trois familles d'activité décrites dans le chapitre \ref{c:activities} peuvent être définies en termes de $\Delta$-activités. 
%Comme attendu, ces familles sont constituées d'activités qui décrivent le polynôme de Tutte. 
Le \textbf{chapitre \ref{c:bloact}} introduit une quatrième famille de $\Delta$-activités, appelées \textit{activités bourgeonnantes}. Sa définition s'inspire de la transformation d'une carte forestière en arbre bourgeonnant munie d'une forêt couvrante, décrite dans la chapitre \ref{c:enrichis}. Dans le contexte des cartes, les activités bourgeonnantes peuvent constituer une alternative intéressante aux activités de Bernardi.

Enfin, l'ultime chapitre (\textbf{chapitre \ref{s:com}}) donne quelques extensions des $\Delta$-activités, comme la généralisation de cette notion aux matroïdes. Nous concluons ce mémoire par une conjecture qui justifierait à elle seule l'intérêt des $\Delta$-activités. Grossièrement, elle stipule que les activités qui décrivent le polynôme de Tutte avec une propriété \textit{à la Crapo} coïncident précisément avec les $\Delta$-activités.

\chapter{Bijections entre cartes planaires et arbres}
\label{c:cartesarbres}

Le but de ce chapitre est de présenter les bijections (connues) entre cartes planaires et arbres. La structure  récursive des arbres  nous permettra de trouver des formules énumératives qui nous seront utiles par la suite.

\section{Des cartes planaires aux arbres bourgeonnants}
\label{s:bourgeonnants}

Dans cette partie, nous allons présenter quelques notions et résultats issus de l'article \cite{bdg2002} écrit par Jérémie Bouttier, Philippe Di Francesco et Emmanuel Guitter. Ce papier généralise la bijection \cite{Sch97} de Gilles Schaeffer entre cartes eulériennes et \textit{arbres bourgeonnants} (eulériens). L'idée de ces travaux est qu'une carte planaire peut être transformée en un arbre plan (avec certaines contraintes) sans changer la distribution des degrés des sommets. 

Commençons par définir les arbres en question. Un arbre \textit{bourgeonnant} est un arbre plan enraciné avec des demi-arêtes qui pendent sur certains sommets. Ces demi-arêtes sont de trois types différents : la \textit{racine}, qui est unique et représentée par un râteau, les \textit{bourgeons}, portant chacun une \textit{charge} $-1$ et représentés par des flèches blanches, et les \textit{feuilles}, portant chacune une \textit{charge} $+1$ et représentées par des flèches noires\footnote{Les couleurs des feuilles et des bourgeons ont été échangées entre \cite{Sch97} et \cite{bdg2002}. Nous reprenons la représentation originelle de Gilles Schaeffer : les bourgeons sont blancs et les feuilles noires.} (voir la figure \ref{exRS}). La \textit{charge} d'un arbre bourgeonnant est alors la différence entre le nombre de feuilles et le nombre de bourgeons dans cet arbre. En coupant en deux une arête\footnote{Avertissement pour la suite : un bourgeon n'est pas une arête, mais une demi-arête. La notion de sous-arbre ne peut donc pas s'appliquer à un bourgeon seul.} d'un arbre bourgeonnant, on obtient deux arbres : celui qui ne contient pas la racine est un \textit{sous-arbre}. Un sous-arbre est naturellement enraciné sur la demi-arête où il y a eu coupure. 

\begin{defi}  Un arbre bourgeonnant est appelé \textbf{R-arbre} (resp. \textbf{S-arbre}) quand :
\begin{enumerate}
\item[(i)] sa charge vaut $1$ (resp. $0$),
\item[(ii)] tout sous-arbre a pour charge $0$ ou $1$.
\end{enumerate}
\end{defi}

\fig{[width=\textwidth]{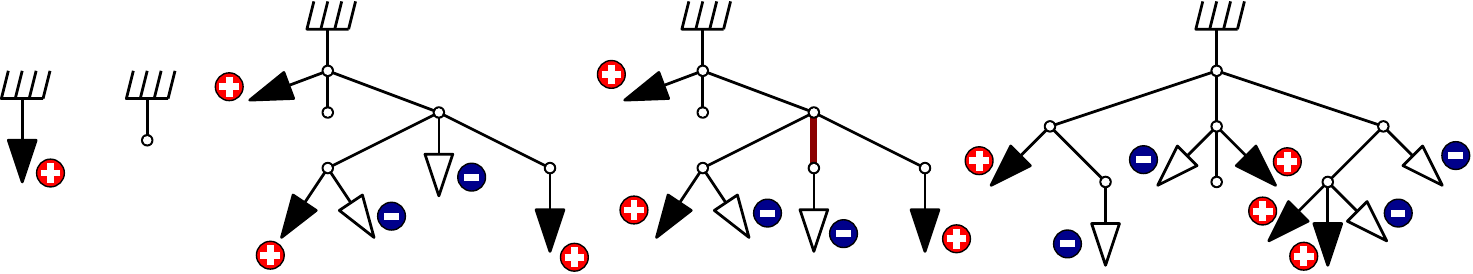}}{De gauche à droite : le plus petit R-arbre (une feuille), le plus petit S-arbre (un sommet de degré $1$ incident à une racine), un R-arbre, un arbre bourgeonnant qui \textit{n'est pas} un R-arbre (car le sous-arbre associé à l'arête en trait gras est de charge $-1$), un S-arbre.}{exRS}

Par convention, on considèrera qu'une feuille simple (sans sommet) est également un R-arbre. La figure \ref{exRS} donne des exemples de R- et S-arbres.

Les R- et les S-arbres admettent clairement une propriété d'hérédité : tout sous-arbre d'un R-arbre ou d'un S-arbre est lui-même un R-arbre ou un S-arbre. Cette première observation va nous permettre de décrire les séries génératrices desdits objets.  On note $R(z,u)$ (resp. $S(z,u)$) la série génératrice des R-arbres (resp. S-arbres) comptés selon le nombre de feuilles (variable $z$) et le nombre de sommets (variable $u$). De plus, afin que l'énumération puisse s'adapter à différentes classes de cartes planaires,\textbf{ chaque sommet de degré $\boldsymbol k$ est pondéré par une variable $\boldsymbol {g_k}$}. Par souci de concision, ces variables de poids n'apparaîtront pas dans les variables de nos séries génératrices.

\begin{prop} \label{heqRS}
Les séries génératrices $R(z,u)$ et $S(z,u)$ des R- et S-arbres satisfont $R(0,u) = S(0,u) = 0$ et le système d'équations fonctionnelles
\begin{equation}
R = z + u \sum_{i \geq 1} \sum_{j \geq 0} g_{2i+j} { 2i+j-1 \choose i-1,i,j } R^i S^j, 
\label{equerre}
\end{equation}
\begin{equation}
 S = u \sum_{i \geq 0} \sum_{j \geq 0} g_{2i+j+1} { 2i+j \choose i,i,j } R^i S^j,
 \label{equesse}
\end{equation}
où ${ a+b+c \choose a,b,c }$ désigne le coefficient trinomial $(a+b+c)!/(a!b!c!)$.
\end{prop}

\begin{proof} Un R-arbre est soit réduit à une feuille (contribution : $z$), soit contient un sommet racine\footnote{Cela veut dire incident à la racine.}. Dans le second cas, le sommet racine est incident à plusieurs demi-arêtes (dont la racine, et potentiellement des bourgeons et des feuilles) et à des sous-arbres. D'après la propriété d'hérédité décrite plus haut, chacun de ces sous-arbres est soit un R-arbre (non réduit à une feuille), soit un S-arbre. Puisqu'une feuille est également un R-arbre, on peut dire que le sommet racine est uniquement incident à la racine, à des bourgeons, à des R-arbres et à des S-arbres. Si on compte $i$ R-arbres attachés au sommet racine, alors exactement $i-1$ bourgeons sont attachés au sommet racine car la charge totale de l'arbre doit être égale à $1$. Le nombre $j$ de S-arbres est quant à lui arbitraire. Le sommet racine a alors pour degré $1 + (i-1) + i + j = 2i+j$ et est donc pondéré par $u g_{2i+j}$. On retrouve donc \eqref{equerre}, où le coefficient trinomial provient du nombre de mélanges possibles entre bourgeons, R-arbres et S-arbres.

L'équation \eqref{equesse} se prouve de manière similaire. (Dans ce cas-là, le nombre de R-arbres et le nombre de bourgeons attachés au sommet racine sont les mêmes.)
\end{proof}

Les séries génératrices $R$ et $S$ sont en réalité caractérisées par les équations de la proposition précédente, comme le stipule le lemme qui suit. 

\begin{lem} \label{caracteresse}
Considérons une suite $(g_i)_{i \geq 0}$ à valeurs dans un corps $\mathbb K$.
Il existe un unique couple $(R,S)$ de séries formelles bivariées en $z$ et $u$, à coefficients dans $\mathbb K$, qui satisfont les équations \eqref{equerre} et \eqref{equesse}.

\end{lem}
\begin{proof} Il suffit de prouver l'unicité du couple puisque l'existence est garantie par la proposition \ref{heqRS}. Plus précisément, nous montrons par récurrence sur $a+b$ que les coefficients de $z^a\,u^b$ sont déterminés de manière unique par les équations \eqref{equerre} et \eqref{equesse}. Le cas de base $a=b=0$ se résout en substituant $z$ et $u$ par $0$ dans \eqref{equerre} et \eqref{equesse} : on voit alors que les coefficients constants de $R$ et $S$ sont nuls. Pour la propriété d'hérédité, il n'est pas difficile de voir que si on connaît pour chaque $(a,b)$ avec $a+b < n$, le coefficient de $z^a \, u^b$ dans $R$ et $S$, alors pour chaque  $(a,b)$ avec $a+b = n$, l'équation \eqref{equerre} détermine le coefficient de $z^a \,  u^b$ dans $R$   et l'équation \eqref{equesse} celui dans $S$.  \end{proof}

Calculons les premiers termes de $R(z,u)$ et de $S(z,u)$ :
$$
R = z + g_2 z u +  (g_2^2 + 2 g_1 g_3) z u^2 + 3 g_4 u z^2 + O\pare{(z+u)^4},
$$
$$
S = g_1 u + g_1 g_2 u^2 + 2 g_3 z u + (g_1^2 g_3 + g_1 g_2^2) u^3 + (6 g_1 g_4 + 4 g_2 g_3) z u^2 + 6 g_5 u z^2 + O\pare{(z+u)^4}.
$$
(La notation $O\pare{(z+u)^4}$ signifie que les  termes  manquants sont en $z^a u^b$ avec $a+b \geq 4$.)
Nous pouvons remarquer sur ces petites valeurs qu'un S-arbre comporte forcément un sommet de degré impair. Cette propriété est vraie en toute généralité : le nombre de demi-arêtes accrochées à un S-arbre (en incluant la racine) est impair, ce qui ne peut arriver si l'arbre n'a pas de sommets de degré impair. Si on raisonne sur les séries génératrices, cela signifie que si $g_{2k+1}=0$  pour tout $k \geq 0$ alors $S=0$.

Nous allons maintenant voir que les arbres bourgeonnants constituent  le squelette des cartes planaires. Plus précisément, nous allons nous intéresser aux cartes planaires non plus enracinées sur un coin, mais sur une demi-arête (voir figure \ref{cg1}) et à leur série génératrice $\Gamma_1$. Plus que le résultat en lui-même, ce qui nous intéressera ici sera la méthode pour y arriver.

\fig{[scale=1]{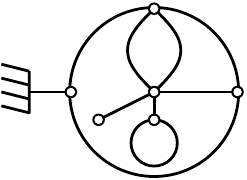}}{Une carte enracinée sur une demi-arête.}{cg1}

\begin{theo} \label{g1rs}
Soit $\Gamma_1(z,u)$ la série génératrice des cartes planaires enracinées sur une demi-arête, avec un poids $z$ par face, un poids $u$ par sommet et un poids $g_k$ par sommet de degré $k$. On a :
$$
\frac {\partial \Gamma_1} {\partial z}(z,u) = S(z,u).
$$ 
%De plus, on a la relation suivante, où aucune dérivée n'apparaît :
%\begin{equation}
%\end{equation}
%où $R$ est la série génératrice des R-arbres.
\end{theo}

\begin{proof} \textbf{1. Définition des $\boldsymbol  S$-arbres équilibrés.} Nous commençons par associer à chaque S-arbre un chemin  constitué de pas montants et de pas descendants : partant de la racine, nous faisons le tour du S-arbre dans le sens trigonométrique. Nous notons la visite d'un bourgeon par un pas montant et la visite d'une feuille par un pas descendant. Une fois revenus à la racine, nous obtenons un chemin qui comporte autant de pas montants que de pas descendants. Un exemple est montré sur la figure \ref{conj1}.

\fig{[scale=1]{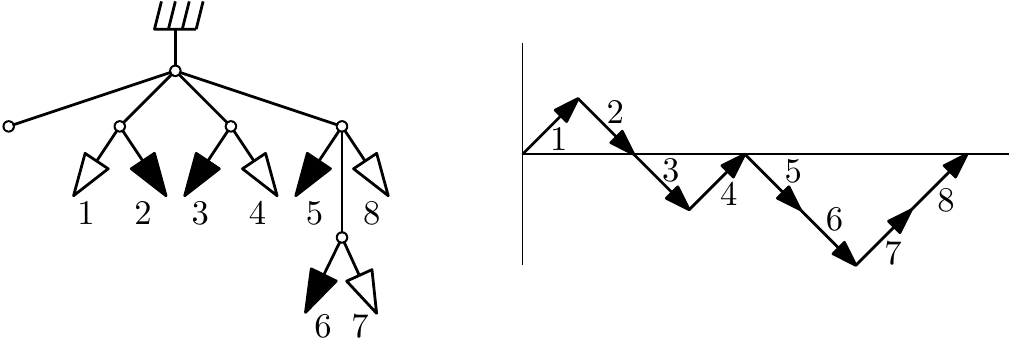}}{Un S-arbre \textit{non }équilibré et le chemin associé.}{conj1}

Un S-arbre est dit \textit{équilibré} lorsque le chemin en question est un chemin de Dyck (voir la figure \ref{conj2}). On appelle $E(z)$ la série génératrice des arbres $S$-équilibrés. 

\textbf{2. Montrons la relation dite de \textit{conjugaison} $\boldsymbol {(zE)'=S}$.} Dans chacun des S-arbres, remplaçons la racine par une feuille de sorte que la charge totale de l'arbre soit maintenant égale à $1$. Autrement dit, les S-arbres sont maintenant enracinés sur une feuille.

\fig{[scale=1]{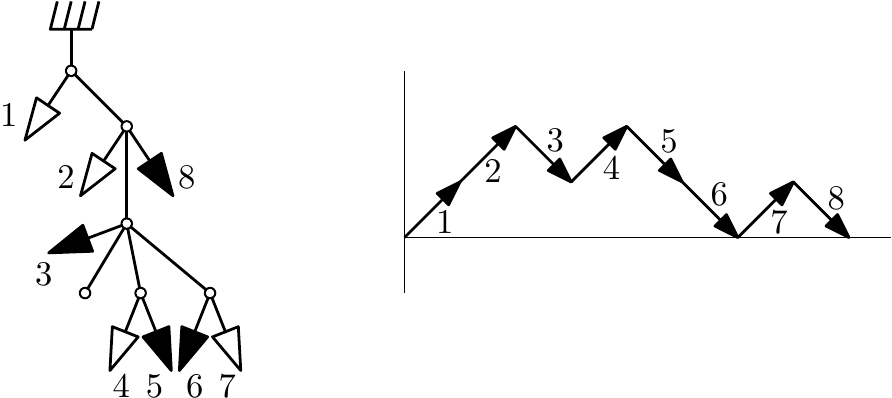}}{Un S-arbre équilibré, conjugué au S-arbre de la figure \ref{conj1}, et le chemin (de Dyck) associé.}{conj2}

Soit $\tau_1$ un S-arbre. Distinguons une deuxième feuille dans l'arbre que nous prenons comme racine, puis oublions la première feuille. Nous obtenons un nouvel arbre bourgeonnant $\tau_2$. Il n'est pas difficile de montrer que $\tau_2$ est un S-arbre. 
%Quand nous coupons en deux une arête de $\tau_1$, nous obtenons deux arbres de charge respective $0$ et $1$, car $\tau_1$ est un S-arbre. Par conséquent, cela est également vrai pour $\tau_2$. Ainsi $\tau_2$ est un S-arbre.  
On dit que les deux S-arbres $\tau_1$ et $\tau_2$ sont \textit{conjugués}. Les arbres des figures \ref{conj1} et \ref{conj2} constituent un exemple de S-arbres conjugués. (La feuille distinguée du deuxième arbre est la feuille numéro 6 du premier.)

Montrons que tout S-arbre $\tau$ est conjugué à un unique arbre équilibré. Pour rester cohérent avec la nouvelle convention qui consiste à remplacer la racine par une feuille marquée, nous rajoutons un pas descendant à la fin du chemin associé à $\tau$. Ainsi la différence entre nombre de pas descendants et de pas montants dans ce chemin est de $1$. On utilise alors le \textit{lemme cyclique} \cite{cycliclemma} : pour tout chemin $C$, il existe une unique permutation circulaire de $C$ tel que le chemin obtenu soit un chemin de Dyck auquel on a ajouté un pas descendant à la fin. En terme de S-arbre, cela veut dire qu'il existe une unique feuille de $\tau$ sur laquelle l'enracinement mène à un S-arbre équilibré. On prouve de cette manière que $\tau$ est conjugué à un unique S-arbre équilibré.

Or, tout S-arbre peut être réenraciné de $(n+1)$ manières différentes\footnote{Il y a une petite subtilité à vérifier ici : deux enracinements sur des feuilles différentes doivent mener à des S-arbres différents -- c'est une conséquence du lemme cyclique.} où $n$ est le nombre de feuilles non racine. Par conséquent l'ensemble des S-arbres avec $n$ feuilles non racine peut être partitionné en classes de conjugaison de cardinal $n+1$ possédant chacune un unique S-arbre équilibré. En termes de séries génératrices, cela signifie que $(zE)'=S$. (On rappelle que $z$ ne compte pas la feuille racine, ni pour $E$, ni pour $S$.)

\textbf{3. Transformons une carte $\boldsymbol C$ enracinée  sur une demi-arête en un arbre bourgeonnant.} La figure \ref{trans} illustre les étapes de cette transformation. Commençant par la racine, nous marchons dans le sens trigonométrique le long des arêtes de la face incidente à la racine. Après avoir longé une arête qui n'est pas un isthme\footnote{Un isthme est une arête dont la suppression déconnecte le graphe.}, nous coupons cette arête en deux. La première demi-arête est un bourgeon, la deuxième demi-arête une feuille. Nous nous arrêtons quand il ne reste plus qu'une seule face : nous avons obtenu notre arbre bourgeonnant. 

Remarquons que les faces non racine de $C$ sont en correspondance avec les feuilles de l'arbre bourgeonnant. En effet, à chaque fois qu'une arête est coupée, une feuille se forme et une face non racine est détruite.

\fig{[width=\textwidth]{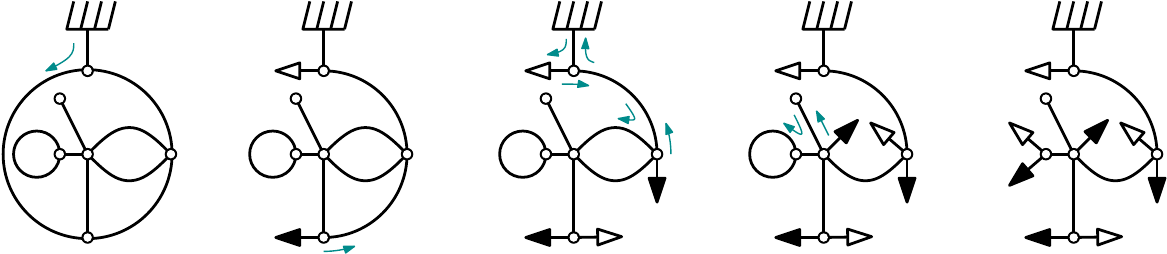}}{Transformation d'une carte enracinée sur une demi-arête en un S-arbre bourgeonnant équilibré.}{trans}

\textbf{4. L'arbre bourgeonnant obtenu est un $\boldsymbol S$-arbre équilibré.} Appelons $\tau_C$ l'arbre bourgeonnant dérivant de la carte $C$. Il est clairement de charge totale nulle. Considérons une arête de $\tau_C$. On veut montrer que le sous-arbre associé à cette arête dans $\tau_C$ a une charge $0$ ou $1$. Appelons $E_2$ l'ensemble des sommets de cet arbre et $E_1$ son complémentaire dans $\Som C$.
Nous pouvons facilement voir que seules les coupures d'arêtes qui ont une extrémité dans $E_1$, l'autre dans $E_2$, peuvent changer la charge totale de $E_2$. Néanmoins les passages de $E_1$ vers $E_2$ et de $E_2$ vers $E_1$ alternent, de sorte que les charges induites par deux coupures lors de deux passages consécutifs se compensent. La charge de $E_2$ ne dépend donc que de la parité du nombre de passages entre $E_1$ et $E_2$ avec coupure d'arête : s'il est pair, la charge de $E_2$ vaut $0$, sinon elle vaut $1$. Quelle que soit la parité, on a bien montré que $\tau_C$ est un S-arbre. 

En outre, si nous regardons comment se construit le chemin associé à $\tau_C$ à partir du chemin vide, nous observons que couper une arête revient à insérer un pas montant et un pas descendant accolés quelque part dans le chemin, qui de fait reste positif. L'arbre $\tau_C$ est bien un S-arbre équilibré.

\textbf{5. Montrons que $\boldsymbol{ \Gamma_1=z\,E}$.} Remarquons que $C$ peut être facilement retrouvée à partir de $\tau_C$ ; il suffit de relier par une arête tout couple de bourgeon et de feuille qui correspondent à deux pas qui se font face dans le chemin associé à $\tau_C$ (voir la figure \ref{appareil}).  Ainsi, l'application $C \mapsto \tau_C$ est une bijection entre les cartes enracinées sur une demi-arête et les S-arbres équilibrés qui conserve l'ensemble des sommets qui transforme les faces non racine en des feuilles. Retranscrite en terme de séries génératrices, cette bijection montre l'égalité $\Gamma_1=zE$. Le théorème est ainsi prouvé. \end{proof}

\fig{[scale=1]{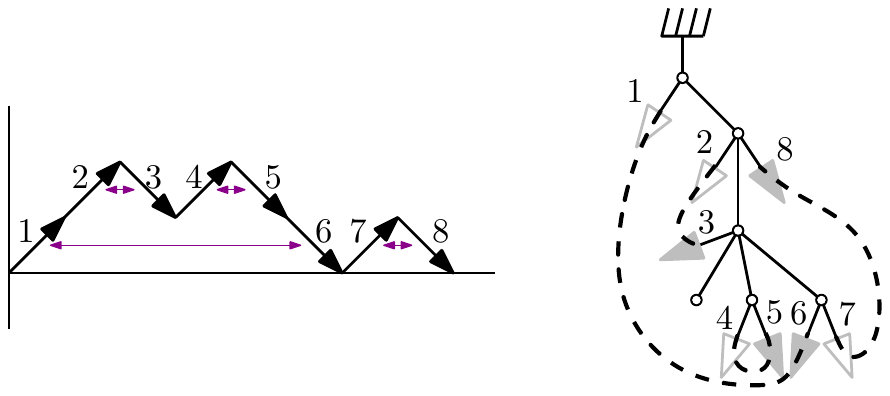}}{Transformation du S-arbre équilibré de la figure \ref{conj2} en carte enracinée sur une demi-arête.}{appareil}

La bijection de la preuve précédente, que nous avons détaillée pour illustrer l'intérêt des arbres bourgeonnants, admet plusieurs variantes qui permettent d'interpréter combinatoirement de nombreuses équations fonctionnelles. Le théorème suivant liste quelques unes de ces équations, extraites de \cite{bdg2002}, que nous utiliserons dans ce mémoire.

\begin{theo} \label{cartesgk}
Appelons
\begin{itemize}
\item $M(z,u)$ la série génératrice des cartes planaires (enracinées sur un coin),
\item $\Gamma_1(z,u)$ la série génératrice des cartes planaires enracinées sur une demi-arête,
\item $\Gamma_2(z,u)$ la série génératrice des cartes planaires enracinées sur une demi-arête et avec une seconde demi-arête incidente à la face racine \footnote{Contrairement au papier originel \cite{bdg2002}, nous excluons la carte vide.} (voir la figure \ref{cg2}),
\item $R(z,u)$ la série génératrice des R-arbres,
\item $S(z,u)$ la série génératrice des S-arbres,
\end{itemize}
où la variable $z$ compte le nombre de faces dans les séries génératrices de cartes et le nombre de feuilles dans les séries génératrices d'arbres bourgeonnants, et où la variable $u$ compte les sommets.
Ces séries satisfont les formules suivantes (où aucune dérivée n'apparaît) :
$$
M = \frac {\Gamma_1^2} z + z \Gamma_2,
$$

$$
\Gamma_1 = S - u \sum_{\substack{i \geq 2 \\ j \geq 0}} g_{2i+j-1} { 2i - j - 2 \choose i,i-2,j} R^i S^j,
$$
$$
\Gamma_2 = R - z - \frac u z \sum_{\substack{i \geq 3 \\ j \geq 0}}   g_{2i+j-2} { 2i - j - 3 \choose i,i-3,j} R^i S^j - \left(\frac u z \sum_{\substack{i \geq 2 \\ j \geq 0}}  g_{2i+j-1} { 2i - j - 2 \choose i,i-2,j} R^i S^j \right)^2.
$$
En éliminant $\Gamma_1$ et $\Gamma_2$ dans les précédentes équations, nous pouvons exprimer $M$ en fonction de $R$ et $S$ :
$$
M = zR + zS^2 - z^2 - 2 u S \sum_{\substack{i \geq 2 \\ j \geq 0}}  g_{2i+j-1} { 2i - j - 2 \choose i,i-2,j} R^i S^j - u \sum_{\substack{i \geq 3 \\ j \geq 0}}  g_{2i+j-2} { 2i - j - 3 \choose i,i-3,j} R^i S^j.
$$
\end{theo}

\fig{[scale=1]{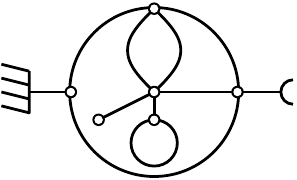}}{Une carte enracinée sur une demi-arête avec une deuxième demi-arête incidente à la face racine.}{cg2}

\section{Des cartes planaires aux mobiles}
\label{s:mobiles}

Que deviennent les arbres bourgeonnants de la partie précédente une fois transposés dans le modèle dual ? La réponse a été apportée par Jérémie Bouttier, Philippe Di Francesco et Emmanuel Guitter  dans \cite{bouttier-mobiles} où sont définis les \textit{mobiles}. Ces objets généralisent les arbres bien étiquetés de Robert Cori et Bernard Vauquelin \cite{cori-vauquelin} et Gilles Schaeffer \cite{schaeffer-these}. Définir les mobiles dans ce mémoire serait hors-sujet ; la figure \ref{mobile} peut néanmoins donner un aperçu de la notion.

\fig{[width=\textwidth]{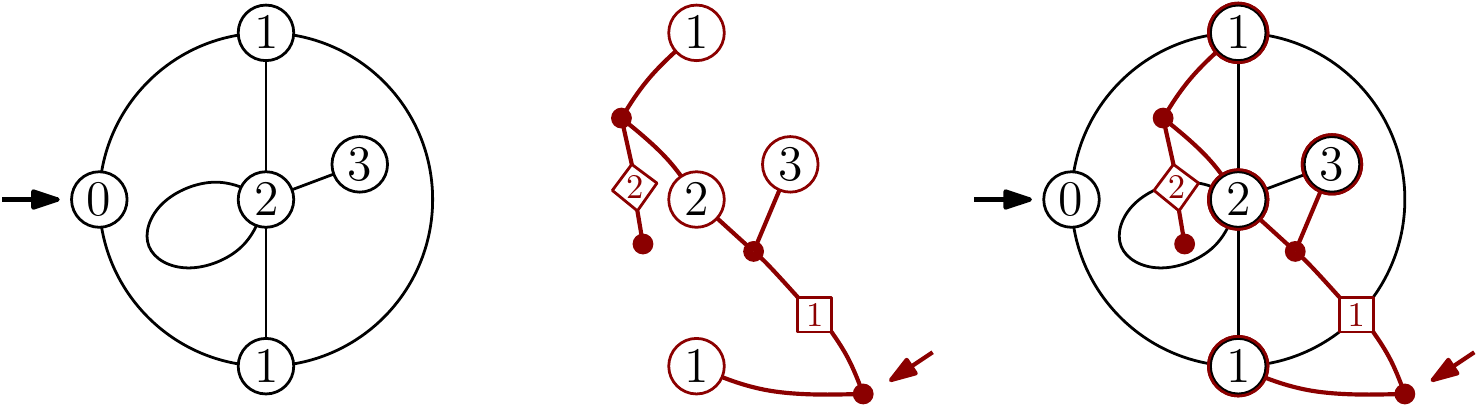}}{\`A gauche : une carte planaire enracinée. Au milieu : le mobile correspondant. \`A droite : leur superposition.}{mobile}

Les mobiles sont les objets idéaux quand on souhaite travailler sur les incidences face-arête sur les cartes planaires, à l'opposé des arbres bourgeonnants qui sont plus adaptés pour les incidences sommet-arête. Toutefois il faut savoir que les mobiles et ce qu'on appelle les \textit{demi-mobiles} jouent un rôle dual par rapport aux R- et aux S-arbres ; leurs séries génératrices satisfont d'ailleurs les mêmes équations \eqref{equerre} et \eqref{equesse}.

En outre, Jérémie Bouttier et Emmanuel Guitter dans \cite{bg-continuedfractions} ont réussi à trouver  grâce aux mobiles  une équation décrivant la série génératrice des cartes où le degré de la face racine\footnote{ou du sommet racine si on raisonne par dualité} est fixé. Elle peut être formulée de la manière suivante.

\begin{prop} 
\label{prop:mp}
Soit $\overline M(z,u)$ la série génératrice des cartes avec un poids $z$ par face, un poids $u g_k$ par sommet non racine de degré $k$ et un poids $u h_k$ par sommet racine\footnote{Même si la formulation est étrange, chaque carte ayant un seul sommet racine.} de degré $k$. Alors 
\begin{equation} 
\label{eqmp}
\frac {\partial \overline M} {\partial z} = u \sum_{i \geq 0} \sum_{j \geq 0} h_{2i+j} {2i+j \choose i,i,j} R^i S^j,
\end{equation}
où $R$ et $S$ sont les séries bivariées définies par \eqref{equerre} et \eqref{equesse}.
\end{prop}

La formule précédente est concise mais comporte une dérivée selon $z$ potentiellement gênante pour la suite. Il faut en effet savoir que les cartes pointées -- ici cela veut dire qu'une face a été distinguée -- sont plus faciles à énumérer que les cartes seules, ce qui justifie la présence de la dérivée. Toutefois, toujours dans \cite{bg-continuedfractions}, se trouve une formule sans dérivée, mais plus complexe, liant la précédente série $\overline M$ aux séries $R$ et~$S$ : 

\begin{theo}
\label{vintegrale}
Soient $\overline M(z,u)$, $R(z,u)$ et $S(z,u)$ les séries de la proposition \ref{prop:mp}. Nous avons la relation
$$
\overline M = u \alpha(R,S)-u^2\,\beta(R,S), 
$$
où $\alpha$ et $\beta$ sont des séries formelles en $x,y,g_1,g_2,\dots,h_1,h_2,\dots$ définies par
$$
\alpha(x,y) =  \sum_{i \geq 0} \sum_{j \geq 0} h_{2i+j} \frac{(2i+j)!}{i! (i+1)! j!}  x^{i+1} y^j,
$$
$$
\beta(x,y) =   \sum_{i \geq 0} \sum_{j \geq 0}  \sum_{k \geq 0} \sum_{\ell \geq 0} \sum_{q=0}^{2i+j} h_{2i+j-q} \, g_{2k+\ell+q+2} \, \frac{(2i+j)!}{i! (i+1)! j!}  \frac{(2k+\ell)!}{(k!)^2 \ell!} x^{i+k} y^{j+\ell}.
$$
\end{theo}

\section{Une autre bijection bourgeonnante}
\label{s:bb}

Nous donnons dans cette section une interprétation combinatoire de l'identité \eqref{eqmp} en termes d'arbres bourgeonnants. Cette interprétation a été décrite pour la première fois par Olivier Bernardi dans \cite[Théo. 8]{bernardi-zagier} sans faire intervenir explicitement les arbres bourgeonnants. Signalons également que cette transformation peut être déduite d'un article récent de Marie Albenque et Dominique Poulalhon \cite{AlPo}.

Cette bijection a l'avantage par rapport à celles décrites dans la section \ref{s:bourgeonnants} de ne pas reposer sur un argument de conjugaison. En d'autres termes, nous gardons le contrôle sur le degré du sommet racine.

Cette bijection revêt un intérêt particulier dans ma thèse ; nous la réutiliserons quand la carte sera munie d'une forêt couvrante. En résultera une nouvelle description du polynôme de Tutte basée sur ce qu'on appellera \textit{l'activité bourgeonnante}.

\subsection{T-arbres}
\label{ss:tarbre}

Un \textit{$T$-arbre } est un arbre bourgeonnant $T$ tel que 
$(i)$ $T$ ne comporte pas de demi-arête racine mais est enraciné sur un coin,
$(ii)$ la charge totale de $T$ est $0$,
$(iii)$ tout sous-arbre a pour charge $0$ ou $1$.
La figure \ref{tartre} montre un exemple d'un tel arbre. 
 
\fig{[scale=1.2]{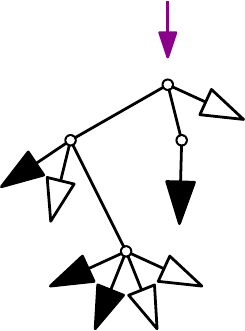}}{Exemple de T-arbre}{tartre}

\begin{prop}\label{p:tarbre}
La série génératrice $T(z,u)$ des T-arbres avec un poids $z$ par feuille, un poids $u$ par sommet, un poids $g_k$ par sommet non racine de degré $k$ et un poids $h_k$ si le sommet racine est de degré $k$ satisfait
$$T(z,u) = u \sum_{i \geq 0} \sum_{j \geq 0} h_{2i+j} \, {2i+j \choose i,i,j} \, R^i \, S^j,$$
où les séries $R$ et $S$ sont définies par \eqref{equerre} et \eqref{equesse}.
\end{prop}
\begin{proof}La preuve est similaire à celle de la proposition \ref{heqRS}. Un T-arbre est un arbre bourgeonnant enraciné sur un coin dans lequel $i$ R-arbres, $j$ S-arbres et $k$ bourgeons sont attachés au sommet racine. Pour que la charge de l'arbre soit nulle, il faut que $i = k$. Mais les séries génératrices des R- et S-arbres sont décrites par \eqref{equerre} et \eqref{equesse} ; nous en déduisons l'expression de $T$ annoncée.
\end{proof}

Décrivons dans la sous-section ci-dessous une bijection $\Theta$ entre T-arbres et cartes planaires $C$ avec une face  marquée $\chi$ :
\begin{itemize}
\item qui transforme les feuilles en faces non racine,
\item qui préserve le nombre de sommets ainsi que leurs degrés,
\item qui préserve le degré du sommet racine.
\end{itemize}
Cette bijection fournit une interprétation combinatoire de \eqref{eqmp}.

\subsection{Des T-arbres aux cartes planaires avec face marquée}
\label{ss:cloture}

La transformation d'un T-arbre enrichi en une carte planaire est peut-être le sens le plus naturel de la bijection. Considérons $T$ un  T-arbre.

Nous réutilisons  le procédé que nous avions employé dans le point 5 de la preuve du théorème \ref{g1rs}  (mais formulé de manière différente) : si une feuille suit immédiatement un bourgeon lorsque nous longeons la face externe dans le sens trigonométrique, alors nous relions les deux pour former une arête. Cette arête est tracée de sorte que si on l'oriente dans le sens bourgeon-feuille, la face externe se trouve à la droite de cette arête. Il se peut que le coin racine se trouve entre un bourgeon et une feuille. Cela ne change rien, le bourgeon et la feuille seront reliés et le coin racine sera dans une face non externe. Nous répétons ce processus jusqu'à qu'il n'y ait plus de feuille et de bourgeon. Nous obtenons alors une carte planaire $C$. Nous marquons la face externe que nous appelons $\chi$. On note $(C,\chi) = \Theta(T)$. La figure \ref{tartransfo} illustre les différentes étapes de cette transformation.

\fig{[width = \textwidth]{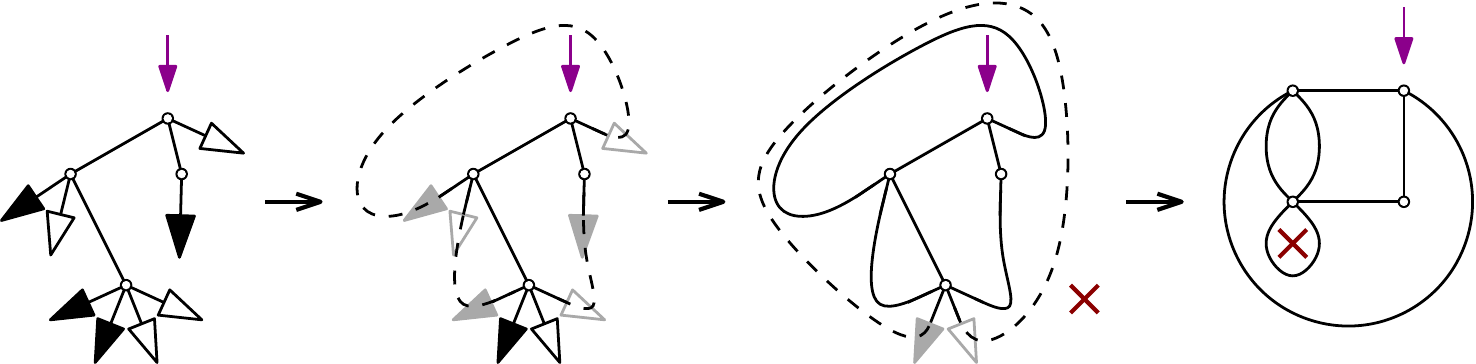}}{Les différentes étapes de la transformation $\Theta$. Les deux dernières cartes sont exactement les mêmes, la dernière est dessinée de sorte que la racine soit sur la face externe.}{tartransfo}

Notons que la carte obtenue ne dépend pas de l'ordre dans lequel on a choisi de regrouper les couples bourgeon/feuille. En effet, chaque bourgeon est associé de manière unique à une feuille. Pour le voir, il suffit de considérer le chemin de pas montants et descendants comme décrit dans le point 1 de la démonstration du théorème \ref{g1rs} : un bourgeon et une feuille sont associés si et seulement si les pas correspondants se font face.

De plus, nous pouvons voir que chaque feuille de $T$ donne naissance  une face non racine quand elle est reliée à un bourgeon. Le nombre de faces non racine de $C$ est donc égal au nombre de feuilles dans $T$. En outre, comme nous l'avions annoncé, les degrés des sommets, dont celui la racine, sont préservés par $\Theta$.

\subsection{Des cartes planaires avec face marquée aux T-arbres}
\label{ss:ouverture}

Décrivons maintenant l'application réciproque de $\Theta$, que nous notons provisoirement $\Theta^*$.

Considérons  une carte planaire $C$ avec une face marquée $\chi$. Commençons par tracer $C$ dans le plan de manière à ce que $\chi$  devienne la face externe.

Décrivons une itération de la transformation. Considérons $E$ l'ensemble des arêtes  incidentes à la face extérieure qui ne sont pas des isthmes. Nous orientons les arêtes de $E$ selon le sens trigonométrique dans la face externe. Pour chaque arête $e$ de $E$, nous procédons comme suit. Nous notons $I_e$ l'ensemble des arêtes $e'$ de $E$ telles que $e$ et $e'$ sont incidentes aux deux mêmes faces. (L'ensemble $I_e$ est représenté schématiquement sur la figure \ref{ie}.)  Si $I_e$ est réduit à $e$, alors on coupe $e$ en deux : la première demi-arête (selon l'orientation de l'arête) est un bourgeon, la seconde est une feuille. Si $I_e$ a plus de deux éléments, alors la suppression de deux arêtes de $I_e$ déconnecte le graphe. Il faut donc déterminer quelle arête de $I_e$ nous allons couper. Considérons la composante de $C \backslash I_e$ contenant la racine et l'arête $\hat e$ qui est l'unique arête de $I_e$ qui sort de cette composante. Nous remplaçons $\hat e$ par un couple bourgeon/feuille en respectant cet ordre, en conservant les autres arêtes de $I_e$.

\fig{[scale=0.7]{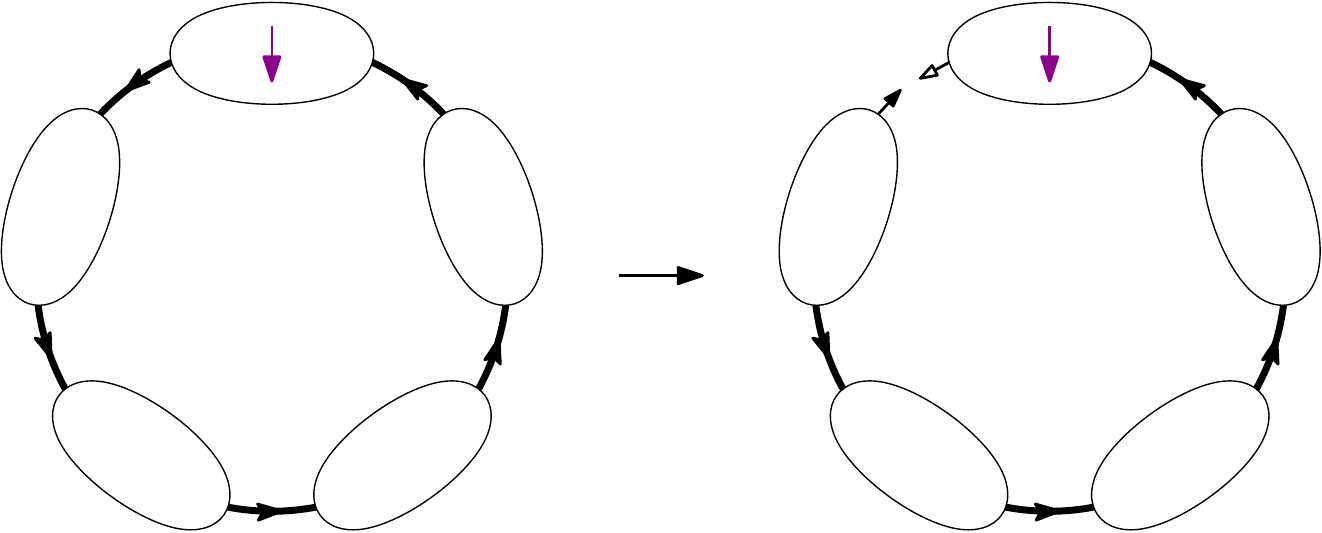}}{En gras : l'ensemble $I_e$. Nous coupons l'arête qui sort de la composante racine.}{ie}

Nous répétons la précédente itération jusqu'à qu'il n'y ait plus que  des isthmes sur la face extérieure. Autrement dit, nous itérons jusqu'à que nous obtenions un arbre bourgeonnant  $T$. On pose $T = \Theta^*(C,\chi)$. Les différentes étapes de cette transformation sont illustrées sur la figure \ref{tof}.

\fig{[width = \textwidth]{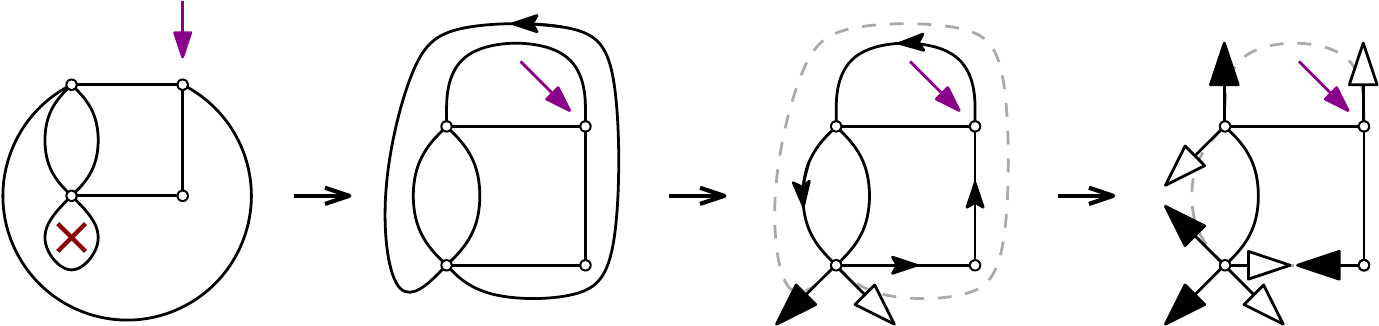}}{Les différentes étapes de la transformation $\Theta^*$.}{tof}

Montrons que nous avons bien obtenu un T-arbre.

\begin{prop}
L'image $T$ par $\Theta^*$ de n'importe quelle carte planaire avec face marquée $(C,\chi)$  est un T-arbre.
\end{prop}
\begin{proof}
Les points $(i)$ et $(ii)$ de la définition de \textit{T-arbre} sont trivialement vérifiés. Considérons une arête  $e$ de $T$. Appelons $S_2$ l'ensemble des sommets du sous-arbre associé à $e$ et $S_1$ le complémentaire de $S_2$ dans $T$. L'ensemble $\gamma$ des arêtes de $C$ avec une extrémité dans $S_1$, l'autre dans $S_2$, est un \textit{cocycle} de $C$, c'est-à-dire un ensemble d'arêtes minimal pour l'inclusion dont la suppression augmente le nombre de composantes connexes de $1$. Notons que seules les arêtes de $\gamma$ peuvent augmenter ou réduire totale la charge de $S_2$ au cours de la transformation $\Theta^*$. Itérons le processus jusqu'à qu'une arête de $\gamma$ soit incidente à la face extérieure. Prouvons alors par récurrence\footnote{La récurrence porte sur des objets un peu plus généraux que les cartes planaires dans la mesure où des bourgeons et des feuilles sont attachés aux sommets. Quand on parle de "charge d'un ensemble de sommets", on parle de la différence entre feuilles et bourgeons attachés à cet ensemble de sommets.} sur le cardinal de $\gamma$ que la charge de $S_2$ est égale à $0$ ou $1$.

Si $\gamma$ a un seul élément, alors il s'agit de $e$ : cette arête est et restera un isthme et la charge de $S_2$ est nulle.

Si $\gamma$ comporte au moins trois éléments, alors exactement deux arêtes de $\gamma$ sont incidentes à la face racine de $C$, donc appartiennent à l'ensemble $E$ des arêtes incidentes à la face externe qui ne sont pas des isthmes. Une fois qu'une arête appartient à  $E$, soit elle est supprimée, soit elle devient un isthme. Or ces deux arêtes ne peuvent pas devenir des isthmes car $\gamma$ contient au moins trois éléments. Elles sont donc supprimées. En outre, une des deux arêtes est orientée dans le sens $S_1$ vers $S_2$ alors que l'autre est orientée de $S_2$ vers $S_1$. Lors de leur coupure, ces arêtes laissent donc un bourgeon et une feuille sur $S_2$. En conclusion, après une itération, $\gamma$ perd deux arêtes et la charge de $S_2$ reste nulle. On peut donc appliquer l'hypothèse de récurrence à la carte obtenue après cette itération.

Si $\gamma$ comporte exactement deux arêtes, alors ces deux arêtes sont incidentes à la face externe et une même autre face. Elles appartient donc à $I_e$. Comme précédemment, une est orientée de $S_1$ vers $S_2$ et l'autre de $S_2$ vers $S_1$. Mais l'arête de $I_e$ qui est supprimée est celle qui sort de la composante de $C \backslash I_e$ contenant la racine. Or $S_2$ ne contient pas la racine, donc c'est l'arête qui est orientée de $S_1$ vers $S_2$ qui est supprimée. L'ensemble $S_2$ sera donc muni d'une feuille, sa charge est et restera égale à $1$.

Nous avons donc bien prouvé par récurrence que la condition $(iii)$ était vraie. \end{proof}

Expliquons rapidement pourquoi $\Theta$ et $\Theta^*$ sont deux applications réciproques. On se convainc facilement que $\Theta \circ \Theta^*$ est l'identité. En effet, $\Theta$ relie les couples bourgeon/feuille que la transformation $\Theta^*$ a coupés. 
L'application $\Theta^* \circ \Theta$ est aussi égale à l'identité, mais cela s'avère fastidieux à prouver directement sur les objets combinatoires. Toutefois nous pouvons utiliser un simple argument de cardinalité. Plus précisément, comme $\Theta \circ \Theta^*$ vaut l'identité, nous savons que $\Theta$ est surjective de l'ensemble des T-arbres  à $f$ feuilles et $s$ sommets vers l'ensemble des cartes planaires avec face marquée à $f+1$ faces et $s$ sommets. Or d'après la relation \eqref{eqmp}, ces deux ensembles ont même cardinal. Donc $\Theta$ est bijective ; les applications $\Theta$ et $\Theta^*$ sont réciproques.

\part{Cartes forestières}

% *!*!*!*!*!*!*!*!*!*!*!*!*!*!*!*!*!*!*!*!*!*!*!*!*!*!*!*!*!*!*!*!*!*!*!*!*!*!*!*! %
 \chapter{Série génératrice des cartes forestières }
\label{c:cartesforestieres} 
 
% *!*!*!*!*!*!*!*!*!*!*!*!*!*!*!*!*!*!*!*!*!*!*!*!*!*!*!*!*!*!*!*!*!*!*!*!*!*!*!*! %
% *!*!*!*!*!*!*!*!*!*!*!*!*!*!*!*!*!*!*!*!*!*!*!*!*!*!*!*!*!*!*!*!*!*!*!*!*!*!*!*! %

\subsection*{Introduction} 

Le chapitre \ref{c:cartesarbres} nous a montré que l'énumération des cartes planaires est très bien comprise aujourd'hui. Ce sont deux problèmes plus généraux qui concentrent maintenant l'attention : l'énumération des cartes en genre quelconque \cite{bender-surface-I,chapuy-marcus-schaeffer} et l'énumération des cartes munies d'une structure additionnelle. Cette dernière question est particulièrement pertinente d'un point de vue physique ; les surfaces sur lesquelles rien ne se passe ("gravité pure") n'ont que peu d'intérêt. Ont ainsi été étudiées les cartes munies d'un polymère \cite{DK88}, d'un modèle d'Ising  \cite{Ka86,mbm-schaeffer-ising,BDG-hard-part-blossoming}, d'une coloration propre \cite{tutte-triangulations}, d'un arbre couvrant \cite{mullin-boisees}... En particulier, de nombreux papiers ont été consacrés ces vingt dernières années à l'étude du modèle de Potts sur des familles de cartes planaires \cite{baxter82,daul,DF-Eynard-Guitter,eynard-bonnet-potts,zinn-justin-dilute-potts}. Comme nous l'avons souligné dans le chapitre introductif, cela revient à énumérer les cartes planaires pondérées par leur polynôme de Tutte.

Olivier Bernardi et Mireille Bousquet-Mélou ont récemment prouvé \cite{bernardi-mbm-de} que la fonction de partition du modèle de Potts était différentiellement algébrique (voir Sous-section \ref{ss:nature}). Ceci a été au moins vérifié pour les cartes planaires générales et pour les triangulations. Toutefois, la méthode qu'ils utilisent pour calculer ces équations différentielles est particulièrement technique et ne met pas en évidence la structure combinatoire des cartes coloriées. En outre, il semble difficile de trouver par cette approche le comportement asymptotique du nombre de cartes coloriées ou l'emplacement des transitions de phase.

Le but de cette première partie de thèse est de remédier à ces problèmes, du moins pour une spécialisation à une variable du polynôme de Tutte. Cette spécialisation est obtenue en fixant une de ses deux variables à $1$. Combinatoirement, nous énumérons tout simplement les cartes munies d'une forêt couvrante, que nous appelons \textit{cartes forestières}. La série génératrice de telles cartes est notée $F(z,u,t)$, où $z$ compte les faces, $u$ les composantes non racine de la forêt et $t$ les arêtes.

Notre approche est purement combinatoire. Utilisant les théorèmes du chapitre \ref{c:cartesarbres}, la série $F(z,u,t)$ est exprimée en termes de deux séries $R$ et $S$, solutions couplées d'un système fonctionnel. Nous exploitons par la suite ce système afin de montrer que $F$ est différentiellement algébrique, consolidant ainsi le résultat susmentionné d'Olivier Bernardi et Mireille Bousquet-Mélou.

Puis, pour $u \geq -1$ et $t=1$, nous étudions les singularités de $F(z,u,t)$ et le comportement asymptotique de son $n$-ième coefficient. Pour $u > 0$, nous retrouvons  le régime asymptotique standard avec un facteur en $n^{-5/2}$. En $u = 0$, nous observons une transition de phase avec un facteur $n^{-3}$. Enfin, pour $u \in [-1,0[$, nous obtenons un comportement extrêmement atypique en $n^{-3} \pare{\ln n}^{-2}$. \`A notre connaissance, c'est une nouvelle "classe d'universalité" pour les cartes planaires.

Une grande partie de ces résultats provient d'un article conjointement réalisé avec ma directrice de thèse, Mireille Bousquet-Mélou \cite{regular-fpsac,mbm-courtiel}.

\subsection*{Remarque préliminaire}

Nous souhaitons mener l'étude des séries génératrices de différentes classes de cartes (cartes cubiques, cartes tétravalentes,  cartes eulériennes, etc.). Afin d'éviter les redondances dans les calculs, cette étude s'effectuera dans le cadre général suivant : pour chaque carte planaire enracinée, chaque sommet de degré $k$ est muni d'une variable de poids $d_k$. Plusieurs classes de cartes peuvent être ainsi retrouvées par spécialisation des variables $d_k$. Par exemple, l'ensemble des cartes cubiques est obtenu en posant $d_3=1$ et $d_k=0$ pour $k \neq 3$. 

% !!!!!!!!!!!!!!!!!!!!!!!!!!!!!!!!!!!!!!!!!!!!!!!!!!!!!!!!!!!!!!!!!!!!!!!!!!!!!!!! %
\section{Premières descriptions}

% !!!!!!!!!!!!!!!!!!!!!!!!!!!!!!!!!!!!!!!!!!!!!!!!!!!!!!!!!!!!!!!!!!!!!!!!!!!!!!!! %

\subsection{Définition}

Considérons une carte planaire enracinée $C$. On rappelle que $\Som C$ désigne l'ensemble des sommets de $C$ et $\Ar C$ l'ensemble des arêtes. Une \textit{forêt couvrante} de $C$ est un graphe $F$ tel que l'ensemble des sommets soit égal à $\Som C$ et l'ensemble des arêtes  soit un sous-ensemble de $\Ar C$ ne comportant pas de cycle. Chaque composante connexe de $F$ est un arbre : la \textit{composante racine} désigne l'arbre contenant le sommet racine. Une arête de $C$ appartenant à la forêt $F$ est dite \textit{interne} ; à l'inverse elle est dite \textit{externe}. Tout couple de la forme $(C,F)$ est appelé \textit{carte forestière}. La figure \ref{forestierenormale} illustre cette définition.

\begin{figure}[h!]
\begin{center}
\includegraphics[scale=1.3]{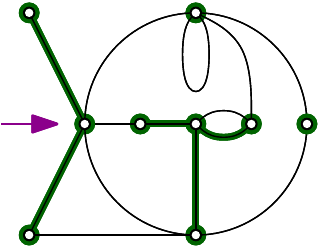}
\end{center}
\vspace{-15pt}
\caption{Une carte forestière avec 3 composantes non-racine.}
\label{forestierenormale}
\end{figure}

Nous allons étudier la série génératrice $F(z,u,t)$ des cartes forestières, comptées selon le nombre de faces (variable $z$), le nombre de composantes non racine (variable $u$) et le nombre d'arêtes (variable $t$) :
\begin{equation}
F(z,u,t) = \sum_{\substack{C\textrm{ carte planaire}\\F\textrm{ forêt couvrante de }C }} \pare{\prod_{k \geq 1} d_{k}^{\ \som_{k}(C)}} z^{\face(C)}\,u^{\cc(F)-1}\,t^{\arete(C)},
\label{Fdef}
\end{equation}
où $\som_{k}(.)$, $\face(.)$, $\cc(.)$, $\arete(.)$ désignent respectivement le nombre de sommets de degré $k$, de faces, de composantes connexes, d'arêtes.

Prenons l'exemple des cartes cubiques (rappel : $d_3=1$ et $d_k=0$ pour $k \neq 3$). Les premiers termes de la série $F$ selon les suivants :
\begin{equation}
F(z,u,t) =  (6 + 4 u) z^3t^3 + (140 + 234 u + 144 u^2 + 32 u^3) z^4t^4 + O(z^5).
\label{premierstermesF}
\end{equation}
(On remarquera que la variable $t$ est donc redondante avec $z$.) Combinatoirement, le terme $(6 + 4 u) \,z^3\,t^3$ signifie qu'il y a exactement $10$ cartes forestières cubiques avec trois faces : six d'entre elles sont munies d'un arbre couvrant, les quatre autres d'une forêt à deux composantes (voir figure \ref{dixcartes}).

\begin{figure}[h!]
\begin{center}
\includegraphics[scale=1.6]{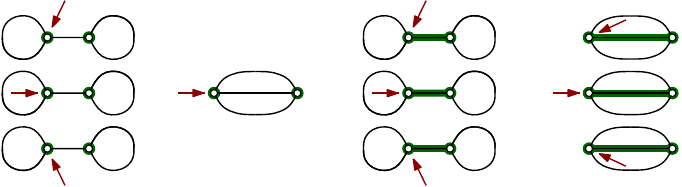}
\end{center}
\vspace{-15pt}
\caption{Les dix cartes forestières cubiques à trois faces.}
\label{dixcartes}
\end{figure}

\subsection{Cartes forestières, polynôme de Tutte et modèles apparentés}
\label{ss:modele}

Faisons maintenant le rapprochement entre la série $F$ et le polynôme de Tutte défini en section \ref{s:poltutte}. Rappelons que le polynôme de Tutte d'un graphe connexe $G$ a pour expression
\begin{equation}
T_G(x,y) = \sum_{S\textrm{ sous-graphe couvrant de }G}(x-1)^{\cc(S)-1}\,(y-1)^{\cycl(S)},
\label{deftutte}
\end{equation}
où $\cycl(S)$ est le nombre cyclomatique du sous-graphe $S$, c'est-à-dire le nombre minimal d'arêtes à supprimer pour que $S$ devienne acyclique. 
En posant $y=1$, tous les termes de la précédente somme s'annulent sauf ceux qui sont associés à des sous-graphes $S$ tels que $\cycl(S) = 0$, autrement dit des forêts couvrantes. Ainsi, notre série $F$ admet une description en termes du polynôme de Tutte :
\begin{equation}
F(z,u,t) = \sum_{C\textrm{ carte planaire}}  \pare{\prod_{k \geq 1} d_{k}^{\ \som_{k}(C)}} T_C(u+1,1)\, z^{\face(C)} \, t^{\arete(C)}.
\label{lientutteF}
\end{equation}
\'Ecrit en termes d'activité interne (voir l'équation \eqref{tuttepremiereecriture} p. \pageref{tuttepremiereecriture}), on obtient
\begin{equation}
F(z,u,t) = \sum_{\substack{C\textrm{ carte planaire}\\T\textrm{ arbre couvrant de }C }}  \pare{\prod_{k \geq 1} d_{k}^{\ \som_{k}(C)}} (u+1)^{\inte(T)} \,  z^{\face(C)} \, t^{\arete(C)},
\label{Fintact}
\end{equation}
où $\inte(T)$ désigne le nombre d'arêtes internes actives\footnote{Il existe plusieurs notions d'activité qui peuvent convenir ici. La deuxième partie de ce mémoire est d'ailleurs consacrée à ce sujet.} d'un arbre couvrant $T$.
Remarquons que nous pouvons ainsi donner un sens combinatoire à la variable $\mu = u+1$ et considérer $u$ dans l'intervalle $[-1,+\infty[$.

Nous allons maintenant donner cinq autres descriptions de $F$ en termes de cartes duales, c'est-à-dire des cartes où la variable $d_k$ compte non plus les sommets mais les faces de degré $k$. En effet, pour toute carte \textbf{planaire} $C$, le polynôme de Tutte de $C$ est le symétrique du polynôme de Tutte de $C^*$ (voir équation \eqref{tuttedual} p.~\pageref{tuttedual}) :
$$T_{C^*}(x,y) = T_C(y,x).$$
Dès lors, il existe une version duale à \eqref{lientutteF} :
\begin{equation}
F(z,u,t) = \sum_{C\textrm{ carte planaire}}  \pare{\prod_{k \geq 1} d_{k}^{\ \face_{k}(C)}} T_{C}(1,u+1)\, z^{\som(C)} \, t^{\arete(C)},
\label{Fdual}
\end{equation}
où $\som(C)$ est le nombre de sommets de la carte $C$ et $\face_{k}(C)$ son nombre de faces de degré $k$. Les expressions duales de \eqref{Fdef} et \eqref{Fintact} s'en déduisent :
$$
F(z,u,t)  = \sum_{\substack{C \textrm{ carte planaire} \\ K\textrm{sous-graphe connexe de }C}} \pare{\prod_{k \geq 1} d_{k}^{\ \face_{k}(C)}} u^{\cycl(K)} \, z^{\som(C)} \, t^{\arete(C)} 
$$
\begin{equation} \label{interp1}
 	\phantom{F(z,u,t) = i} =  \sum_{\substack{C \textrm{ carte planaire} \\ T\textrm{arbre couvrant de }C}} \pare{\prod_{k \geq 1} d_{k}^{\ \face_{k}(C)}}  (u+1)^{\ext(T)} \, z^{\som(C)} \, t^{\arete(C)},
\end{equation}
où $\ext(T)$ désigne le nombre d'arêtes externes actives de l'arbre couvrant $T$. Par exemple, le terme $(6 + 4 u) \,z^3\,t^3$ qui apparaît dans le développement \eqref{premierstermesF} de $F$ avec $d_3=1$ et $d_k=0$ pour $k \neq 3$ correspond aux 10 triangulations à 3 sommets munies d'un sous-graphe couvrant connexe. Elles sont dessinées figure \ref{triangulations}.

\begin{figure}[h!]
\begin{center}
 \includegraphics[width=\textwidth - 20 pt]{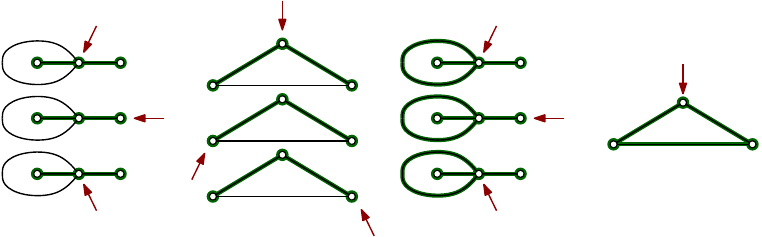}
 \end{center}
\vspace{-15pt}
\caption{ Les dix triangulations à trois sommets équipées d'un sous-graphe couvrant connexe.}
\label{triangulations}
\end{figure}

La série $F$ peut être également interprétée dans le modèle du tas de sable sur des cartes. En effet, il est connu \cite{merino,cori-borgne} que $T_C(1,y)$ compte les configurations \textit{récurrentes} (voir sous-section~\ref{ss:sable} p.~\pageref{ss:sable}) du modèle du tas de sable sur $C$ selon le niveau, noté $\ell(\gamma)$ pour une configuration $\gamma$:
$$T_C(1,y) = \sum_{\gamma\textrm{ configuration récurrente}} y^{\ell(\gamma)}.$$
Par conséquent, 
\begin{equation} \label{interp2}
F(z,u,t)  =   \sum_{\substack{C \textrm{ carte planaire} \\ \gamma \textrm{ configuration récurrente de }C}} \pare{\prod_{k \geq 1} d_{k}^{\ \face_{k}(C)}} (u+1)^{\ell(\gamma)} \, z^{\som(C)} \, t^{\arete(C)} 
\end{equation}
est également la série génératrice des cartes planaires équipées d'une configuration récurrente $\gamma$ du modèle du tas de sable, avec un poids $d_k$ pour chaque face de degré $k$, comptées selon le nombre de sommets  et d'arêtes de $M$ et le niveau de $\gamma$. Le précédent terme $(6 + 4 u) \,z^3\,t^3 =  4 (u+1)\,z^3\,t^3  + 2\,z^3\,t^3 $ est cette fois illustré par la figure \ref{triangulations2} où sont dessinées les 6 triangulations munies d'une configuration récurrente. 

\begin{figure}[h!]
\begin{center}
 \includegraphics[width=\textwidth]{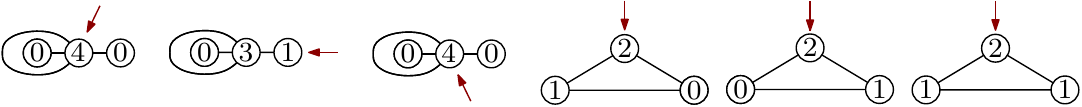}
 \end{center}
\vspace{-15pt}
\caption{ Les six triangulations à 3 sommets équipées d'une configuration récurrente du modèle de tas de sable.}
\label{triangulations2}
\end{figure}

Pour finir, nous allons exprimer notre série $F$ en termes du \textit{modèle de Potts}
% Pour $q \in \N$, une coloration à $q$ couleurs d'un graphe $G$ est une application de $\Som G$ vers $\ens{1,\dots,q}$.
(voir sous-section~\ref{ss:potts} p.~\pageref{ss:potts}). 
 On rappelle que la \textit{fonction de partition du modèle de Potts} $P_G$ sur un graphe $G$ compte les colorations selon le nombre d'arêtes monochromatiques, noté $m(c)$ pour une coloration $c$ : 
 $$P_G(q,\nu) = \sum_{c \, : \,  \Som G \rightarrow \ent 1 q} \nu^{m(c)}.$$
 Fortuin et Kasteleyn ont établi dans \cite{fk} l'équivalence entre  le \textit{polynôme réduit de Potts} $\tilde P_G(q,\nu)$ , défini par 
 $$ \tilde P_G(q,\nu) = \frac 1 q \,  P_G(q,\nu),$$
 et le polynôme de Tutte d'un graphe (ici connexe) $G$ selon la relation :
 $$ \tilde P_G\left( (\mu-1)(\nu -1) ,\nu\right) = (\nu - 1)^{\som (G) -1} T_G(\mu,\nu).$$ 
 En posant $(\mu,\nu) = (1,u + 1)$ et en injectant la précédente relation dans \eqref{Fdual}, on obtient 
 \begin{equation} \label{interp3}
  F(z,u,t) = u  \sum_{C\textrm{ carte planaire}} \pare{\prod_{k \geq 1} d_{k}^{\ \face_{k}(C)}}
 \tilde{P}_{C}(0,u+1)\, (z/u)^{\som(C)}\, t^{\arete(C)}.
 \end{equation}

% !!!!!!!!!!!!!!!!!!!!!!!!!!!!!!!!!!!!!!!!!!!!!!!!!!!!!!!!!!!!!!!!!!!!!!!!!!!!!!!! %
\section{\'Equations pour les cartes forestières}
% !!!!!!!!!!!!!!!!!!!!!!!!!!!!!!!!!!!!!!!!!!!!!!!!!!!!!!!!!!!!!!!!!!!!!!!!!!!!!!!! %

Dans cette section, nous donnons plusieurs équations qui caractérisent la série génératrice des cartes forestières $F(z,u,t)$ et des séries génératrices apparentées. Pour cela, nous établissons une bijection entre cartes forestières et \textit{cartes décorées} par des arbres, c'est-à-dire des cartes équipées d'une collection $(T_s)$ d'arbres indexée par les sommets.  

\subsection{Des cartes forestières aux cartes décorées}
\label{ss:cf}

Dans ce chapitre, nos arbres auront des \textit{pattes}, c'est-à-dire des demi-arêtes accrochées à certains sommets\footnote{Nous voulons éviter le terme de "feuilles", afin de les distinguer des sommets de degré $1$ ou des feuilles des arbres bourgeonnants de la section \ref{s:bourgeonnants}.}. Notons $T_\ell(t)$ la série génératrice des arbres à $\ell$ pattes, enracinés sur une patte, avec un poids $t$ par arête (on rappelle que les pattes ne comptent pas pour une arête) et un poids $d_k$ pour chaque sommet de degré $k$ (voir la figure \ref{arbresapattes}). Ces séries sont facilement calculables. En effet, si on pose $\mathcal T(t,y) = \sum_{\ell \geq 1} T_\ell(t) y^{\ell-1}$, alors par une décomposition à la racine,
\begin{equation} \label{deftl}
\mathcal T(t,y) = \sum_{k \geq 1} d_k {\left(t\,\mathcal T(t,y)+y\right)}^{k-1}.
\end{equation}
Pour revenir à notre exemple cubique ($d_3=1$ et $d_k=0$ pour $k$ différent de $3$), les arbres à pattes sont comptés par les nombres de Catalan. Plus précisément, si $\Cat(n)$ désigne le $n$-ième nombre de Catalan, alors $T_1(t)=T_2(t)=0$ et $T_\ell(t) = \Cat(\ell-2) t^{\ell-3}$ pour $\ell\geq 3$.  

\fig{[scale=1.5]{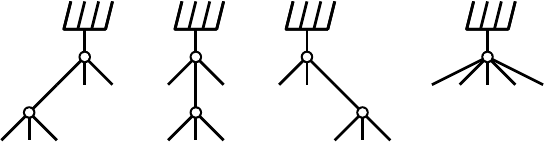}}{Les quatre arbres à $6$ pattes, enracinés sur une patte, où on a supposé $d_1=d_2=d_3=0$. On a alors $T_6(t) = 3 \,d_4^2\, t + d_6$.}{arbresapattes}

\begin{theo}  \label{tfemme}
Appelons $T_\ell(t)$ la série génératrice des cartes à $\ell$ pattes précédemment définie et $\overline M(z,u;g_1,g_2,g_3,\dots;h_1,h_2,h_3,\dots)$ la série génératrice des cartes planaires avec un poids $z$ par face, un poids $u$ par sommet, un poids $g_k$ par sommet non racine de degré $k$ et un poids $h_k$ par sommet racine de degré $k$. La série génératrice $F(z,u,t)$ des cartes forestières se déduit de la série $\overline M$ grâce à la formule
\begin{equation}
\label{femme}
F(z,u,t)= \frac 1 u \, \overline M( z ; g_1(t), g_2(t) , g_3(t), \dots; h_1(t), h_2(t) , h_3(t), \dots),
\end{equation}
où $g_k(t)= t^{k/2} T_k(t)$ et $h_k(t) = t^{k/2}(\frac 2 k t T'_k(t) + T_k(t))$ pour $k \geq 1$. 
\end{theo}

\begin{proof} La preuve de ce théorème repose sur une idée simple : en contractant chaque arbre d'une carte forestière, nous obtenons une carte planaire sans structure, donc facilement énumérable. Ce raisonnement nous a été inspiré par \cite[Appendice A]{BDG-blocked}, où les auteurs étudient les cartes forestières tétravalentes dans lesquelles l'arête racine n'appartient à la forêt. Cette idée est également présente dans \cite{sportiello}.

Passons aux détails de la preuve. Montrons tout d'abord que \mbox{$\frac 2 k t T'_k(t) + T_k(t)$} est la série génératrice $T^c_k(t)$ des arbres à $k$ pattes enracinés sur un coin (voir la figure \ref{arbresacoin}). Plus précisément, montrons l'égalité 
$$ k\,T^c_k(t) =  2 \,  t \, T'_k(t) + T_k(t).$$ Pour expliquer cette formule, il suffit de procéder à un double comptage, celui des arbres à $k$ pattes enracinés à la fois sur une patte et sur un coin. Le membre de gauche correspond aux arbres à $k$ pattes enracinés sur un coin dans lesquels une des $k$ pattes a été marquée. Le membre de droite correspond quant à lui aux arbres à $k$ pattes enracinés sur une patte dans lesquels un des coins a été marqué. En effet, le nombre de coins d'un tel arbre est $2m + k$, où $m$ est le nombre d'arêtes.

\fig{[scale=1.5]{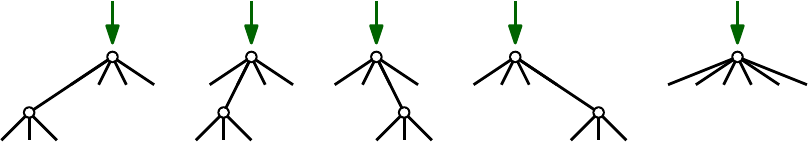}}{Les cinq arbres à $6$ pattes, enracinés sur un coin, où on a supposé $d_1=d_2=d_3=0$. Ils sont comptés par $\frac 1 3 t T'_6(t) + T_6(t) = 4 \,d_4^2\, t + d_6$.}{arbresacoin}

Décrivons maintenant une bijection $\Phi$ entre cartes forestières $(C,F)$ et \textit{cartes décorées par des arbres}, c'est-à-dire des cartes $C'$ munies d'une collection d'arbres $(T_s)_{s \in \mathcal S (C')}$ telle que :
\begin{itemize}
\item pour tout sommet $s$ de $C'$, le nombre de pattes de $T_s$ correspond au degré de $s$,
\item si $\hat s$ désigne le sommet racine de $C'$, alors $T_{\hat s}$ est enraciné sur un coin,
\item pour tout sommet non racine $s$ de $C'$, l'arbre $T_s$ est enraciné sur une patte.
\end{itemize} 
Contractons toutes les arêtes de la forêt $F$ dans $C$ et notons la carte obtenue $C'$. Elle est enracinée sur un coin -- la flèche qui indique la racine reste à la même place. De plus, il y a une correspondance canonique entre les sommets $s$ de $C'$ et les composantes connexes $K_s$ de $F$. En coupant en deux chaque arête de $C$ qui n'appartient pas à $F$, on obtient une famille d'arbres à pattes. L'arbre $T_s$ se définit alors comme l'unique arbre à pattes contenant $K_s$. L'arbre $T_{\hat s}$, où $\hat s$ est le sommet racine de $C'$, hérite naturellement de l'enracinement sur un coin de $C$. 

\fig{[scale=1]{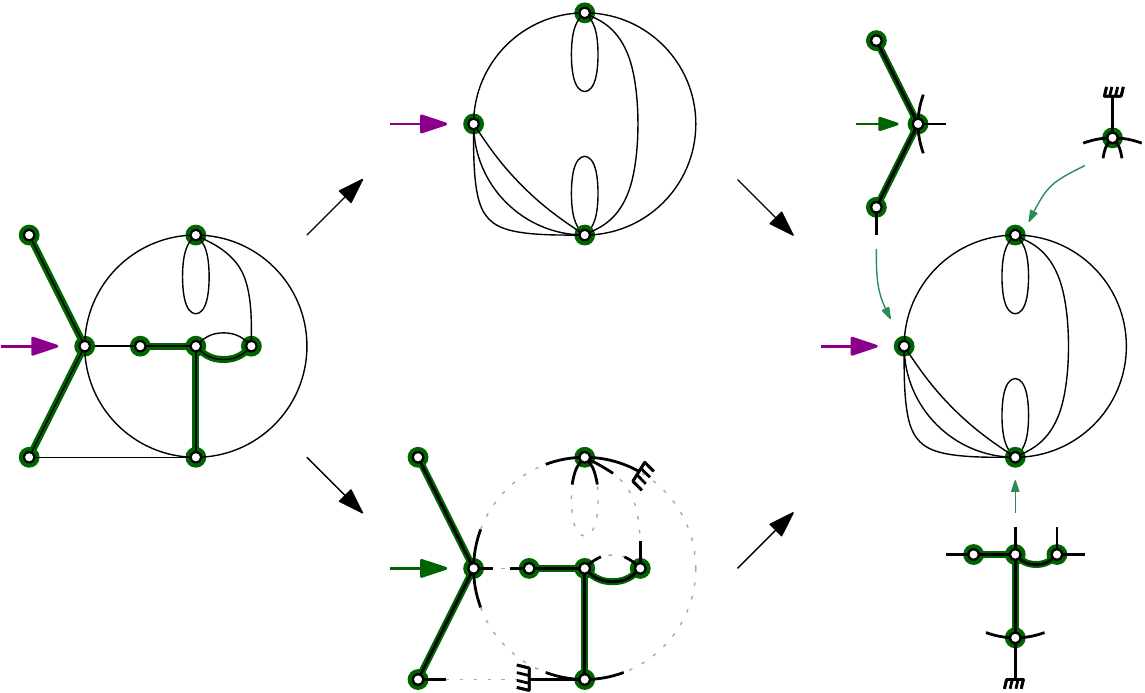}}{\`A gauche : Une carte forestière $(C,F)$. En haut : La carte contractée $C'$. En bas : La collection d'arbres $(T_s)$. \`A droite : La carte décorée $\Phi(C,F)$.}{decompo}

Soit $s$ un sommet non racine. Expliquons comment enraciner l'arbre $T_s$ sur une feuille. Une carte enracinée n'a pas de symétrie : nous pouvons donc fixer de manière canonique un ordre total sur les demi-arêtes de $C'$. Cet ordre peut avoir une origine combinatoire (un bon choix ici serait l'ordre de première visite lors de la transformation d'une carte en arbre bourgeonnant -- voir section \ref{s:bourgeonnants}) mais ce n'est pas obligatoire. L'arbre $T_s$ est alors enraciné sur la patte correspondant à la plus petite demi-arête pour cet ordre.

Posons $\Phi(C,F) = (C',(T_s))$ et montrons que $\Phi$ est bijective.
Pour retrouver la carte forestière $(C,F)$ à partir de $(C',(T_s))$, ôtons chaque sommet $s$ de $C'$ sans toucher aux arêtes et insérons dans l'espace ainsi créé l'arbre $T_s$. Si $\hat s$ est le sommet racine, nous orientons $T_{\hat s}$  de sorte que sa racine coïncide avec la racine de $C'$. Si $s$ est un sommet non racine, nous orientons $T_s$ de sorte que la patte racine coïncide avec la plus petite demi-arête de $C'$ pour l'ordre précédemment choisi. Alors, en recollant pattes des arbres et arêtes de $C'$, nous retombons sur la carte forestière $(C,F)$. Compte tenu que cette construction inverse peut être appliquée à toute carte décorée d'arbres, l'application $\phi$ est bien bijective.

Du reste, nous pouvons facilement vérifier que :
\begin{itemize}
\item $C$ et $C'$ ont le même nombre de faces,
\item le nombre de composantes de $F$ est égal au nombre de sommets de $C'$,
\item le nombre d'arêtes de $C$ dans la forêt $F$ est égal au nombre total d'arêtes parmi tous les arbres $T_s$,
\item le nombre d'arêtes de $C$ en dehors de la forêt $F$ est égal au nombre d'arêtes\footnote{Cela peut se réécrire, grâce à la formule des poignées de main, comme la moitié de la somme des degrés des sommets, ce qui revient bien à pondérer chaque sommet de degré $k$ par $t^{k/2}$.} de $C'$.
\end{itemize}
En transposant toutes ces observations au niveau des séries génératrices, nous obtenons la formule \eqref{femme}. \end{proof}

\subsection{Les équations}

Le théorème précédent permet de voir notre série $F(z,u,t)$ comme une spécialisation de la série génératrice des cartes planaires $\overline M(z,u;g_1,g_2,g_3,\dots;h_1,h_2,h_3,\dots)$. Tous les résultats énoncés dans la section \ref{s:mobiles} sur cette série $\overline M$ peuvent être appliqués.

\begin{theo} \label{central}
Soient $\theta$, $\phi_1$ et $\phi_2$ les séries 
$$\theta(t,x,y) =  \sum_{i \geq 0} \sum_{j \geq 0}  \left(\frac {2\,t} {2i+j} T'_{2i+j}(t) + T_{2i+j}(t)\right) {2i+j \choose i,i,j}  x^i y^j,$$
$$
\phi_1(t,x,y) = \sum_{i \geq 1} \sum_{j \geq 0} T_{2i+j}(t) { 2i+j-1 \choose i-1,i,j } x^i y^j, 
$$
$$
\phi_2(t,x,y) = \sum_{i \geq 0} \sum_{j \geq 0} T_{2i+j+1}(t) { 2i+j \choose i,i,j } x^i y^j,
$$
où la série $T_\ell(t)$ désigne la série génératrice des arbres à $\ell$ pattes comptés selon les arêtes définie par \eqref{deftl}.

Il existe un unique couple $(R,S)$ de séries formelles en $z$, $u$, $t$, $d_1,d_2,d_3,\dots$ qui satisfait
\begin{equation}
\label{forestiere}
R = t\,z + t\,u\,\phi_1(t,R,S),
\end{equation}
\begin{equation}
\label{forteresse}
S = t\,u\, \phi_2(t,R,S).
\end{equation}

La série génératrice $F(z,u,t)$ des cartes forestières est décrite par les relations $F(0,u,t) = 0$ et 
\begin{equation} \label{forfait}
\frac {\partial F} {\partial z} (z,u,t) = \theta(t,R,S).
\end{equation}
\end{theo}
\begin{proof}
On utilise la proposition \ref{prop:mp} couplée au théorème précédent. Nous remplaçons donc dans \eqref{equerre}, \eqref{equesse} et \eqref{eqmp} la variable $g_k$ par $t^{k/2} T_k(t)$ et la variable $h_k(t)$ par $t^{k/2}(\frac 2 k t T'_k(t) + T_k(t))/u$ pour chaque $k \geq 1$. Les équations \eqref{forestiere}, \eqref{forteresse} et \eqref{forfait} sont alors obtenues en effectuant la substitution bijective $(R,S) \mapsto (t^{-1}R,t^{-1/2}S)$. Il n'est pas très difficile de voir que $R$ et $S$ sont devenues des séries formelles en $t$. En outre, l'unicité du couple $(R,S)$ provient du lemme \ref{caracteresse}. \end{proof}

\noindent \textbf{\textit{Remarque 1.}} Le théorème \ref{vintegrale} permet de trouver une expression sans dérivation de $F(z,u,t)$  en fonction des séries $R$ et $S$. Toutefois cette relation s'avèrera moins pratique à manipuler à cause des cinq\footnote{voire sept si on compte les sommes naturellement incluses dans $T_\ell(t)$ !} sommes imbriquées dans la série $\beta$.

\noindent \textbf{\textit{Remarque 2.}} Lorsque $u=0$, le système se simplifie considérablement. En effet, nous trouvons $R=tz$ et $S=0$, de sorte que \eqref{forfait} donne :
$$F(z,0,t) =  \sum_{i \geq 1} \left(\frac {t} {i} T'_{2i}(t) + T_{2i}(t)\right) {2i \choose i}  t^i \frac{z^{i+1}} {i+1}. $$
Nous obtenons de cette manière une expression explicite de $F(z,0,t)$ sous condition qu'on en trouve une pour $T_{2i}$ (généralement le  théorème d'inversion de Lagrange suffit pour en exhiber une, voir par exemple le lemme \ref{l:euln}). La série $F(z,0,t)$  compte les cartes planaires munies d'un arbre couvrant, et cette expression peut se déduire du papier de Mullin \cite{mullin-boisees}.

\noindent \textbf{\textit{Remarque 3.}} Le lecteur désireux de voir  des exemples pourra se référer à la section \ref{spanorama} où  les précédentes équations sont explicitées pour  plusieurs classes de cartes forestières.

\section{Variations sur les cartes forestières et leurs équations}

Nous allons définir et étudier deux classes particulières de cartes forestières : celles qui sont enracinées sur une feuille et celles où l'arête racine n'appartient pas à la forêt. Nous retrouvons la logique de la section précédente ; la même méthode s'applique et permet de décrire les séries génératrices associées.

\subsection{Cartes forestières enracinées sur une feuille}

Une carte est \textit{enracinée sur une feuille} si le sommet racine a degré $1$. Notons $G(z,u,t)$ la série génératrice des cartes forestières enracinées sur une feuille, avec un poids $z$ par face, un poids $u$ par composante non racine, un poids $t$ par arête et un poids $d_k$ pour chaque sommet non racine de degré $k$ (voir la figure \ref{forestierefeuille}).

\fig{[scale = 1.1]{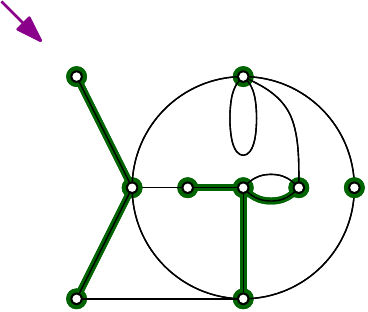}}{Une carte forestière enracinée sur une feuille.}{forestierefeuille} 

\begin{prop}
\label{GRS}
La série génératrice des cartes forestières enracinées sur une feuille est 
$$G(z,u,t) = \left(1 + \frac 1 u \right) \left( S  - u  \sum_{i \geq 2} \sum_{j \geq 0} T_{2i+j-1}(t) { 2i - j - 2 \choose i,i-2,j} R^i S^j \right),$$
où les séries $R$ et $S$ sont définies par \eqref{forestiere} et \eqref{forteresse} et où $T_\ell(t)$ désigne la série génératrice des arbres à $\ell$ pattes, définie par \eqref{deftl}. De plus,  
$$\frac {\partial G}{\partial z}(z,u,t) = \left(1 + \frac 1 u \right) S(z,u,t). $$
\end{prop}

Comme dans la section précédente, il est possible de transformer ces cartes  en cartes décorées. De cette manière, nous pouvons exprimer la série $G$ en fonction de la série $\Gamma_1(z,u;g_1,g_2,g_3,\dots)$ des cartes enracinées sur une demi-arête\footnote{que nous avons étudiées section \ref{s:bourgeonnants}}, avec un poids $z$ par face et un poids $u g_k$ par sommet de degré $k$.

\begin{lem} \label{GM}
 L'équivalent de \eqref{femme} pour les cartes forestières enracinées sur une feuille est 
$$G(z,u,t) =  \sqrt{t} \left(1 + \frac 1 u \right) \Gamma_1(z,u;t^{1/2} T_1(t),t T_2(t),t^{3/2} T_3(t), \dots).$$
\end{lem}

\begin{proof} Soit $E$ l'ensemble des cartes forestières enracinées sur une feuille dans lesquelles l'arête racine n'appartient à la forêt.  Nous pouvons facilement vérifier que l'application $\Phi$ (de la preuve du théorème \ref{tfemme}) induit une bijection entre $E$ et les cartes décorées par des arbres, enracinées sur une feuille, et où l'arbre associé au sommet racine est l'arbre trivial (un sommet et une patte). En retirant ce sommet et la demi-arête racine (poids total $\sqrt{t}$), nous obtenons des cartes décorées enracinées sur une demi-arête, comptées par $\Gamma_1(z,u;t^{1/2} T_1(t),t T_2(t),t^{3/2} T_3(t), \dots)$. La série génératrice associée à $E$ est donc 
$$\sqrt{t}\,\Gamma_1(z,u;t^{1/2} T_1(t),t T_2(t),t^{3/2} T_3(t), \dots).$$
Pour obtenir les autres cartes forestières enracinées sur une feuille, il suffit de considérer une carte forestière de $E$ et de rajouter l'arête racine à la forêt. En faisant ainsi, nous perdons une composante connexe, d'où le facteur $\frac 1 u$.
\end{proof}

La proposition \ref{GRS} se déduit directement du lemme précédent et des formules des théorèmes \ref{g1rs} et \ref{cartesgk}. 

\subsection{Où l'arête racine n'appartient pas à la forêt}
\label{hforet}

Dans la sous-section précédente, nous avons constaté que l'enracinement sur une feuille évite les considérations sur la composante racine, et donc facilite l'énumération. Cette idée se transpose naturellement sur l'ensemble des cartes forestières où l'arête racine\footnote{On rappelle que l'arête racine est l'arête qui suit la racine dans le sens trigonométrique.} n'appartient pas à la forêt (voir figure \ref{forestiereout}). Appelons $H(z,u,t)$ la série génératrice de telles cartes, où la variable $z$ compte les faces, $u$ les composantes non racine et $t$ les arêtes. Les sommets de degré $k$ sont toujours pondérés par une variable $d_k$.

\fig{[scale = 1.1]{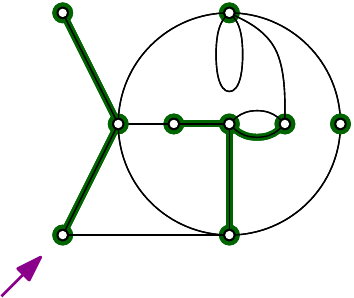}}{Une carte forestière dans laquelle l'arête racine n'appartient pas à la forêt. A contrario, les arêtes racines des cartes des figures \ref{forestierenormale} et \ref{forestierefeuille} sont dans la forêt.}{forestiereout} 

\begin{prop} 
\label{HRS}
La série génératrice des cartes forestières où l'arête racine n'appartient pas à la forêt est 
\begin{multline} H(z,u,t) = \frac z {ut} R + \frac z {ut} S^2 - \frac {z^2} {u} \\
-  \frac 2 t\,  S\, \sum_{\substack{i \geq 2 \\ j \geq 0}}  T_{2i+j-1}(t) { 2i - j - 2 \choose i,i-2,j} R^i S^j - \frac 1 t \sum_{\substack{i \geq 3 \\ j \geq 0}}  T_{2i+j-2}(t) { 2i - j - 3 \choose i,i-3,j} R^i S^j. \label{H&M}
 \end{multline}
 On rappelle que les séries $R$ et $S$ sont définies par \eqref{forestiere} et \eqref{forteresse} et $T_\ell(t)$ par \eqref{deftl}.
\end{prop}

Là encore, la transformation en cartes décorées nous permet d'exprimer $H$ en termes d'une série déjà étudiée : ici, il s'agit de la série génératrice $M(z,u,t;g_1,g_2,\dots)$ des cartes planaires où les faces sont pondérées par $z$, les arêtes par $t$ et les sommets de degré $k$ par $u g_k$.

\begin{lem} \label{HM}
 L'équivalent de \eqref{femme} pour les cartes forestières   où l'arête racine n'appartient pas à la forêt est 
$$H(z,u,t) =  \frac 1 u M(z,u;t^{1/2} T_1(t),t T_2(t),t^{3/2} T_3(t), \dots).$$
\end{lem}
\begin{proof} Pour une carte $C$ décorée par des arbres, on note $\hat T_C$ l'arbre associé au sommet racine de $C$.
Il est facile de décrire l'image par $\Phi$ (voir la preuve du théorème \ref{tfemme} pour la définition de $\Phi$) des cartes forestières dans lesquelles l'arête racine n'est pas dans la forêt. En effet, il s'agit des cartes $C$ décorées par des arbres telles que la demi-arête qui suit la racine dans $\hat T_C$ est une patte. Nous pouvons alors enraciner $\hat T_C$ sur cette patte sans perte d'information. \`A un facteur $u$ près, qui provient du fait que la composante racine n'est pas pondérée dans $H(z,u,t)$, nous retrouvons alors la série génératrice $M(z,u;t^{1/2} T_1(t),t T_2(t),t^{3/2} T_3(t), \dots)$ des cartes planaires où chaque sommet de degré $k$ est pondéré par le nombre d'arbres à $k$ pattes enracinés sur une patte.
\end{proof}

Le lemme précédent et la caractérisation de la série $M$ (voir théorème \ref{cartesgk}) prouvent directement la relation \eqref{H&M}.

La série $H$ est étroitement liée à la série génératrice $F$ des cartes forestières "normales", comme l'indique la relation suivante.

\begin{prop} \label{FH} Soit $F_{f,k,a}$ (resp. $H_{f,k,a}$) le coefficient de $z^f u^k t^a$ dans la série $F(z,u,t)$ (resp. $H(z,u,t)$). Alors
$$a \,H_{f,k,a} = (f + k - 1)\,F_{f,k,a}.$$
\end{prop}
\begin{proof} On utilise un argument de double comptage. En effet, chaque membre de l'égalité du dessus dénombre les cartes forestières doublement enracinées avec $f$ faces, $k+1$ composantes et $a$ arêtes, où la première racine se trouve sur un coin qui précède (dans le sens trigonométrique) une arête hors forêt (nombre de tels coins\footnote{La formule qui suit se prouve par récurrence.} : $2(f + k - 1)$), et où la seconde racine se trouve sur un coin quelconque (nombre de tels coins : $2a$).
\end{proof}

\section{Cartes forestières eulériennes}
\label{seulerien}

Cette section est consacrée à l'étude des cartes forestières eulériennes et des équations qui les régissent. En effet, de nombreuses simplifications s'opèrent dans ce contexte particulier. Nous allons voir lesquelles.

Une carte est \textit{eulérienne} si le degré de chaque sommet est pair. En termes de poids sur les sommets, cela veut dire que $d_{2k+1}=0$ pour chaque $k \geq 0$ ; c'est ce qui sera supposé dans cette section.

\fig{[scale = 1.1]{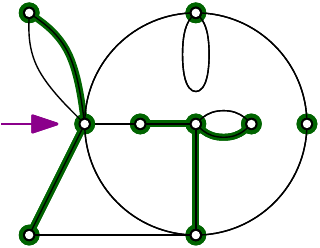}}{Une carte forestière eulérienne.}{forestiereeulerienne} 

Nous pouvons dans un premier temps remarquer qu'il n'existe pas d'arbres avec un nombre impair de pattes et un poids $d_k$ pour chaque sommet de degré $k$. En d'autres mots, la série génératrice $T_{2\ell+1}(t)$ des arbres à $2 \ell + 1$ pattes est nulle pour chaque $\ell \geq 0$.  Par conséquent, comme indiqué dans la section \ref{s:bourgeonnants}, la série $S$ apparaissant dans \eqref{forteresse} est elle aussi nulle 
$S = 0$.
Par conséquent le théorème \ref{central} se lit :

\begin{prop} \label{teulerien}
Supposons que $d_{2k+1} = 0$ pour tout $k \geq 0$. Soient $\theta$ et $\phi$ les séries 
$$\theta(t,x) =  \sum_{i \geq 0} \left(\frac {t} {i} T'_{2i}(t) + T_{2i}(t)\right) {2i \choose i}  x^i,$$
$$
\phi(t,x) = \sum_{i \geq 1} T_{2i}(t) { 2i-1 \choose i } x^i.
$$

Il existe une unique série formelle $R$ en $z$, $u$, $t$, $d_2,d_4,d_6\dots$ et solution de l'équation
\begin{equation}
\label{eulR}
R = t\,z + t\,u\,\phi(t,R).
\end{equation}
La série génératrice $F(z,u,t)$ des cartes forestières eulériennes est décrite par les relations $F(0,u,t) = 0$ et 
\begin{equation} \label{eulF}
\frac {\partial F} {\partial z} (z,u,t) = \theta(t,R).
\end{equation}
\end{prop}

De même, la formule \eqref{H&M} se réécrit 
\begin{equation}
H(z,u,t) = \frac z {ut} R  - \frac {z^2} {u} 
 - \frac 1 t \sum_{i \geq 3 }  T_{2i-2}(t) { 2i - 3 \choose i} R^i ,
\end{equation}
où $H$ est la série génératrice des cartes forestières dans lesquelles l'arête racine n'appartient pas à la forêt (voir sous-section \ref{hforet}).

De plus, en combinant les relations \eqref{eqmp} et \eqref{eulR}, nous obtenons pour toute suite de poids $(g_{2k})_{k \geq 1}$ l'égalité
$$t \frac{\partial \overline M}{\partial z} (z,u;0,t\,T_2(t),0,t^2\,T_4(t),0,\dots;0,t\,T_2(t),0,t^2\,T_4(t),0,\dots) = 2(R - tz),$$
où $\overline M(z,u;g_1,g_2,\dots,h_1,h_2,\dots)$ est la série génératrice des cartes planaires et où $R(z,u,t)$ est toujours la solution de l'équation \eqref{eulR}. D'après le lemme \ref{HM}, le premier terme de cette égalité vaut également $ ut \pd H z(z,u,t)$ . 
Nous obtenons alors la formule très simple
\begin{equation} \label{eulHR}
\pd H z(z,u,t)= \frac 2 {ut} (R - tz).
\end{equation}

Cette relation présente un avantage : si on connaît les coefficients de $R$ (qui peuvent être calculés rapidement avec la méthode de Newton appliquée à l'équation \eqref{eulR}), on peut déduire immédiatement les coefficients de $H$, et donc de $F$ (grâce à la proposition \ref{FH}).

\noindent \textbf{Remarque.} Il n'existe pas d'équivalent à la proposition \ref{GRS} compte tenu qu'une carte eulérienne n'a pas de feuille !

\section{Spécialisation des équations à différentes classes de cartes}
\label{spanorama}

Dans cette section, nous dressons la liste des équations portant sur $F(z,u,t)$, la série génératrice  des cartes forestières, parmi différentes classes de cartes. Nous voulons en effet montrer un panorama des équations que nous verrons dans la suite de ce mémoire. La figure \ref{recapitulatif} donne un tableau récapitulatif de toutes ces équations. %Nous avons essayé de les ordonner de manière à mettre en évidence les strates de difficulté qui apparaissent à chaque fois.

\fig{[width=\textwidth]{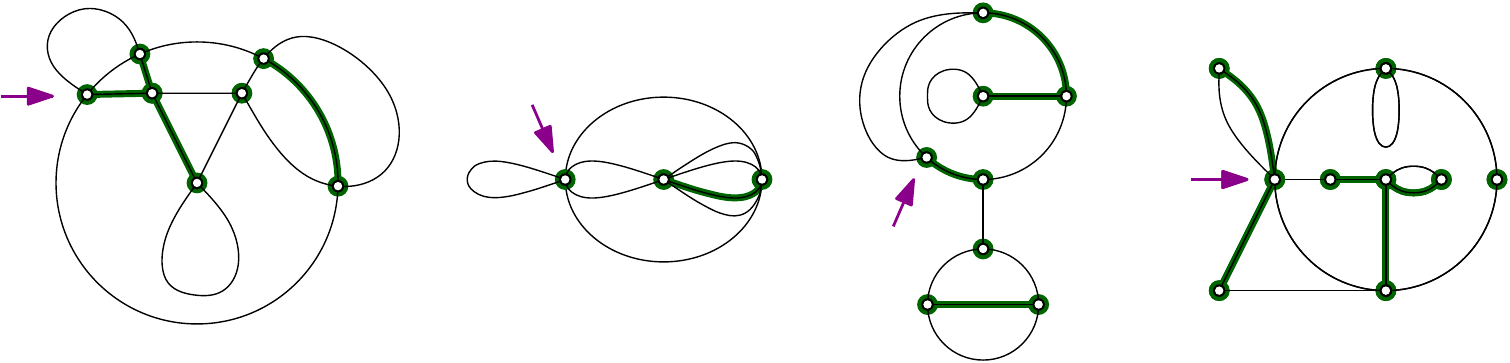}}{De gauche à droite : une carte forestière tétravalente, une carte forestière $6$-régulière, une carte forestière cubique, une carte forestière eulérienne.}{tetrex}

\subsection{Cartes forestières régulières}

Nous allons voir ce que deviennent les équations du théorème \ref{central} dans le cadre particulier des cartes forestières \textit{$p$-régulières}. 

Une carte est dite \textit{$p$-régulière} (ou \textit{régulière}) si chaque sommet a pour degré $p$, avec $p \geq 3$. En termes de variables de poids, cela veut dire que $d_p = 1$ et $d_k = 0$ pour $k$ différent de $p$.

En utilisant la relation d'Euler et le fait que chaque sommet est incident à exactement $p$ demi-arêtes, nous prouvons  la relation
\begin{equation} \label{relreg}
2 (p-2) \arete(C) = p (p-2)  \som(C) = 2 p (\face(C) - 2)
\end{equation} 
pour toute carte $p$-régulière $C$, où $\arete(C)$, $\som(C)$ et $\face(C)$ désignent respectivement le nombre d'arêtes, sommets, faces de $C$. Par conséquent, la variable des arêtes est redondante par rapport à la variable des faces ; nous pouvons la supprimer (c'est-à-dire la substituer par $1$)\footnote{Au vu des équations du théorème \ref{central}, il est plus adéquat de compter les cartes forestières selon les faces que selon les arêtes.}. Nous notons donc $F(z,u)$ la série génératrice des cartes forestières $p$-régulières dans laquelle la variable $z$  représente les faces et $u$ les composantes non racine.

Un arbre $p$-régulier à $\ell$ pattes n'existe que si $\ell = (p-2)k + 2$ avec $k \geq 1$. Le nombre $T_\ell$ de tels arbres est explicite \cite[Théorème 5.3.10]{stanley-vol-2} et vaut $$T_{\ell} = \frac {((p-1)k)!}{k!((p-2)k+1)!}$$ si $\ell = (p-2)k + 2$ et vaut $0$ sinon. 
 Quant à 
$(2 t T'_\ell(t) /\ell + T_\ell(t))$, sa valeur en $t=1$ est après calcul
$$T^c_\ell= p  \frac {((p-1)k)!}{(k-1)!((p-2)k+2)!},$$
si  $\ell = (p-2)k + 2$ avec $k \geq 1$, et  $0$ dans le cas contraire. Nous rappelons que cela correspond au nombre d'arbres à $\ell$ pattes enracinés sur un coin.

 On déduit alors du théorème central \ref{central} la proposition suivante.
 
\begin{prop} \label{regulierecentral}
Soient $\theta$, $\phi_1$ et $\phi_2$ les séries 
$$\theta(x,y) =  \sum_{i \geq 0} \sum_{j \geq 0}  T^c_{2i+j} {2i+j \choose i,i,j}  x^i y^j,$$
\begin{equation} \label{rphi}
\phi_1(x,y) = \sum_{i \geq 1} \sum_{j \geq 0} T_{2i+j} { 2i+j-1 \choose i-1,i,j } x^i y^j, \quad
\phi_2(x,y) = \sum_{i \geq 0} \sum_{j \geq 0} T_{2i+j+1} { 2i+j \choose i,i,j } x^i y^j,
\end{equation}
où $T_\ell$ et $T^c_\ell$ sont précédemment définis.

Il existe un unique couple $(R,S)$ de séries formelles  en $z$ et $u$  qui satisfait
$$
R = z + u\,\phi_1(R,S),\quad S = u\, \phi_2(R,S).
$$

La série génératrice $F(z,u)$ des cartes forestières $p$-régulières est caractérisée par $F(0,u) = 0$ et 
$$
\frac {\partial F} {\partial z} (z,u) = \theta(R,S).
$$
\end{prop} 

\noindent \textbf{Remarque 1.} L'unicité du couple $(R,S)$ ne peut pas être déduite du théorème \ref{central} telle quelle. Par contre, elle est une application directe du lemme \ref{caracteresse}.

\noindent \textbf{Remarque 2.} La proposition précédente permet de calculer les premiers termes de la série $F$. Par exemple pour $p=4$,
$$
F(z,u) = 2 z^3 + (9 u + 15) z^4 + (54 u^2 + 180 u + 168 ) z^5 + O(z^6),
$$
et pour $p=3$,
$$F(z,u) =  (6 + 4 u) z^3 + (140 + 234 u + 144 u^2 + 32 u^3) z^4 + O(z^5).$$
%Il est intéressant de remarquer que le développement des cartes forestières cubiques est visiblement plus compliqué que celui des cartes forestières tétravalentes\footnote{pour cause : les cartes tétravalentes sont eulériennes, les cartes cubiques ne le sont pas !}.

\noindent \textbf{Remarque 3.} Nous avons ici une expression explicite 
%du nombre d'arbres à pattes. \na{!} Nous pouvons donc calculer une expression explicite
 de $F(z,0)$, la série génératrice des cartes équipées d'un arbre couvrant :
\begin{equation} \label{Fz0reg}
F(z,0) = \sum_{k \geq 1} \frac{p((p-1)k)!} {(k-1)! (1 + (p-2)k/2)! (2 + (p-2)k/2)!} z^{2 + (p-2)k/2},
\end{equation}
où  $k$ est restreint aux entiers pairs quand $p$ est impair.

\noindent \textbf{Remarque 4.} Lorsque $p$ est pair, les cartes deviennent eulériennes. Par conséquent, la série $S$ disparaît -- voir  section \ref{seulerien}. Les équations obtenues figurent dans le tableau récapitulatif de la figure \ref{recapitulatif}. 

Quand $p = 2q$ avec $q \geq 3$, nous remarquons que $F(z,u)$ n'est plus apériodique. En effet, le nombre de faces dans une carte planaire $2q$-régulière est de la forme $2 + (q-1)k$, avec $k \geq 1$, comme le montre la relation \eqref{relreg}. Nous pouvons par exemple l'observer sur les premiers termes de la série génératrice des cartes forestières $6$-régulières : 
$$F(z,u) = 5z^4+(100u+252)z^6+(3000 u^2 + 14175 u+ 19305) z^8 + O(z^{10}).$$

\noindent \textbf{Remarque 5.} Lorsque $p$ est impair, les simplifications du cas eulérien ne s'opèrent pas. En particulier, les séries $\phi_1$ et $\phi_2$ définies par \eqref{rphi} sont bivariées en $x$ et $y$. Toutefois quand $p=3$, ces séries peuvent être réduites à des séries à une seule variable. En effet, on peut vérifier que 
\begin{equation} \label{phi1cubique}
\phi_1(x,y) = (1-4y)^{3/2}\,\psi_1\left(\frac x {(1-4y)^2}\right) - x,
\end{equation}
\begin{equation}
\phi_2(x,y) = \sqrt{1-4y}\,\psi_2\left(\frac x {(1-4y)^2}\right) + \frac 1 4 \, (1 - \sqrt{1 - 4y})^2,
\end{equation}
où $\psi_1$ et $\psi_2$ sont les séries monovariées
\begin{equation} \label{psicubique}
\psi_1(z) = \sum_{i \geq 1} \frac{(4i-4)!}{i!(i-1)!(2i-2)!}z^i \quad\textrm{et}\quad \psi_2(z) = \sum_{i \geq 1} \frac{(4i-2)!}{i!^2(2i-1)!}z^i.
\end{equation}
Cette réécriture astucieuse nous sera salvatrice pour la suite de ce mémoire -- l'étude des séries à plusieurs variables est beaucoup plus délicate que celle des séries à une seule variable.

\noindent \textbf{Remarque 6.} Le cas $p=2$ est trivial. En effet, les cartes $2$-régulières sont les cycles et la série génératrice des cycles forestiers avec un poids $u$ par composante non racine et un poids $t$ par arête\footnote{Ici nous sommes obligés  de compter selon les arêtes plutôt que les faces car le nombre de faces est égal à $2$.} est égale à\footnote{Expliquons pourquoi les cycles forestiers sont comptés par $\frac t {(1 - t)(1 - (1+u)t)}$. En partant de la racine, nous regardons quelles suites d'arêtes internes/externes nous pouvons obtenir en faisant le tour du cycle (par exemple dans le sens trigonométrique). Cette suite commence par une succession d'arêtes internes (poids $1/1-t$), qui constituent une partie de la composante racine (l'autre partie étant de l'autre côté de la racine). Ces premières arêtes sont forcément suivies par une arête externe (poids $t$) -- nous ne voulons pas que les arêtes internes forment un cycle. Après quoi, nous visitons une succession (potentiellement vide) d'arêtes internes/externes. Une arête interne prolonge la composante qui précède donc est simplement pondérée par $t$, tandis qu'une arête externe clôt une composante donc est pondérée par $ut$.  Les dernières arêtes internes que nous visitons ne sont pas pondérées par $u$ car elles appartiennent à la composante racine. Ce raisonnement nous mène à la série génératrice attendue.}
\begin{equation} \label{cycle}
F(u,t) = \frac t {(1 - t)(1 - (1+u)t)}.
\end{equation}

\subsection{Cartes forestières eulériennes et $\boldsymbol 4$-eulériennes}

Nous supposons maintenant que $d_{2k-1}=0$ et $d_{2k}=1$ pour tout $k \geq 1$. Nous obtenons alors l'ensemble des cartes eulériennes considérées uniformément, c'est-à-dire sans tenir compte des poids.

Pour appliquer la proposition \ref{teulerien}, nous devons d'abord calculer le nombre d'arbres à pattes eulériens.

\begin{lem} \label{l:euln} Pour tout $i \geq 1$ et $j \geq 0$, le nombre d'arbres eulériens à $2i$ pattes et $j$ arêtes, enracinés sur une pattes, est 
$$
\frac 1 {2i-1}  { 2i + j - 1  \choose j + 1 } { i + j - 1 \choose j }.
$$
\end{lem}
\begin{proof} Soit $E(x,y)$ la série génératrice des arbres eulériens à pattes comptés selon le nombre de pattes non racine ($x$) et le nombre d'arêtes plus un ($y$). Nous considérons l'arbre réduit à une patte comme un arbre en tant que tel. La décomposition d'un arbre à pattes à la racine montre que
$$
E = x + y \frac E {1 - E^2}.
$$
(En effet, soit l'arbre est réduit à une feuille, soit un nombre impair d'arbres à pattes sont attachés au sommet racine.)
D'après le théorème d'inversion de Lagrange \cite[\'Equation 3.6.6 p. 14]{AS}\footnote{Cette version n'est pas forcément la plus connue -- elle énonce que si $f(x,y) = x + y \, \phi(f(x,y))$, alors $$f(x,y) = x + \sum_{k \geq 1} \frac{y^k} {k!} \pare{\pd \, x}^{k-1} \pare{\phi(x)^k}.$$ } , nous avons
$$
E(x,y) = x + \sum_{j \geq 1} \frac {y^j} {j!} \left( \frac \partial {\partial x} \right)^{j-1} \left( \frac{x}{1-x^2} \right)^{j}.
$$
Après avoir utilisé l'identité
$$ (1 - x^2)^{-j} = \sum_{i \geq 0} { i + j - 1 \choose i }
 x^{2i},$$
nous obtenons
$$E(x,y) = x  + \sum_{i \geq 0} \sum_{j \geq 1} \frac  1 {2i+1}  { 2i + j  \choose j } { i + j - 1 \choose i } x^ {2i+1} y^j,$$
ce qui donne le cardinal voulu en remplaçant $i$ par $i-1$ et $j$ par $j+1$.
\end{proof}

D'après le lemme précédent, la série $T_{2i}(t)$ vaut 
$$ \sum_{i \geq 1} \sum_{k \geq 0} \frac 1 {2i-1}  { 2i + j - 1  \choose j + 1 } { i + j - 1 \choose j } x^i t^k.$$
Donc d'après la proposition \ref{teulerien}, la série génératrice $F(z,u,t)$ des cartes forestières eulériennes est solution du système :
$$\pd F z = \theta(t,R),\quad R = tz + tu\phi(t,R)$$
où $\theta$ and $\phi$ sont définies par
$$
\theta(t,x) = 2 \sum_{i \geq 1} \sum_{k \geq 0} \frac{ (2i+k-1)!(i+k)!      } {i!^2(i-1)!k!(k+1)!} x^i t^k,
$$
$$
\phi(t,x) = \sum_{i \geq 1} \sum_{k \geq 0} \frac{ (2i+k-1)!(i+k-1)!      } {i!(i-1)!^2k!(k+1)!} x^i t^k,
$$
les coefficients de $\phi$ et de $\theta$ étant déduits du lemme précédent.

Comme $d_2 \neq 0$, il existe une infinité de cartes eulériennes avec un nombre de faces fixé. Autrement dit, quel que soit l'entier $f \geq 2$, le coefficient de $z^f$ dans $F(z,u,t)$ n'est pas un polynôme en $t$. Par exemple, le coefficient de $z^2$ compte les cycles munis d'une forêt couvrante, dont la série génératrice est décrite par \eqref{cycle}.

Toutefois, le coefficient de $t^a$ dans $F(z,u,t)$ est un polynôme en $z$, pour tout $a \geq 1$. La variable naturelle pour $F$ est donc la variable des arêtes $t$. Montrons par exemple le développement de $F$ à l'ordre $3$ en $t$ :
$$F(z,u,t) = tz^2+((2+u)z^2+2z^3)t^2+ ((3+3u+u^2)z^2+(12+6u)z^3+5z^4)t^3 + O(t^4).$$
Néanmoins l'étude asymptotique de $F$ par rapport à la variable des arêtes n'est donc pas chose aisée. En effet, il est difficile d'exploiter le système ci-dessus sachant qu'il fait intervenir la dérivée de $F$ par rapport à la variable des faces $z$. 

C'est pourquoi nous allons contourner le problème en étudiant plutôt les cartes forestières \textit{$4$-eulériennes}. Une carte $4$-eulérienne est une carte où le degré de chaque sommet est  pair, et au moins égal à $4$. Comme nous avons interdit les sommets de degré $2$, il n'existe qu'un nombre fini de cartes forestières $4$-eulériennes avec un nombre de faces fixé. Ces cartes peuvent donc être énumérées selon la variable $z$. 

Les cartes $4$-eulériennes ne sont pas des objets si éloignés des cartes eulériennes, comme en témoigne le lemme suivant.
   
\begin{prop} \label{correspondanceeul}
La série génératrice $F(z,u,t)$ des cartes forestières eulériennes et la série génératrice $F_4(z,u,t)$ des cartes forestières $4$-eulériennes sont liées par la relation
$$F(z,u,t) =   F_4 \pare{\frac{z} {1- t(1+u)},\frac{u} {1- t(1+u)},\frac t {1-t}}  +  \frac {z^2 t} {(1 - t)(1-t(1+u))} .$$
\end{prop}
\begin{proof} Dans une carte enracinée eulérienne qui n'est pas un cycle, les sommets de degré $2$ peut être organisés en chaînes maximales, connectées par des sommets de degré au moins $4$ en leurs extrémités.

\fig{[width=\textwidth]{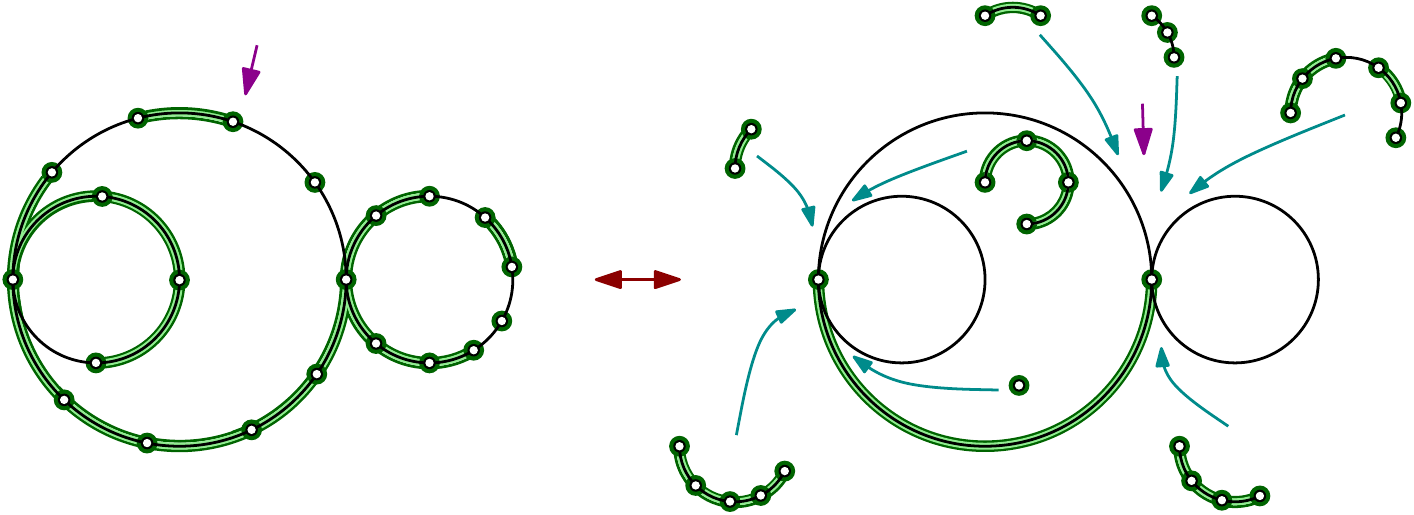}}{Correspondance entre cartes forestières eulériennes et cartes forestières $4$-eulériennes.}{4eul}

Considérons $C$ une carte forestière eulérienne qui n'est pas un cycle. Nous remplaçons chaque chaîne maximale par une arête. Si cette chaîne comportait une arête externe, alors la nouvelle arête est externe\footnote{Cette condition est nécessaire pour ne pas créer de cycles au sein de la forêt.} ; si elle n'était constituée que d'arêtes internes, alors la nouvelle arête est interne. Si la racine était au milieu d'une chaîne maximale $m$, nous réenracinons la carte sur le coin précédant $m$. Nous obtenons ainsi une carte forestière $4$-eulérienne. 

Décrivons maintenant l'application inverse. Soit $C'$ une carte forestière $4$-eulé\-rienne. On insère avant\footnote{Le terme "avant" présuppose que l'arête est orientée. Cette orientation peut être arbitraire.} chaque arête $e$ de $C'$ une  chaîne arbitrairement longue (et potentiellement vide) d'arêtes internes (poids en plus par arête : $1/(1-t)$). Si  $e$ était de surcroît externe, nous rajoutons après $e$ une chaîne d'arêtes  qui peuvent être individuellement internes ou externes de longueur également arbitrairement longue (poids en plus par arête externe : $1/(1-t(1+u))$, une arête interne ne change pas le nombre de composantes connexes, une arête externe l'augmente de $1$). Enfin, nous rajoutons une dernière chaîne d'arêtes entre la racine et l'arête qui précède la racine (poids :  $1/(1-t(1+u))$). Nous obtenons ainsi une carte forestière eulérienne qui n'est pas un cycle. La figure \ref{4eul} illustre cette correspondance.

L'ensemble de ces cartes est énuméré par $F - \frac{z^2 t} {(1 - t)(1-t(1+u))}$ (voir équation \eqref{cycle}).  De plus, nous pouvons prouver par récurrence sur la taille qu'une carte forestière $4$-eulérienne avec $n$ faces et $k$ composantes non racine comporte exactement $n + k - 1$ arêtes externes. Par conséquent, d'après la correspondance décrite ci-dessus, ce même ensemble est énuméré par
$$
\frac 1 {1-u(1+t)} \pare{1-u(1+t)} F_4 \pare{\frac{z} {1- t(1+u)},\frac{u} {1- t(1+u)},\frac t {1-t}},
 $$
ce qui prouve l'égalité voulue. 
\end{proof}

En utilisant l'inversion de Lagrange à l'équation $E = x + y \frac{E^3}{1- E^2}$, nous pouvons établir l'équivalent du lemme \ref{l:euln} pour les cartes $4$-eulériennes :

\begin{lem} \label{l:4euln} Pour tout $i \geq 1$ et $i - 2 \geq j \geq 0$, le nombre d'arbres à $2i-1$ pattes and $j$ arêtes sans sommet  de degré $2$ est 
$$
\frac 1 {2i-1}  { 2i + j - 1  \choose j + 1 } { i - 2 \choose j }.
$$
\end{lem}

Nous pouvons alors appliquer la proposition \ref{teulerien} et trouver un système fonctionnel caractérisant la série génératrice des cartes forestières $4$-eulériennes. Nous retrouvons ce système dans le tableau récapitulatif de la figure \ref{recapitulatif}.

\begin{figure}[h!]

\begin{center}
\begin{tabular}{|c|c|c|c|} \hline
\begin{minipage}{0.13 \textwidth} \vspace{0.15cm} \begin{center}
Classe de cartes
\end{center} \vspace{-0.2cm} \end{minipage} & Système & Fonctions $\theta$ et $\phi$ & Commentaires \\ \hline
\begin{minipage}{0.13 \textwidth} \begin{center}
Cartes forestières tétra\-valentes
\end{center} \end{minipage} & \begin{minipage}{0.15 \textwidth} \footnotesize $$\pd F z = \theta(R) $$ $$ R = z + u \phi(R) $$\end{minipage} & \begin{minipage}{0.38 \textwidth} $$\theta(x) = 4 \,\sum_{i \geq 2} \frac {(3i-3)!} {i!^2\,(i-2)!} x^i $$ $$ \phi(x) = \sum_{i \geq 2} \frac {(3i-3)!} {i!\,(i-1)!^2} x^i$$ \end{minipage}  & \begin{minipage}{0.2 \textwidth} \vspace{0.15cm} \ding{51} Pas de série $S$ \\
\ding{51}\,Séries mono\-variées \\ \ding{51} Apériodicité \\ \ding{51}\,Cas le plus simple \end{minipage}  \\ \hline
\begin{minipage}{0.13 \textwidth} \begin{center}
Cartes forestières $2q$-régulières, $q \geq 3$
\end{center} \end{minipage} & \begin{minipage}{0.15 \textwidth} \footnotesize $$\pd F z = \theta(R) $$ $$ R = z + u \phi(R) $$\end{minipage} & \begin{minipage}{0.38 \textwidth} \footnotesize $$\theta(x) = 2q\,\sum_{i \geq 1}  \frac{((2q-1)i)! x^{(q-1)i+1}}{ (i-1)! ((q-1)i+1)!^2 } $$ $$ \phi(x) = \sum_{i \geq 1} \frac{((2q-1)i)! x^{(q-1)i+1}}{ i! ((q-1)i)! ((q-1)i+1)!  } $$ \end{minipage}  & \begin{minipage}{0.2 \textwidth} \vspace{0.15cm} \ding{51} Pas de série $S$ \\
\ding{51}\,Séries mono\-variées \\ \ding{55}~Période $q-1$, supérieure à $1$ \\ \ding{55}\:Dépend d'un paramètre $q$ \end{minipage}  \\ \hline
\begin{minipage}{0.13 \textwidth} \begin{center}
Cartes forestières cubiques
\end{center} \end{minipage} & \begin{minipage}{0.165 \textwidth} \footnotesize $$\pd F z = \theta(R,S)$$ $$\hspace{-0.1cm} R = z + u \phi_1(R,S)$$ $$S = u \phi_2(R,S)$$ \end{minipage} & \begin{minipage}{0.385 \textwidth} \footnotesize $$\theta(x,y) = 3 \sum_{\substack{i \geq 0 \\ j \geq 0 \\ 2i+j \geq 3}} \frac {(4i+2j-4)!} {i!^2j!(2i+j-3)!} x^i y^j
$$ $$ \hspace{-0.05cm}
\phi_1(x,y) =  \sum_{\substack{i \geq 0 \\ j \geq 0 \\ 2i+j \geq 3}} \frac {(4i+2j-4)!x^i y^j } {i!(i-1)!j!(2i+j-2)!} 
$$ $$
\phi_2(x,y) = \sum_{\substack{i \geq 0 \\ j \geq 0 \\ 2i+j \geq 2}} \frac {(4i+2j-2)!} {i!^2j!(2i+j-1)!} x^i y^j
$$ \end{minipage}  & \begin{minipage}{0.2 \textwidth} \vspace{0.15cm} \ding{51} Apériodicité  \\ \ding{51}\,Séries biva\-riées mais réductibles à des séries monovariées  \\
\ding{55}\,Retour de la série $S$ \\ \ding{55} Système fonctionnel avec deux équations couplées\end{minipage}  \\ \hline
\begin{minipage}{0.13 \textwidth} \begin{center}
Cartes forestières eulériennes non pondérées
\end{center} \end{minipage} & \begin{minipage}{0.165 \textwidth} \footnotesize $$\pd F z = \theta(t,R) $$ $$\hspace{-0.035cm} R = tz + tu \phi(t,R) $$\end{minipage} & \begin{minipage}{0.390 \textwidth} \footnotesize $$
\theta(t,x) = 2 \sum_{\substack{ i \geq 1 \\ k \geq 0 }} \frac{ (2i+k-1)!(i+k)!      } {i!^2(i-1)!k!(k+1)!} x^i t^k
$$
$$ \hspace{-0.1cm}
\phi(t,x) = \sum_{\substack{ i \geq 1 \\ k \geq 0 }} \frac{ (2i+k-1)!(i+k-1)!      } {i!(i-1)!^2k!(k+1)!} x^i t^k\hspace{-0.3 cm} \ $$
\end{minipage}  & \begin{minipage}{0.2 \textwidth} \vspace{0.15cm} \ding{51} Pas de série $S$ \\
\ding{51} Apériodicité \\ \ding{55} Variable naturelle $= t$ (arêtes) \\ \ding{55} Séries bivariées\end{minipage}  \\ \hline
\begin{minipage}{0.13 \textwidth} \begin{center}
Cartes forestières $4$-eulé\-riennes non pondérées
\end{center} \end{minipage} & \begin{minipage}{0.165 \textwidth} \footnotesize $$\pd F z = \theta(t,R) $$ $$\hspace{-0.035cm} R = tz + tu \phi(t,R) $$\end{minipage} & \begin{minipage}{0.390 \textwidth} \footnotesize \vspace{-0.3cm}
\begin{multline*}
\theta(t,x) =  2 \times \\
 \hspace{-0.9cm} \sum_{\substack{ i \geq 2 \\  0 \leq k \leq i - 2  }} \hspace{-0.05cm} \frac{ (i+k)(2i+k-1)!(i-2)!} {i!^2(i-k-2)!k!(k+1)!} x^i t^k \hspace{-0.3 cm} \ 
\end{multline*}
\vspace{-0.6cm}
\begin{multline*}
\phi(t,x) = \\
\hspace{-0.9cm} \sum_{\substack{ i \geq 2 \\  0 \leq k \leq i - 2  }} \hspace{-0.25cm} \frac{ (2i+k-1)!(i-2)!      } {i!(i-1)!(i-k-2)!k!(k+1)!} x^i t^k\hspace{-0.45 cm} \ \end{multline*}
\end{minipage}  & \begin{minipage}{0.2 \textwidth} \vspace{0.15cm} \ding{51} Pas de série $S$ \\ \ding{51} Apériodicité \\ \ding{51} On peut compter selon les faces. \\ \ding{55} Séries bivariées\end{minipage}  \\ \hline
 \end{tabular}
\end{center}

\caption{Tableau récapitulatif.}
\label{recapitulatif}

\end{figure}

\chapter{\'{E}quations différentielles}
\label{c:ed}

Comme nous l'avions dit en introduction (voir sous-section~ \ref{ss:nature} p.~\pageref{ss:nature}), il est intéressant de connaître la nature des séries génératrices étudiées (à savoir où se placent ces séries dans la hiérarchie rationnelle / algébrique / holonome / différentiellement algébrique). Dans ce chapitre, nous prouvons que la série génératrice des cartes forestières est différentiellement algébrique (sous des hypothèses raisonnables). Nous confirmons ainsi le résultat d'Olivier Bernardi et Mireille Bousquet-Mélou \cite{bernardi-mbm-de} qui traitait de l'algébricité différentielle de la série génératrice de Potts, mais uniquement pour des cartes planaires générales et pour les triangulations. Ce résultat étant d'abord obtenu de manière théorique, nous indiquons dans un second temps comment obtenir en pratique les équations différentielles.

\section{Algébricité différentielle  chez les cartes décorées}

L'algébricité différentielle sera établie dans un cadre plus général que celui des séries génératrices de cartes forestières. Nous allons montrer que sous de bonnes hypothèses d'holonomie, la série génératrice des cartes décorées (c'est-à-dire des cartes où chaque sommet est pondéré par un objet combinatoire tel un arbre ou une carte) est différentiellement algébrique. Nous appliquerons par la suite ce résultat à nos cartes forestières.

%\newcommand{\decore}[2]{$\pare{\mathcal #1,\mathcal #2}$}

%Soient $\mathcal G$ et $\mathcal H$ deux classes d'objets combinatoires (exemple : arbres planaires, cartes planaires, etc.). Nous supposons qu'à chaque objet $x$ de ces deux classes nous avons fait correspondre un entier naturel $n$ (exemple : nombre de feuilles chez les arbres planaires, degré du sommet racine chez les cartes planaires). Nous dirons alors que l'objet $x$ est \textit{indexé par $n$}. Nous autorisons à ce que le nombre d'objets indexés par un entier naturel fixé ne soit pas fini\footnote{c'est pourquoi nous n'employons pas le terme de "taille" ici}. 
%% Dans ce cas, nous rajouterons d'autres statistiques (exemple : nombre d'arêtes) de sorte que le nombre d'objets indexés par un entier à statistiques fixés soit fini.
%Une carte $C$ est dite \textit{décorée par \decore G H} si chaque sommet non racine de degré $k$ est muni d'un objet de $\mathcal G$ indexé par $k$ et si le sommet racine est muni d'un objet de $\mathcal H$ indexé par le degré du sommet racine. 
%
%Les cartes forestières constituent un exemple de cartes décorées. En effet, comme le prouve la proposition \ref{femme} p. \pageref{femme}, les cartes forestières correspondent aux cartes décorées par \decore T {T^c}, où $\mathcal T$ (resp. $\mathcal T^c$) désigne la classe des arbres à pattes enracinés sur une patte (resp. sur un coin) indexés selon le nombre de pattes.
%
\noindent \textbf{Rappel. }Une série (multivariée) est dite \textit{holonome} quand l'espace vectoriel engendré par toutes ses dérivées partielles est de dimension finie. Une série est dite \textit{différentiellement algébrique} en une variable $x$ si elle satisfait une équation différentielle polynomiale par rapport à $x$.

\begin{theo} \label{t:dalg}
%Soient $\mathcal G$ et $\mathcal H$ deux classes d'objets combinatoires. 
Soient  $g_k(z,u,t)$ et $h_k(z,u,t)$ des séries de variables $z$, $u$ et $t$ à coefficients rationnels indexées par $k \geq 1$. Notons $C(z,u,t)$ la série génératrice des cartes planaires avec un poids $z$ par face, un poids $u$ par sommet non racine, un poids $t$ par arête, un poids $g_k(z,u,t)$ par sommet non racine de degré $k$ et un poids $h_k(z,u,t)$ si le sommet racine a degré $k$.

%Appelons $C(z,u,t)$ la série génératrice des cartes décorées par \decore G H, comptées selon le nombre de faces ($z$), de sommets non racine ($u$) et d'arêtes ($t$). Les objets de $\mathcal G$ et $\mathcal H$ peuvent également comporter des faces, des sommets et des arêtes, de sorte qu'ils soient pondérés eux aussi par $z$, $u$ et $t$. Notons $G = \sum g_k(z,u,t) v^k$ (resp. $H = \sum h_k(z,u,t) v^k$) où $g_k(z,u,t)$ (resp. $h_k(z,u,t)$) désigne la série génératrice des objets de $\mathcal G$ (resp. $\mathcal H$) indexés par $k$. 

Si $\sum_{k \geq 1} g_k(z,u,t) v^k $ et $\sum_{k \geq 1} h_k(z,u,t) v^k$ sont deux séries holonomes (par rapport au jeu de variables $z,u,t,v$), alors $C$ est différentiellement algébrique en $t$, en $u$ et en $z$.
\end{theo}
\begin{proof} Nous prouvons ici l'algébricité différentielle de $C$ selon $t$. Celle selon $u$ et $z$ se démontre de manière similaire.

Nous utilisons le théorème \ref{vintegrale} p. \pageref{vintegrale} dû à Jérémie Bouttier et Emmanuel Guitter. Dans le cadre du présent théorème, le poids des sommets non racine de degré $k$ est $ t^{k/2} g_k(z,u,t)$ et le poids du sommet racine de degré $k$ est $t^{k/2} h_k(z,u,t)$. Le théorème \ref{vintegrale} se réécrit donc comme suit. Il existe deux séries $R$ et $S$ (attention, ce ne sont pas tout à fait les mêmes que celles du théorème \ref{vintegrale} -- la série $R$ a été multipliée par $t$ et $S$ par $\sqrt t$ comme dans la preuve du théorème \ref{tfemme} p. \pageref{tfemme}) telles que
\begin{equation} \label{firs}
R = t \,z + t \,u \,\phi_1(z,u,t,R,S), \quad S = t\, u\, \phi_2(z,u,t,R,S)
\end{equation}
avec
$$\phi_1(z,u,t,x,y) = \sum_{i \geq 1} \sum_{j \geq 0} g_{2i+j}(z,u,t) \frac{(2i+j-1)!}{i! (i-1)! j!} x^i y^j,$$
$$\phi_2(z,u,t,x,y) = \sum_{i \geq 0} \sum_{j \geq 0} g_{2i+j+1}(z,u,t) \frac{(2i+j)!}{(i!)^2 j!} x^i y^j.$$
La série $C$ vérifie alors
\begin{equation} \label{Calphabeta}
C = \alpha(z,u,t,R,S)-u\,\beta(z,u,t,R,S), 
\end{equation}
où $\alpha(z,u,t,x,y)$ et $\beta(z,u,t,x,y)$ sont deux séries définies par
$$
\alpha =  \sum_{i \geq 0} \sum_{j \geq 0} h_{2i+j}(z,u,t) \frac{(2i+j)!}{i! (i+1)! j!}  x^{i+1} y^j,
$$
$$
\beta =   \sum_{i, j, k, \ell \geq 0} \sum_{q=0}^{2i+j} h_{2i+j-q}(z,u,t) \, g_{2k+\ell+q+2}(z,u,t) \, \frac{(2i+j)!}{i! (i+1)! j!}  \frac{(2k+\ell)!}{(k!)^2 \ell!} x^{i+k} y^{j+\ell}.
$$
Commençons par prouver le lemme suivant.

\begin{lem}Les séries $\phi_1$, $\phi_2$, $\alpha$ et $\beta$ sont holonomes (pour le jeu de variables $z,u,t,x,y$).
\end{lem}
\begin{proof} Nous allons utiliser les propriétés de clôture des séries holonomes montrées par Lipshitz \cite{lipshitz-diag,lipshitz-df}. Nous allons plus précisément utiliser :
\begin{itemize}
\item la clôture par produit de Cauchy,
\item la clôture par produit de Hadamard (c'est-à-dire le produit terme à terme des coefficients),
\item la clôture par décalage d'indice.
\end{itemize}
Nous utilisons également le fait que les séries formelles dont les coefficients sont des produits de factorielles (par exemple $\sum i!/(j!(j-1)!(k+2)!) x^i y^i z^j$) sont holonomes\footnote{En effet, $\exp(x) = \sum_{i \geq 0} x^i/i!$ est trivialement holonome, et le reste se déduit par décalage d'indice et produit de Hadamard.}.
%On dit qu'une suite multi-indexée $(a_{n_1,\dots,n_m})$ est \textit{holonome} si la série génératrice correspondante est holonome.

Prouvons que si $\sum a_{n_1,\dots,n_m} \prod_{i=1}^m z^{n_i}$ est holonome, alors  $\sum a_{n_1,\dots,n_m+n_{m+1}} \prod_{i=1}^{m+1} z^{n_i}$  et $\sum a_{n_1,\dots,2n_m+n_{m+1}} \prod_{i=1}^{m+1} z^{n_i}$ sont holonomes. Soit $k \in \, \ens{1,2}$.
La série $$\sum_{n_1,\dots,n_m,n_{m+1},j} a_{n_1,\dots,n_{m-1},j}  \, w^j \,  \prod_{i=1}^{m+1} z_i^{n_i} $$ est holonome par hypothèse, tout comme $$\sum_{n_1,\dots,n_m,n_{m+1}} \pare{z_{m}\,w^k}^{n_m}  \, \pare{z_{m+1}\,w}^{n_{m+1}} \, \prod_{i=1}^{m-1} z_i^{n_i}$$
qui est rationnelle. En prenant le produit de Hadamard de ces deux séries par rapport aux variables $z_1,\dots,z_{m-1},w$, nous voyons que $$\sum_{n_1,\dots,n_m,n_{m+1}} a_{n_1,\dots,n_{m-1},kn_m+n_{m+1}} \, y^{kn_m+n_{m+1}} \prod_{i=1}^{m+1} z_i^{n_i},$$
est holonome, ce qui montre l'holonomie des deux séries.

En utilisant cela et la clôture par produit de Hadamard et par décalage d'indice, il est facile de montrer que $\phi_1$, $\phi_2$ et $\alpha$ sont holonomes. Par exemple, pour $\phi_1$, nous commençons par appliquer la propriété ci-dessus à la série (holonome) 
$$\sum_{k \geq 0} g_k(z,u,t) \, (k-1)! \, z^k$$
et ainsi montrer l'holonomie de 
$$\sum_{i \geq 0} \sum_{j \geq 0} g_{2i-j}(z,u,t) \,(2i+j-1)!\, x^i \,y^j.$$
Nous concluons en effectuant le produit de Hadamard de cette série et $$\sum_{i \geq 1} \sum_{j \geq 0} \frac 1 {i!(i-1)!j!}\, x^i \,y^j.$$

Passons à l'holonomie de $\beta$. Par clôture du produit de Cauchy,  $$\left(\sum_{q \geq 0} h_q(z,u,t) w^q \right) \times  \left(\sum_{m \geq 0}\sum_{q \geq 0} g_{m+q+2}(z,u,t) x^m w^q \right)$$ est holonome, ce qui peut se réécrire 
$$\sum_{m \geq 0}\sum_{p \geq 0} \left(\sum_{q=0}^p h_{p-q}(z,u,t) g_{m+q+2}(z,u,t) \right) x^m w^p.$$
Ainsi $$ \sum_{m \geq 0}\sum_{p \geq 0} p! m!  \sum_{q=0}^p h_{p-q}(z,u,t) g_{m+q+2}(z,u,t) x^m w^p $$ est une série holonome, donc d'après l'implication que nous avons prouvée au début du lemme,  $$(2i+j)! (2k+\ell)! \left(\sum_{q=0}^{2i+j} h_{2i+j-q}(z,u,t) \, g_{2k+\ell+q+2}(z,u,t) \right)$$ est également une suite holonome. Par suite, $\beta$ est bien holonome.
\end{proof}

La différentiation des deux relations de \eqref{firs} par rapport à $t$ donne un système linéaire en $\pd R t$ et $\pd S t$. Le déterminant de la matrice associée, qui est une série en $z$, $u$ et $t$ à coef\-ficients dans $\Q $, est non nul car son coefficient constant vaut $1$. Par conséquent, $\pd R t$ et $\pd S t$ s'expriment de manière rationnelle en termes de $z$, $u$ et des dérivées partielles $\pd {\phi_\ell} t$, $\pd {\phi_\ell} x$ et $\pd {\phi_\ell} y$, évaluées en $(z,u,t,R,S)$, avec $\ell \in \ens{1,2}$.

Notons $\K$ le corps $\Q(z,u,t)$. Grâce à la remarque précédente, il est facile de prouver par récurrence (en partant de la relation $C = \alpha(z,u,t,R,S)-u\,\beta(z,u,t,R,S)$) que pour tout $k \geq 0$, il existe une expression rationnelle de $\pdd k C t$ en fonction de 
$$\ens{\pdtrip {i+j+k} {\phi_\ell} {t^i \partial x^j \partial y^k} (z,u,t,R,S), \pdtrip {i+j+k} \alpha {t^i \partial x^j \partial y^k} (z,u,t,R,S), \pdtrip {i+j+k} \beta {t^i \partial x^j \partial y^k} (z,u,t,R,S)}_{i, j, k \geq 0, \ell \in \, \ens{1,2}}$$
à coefficients dans $\K$. 
Or, d'après le lemme précédent, les séries $\phi_1$, $\phi_2$, $\alpha$ et $\beta$ sont holonomes. Donc l'ensemble ci-dessus engendre un espace vectoriel de dimension finie sur $\K(R,S)$ notée $d$. En d'autres termes, il existe $d$ fonctions $f_1, f_2, \dots, f_d$ dans cet ensemble et des fonctions rationnelles $A_k \in \, \K(x,y,f_1,\dots,f_d)$ (indexées sur $\N$), telles que $\pdd k C t  = A_k(R,S,f_1,\dots,f_d)$ pour tout $k \geq 1$.

Mais comme le degré de transcendance (pour cette notion voir par exemple \cite[p. 254]{lang}) de $\K(R,S,f_1,\dots,f_d)$  est (au plus) $d+2$ sur $\K$, les $d+3$ séries $\pdd k C t$, avec $k \in \, \ens{0,\dots,d+2}$, sont algébriquement liées. En d'autre termes, $C$ est différentiellement algébrique.
\end{proof}

Comme nous le mentionnons, nous pouvons appliquer ce théorème aux cartes forestières.

\begin{cor} Supposons que la série $D(x) = \sum_k d_k x^{k-1}$ à coefficients dans $\Q$ soit algébrique. Alors la série génératrice $F(z,u,t)$ des cartes forestières avec un poids $d_k$ pour chaque sommet de degré $k$ est différentiellement algébrique par rapport à $z$, par rapport à $u$ et par rapport à $t$. 

Il en est de même pour la série génératrice $G(z,u,t)$ des cartes forestières enracinées sur une feuille et la série génératrice $H(z,u,t)$ des cartes forestières  dans lesquelles l'arête racine n'appartient pas à la forêt (voir propositions \ref{GRS} et \ref{HRS} p. \pageref{GRS} et \pageref{HRS}).
\end{cor}
\begin{proof}La proposition \ref{femme} montre que $F(z,u,t)$ est une spécialisation de la précédente série $C(z,u,t)$ avec $g_k(z,u,t) = T_k(t)$ et $h_k(z,u,t) = 2\,t\,T'_k(t)/k + T_k(t)$, où $T_k$ est la série définie par \eqref{deftl}. 

D'après le théorème précédent, il suffit  de prouver que  $\mathcal T(t,v) = \sum_{k \geq 1}  T_k(t) $ et $\mathcal T^c(t,v) = \sum_{k \geq 1}  \pare{2\,t\,T'_k(t)/k + T_k(t)}  v^{k}$  sont holonomes. Or d'après \eqref{deftl}, la série  $\mathcal T(t,v)$ satisfait
$$\mathcal T = D\pare{t \mathcal T + v}.$$
Par hypothèse, $D$ est algébrique : il existe donc un polynôme $P(x,y)$ à coefficients dans $\Q$ tel que $P(x,D(x)) = 0$. Ceci implique que
$$P(t \mathcal T + v,\mathcal T) = P\pare{t \mathcal T + v,D(t \mathcal T + v)} = 0.$$
Par conséquent, $\mathcal T$ est aussi algébrique et donc holonome. Quant à $\mathcal T^c$, les propriétés classiques de clôture sur l'holonomie permettent de conclure.

Pour la série  $H(z,u,t)$, il suffit d'utiliser le lemme \ref{HM} p. \pageref{HM}. On voit alors que $H$ est une spécialisation de $C(z,u,t)$ avec $g_k(z,u,t)=h_k(z,u,t)=T_k(t)$. Nous venons de prouver que $\sum_{k \geq 1}  T_k(t) v^k$ était holonome : la série $H$ est différentiellement algébrique en $z$, en $u$ et en $t$.

L'algébricité différentielle de la série  $G(z,u,t)$  est moins directe. Nous utilisons le lemme \ref{GM} p. \pageref{GM} : la série $G(z,u,t)$ est, à un facteur $\sqrt t (1 + 1/u)$ près, la série génératrice des cartes enracinées sur une demi-arête avec un poids $z$ par face, un poids $t$ par arête et un poids $u T_k(t)$ par sommet de degré $k$. La suppression de la demi-arête racine est réversible, pour peu qu'on réenracine sur le coin où était accrochée la demi-arête racine. On voit alors que la série génératrice des cartes obtenues par suppression de la demi-arête racine est $C(z,u,t)$ avec $g_k(z,u,t) = T_k(t)$ et $h_k(z,u,t) = u T_{k+1}(t)$. On prouve de cette manière que $G$ est différentiellement algébrique en $z$, en $u$ et en $t$.
\end{proof}

\noindent \textbf{Remarque. } Supposons que la suite des poids $(d_k)$ soit à valeurs dans $\ens{0,1}$ (cela revient à autoriser seulement certains degrés, sans pondération). D'après un théorème du à Carlson \cite{carlson}, si la série $D(x) = \sum_{k} d_k x^{k-1}$ est algébrique, alors elle est rationnelle. Dans ce cas-là, l'hypothèse d'algébricité du précédent corollaire est inutilement faible, nous pouvons la remplacer par une hypothèse de rationalité.

\section{Calcul explicite d'équations différentielles}

Nous appliquons  la méthode utilisée dans la section précédente pour construire des équations différentielles explicites dans différentes classes de cartes forestières et ainsi montrer que l'approche n'est pas seulement théorique.

\subsection{Le cas tétravalent}
\label{ss:tetraed}

Nous cherchons à expliciter les équations différentielles selon la variable $z$ comptant les faces dans le cas des cartes forestières tétravalentes\footnote{Rappel : $d_4 = 1$ et $d_k = 0$ pour $k \neq 4$, la variable des arêtes est alors redondante avec celle des faces.}. Nous adopterons la même approche que celle utilisée dans la preuve du théorème \ref{t:dalg}, à part que le raisonnement se basera sur l'équation \eqref{forfait} p.~\pageref{forfait} plutôt que sur l'équation \eqref{Calphabeta}. De manière générale, la fonction $\beta$ qui intervient dans cette dernière équation est beaucoup trop complexe\footnote{Par exemple, le simple calcul d'une équation différentielle linéaire pour $\beta$ semble un véritable tour de force !}  par rapport à $\theta$, si bien que nous préférons raisonner sur l'équation \eqref{forfait} quand nous travaillons avec la variable $z$. Cela revient à trouver des équations différentielles pour $\pd F z$ plutôt que $F$ elle-même.

Rappelons que dans le cas tétravalent (voir par exemple le tableau récapitulatif de la figure \ref{recapitulatif}), le système d'équations satisfait par la série génératrice $F$ des cartes forestières  tétravalentes est
$$R = z + u \phi(R), \quad  \pd F z = \theta(R),$$
où
\begin{equation}
 \phi(x) = \sum_{i \geq 2}
 \frac{(3i-3)!}{(i-1)!^2i!} x^i
  \quad \hbox{et} \quad 
 \theta(x) = 4 \sum_{i \geq 2} \frac{(3i-3)!}{(i-2)!i!^2} x^i.
\label{fiq}
\end{equation}

Les séries $\phi$, $\theta$ et leurs dérivées appartiennent à un espace vectoriel de dimension $3$ sur $\Q(x)$ engendré par $1$, $\phi$ et $\phi'$ (par exemple). Plus précisément,
$$
x (27x-1 ) \phi'' ( x )+6\phi( x )  +6x = 0,
$$ 
\begin{equation}
 3\theta(x)= 2(27x-1 ) \phi'(x) -42\phi(x) +  12x.
 \label{thetaphi4}
\end{equation}
Comme dans la preuve du théorème \ref{t:dalg}, nous exprimons $\pd R z$, puis $\pd F z$ et toutes ses dérivées comme fonctions rationnelles de $u$, $R$, $\phi(R)$ et $\phi'(R)$. Mais comme $R = z + u \phi(R)$, ces fonctions sont également rationnelles en $u$, $z$, $R$ et $\phi'(R)$. Une fois ces expressions rationnelles calculées explicitement pour $\pd F z$, $\pdd 2 F z$ et $\pdd 3 F z$, nous éliminons  $R$ et $\phi'(R)$ de ces trois équations en utilisant des résultants. Nous obtenons ainsi une équation différentielle d'ordre $2$ et de degré $7$ satisfaite par $\pd F z$ par rapport à $z$. Le détail de cette équation
\begin{equation} \label{edF4}
  \substack{9\,{F'}^{2}{F''}^{5}{u}^{6}+36\,{F'}^{2}{F''}^{3
}F'''\,{u}^{5}z+144\,{F'}^{2}{F''}^{4}{u}^{5}-12\,
 \left( 21\,z-1 \right) F'\,{F''}^{5}{u}^{5}+432\,{F'
}^{2}{F''}^{2}F'''\,{u}^{4}z \\ 
-48\, \left( 24\,z-1 \right)
F'\,{F''}^{3}F'''\,{u}^{4}z
 +864\,{F'}^{2}{F''
}^{3}{u}^{4}-96\, \left( 27\,z-2 \right) F'\,{F''}^{4}{u}^{
4}+4\, \left( 27\,z-1 \right)  \left( 15\,z-1 \right) {F''}^{5}{u
}^{4} \\ 
+1728\,{F'}^{2}F''\,F'''\,{u}^{3}z - 288\, \left( 21
\,z-2 \right) F'\,{F''}^{2}F'''\,{u}^{3}z
+10368\,F'\,{F'''}^{2}{u}^{2}{z}^{3}+16\, \left( 27\,z-1 \right)  \left( 
21\,z-1 \right) {F''}^{3}F'''\,{u}^{3}z
\\
+2304\,{F'}^{2}{
F''}^{2}{u}^{3} -288\, \left( 31\,z-4 \right) F'\,{F''}^{3}{u}^{3}
-64\, \left( 6\,uz-162\,{z}^{2}+33\,z-1 \right) {F''}^
{4}{u}^{3}+2304\,{F'}^{2}F'''\,{u}^{2}z
\\ -2304\, \left( 6\,z-1
 \right) F' \,F''\,F'''\,{u}^{2}z 
-192\, \left( 8\,uz-54\,{z}^{2}+29\,z-1 \right) {F''}^{2}F'''\,{u}^{2}z-768\,
 \left( 2\,u+189\,z-7 \right) {F'''}^{2}u{z}^{3}
 \\+2304\,{F'}^
{2}F''\,{u}^{2}-3072\, \left( 3\,z-1 \right) F'\,{F''}
^{2}{u}^{2}
-192\, \left( 24\,uz-27\,{z}^{2}+55\,z-2 \right) {F''}
^{3}{u}^{2}
-1536\, \left( 21\,z-2 \right) F'\,F'''\,uz
\\ -768\,
 \left( 12\,uz+81\,{z}^{2}+24\,z-1 \right) F''\,F'''\,uz+1536
\, \left( 9\,z+2 \right) F'\,F''\,u
-512\, \left( 39\,uz+81
\,{z}^{2}+51\,z-2 \right) {F''}^{2}u \\ +36864\,F'\,z-1024\,
 \left( 12\,uz-162\,{z}^{2}+33\,z-1 \right) F'''\,z-1024\, \left( 
36\,uz+27\,z-1 \right) F''-24576\,z
=0}
\end{equation}
nous convaincra que l'intérêt de la démarche réside plus dans la méthode employée que l'équation elle-même. Nous verrons que $F$ ne satisfait pas d'équation différentielle d'ordre $2$ et que par conséquent \eqref{edF4} constitue l'équation différentielle d'ordre minimal pour $F$ (à constante multiplicative près). Pour plus de détails, nous nous réfèrerons à la fin de cette sous-section.

Appliquons maintenant la même méthode à la série $H$ de la proposition \ref{HRS} :
$$
H(z,u)=\frac z u R-\frac {z^2} u-\Lambda(R)
$$
avec
$$
\Lambda(x)= \sum_{i\ge 3} \frac{(3i-6)!}{(i-3)!(i-2)!i!}x^i.
$$
Cette dernière série satisfait
$$
30\Lambda(x)=x(27x-1) \phi'(x)+(1-24x)\phi(x)+3x^2.
$$
Nous trouvons alors une équation différentielle pour $H$ d'ordre $2$ et de degré $3$ par rapport à $z$ :
\begin{multline*}
3\, ( u+1 ) {u}^{2}{H'}^{2}H'' +  12\,{u}^{2}zH' H''  
+6\, ( u-8 ) u{H'}^{2}
+240\,H  \\
+4\, ( 6\,uz-54\,z+1 ) H'
  +4\, (3\,u{z}^{2}+30\,uH  +27\,{z}^{2}-z )H''+24\,{z}^{2}=0 .
\end{multline*}
La simplicité relative de cette équation différentielle n'est pas surprenante au vu de l'équation \eqref{eulHR}, à savoir $\pd H z = 2(R-z)/u$.

\subsection*{Une équation différentielle d'ordre 2 sur $\boldsymbol F$ ?}
\label{s:ed2?}

L'équation différentielle obtenue précédemment pour la série génératrice $F$ des cartes forestières tétravalentes est en réalité une équation différentielle d'ordre $2$ portant sur $\pd F z$. Nous pouvons naturellement nous demander si $F$ satisfait elle-même une équation différentielle d'ordre $2$.

Utilisons pour cela le théorème \ref{vintegrale} p. \pageref{vintegrale} couplé avec le théorème \ref{tfemme} p. \pageref{tfemme}. Dans le cas tétravalent, nous obtenons le système fonctionnel
$$F = \psi(R), \quad R = z + u \phi(R) $$ où $\phi$ est définie par \eqref{fiq} et $\psi$ par
$$\psi(x) = 4 \sum_{i \geq 2} \frac{(3i-3)!}{(i-2)!i!^2} \, \frac{x^{i+1}} {i+1} - u \sum_{\substack{ i \geq 2 \\ j \geq 1}} \frac{(3i-3)!}{(i-2)!i!^2} \, \frac{(3j)!}{j!^3} \,  \frac{x^{i+j+1}}{i+j+1}.$$
(La dernière somme a été simplifiée : un calcul rapide nous permet de passer d'une somme triple à une somme double.)

Supposons que $F$ satisfait une équation différentielle d'ordre $2$. Il existe alors un polynôme non nul $P$ tel que
$$P\pare{F,\pd F z,\pdd 2 F z,z,u} = 0,$$
soit de manière équivalente
$$P\pare{\psi(R),\theta(R),R'\theta'(R),z,u} = 0,$$
avec $\theta$ donnée par \eqref{fiq}. 
Mais les séries $\theta$, $\phi$ et $\phi'$ sont liées par la relation \eqref{thetaphi4} et $R = z + u \phi(R)$. Il existe donc un polynôme $Q$ non nul tel que
$$Q\pare{\psi(R),\phi(R),\phi'(R),R,u} = 0.$$
Comme $R(z,u) = z + O(z^2)$, il existe une série formelle $Z(r,u)$ telle que $R(Z(r,u),u) = r.$ En subtituant $z$ par $Z$ dans l'égalité du dessus, nous obtenons 
$$Q\pare{\psi(r),\phi(r),\phi'(r),r,u} = 0.$$
% $z = R - u \phi(R)$, $R' = \pare{1 - u \phi'(R)}^{-1}$ et que
Par conséquent se dessinent deux possibilités : soit la série $\psi(x)$ est algébrique sur $\Q(x,u,\phi(x),\phi'(x))$, soit $ \phi$ et $\phi'$ sont algébriquement liées sur $\Q(x)$. Examinons séparément chacune de ces deux possibilités.

\noindent 1. Est-ce que $\psi(x)$ est algébrique sur $\Q(x,u,\phi(x),\phi'(x))$ ? Nous pouvons constater par le calcul que
$$ 60 \sum_{i \geq 2} \frac{(3i-3)!}{(i-2)!i!^2} \, \frac{x^{i+1}} {i+1} = 54 x^2 - 2(1+81x)\phi(x) + 8x(27x-1)\phi'(x)$$
et que
$$ 3 \sum_{\substack{ i \geq 2 \\ j \geq 1}} \frac{(3i-3)!}{(i-2)!i!^2} \frac{(3j)!}{j!^3}   \frac{x^{i+j+1}}{i+j+1} = 12 x \phi(x) - 2 (1-27x)\phi(x)\phi'(x) - 48 \phi^2(x) + 12 \int \frac{\phi^2(x)} x dx.
$$
Au vu de la définition de $\psi$, l'algébricité de $\psi$ sur $\Q(x,u,\phi(x),\phi'(x))$ est donc équivalente à l'algébricité de $\int \phi^2(x)/x dx$ sur $\Q(x,\phi(x),\phi'(x))$.
 Par suite, nous nous ramenons à la série hypergéométrique $$f(x) = \, {_2F _1} \pare{\frac 1 3 , \frac 2 3 ; 2 ;  x }$$
grâce à la relation 
$\phi(x) = x \pare{ f(x) - 1}$,
où $_2 F_1 (a,b;c;x)$ désigne la série hypergéométrique standard de paramètres $a$, $b$, $c$. Comme 
$$20 \int x f(x) dx = 9 x^2f(x) + 9x^2(1-x) f'(x),$$
l'algébricité de $\int \phi^2(x)/x dx$ sur $\Q(x,\phi(x),\phi'(x))$ est équivalente à celle de $g = \int  x f^2(x) dx$ sur $\Q(x,f(x),f'(x))$.

La suite du raisonnement utilise la théorie de Galois différentielle. Elle nous a été soufflée par Alin Bostan, Bruno Salvy et Michael Singer, que nous remercions vivement pour leur aide.  Il s'agit d'un raisonnement dont je ne maîtrise pas toutes les subtilités, c'est pourquoi je vais juste mentionner l'argument principal. Il se base sur l'équation différentielle  d'ordre $3$ satisfaite par $g'$
$$ -4g' + 8x(x-1)g'' + 27x(x-1)^3 g^{(3)} + 9 x^2 (x-1)^2 g^{(4)}(x) =0$$ 
%  Supposons que $g$ soit algébrique sur $\Q(x,f(x),f'(x))$. En considérant le polynôme annulateur minimal de $g$ à coefficients dans $\Q(x,f(x),f'(x))$, il est possible de prouver que $g$ est une intégrale rationnelle en $f$ et $f'$. Nous considérons alors l'équation différentielle d'ordre $3$ satisfaite par $g'$ :
%$$ -4g' + 8x(x-1)g'' + 27x(x-1)^3 g^{(3)} + 9 x^2 (x-1)^2 g^{(4)}(x) =0.$$ 
%Cette équation différentielle est irréductible. 
et un résultat de Bertrand \cite{bertrand}, qui implique que si la série $g$ est algébrique sur $\Q(x,f(x),f'(x))$, alors $g$ est combinaison linéaire de $1$, de $f$ et ses dérivées. Or la commande \textit{maple} \texttt{ode\_int\_y}, appliquée à l'équation différentielle du dessus, permet de tester s'il existe une telle combinaison linéaire. La réponse étant négative, nous  en déduisons que $\psi$ n'est pas algébrique sur $\Q(x,u,\phi(x),\phi'(x))$.

\noindent 2. Est-ce que $\phi$ et $\phi'$ sont algébriquement liées sur  $\Q(x)$ ?  C'est impossible car $f'$ diverge en $1$ comme $\ln(1-x)$ et $f(1) = \sqrt 2/(12 \pi)$ est fini et transcendant (voir \cite[\'Equation (15.3.11)]{AS}).

En conclusion, il n'existe pas d'équation différentielle d'ordre $2$ satisfaite par la série génératrice $F(z,u)$ des cartes forestières tétravalentes.

\subsection{Le cas cubique}

Nous raisonnons ici sur le système fonctionnel vérifié par $\pd F z$ dans le cas cubique  (voir par exemple la figure \ref{recapitulatif}) :
$$\pd F z = \theta(R,S), \quad R = z + u \phi_1(R,S), \quad S = u \phi_2(R,S)$$ 
avec
$$
 \theta(x,y) = 3 \sum_{i \geq 0} \sum_{\substack{j \geq 0 \\ 2i+j \geq 3 }} {\frac { \left( 4\,i+2\,j-4 \right) !}{ \left( 2i+j-3 \right) !\,
  i!^{2}j!}}
 x^iy^j, 
$$
$$
 \phi_1(x,y) = \sum_{i \geq 1}\sum_{\substack{j \geq 0 \\ 2i+j \geq 3 }} \frac{(4i+2j-4)!}{(2i+j-2)!\,(i-1)!\,i!\,j!}x^iy^j,
$$
$$
 \phi_2(x,y) = \sum_{i \geq 0}\sum_{\substack{j \geq 0 \\ 2i+j \geq 2 }} \frac{(4i+2j-2)!} {(2i+j-1)!i!^2j!}x^iy^j.
$$
La méthode reste la même que précédemment, bien que nous ayons  affaire maintenant à trois séries bivariées. Commençons par observer que
$$
\theta(x,y)=-2\phi_1(x,y)+(1-y)\phi_2(x,y)-2x-y^2.
$$
Cette équation, spécialisée en $(x,y) = (R,S)$, donne
\begin{equation}
\label{F-cubic}
F'=2 \frac z u +\frac S u - \pare{1+\frac 1 u}(2R+S^2).
\end{equation} 
Puis, rappelons les identités \eqref{phi1cubique}-\eqref{psicubique} p. \pageref{phi1cubique} exprimant les séries $\phi_1$ et $\phi_2$ en termes des deux séries monovariées $\psi_1$ et $\psi_2$ définies par 
$$\psi_1(z) = \sum_{i \geq 1} \frac{(4i-4)!}{i!(i-1)!(2i-2)!}z^i \quad\textrm{et}\quad \psi_2(z) = \sum_{i \geq 1} \frac{(4i-2)!}{i!^2(2i-1)!}z^i.
$$
Réécrit en fonction de ces deux séries, notre système devient 
\begin{eqnarray}
 \label{phi1enpsi}
R &=&z+u \left( 1-4\,S \right) ^{3/2}\,{\Psi_1} \left(
\frac R {(1-4S)^2}\right) -uR,
\\
 \label{phi2enpsi}
S&=&u \sqrt {1-4S}\,{\Psi_2} \left(
\frac R {(1-4S)^2}
 \right) 
+ \frac u 4\left({1-\sqrt{1-4S}}\right)^2.
\end{eqnarray}

Les séries $\psi_1$, $\psi_2$ et leurs dérivées appartiennent à un espace vectoriel de dimension $3$ sur $\Q(x)$ engendré par $1$, $\psi_1$ et $\psi_2$. Nous pouvons vérifier les identités suivantes :
\begin{equation} \label{relpsicubique}
(1-64z) \psi_1'(z)+ 4\psi_1(z)+ 2\psi_2(z)=1 ,
 \quad \quad 
z (1-64  z) \psi_2'(z)+ 6\psi_1(z)+ 16  z\psi_2(z)=8  z.
\end{equation}
En dérivant \eqref{phi1enpsi} et \eqref{phi2enpsi}, nous exprimons $\pd R z$ et $\pd S z$ comme fonctions rationnelles de $u, R, S, \psi_1(T)$ et $\psi_2(T)$, où $T = R/(1-4S)^2$ (résolution d'un système $2 \times 2$). Mais $\psi_1(T)$ et $\psi_2(T)$ peuvent s'exprimer en fonction de $z, u, R$ et $\sqrt{1-4S}$, toujours grâce à \eqref{phi1enpsi} et \eqref{phi2enpsi}. Donc $\pd R z$ et $\pd S z$ peuvent s'exprimer elles aussi en termes de $z, u, R$ et $\sqrt{1-4S}$.  En réalité, ces expressions ne contiennent aucune racine carrée :
\begin{eqnarray}
  \pd R z &=& \frac{R(48z-1+16(u+1)R+2(3+u)S-8(u+1)S^2)}{D} ,\nonumber
\\
\pd S z &=&\frac{2(3z+(u-3)R-12zS+4(u+1)RS)}D,\nonumber
\end{eqnarray}
où
$$
D=36z^2+(24z-1+24uz)R+4(u+1)RS-4(u+1)^2RS^2+4(u+1)^2R^2.
 $$
 En couplant ces deux équations avec \eqref{F-cubic}, nous pouvons exprimer $\pd F z$, $\pdd 2 F z$ et $\pdd 3 F z$ en termes de $u$, $z$, $R$ et $S$. L'élimination de $R$ et $S$ permet alors d'obtenir une énorme équation différentielle satisfaite par $\pd F z$, de degré $17$ mais d'ordre $2$.

Quant à la série génératrice $G(z,u)$ des cartes  cubiques forestières enracinées sur une feuille\footnote{Rappel : la feuille racine n'est pas pondérée. Une carte cubique enracinée sur une feuille est une carte avec uniquement des sommets de degré $3$, sauf un sommet de degré $1$, sur lequel la carte est naturellement enracinée. On dit alors que la carte est quasi-cubique.} (voir proposition \ref{GRS}), il suffit de remplacer \eqref{F-cubic} par 
$$ 10 \, G= \pare{1+\frac 1 u} \left(z-R+6\,z\,S\,+2\,(u+1)\,R\,S\right),$$ 
et nous obtenons une équation différentielle d'ordre $2$ et de degré $5$ portant sur $G$. Cette dernière devient un peu plus simple quand nous réécrivons $G$ sous la forme $G=(W+z/u)/2$ :
\begin{multline*}
\left( 3\,{u}^{4}z {{W}'} ^{4}-{u}^{3} ( 5\,W
      u-uz+z )  {{W}'} ^{3}+4\, ( u+1 ) (
      5\,W u-uz+z ) ^{2} \right) {W}''\\
 -48\,{u}^{2}z ( u+1)  {{W}'} ^{3}+8\,u ( u+1 )  ( 5\,W u-uz+z )  {{W}'} ^{2}+4\, ( u^2-1) ( 5\,W u-uz+z ){{W}'} = 0.
\end{multline*}
Comme le soulignent les travaux d'Olivier Bernardi et Mireille Bousquet-Mélou  \cite{bernardi-mbm-de}, l'introduction de la série $W$ est naturelle dans le cadre du modèle de Potts. D'ailleurs, la précédente équation différentielle a été d'abord obtenue dans leur papier.

\subsection{Les autres classes}

Malheureusement pour les autres classes de cartes, les équations sont trop grosses pour être calculées explicitement. En effet, la taille de ces polynômes croît exponentiellement chaque fois que nous éliminons une variable parmi plusieurs polynômes. Si nous devons éliminer plus de $3$ variables (les cas précédents ne faisaient intervenir que $2$ variables intermédiaires), nous nous retrouvons très rapidement avec des équations trop grandes pour être manipulées. 

Nous pouvons toutefois majorer très simplement l'ordre des équations différentielles. Plaçons-nous dans le cas eulérien afin d'expliquer pourquoi (il s'agit d'appliquer la même méthode que précédemment). Le système fonctionnel est  
$$\pd F z = \theta(t,R),\quad R = tz + tu\phi(t,R).$$
Supposons que l'espace vectoriel engendré par  
$$E = \ens{ \pdd i \phi x (t,x), \pdd i \theta x  (t,x)}_{i\geq 0}$$
est de dimension $d$ sur $\Q(x,t)$, et que $1$ appartient à l'espace vectoriel engendré par $E$. Alors nous pouvons exprimer $\pd F z$ et ses dérivées de manière rationnelle sur $\Q(z,u,t)$ en fonction de $R$, $1$ et $d-1$ éléments de $E$ évalués en $(t,R)$. Mais comme $R =  tz + tu \phi(t,R)$, la série $R$ appartient également à l'espace vectoriel engendré par $E$ avec $x=R$, de sorte que $\pd F z$ et ses dérivées appartiennent à l'algèbre engendrée par $1$ et $d-1$ éléments de $E$. Nous pouvons donc éliminer ces $d-1$ éléments  et  ainsi obtenir une équation différentielle portant sur $\pd F z$ d'ordre au plus $d-1$.  

Le logiciel \textit{maple} permet de découvrir expérimentalement des relations entre les dérivées de $\phi$ et $\theta$. Nous utilisons notamment la commande \texttt{pade2} du paquet \texttt{gfun} programmé par Bruno Salvy et Paul Zimmerman \cite{gfun}. Dans le cas eulérien non pondéré, nous constatons que le précédent ensemble $E$ est de dimension $5$. Par conséquent, il existe une équation différentielle satisfaite par $\pd F z$ d'ordre au plus $4$. Nous observons le même phénomène pour les cartes forestières $4$-eulériennes non pondérées et pour les cartes forestières $6$-régulières (le degré de chaque sommet est $6$.)

De manière empirique, il semblerait que pour les cartes forestières $2q$-régulières, l'ensemble $E$ est de dimension $2q-1$. Si cela se vérifie, alors la dérivée selon $z$ de la série génératrice des cartes forestières $(2q)$-régulières satisfait une équation différentielle d'ordre au plus $2q-2$.

\chapter{Combinatoire des arbres forestiers} \label{c:enrichis}

Nous avons vu dans la sous-section \ref{ss:modele} p. \pageref{ss:modele} que la série génératrice $F(z,u,t)$ des cartes forestières est aussi la série génératrice des cartes planaires $C$ pondérées par $T_C(u+1,1)$, où $T_C$ désigne le polynôme de Tutte. De fait, la série $F(z,\mu-1,t)$ admet nombreuses descriptions combinatoires (voir \eqref{interp1} et  \eqref{interp2}). Il est donc intéressant d'étudier le comportement asymptotique des coefficients de $F(z,u,t)$ quand $u$ décrit l'intervalle $[- 1,+ \infty[$, et pas uniquement $[0,+\infty[$. La difficulté de ce genre d'entreprise réside dans le fait que certaines séries affiliées à $F$ (comme les séries $R$ et $S$, voir le théorème \ref{central}) peuvent comporter des coefficients négatifs, alors que la plupart des théorèmes d'analyse complexe (comme le théorème de Pringsheim) ne s'appliquent que pour des séries à coefficients positifs. 

L'objectif de ce chapitre est de prouver de manière purement combinatoire que certaines séries liées à $F$ admettent des coefficients positifs (en un sens que nous préciserons). Nous serons alors amenés à introduire ce que nous appellerons \textit{arbres bourgeonnants enrichis}, ce qui est en substance la version forestière des arbres bourgeonnants de la section \ref{s:bourgeonnants}. Nous profitons du propos pour montrer une bijection entre arbres bourgeonnants enrichis et cartes forestières. Cette bijection revêt une importance particulière dans cette thèse dans la mesure où elle joue le rôle de passerelle entre les deux parties.

\section{$\boldsymbol{(u+1)}$-positivité}

\subsection{Définition}

Considérons une série formelle $A$ à coefficients dans $\R[u]$ avec pour jeu de variables $z_1,z_2,\dots,z_m$. Si cette série admet des coefficients positifs dans la base algébrique \mbox{$1+u,z_1,z_2,\dots,z_m$} (autrement dit les coefficients de $A(\mu-1,z_1,z_2,\dots,z_m)$ sont positifs en $\mu$), alors la série $A$ est dite \textit{$(u+1)$-positive}.

La série génératrice $F(z,u,t)$ des cartes forestières constitue un exemple non trivial de série $(u+1)$-positive. En effet, comme nous l'avons précisé en introduction, $F$ est la série génératrice des cartes planaires $C$ pondérés par $T_C(u+1,1)$, qui est bien un polynôme en $(u+1)$ avec des coefficients positifs. Par exemple, pour les cartes cubiques, nous avons
\begin{multline*}
F(z,u,t)  =  (6 + 4 u) z^3t^3 + (140 + 234 u + 144 u^2 + 32 u^3) z^4t^4 + O(z^5) \\
  = \pare{2 + 4(u+1)}  z^3t^3 + \pare{18 + 42(u+1) + 48 (u+1)^2 + 32 (u+1)^3 } z^4t^4 + O(z^5).
\end{multline*}

La prochaine sous-section donne une condition suffisante pour que des séries génératrices d'arbres avec une forêt couvrante, avec un poids $u$ par composante connexe de la forêt, soient $(u+1)$-positives.

\subsection{La $\boldsymbol{(u+1)}$-positivité chez les arbres forestiers}
\label{ss:arbresforestiers}

Nous rappelons qu'un \textit{arbre forestier} est un couple $(T,F)$ tel que $T$ est un arbre plan (possiblement décoré) et $F$ une forêt couvrante de  $T$. Considérons $\mathcal F$ une classe d'arbres forestiers. Nous allons définir une propriété sur $\mathcal F$ qui garantit la $(u+1)$-positivité de sa série génératrice $A_{\mathcal F}$ (après division par $u$), où $u$ compte le nombre de composantes  dans la forêt couvrante.

Soient $(T,F)$ un arbre forestier de $\mathcal F$ et $e$ une de ses arêtes. Si $e$ appartient à $F$, on dit qu'elle est \textit{interne} ; elle est dite \textit{externe} sinon. Le  \textit{basculement} de $e$ est l'opération qui consiste à changer le statut interne/externe de l'arête $e$. En d'autres termes, si $e$ est interne, alors on la retire de la forêt $F$ ; si $e$ est externe, alors on l'intègre à $F$ (voir la figure \ref{classestable}). Nous obtenons ainsi une nouvelle forêt couvrante $F'$ de $T$. On dit que $e$ est \textit{basculable}  si $(T,F')$ appartient à $\mathcal F$. La classe d'arbres forestiers $\mathcal F$ est \textit{stable} si pour tout  $(T,F) \in \mathcal F$,
\begin{enumerate}
\item[(i)] chaque arête externe de $(T,F)$ est basculable,
\item[(ii)] l'arbre forestier résultant du basculement de n'importe quelle arête basculable de $(T,F)$ a le même ensemble d'arêtes basculables que $(T,F)$.
\end{enumerate}

\fig{[width=\textwidth]{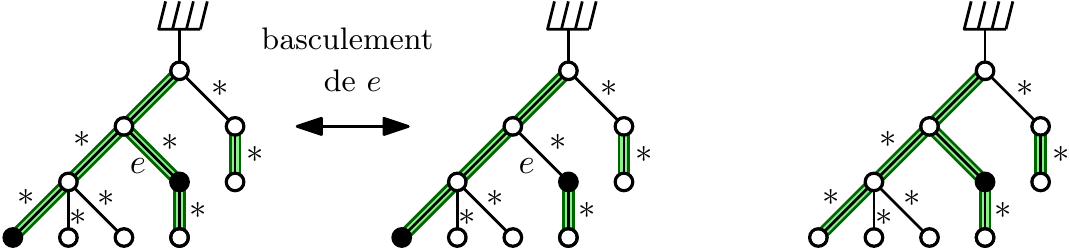}}{\`A gauche : deux arbres forestiers de $\mathcal F_1$ qui ne diffèrent que d'un basculement de l'arête $e$. \`A droite : un arbre forestier de $\mathcal F_2$. Les étoiles indiquent les arêtes basculables.}{classestable}

\noindent \textbf{Exemple 1.} Soit $\mathcal F_1$ la classe des arbres forestiers avec un bicoloriage des sommets en blanc et noir tels que le sommet racine est relié à un sommet noir par un chemin uniquement constitué d'arêtes internes. La classe $\mathcal F_1$ n'est \textbf{pas} stable. En effet, la condition $(i)$ est tout le temps vérifiée mais pas la condition $(ii)$. Le premier arbre forestier de la figure \ref{classestable} constitue un contre-exemple. Le basculement de l'arête $e$ donne bien un arbre forestier de $\mathcal F_1$ mais deux arêtes ne sont plus basculables : si on retire l'une d'entre elles de la forêt couvrante, le sommet racine ne sera plus relié à un sommet noir par des arêtes internes.

\noindent \textbf{Exemple 2.} Soit $\mathcal F_2$ la classe des arbres forestiers où tous les sommets sont coloriés en blanc, sauf un qui est colorié en noir, tels que le sommet racine est relié au sommet noir par un chemin uniquement constitué d'arêtes internes. La classe $\mathcal F_2$ est stable. En effet, la condition $(i)$ est trivialement vérifiée. De plus, dans un arbre forestier de $\mathcal F_2$, une arête est mutable si et seulement si elle n'appartient pas à l'unique chemin reliant le sommet racine au sommet noir\footnote{Si elle appartient à l'unique chemin reliant le sommet racine au sommet noir, alors elle est forcément interne.}. Cette propriété ne dépend pas de la forêt couvrante considérée, d'où la propriété $(ii)$.

\begin{lem} \label{l:stable}
Soient $\mathcal F$ une classe stable d'arbres forestiers et $(T,F) \in \, \mathcal F$. Appelons $\mathcal F_T$ l'ensemble des forêts couvrantes de $T$  telles que $(T,F') \in \, \mathcal F$.  Les éléments de $\mathcal F_T$ ont tous le même nombre d'arêtes basculables. Appelons $b$ ce nombre. Si $T$ a au moins un sommet, la série génératrice de $\mathcal F_T$, où chaque composante connexe est pondérée par un poids $u$, est égale à $u\,(1+u)^b$.
\end{lem} 
\begin{proof} La condition $(i)$ implique que la forêt $F_{\max}$, qui est constituée de toutes les arêtes de $T$, appartient à $\mathcal F_T$. Elle indique en outre que nous pouvons obtenir $F_{\max}$ à partir de n'importe quelle forêt $F'$ de $\mathcal F_T$ en insérant successivement dans la forêt chaque arête externe à $F'$. Donc, d'après la condition $(ii)$, toute forêt de $\mathcal F_T$ a le même ensemble d'arêtes basculables que $F_{\max}$ (et \textit{a fortiori} le même nombre, $b$). Par conséquent, pour obtenir n'importe quelle forêt $F'$ de $\mathcal F_T$ à partir de $F_{\max}$, il suffit de déterminer pour chaque arête basculable si cette arête appartient à $F'$ ou non. Mais $F_{\max}$ a une unique composante (poids $u$ -- l'existence d'une composante est due à la présence d'un sommet dans $T$) et le retrait d'une arête d'une forêt couvrante augmente de $1$  le nombre de composantes (on multiplie par un facteur $u$ pour chaque arête choisie pour être externe). La série génératrice de $\mathcal F_T$ est donc $u (1+u)^b$, comme annoncé.
\end{proof}

La proposition qui suit est une conséquence directe du lemme précédent (on a sommé sur tous les arbres $T$ de $\mathcal F$).

\begin{prop} \label{p:stable}
Soit $\mathcal F$ une classe stable d'arbres forestiers avec au moins un sommet. Associons à chaque arbre $T$ tel que $(T,F) \in \, \mathcal F$ des statistiques\footnote{par exemple : le nombre de feuilles, le nombre d'arêtes, le nombre de sommets blancs, etc.} $n_1(T),n_2(T),\dots,n_m(T)$ (la forêt couvrante $F$ n'influe pas sur ces paramètres) de sorte que le nombre d'arbres $T$ à statistiques $n_1(T),n_2(T),\dots,n_m(T)$ fixées est fini. Appelons  $A_{\mathcal F}(u,z_1,\dots,z_m)$ la série génératrice des arbres forestiers $(T,F)$ de $\mathcal F$, où la variable $u$ compte le nombre de composantes dans $F$ et où $z_i$ compte $n_i(T)$. La série  $A_{\mathcal F}/u$ est $(u+1)$-positive.
\end{prop}

La section suivante constitue un bon exemple d'application de cette proposition.

\section{Arbres bourgeonnants enrichis}
\label{s:abe}

Rappelons que le théorème central \ref{central} p. \pageref{central} fait intervenir deux séries $R$ et $S$ telles que 
\begin{equation} \label{bouger}
R =  tz +  tu \sum_{i \geq 1} \sum_{j \geq 0} T_{2i+j}(t) { 2i+j-1 \choose i-1,i,j } R^i S^j,
\end{equation}
\begin{equation} \label{bouse}
S =  tu \sum_{i \geq 0} \sum_{j \geq 0} T_{2i+j+1}(t) { 2i+j \choose i,i,j } R^i S^j,
\end{equation}
où $T_k(t)$ est la série génératrice des arbres à $k$ pattes, avec un poids $d_\ell$ pour chaque sommet de degré $\ell$, comptés selon le nombre d'arêtes -- cette série est définie par \eqref{deftl} p. \pageref{deftl}. Nous souhaitons obtenir des propriétés de $(u+1)$-positivité sur $R$, $S$ et leurs variantes. Afin d'appliquer la méthode de la section précédente, nous allons interpréter ces séries comme des séries génératrices d'arbres bourgeonnants (voir section \ref{s:bourgeonnants} p. \pageref{s:bourgeonnants}) forestiers particuliers.

Commençons par quelques  rappels. Un arbre \textit{bourgeonnant} est un arbre plan enraciné avec des demi-arêtes pendantes. Ces demi-arêtes sont de trois types différents : la \textit{racine}, qui est unique et représentée par un râteau, les \textit{bourgeons}, portant chacun une charge $-1$ et représentés par des flèches blanches, et les \textit{feuilles}, portant chacune une charge $+1$ et représentées par des flèches noires. La \textit{charge} d'un arbre bourgeonnant est la différence entre le nombre de feuilles et le nombre de bourgeons. En coupant en deux une arête d'un arbre bourgeonnant, on obtient deux arbres : celui qui ne contient pas la racine est un \textit{sous-arbre}.

\begin{defi} Un arbre bourgeonnant $T$ muni d'une forêt couvrante $F$ est un \textbf{R-arbre enrichi} (resp. \textbf{S-arbre enrichi}) si
\begin{itemize}
\item[(i)] sa charge vaut $1$ (resp. $0$)
\item[(ii)] tout sous-arbre associé à une arête externe (à $F$) a pour charge $1$ ou $0$.
\end{itemize}
\end{defi}

\fig{[scale=1.2]{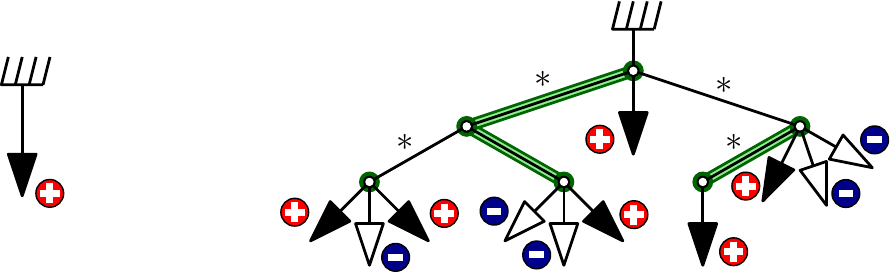}}{\`A gauche : le plus petit R-arbre enrichi. \`A droite : un R-arbre enrichi ayant $6$ feuilles, $5$ bourgeons et $3$ composantes. Les étoiles indiquent les arêtes basculables.}{enrichis}

La figure \ref{enrichis} montre des exemples de R-arbres enrichis.

\noindent \textbf{Remarque. } En contractant toutes les arêtes internes dans les R- et S-arbres enrichis, nous retrouvons les R- et S-arbres bourgeonnants de la section \ref{s:bourgeonnants}. Autrement dit, un R-arbre (resp. un S-arbre) enrichi est un R-arbre (resp. S-arbre) sur lequel on a greffé pour chaque sommet de degré $k$ un arbre à $k$ pattes. Cela justifie \textit{a posteriori} l'appellation de R-arbres et S-arbres "enrichis".

\begin{prop} \label{beqRS}
Les\footnote{L'unicité de ce couple de  séries provient du lemme \ref{caracteresse} p. \pageref{caracteresse}.} séries $R(z,u,t)$ et $S(z,u,t)$ définies par \eqref{bouger} et \eqref{bouse} sont respectivement les séries génératrices des R- et S-arbres enrichis, où la variable $z$ compte le nombre de feuilles, $u$ le nombre de composantes dans la forêt et $t$ le nombre d'arêtes et demi-arêtes\footnote{Pour éviter des confusions, précisons qu'une demi-arête est bien pondérée par $t$ et non pas $\sqrt t$.} qui ne sont pas des bourgeons (donc feuilles/racine)\footnote{Le $R$-arbre réduit à une feuille est pondéré par $zt$ et non pas par $zt^2$.}.
\end{prop}
\begin{proof} 
Nous pourrions adapter  la preuve de la proposition \ref{heqRS} p. \pageref{heqRS}. Mais d'après la remarque ci-dessus, les séries génératrices $R$ et $S$ des arbres enrichis sont les séries génératrices des R-arbres et S-arbres dans laquelle on a substitué $g_k$ par $t^{k/2} T_k(t)$. Il suffit  alors d'appliquer cette proposition \ref{heqRS}.

Vérifions que la statistique associée à $t$ est bien la bonne : au vu des équations \eqref{bouger} et \eqref{bouse}, la variable $t$ compte le nombre d'arêtes internes, plus le nombre de feuilles, plus le nombre de composantes. Or, en amont de chaque composante de la forêt, nous retrouvons soit une arête externe, soit la racine. Cela prouve que $t$ compte le nombre d'arêtes et demi-arêtes qui ne sont pas des bourgeons.
\end{proof}

Revenons maintenant à la $(u+1)$-positivité.

\begin{prop} \label{Rstable}
La classe des R-arbres enrichis privée de l'arbre réduit à une feuille et la classe des S-arbres enrichis sont stables (au sens de la sous-section \ref{ss:arbresforestiers}).
\end{prop}
\begin{proof}La condition $(i)$ de la définition des classes stables est trivialement vérifiée. Quant à la condition $(ii)$, il suffit d'observer qu'une arête est basculable (que ce soit pour les R-arbres enrichis ou pour les S-arbres enrichis) si et seulement si le sous-arbre correspondant a pour charge $0$ ou $1$ et que ceci ne dépend pas de la forêt.
\end{proof}

La proposition \ref{p:stable}, appliquée aux précédentes classes, permet alors de déduire la $(u+1)$-positivité des séries suivantes. (Les statistiques sont ici le nombre de feuilles et le nombre total d'arêtes et de demi-arêtes qui ne sont pas des bourgeons.)

\begin{cor} 
Quelle que soit la suite de poids $(d_k)$, les séries $(R-zt)/u$ et $S/u$ sont $(u+1)$-positives. Quand $u = \mu - 1$, elles comptent respectivement les R- et S-arbres enrichis (non réduits à une feuille) munis de leur unique arbre couvrant selon le nombre de feuilles ($z$), le nombre d'arêtes et de demi-arêtes qui ne sont pas des bourgeons ($t$), et le nombre d'arêtes basculables ($\mu$). \label{c:RSupos}
\end{cor}

\noindent \textbf{Exemple. } Quand les arbres considérés sont  cubiques ($d_3 = 1$ et $d_k = 0$ pour $k \neq 3$), le développement des précédentes séries devient en écrivant $\mu = u + 1$ :
$$
(R-zt)/u = ( 2 + 4 \mu ) z^2 t^4 + (16 + 36\mu + 48 \mu^2 +40 \mu^3 )z^3 t^7 + O\pare{z^4},$$
$$
S/u  = 2zt^2 + \pare{6 + 12 \mu + 12 \mu^2}z^2t^5 + (72 + 176\mu + 240 \mu^2  + 204 \mu^3 + 128 \mu^4)z^3t^8 + O \pare{z^4}.
$$

Nous aurons également besoin du résultat suivant, qui est une variation de ce qui précède.

\begin{lem} \label{l:Stilde}
Soit $\tilde S(z,u,t)$ la série définie par $\tilde S = t u \phi_2(t,z,\tilde S)$ où 
$$
\phi_2(t,x,y) =  \sum_{i \geq 0} \sum_{j \geq 0} T_{2i+j+1}(t) { 2i+j \choose i,i,j } x^i  y^j.
$$
Quelle que soit la suite de poids $(d_k)$, les séries $\tilde S/u$ et $\pd {\phi_2} y (t,z,\tilde S)$ sont $(u+1)$-positives.

De plus, s'il existe $k \geq 3$ avec $d_k \neq 0$, la série $\tilde S(z,-1,t)$ est une série formelle en $z$ et $t$ à coefficients entiers qui n'est pas un polynôme en $z$.
\end{lem}
\begin{proof} Un \textit{$\tilde{S}$-arbre enrichi} $(T,F)$ est un  arbre bourgeonnant forestier tel que chaque composante de $F$ est incidente à autant de feuilles que de bourgeons (exemple sur la figure \ref{tildeS}). Appelons $\tilde S$ la série génératrice des $\tilde S$-arbres enrichis (nous définissons les paramètres $z$, $u$, $t$ comme dans la proposition \ref{beqRS}). Un $\tilde S$-arbre enrichi peut se voir inductivement comme un arbre à pattes\footnote{pour la définition d'arbres à pattes, voir p. \pageref{deftl}} (il s'agit de l'arbre de $F$ incident à la racine) sur lequel on a greffé au niveau des pattes non racine $i$ bourgeons, $i$ feuilles et $j$ $\tilde S$-arbres enrichis, avec $i \geq  0$ et $j \geq 0$. En termes de séries génératrices, cela se traduit bien par $\tilde S = t u \phi_2(t,z,\tilde S)$.

\fig{[width=\textwidth]{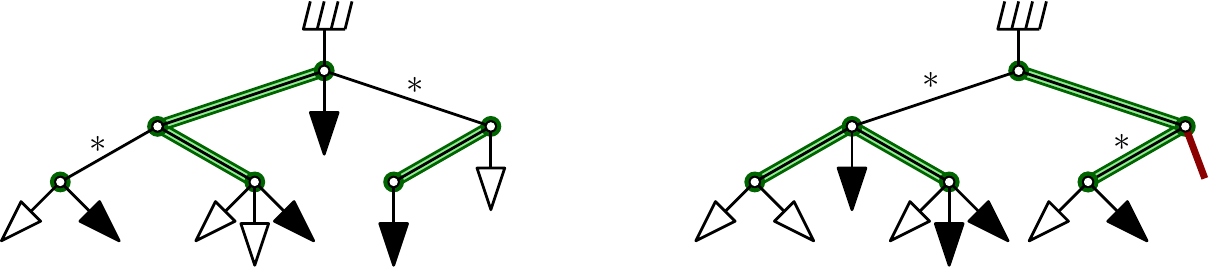}}{\`A gauche : un $\tilde S$-arbre enrichi. \`A droite : un $\tilde S$-arbre enrichi avec une demi-arête pendante (en gras) qui est incidente à la composante racine. Les étoiles indiquent les arêtes basculables.}{tildeS}

La classe des $\tilde S$-arbres enrichis est stable. En effet, une arête est basculable si et seulement si le sous-arbre correspondant est incident à autant de feuilles que de bourgeons. Cette condition est bien indépendante de la forêt considérée.

Nous pouvons étendre la définition des $\tilde S$-arbres enrichis aux arbres bourgeonnants qui ont exactement une demi-arête pendante de charge nulle (différente de la racine, voir la figure \ref{tildeS}). En utilisant les mêmes arguments que plus haut, nous pouvons prouver que  $t u \pd {\phi_2} y (t,z, \tilde S)$ est la série génératrice des $\tilde S$-arbres enrichis avec une demi-arête pendante qui est incidente à la composante racine.

La classe de tels $\tilde S$-arbres enrichis est également stable. En effet, une arête est basculable si et seulement si le sous-arbre correspondant est incident à autant de feuilles que de bourgeons et si elle ne se trouve pas sur l'unique chemin qui relie la racine à la demi-arête pendante.

D'après la proposition \ref{p:stable}, les séries $\tilde S/u$ et $t \pd {\phi_2} y (t,z, \tilde S)$ sont bien $(u+1)$-positives.

\fig{[scale=1.2]{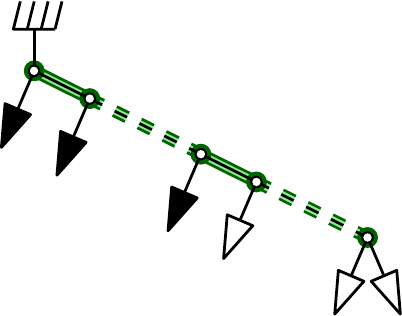}}{Un $\tilde S$-arbre cubique sans arête basculable avec un nombre arbitrairement grand de feuilles.}{noflip}

Pour prouver le dernier point, remarquons que $\tilde S(z,-1,t)$ compte les $\tilde S$-arbres enrichis munis de leur unique arbre couvrant  sans arêtes basculables. Mais de tels arbres peuvent être construits avec un nombre arbitrairement grand de feuilles (voir la figure \ref{noflip} pour  un exemple avec $k=3$), ce qui montre que le degré en $z$ d'une telle série n'est pas borné.
\end{proof}

\section{La $\boldsymbol{(u+1)}$-positivité de $\boldsymbol{(R^n - z^n)/u}$}

Nous prouvons ici un résultat, plutôt difficile, qui nous sera nécessaire lors de l'étude du comportement asymptotique des cartes forestières régulières eulériennes. Sa démonstration fait intervenir une notion de $(u+1)$-positivité plus complexe que celle de la sous-section \ref{ss:arbresforestiers}. Voici son énoncé.

\begin{theo} \label{Pestpositif}
Soit $q$ un entier supérieur ou égal à $2$. Appelons $R(z,u)$ l'unique série formelle telle que 
\begin{equation}
\label{uren}
R = z + u \, \sum_{i \geq 1} \frac{((2q-1)i)!}{ i! ((q-1)i)! ((q-1)i+1)!} R^{(q-1)i+1}.
\end{equation} 
Alors pour tout $n \in \{1,\dots,q\}$, la série $(R^n - z^n)/u$ est $(u+1)$-positive.
\end{theo}

Plus précisément, grâce à l'identité 
$$ R^n - z^n = (R-z)\,\sum_{k=0}^{q-1} z^{q-1-k} R^k,$$
nous déduisons le théorème précédent de la proposition suivante.

\begin{prop} Soient $q \geq 2$ et $R(z,u)$ la série définie par  \eqref{uren}. Alors pour tout entier $k \in \ens{0,\dots,q-1}$, la série $(R-z)R^k/u$ est $(u+1)$-positive.
\end{prop}

\begin{proof} Nous conseillons au lecteur de relire les précédentes sections de ce chapitre avant d'attaquer cette preuve.

 Dans cette preuve, tous les arbres bourgeonnants sont $2q$-réguliers -- autrement dit le degré de chaque sommet est $2q$. Dans ce contexte\footnote{Il n'existe pas de S-arbres enrichis sans sommet de degré impair (voir section \ref{s:bourgeonnants} p. \pageref{s:bourgeonnants} où il est indiqué que $S=0$), et donc les sous-arbres ne peuvent avoir une charge nulle.}, un R-arbre enrichi est un arbre bourgeonnant forestier $R = (T,F)$ tel que $(i)$ la charge totale est $1$ et $(ii)$ tout sous-arbre associé à une arête externe  a pour charge $1$. On définit de même les \textit{B-arbres enrichis} comme les arbres bourgeonnants forestiers de charge $-1$ tels que la condition $(ii)$ soit vérifiée. Par convention, un bourgeon seul (sans sommet) est considéré comme un B-arbre enrichi. Nous nous référerons  à la figure \ref{decomax} pour des exemples de R- et B-arbres maximaux.

\noindent \textbf{1. Principe de la preuve.} Nous interprétons 
$(R-z)R^k$ comme la série génératrice des $(k+1)$-uplets $(R_0,\dots,R_k)$ de R-arbres enrichis $2q$-réguliers tels que $R_0$ n'est pas réduit à une feuille, où la variable $z$ compte les feuilles et  $u$ les composantes dans les $k+1$ forêts. En effet, l'équation \eqref{uren} n'est que la réécriture de  \eqref{bouger} quand $t=1$, $d_{2q} = 1$ et $d_{k} = 0$ pour $k \neq 2q$.

La notion d'arête basculable ne suffit pas ici pour prouver la $(u+1)$-positivité de cette série (divisée par $u$). Tentons d'expliquer pourquoi. Le principal argument de la preuve du lemme \ref{l:stable} réside dans le fait que tout arbre forestier d'une classe stable peut être obtenu à partir d'un arbre forestier ayant une unique composante, et ce en ne changeant uniquement que la forêt de cet arbre.  Considérons maintenant des R-arbres $R_0,R_1,\dots,R_k$ qui ne sont pas feuilles.  Alors en notant $R^{\textrm{max}}_i$ l'arbre $R_i$ dans lequel toutes les arêtes sont internes, l'élément $(R^{\textrm{max}}_0,R^{\textrm{max}}_1,\dots,R^{\textrm{max}}_k)$ n'a pas une seule composante, mais $k+1$. Le précédent argument ne peut donc fonctionner ; il faudrait plutôt se ramener à un $(k+1)$-uplet avec une seule composante, donc forcément un $(k+1)$-uplet dont chaque R-arbre, à part le premier, est une feuille.
%En effet, l'exposant en $u$ associé à un $(k+1)$-uplet de R-arbres enrichis est au moins égal au nombre d'éléments dans ce $(k+1)$-uplet qui ne sont pas des feuilles. Il faut donc changer la structure du $(k+1)$-uplet ! 

L'idée est de greffer successivement sur $R_0$ les arbres $R_1,\dots,R_k$ jusqu'à obtenir un $(k+1)$-uplet $\pare{R'_0,f,\dots,f}$, où $f$ est le R-arbre enrichi réduit à une feuille.  Le point plus difficile est de garantir la caractère bijectif d'une telle transformation. Une fois ces greffes effectuées, nous appliquerons la proposition \ref{p:stable} pour prouver la $(u+1)$-positivité souhaitée.   

\noindent \textbf{2. Décomposition en arbres maximaux.} Fixons un R- ou B-arbre enrichi $A$. Un \textit{sous-R-arbre} (resp. un \textit{sous-B-arbre}) de $A$ est un sous-arbre\footnote{On rappelle qu'un sous-arbre résulte de la coupure d'une arête. Par conséquent, il est strictement inclus dans l'arbre bourgeonnant.} (pas forcément associé à une arête externe) de charge $+1$ (resp. $-1$) ou une feuille (resp. ou un bourgeon). Cet arbre est bien un R-arbre (resp. un B-arbre) enrichi. Il est dit \textit{maximal} s'il n'est inclus dans aucun autre sous-R-arbre ou\footnote{Le "ou" s'applique et pour les R-arbres maximaux et pour les B-arbres maximaux. Ainsi l'intersection de deux arbres maximaux est toujours vide car le contraire impliquerait que l'un des deux est inclus dans l'autre.} sous-B-arbre. Si $R_1,\dots,R_s$ (resp. $B_1,\dots,B_t$) désignent les R-arbres (resp. B-arbres) maximaux de $A$ ordonnés selon un parcours préfixe de l'arbre, alors la suite $(R_1,\dots,R_s,B_1,\dots,B_t)$ est la \textit{décomposition en arbres maximaux} de $A$. Tous ces arbres sont disjoints deux à deux. La décomposition ne dépend pas de la forêt couvrante. La figure \ref{decomax} montre un exemple d'une telle décomposition.

\fig{[scale=1.3]{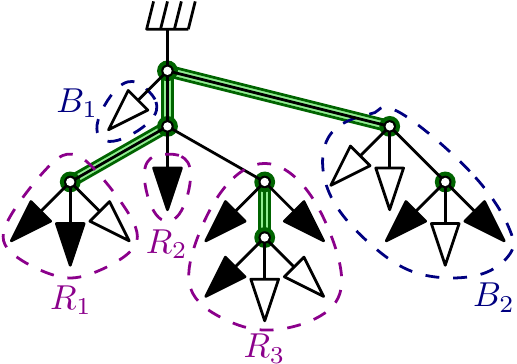}}{La décomposition en arbres maximaux de cet R-arbre enrichi est $(R_1,R_2,R_3,B_1,B_2)$.}{decomax}

Chaque feuille est un sous-R-arbre de $A$, donc est incluse dans un arbre maximal. De même, tout bourgeon fait partie d'un arbre maximal. Par conséquent, la somme des charges attachées aux arbres maximaux de $A$ est égale à la charge totale de $A$. Ainsi, si $A$ est un R-arbre enrichi, on a $s=t+1$, où $s$ désigne le nombre de R-arbres maximaux de $T$ et $t$ le nombre de B-arbres maximaux. De même, si $A$ est un B-arbre enrichi, on a $t=s+1$. 

Du reste, si $A$ a au moins un sommet, le nombre d'arbres maximaux de $A$ est au moins égal au nombre d'enfants du sommet racine, à savoir $2q-1$. Le nombre de B-arbres maximaux dans un R-arbre enrichi non réduit à une feuille est donc au moins égal à $q-1$.

\noindent \textbf{3. Décrivons une application $\boldsymbol \lambda$} qui envoie les couples $(R,B')$ sur les couples $(B,R')$ tels que :
\begin{itemize}
\item $R$ et $R'$ sont des R-arbres enrichis,
\item $B$ et $B'$ sont des B-arbres enrichis,
\item $R$ et $B$ ne sont pas réduits à une feuille ou un bourgeon,
\item $R'$ et $B'$ sont respectivement (strictement) inclus dans $R$ et $B$.
\end{itemize}
Considérons un couple $(R,B')$ et appelons $(R_1,\dots,R_{t+1},B_1,\dots,B_t)$ la décomposition en arbres maximaux de $R$. Dans $R$, intervertissons pour chaque $i \in \ens{1,\dots,t}$ les places de $R_i$ et $B_i$, y compris les arêtes sur lesquelles ces arbres sont enracinés, puis remplaçons $R_{t+1}$ par $B'$. Si $B'$ n'est pas un bourgeon et si l'arête sur laquelle est maintenant enraciné $B'$ n'est pas dans la forêt, nous la rajoutons à la forêt. Notons l'arbre obtenu $B$. On pose alors $R' = R_{t+1}$ et définissons $(B,R') = \lambda(R,B')$. Un exemple est montré figure \ref{il}.

\fig{[width=\textwidth]{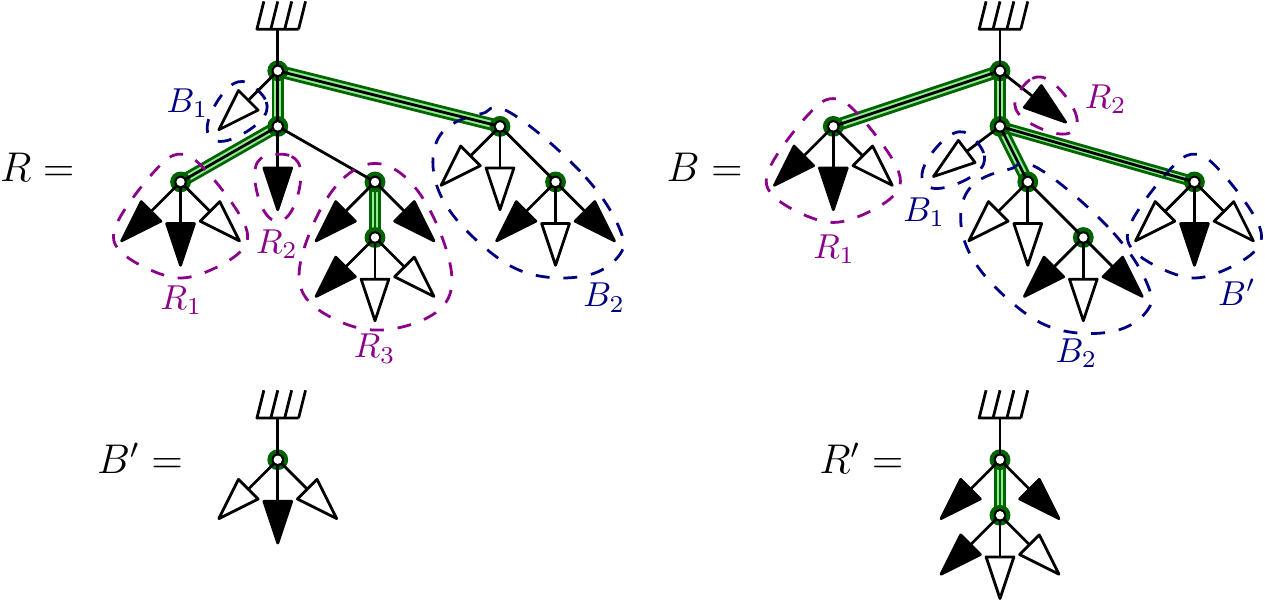}}{Quatre arbres enrichis $R$, $B$, $R'$, $B'$ tels que $(B,R') = \lambda(R,B')$.}{il}

Prouvons que $B$ est un B-arbre enrichi. Comme nous avons enlevé un sous-arbre de charge $+1$ pour le remplacer par un arbre de charge $-1$, la charge totale de $B$ est $-1$. Montrons maintenant la condition $(ii)$. Soit $e$ une arête externe de $B$. Si $e$ est une arête dans $B'$, alors le sous-arbre associé à $e$ est de charge $1$ par définition des B-arbres enrichis. Sinon $e$ correspond à une arête de $R$ car la seule arête de $B$ qui n'appartient ni à $B'$, ni à $R$, est l'arête sur laquelle est enracinée $B'$, mais par construction, cette arête est interne. Le sous-arbre associé à $e$ dans $R$ est donc un sous-R-arbre de $R$. Deux possibilités alors : soit $e$ appartient à un arbre maximal de $R$ différent de $R_{t+1}$, soit $e$ est l'arête sur laquelle est enraciné $R_i$ avec \mbox{$i \in \ens{1,\dots,t}$}. Dans tous les cas, le sous-arbre associé à $e$ dans $B$ est de charge $1$.

Montrons maintenant que  $(R_1,\dots,R_t,B_1,\dots,B_t,B')$ est la décomposition en arbres maximaux de $B$. Soit $S$ un sous-arbre de $R$ contenant (strictement) un arbre maximal de $R$ et notons $c$ sa charge. \'Etant donné que nous avons remplacé les R-arbres maximaux par des B-arbres et les B-arbres maximaux par des R-arbres, la charge de $S$ après application de $\lambda$ est égale à $-c$. Comme $c \notin \, \ens{-1,1} $, on a également $-c \notin \, \ens{-1,1}$ : le sous-arbre $S$ n'est pas maximal dans $B$.

\noindent \textbf{Remarque.} Il était  nécessaire d'échanger  les places de $R_i$ et $B_i$ dans $R$ pour chaque $i \in \ens{1,\dots,t}$. Si on ne faisait que remplacer $R_{t+1}$ par $B'$, alors l'arbre $B'$ ne serait plus forcément le dernier des B-arbres maximaux (et il est également possible qu'il ne soit pas maximal !). La figure \ref{il} permet de s'en convaincre.

\noindent \textbf{4. Les applications réciproques $\boldsymbol{\delta_0}$ et $\boldsymbol{\delta_1}$.}
Nous pouvons voir que l'application $\lambda$ est surjective. Plus précisément, montrons que tout couple de type $(B,R')$ admet exactement deux antécédents $\delta_0(B,R')$ et $\delta_1(B,R')$, sauf quand $R'$ est une feuille auquel cas il n'admet qu'un unique antécédent $\delta_1(B,R')$. 

Si on note $(R_1,\dots,R_{t},B_1,\dots,B_{t+1})$ la décomposition en arbres maximaux de $B$, alors en échangeant les places de $R_i$ et $B_i$ pour chaque $i \in \ens{1,\dots,t}$ et en remplaçant $B_{t+1}$ par $R'$, nous obtenons un R-arbre enrichi\footnote{Pour prouver que $R$ est un R-arbre enrichi, on peut s'inspirer de la preuve ci-dessus où on montre que $B$ est un B-arbre enrichi.} $R$. Lorsque $R'$ n'est pas une feuille, nous pouvons décider si l'arête sur laquelle est enraciné $R$ est interne ou pas. On note $\delta_0(B,R')$ le couple $(R,B_{t+1})$ obtenu en rendant cette arête interne, $\delta_1(B,R')$ en rendant cette arête externe. Lorsque $R'$ est une feuille, le couple $(R,B_{t+1})$ est désigné par $\delta_1(B,R')$. La figure \ref{delta01} illustre cette notion.

\fig{[width=\textwidth]{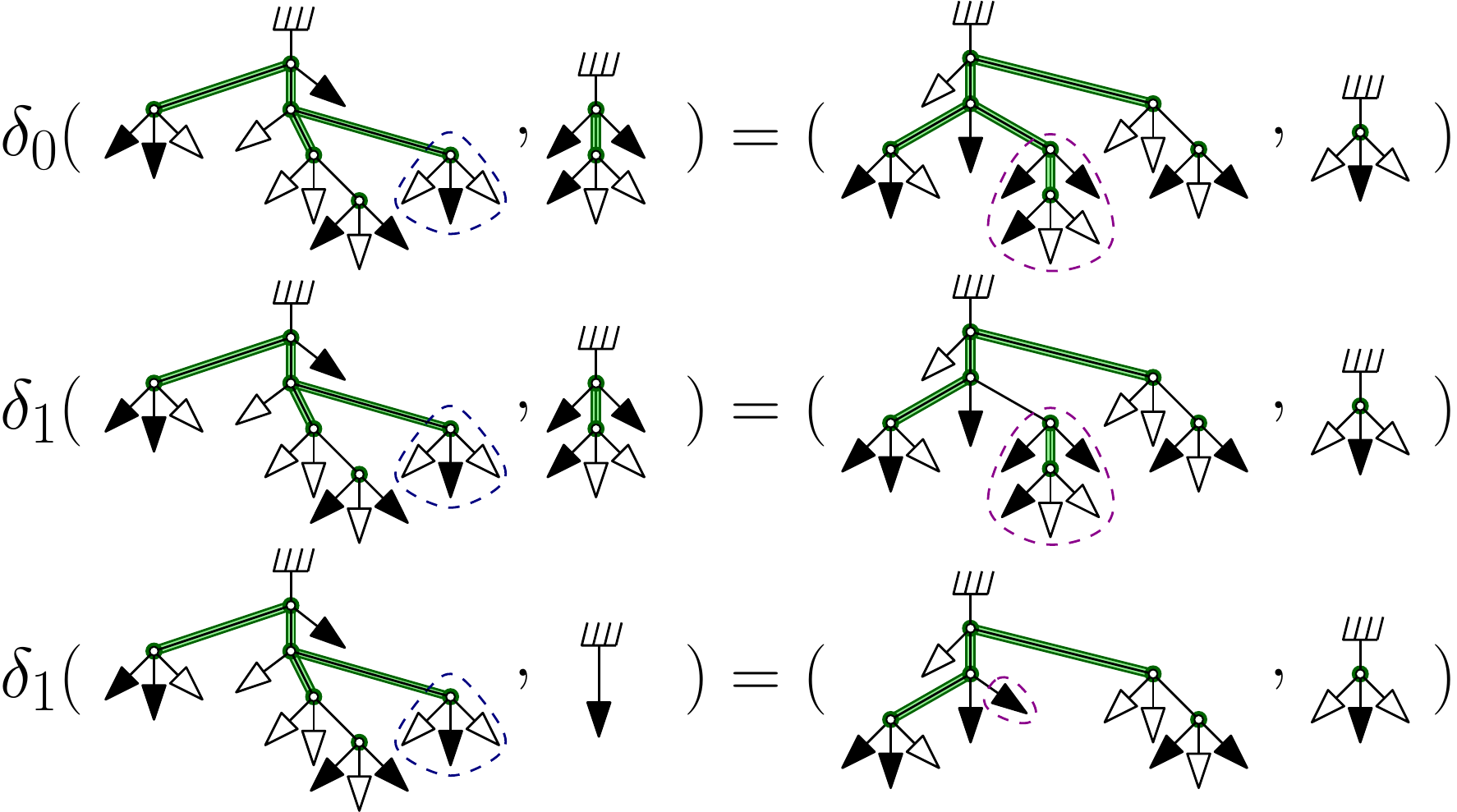}}{Trois exemples d'image par $\delta_0$ ou $\delta_1$. (Nous avons entouré $B'$ et $R'$.)}{delta01}

Nous pouvons alors facilement vérifier l'équivalence
$$\lambda (R,B') = (B,R') \quad \Longleftrightarrow \quad (R,B') = \delta_0(B,R')\ \textrm{ ou }\ (R,B') = \delta_1(B,R').$$
Notons en outre que $\delta_0(B,R')$ a le même nombre de composantes que le couple $(B,R')$ alors que $\delta_1(B,R')$ a exactement une composante en plus. (En effet, $B_{t+1}$ est soit un bourgeon, soit attaché à $B$ par une arête interne.)

\noindent \textbf{5. Les applications $\boldsymbol{ \Lambda_\ell}$, $\boldsymbol{\Delta_{\ell,0}}$, $\boldsymbol{\Delta_{\ell,1}}$.} Pour $k \in \ens{1,\dots,q}$, on note $E$ l'ensemble des $(k+1)$-uplets de R-arbres enrichis $(R_0,\dots,R_k)$ dans lesquels $R_0$ n'est pas réduit à une feuille. Pour $\ell \in \ens{1,\dots,k}$, on définit $E_\ell$ comme le sous-ensemble de $E$ restreint aux éléments $(R_0,\dots,R_k)$ tels que $R_\ell$ est une feuille. Pour un $(k+1)$-uplet qui n'est pas dans $E_\ell$, nous voulons définir une application $\Lambda_\ell$ qui enlève des parties de $R_\ell$ pour les mettre dans $R_0$, plus précisément dans le $\ell$-ième B-arbre maximal\footnote{On rappelle (voir la fin du point 2) qu'un R-arbre a au moins $q - 1$ (et donc $k$) B-arbres maximaux.} de $R_0$. 

Soient $\ell \in \ens{1,\dots,k}$ et $(R_0,\dots,R_k) \in E \backslash E_\ell$. Considérons $B_\ell$ le $\ell$-ième B-arbre maximal dans la décomposition en arbres maximaux de $R_0$ et notons $(\tilde B_\ell, \tilde R_\ell)$ l'image de $(R_\ell,B_\ell)$ par $\lambda$. Remplaçons $B_\ell$ par $\tilde B_\ell$ dans $R_0$ en veillant à ce que l'arête sur laquelle $\tilde B_\ell$ est enraciné appartienne à la forêt et notons $\tilde R_0$ le R-arbre enrichi obtenu. On définit alors $\Lambda_\ell(R_0,\dots,R_k)$ comme $(\tilde R_0,R_1,\dots,R_{\ell-1},\tilde R_\ell,R_{\ell+1},\dots,R_k)$. La figure \ref{LambdaDelta} montre deux exemples d'application de $\Lambda_\ell$, avec $\ell = 1$.

\fig{[width=\textwidth]{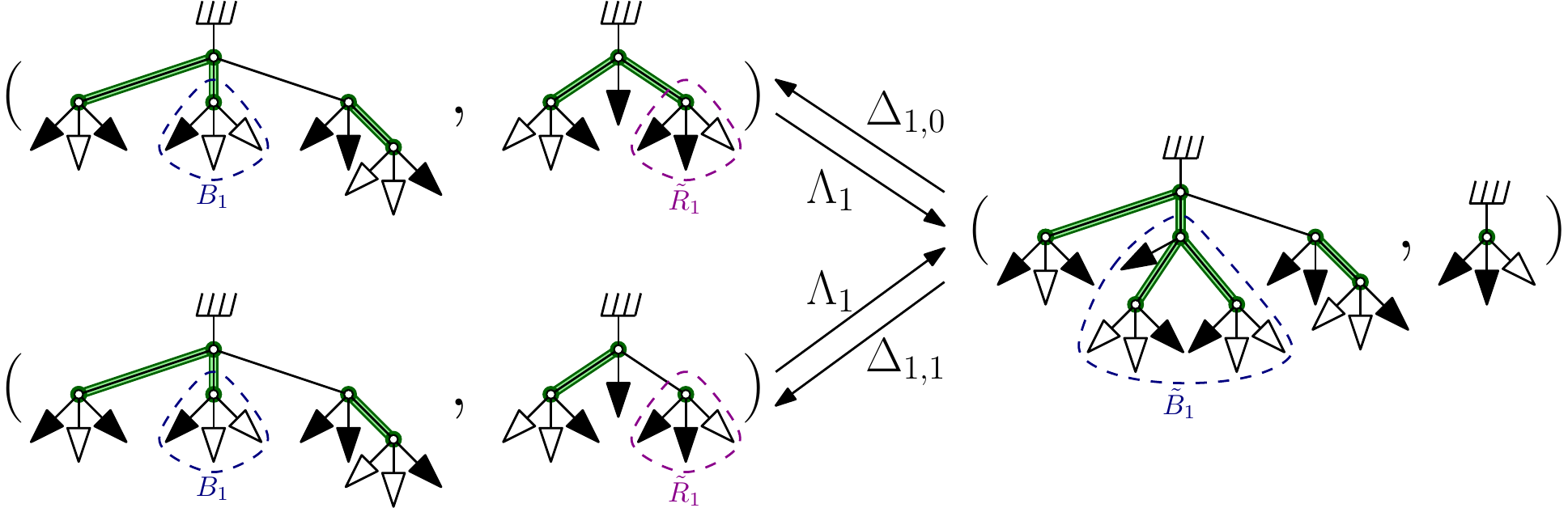}}{Exemples d'application de $\Lambda_\ell$, $\Delta_{\ell,0}$ et $\Delta_{\ell,1}$, avec $\ell=k=1$ et $q=2$.}{LambdaDelta}

Posons $x = (\tilde R_0,R_1,\dots,R_{\ell-1},\tilde R_\ell,R_{\ell+1},\dots,R_k)$. Si $\tilde R_{\ell}$ n'est pas une feuille, alors $x$ admet deux antécédents par $\Lambda$. Plus précisément, considérons $\tilde B_\ell$ le $\ell$-ième B-arbre maximal de $\tilde R_0$ et remplaçons cet arbre par $B'_\ell$ et $\tilde R_\ell$ par $R'_\ell$, où $(R'_\ell,B'_\ell)$ est respectivement égal à $\delta_0(\tilde B_\ell,\tilde R_\ell)$, et à $\delta_1(\tilde B_\ell,\tilde R_\ell)$. On obtient ainsi les deux antécédents qu'on appellera  $\Delta_{\ell,0}(x)$ et $\Delta_{\ell,1}(x)$. Si $\tilde R_\ell$ est une feuille, alors $x$ admet un unique antécédent par $\Lambda_\ell$ qu'on notera $\Delta_{\ell,1}(x)$ et qu'on définit de manière similaire.
Plus généralement, $\Delta_{\ell,0}$ et $\Delta_{\ell,1}$ peuvent être définies pour tout $(k+1)$-uplet $x=(R_0,R_1,\dots,R_k) \in E$ où le $\ell$-ième B-arbre maximal de $R_0$ n'est pas un bourgeon. Dans le cas de $\Delta_{\ell,0}$ sont également exclus les éléments de $E_\ell$. Comme précédemment, nous pouvons noter que $\Delta_{\ell,0}(x)$ a le même nombre de composantes que $x$ alors que $\Delta_{\ell,1}(x)$ a exactement une composante de plus que $x$.

% Enfin notons que pour $i \neq j$, les applications $\Lambda_i$, $\Delta_{i,0}$, $\Delta_{i,1}$ commutent de manière évidente avec $\Lambda_j$, $\Delta_{j,0}$, $\Delta_{j,1}$. \na{Cela a besoin d'une explication ?}

\noindent \textbf{6. Associer à tout $\boldsymbol{x \in E}$ un élément $\boldsymbol {\rho(x) \in \bigcap_{\ell = 1}^k E_\ell}$.} 
%où toutes les coordonnées sont des feuilles, sauf la première.}
Soient $x = (R_0,\dots,R_k) \in E$ et \mbox{$\ell \in \ens{1,\dots,k}$}. Si $x$ n'appartient pas à $E_\ell$ (ce qui veut dire que $R_\ell$ n'est pas une feuille), alors on peut lui appliquer $\Lambda_\ell$. Dans ce cas-là, en notant $(R_0',\dots,R_k') = \Lambda_\ell(R_0,\dots,R_k)$, nous voyons que $R_\ell'$ est un sous-arbre strict de $R_\ell$ -- voir les définitions de $\lambda$ et $\Lambda_\ell$ pour s'en convaincre. Donc en appliquant $\Lambda_{\ell}$ à $x$ suffisamment de fois, nous finissons par obtenir un élément de $E_\ell$ car toute suite strictement décroissante (pour l'inclusion) de R-arbres enrichis termine par une feuille.

Ainsi il existe $k$ entiers naturels $n_1,\dots,n_k$ tels que $$\Lambda^{(n_k)}_k \circ \dots \circ \Lambda^{(n_1)}_1 (x) \in \bigcap_{\ell = 1}^k E_\ell,$$
où $\Lambda^{(n)}$ désigne l'expression $\Lambda \circ \dots \circ \Lambda$, où $\Lambda$ apparaît $n$ fois. On appelle $\rho(x)$ cet élément. % de $\bigcap_\ell E_\ell$. 
La figure \ref{representantdelta} montre un exemple de calcul de $\rho(x)$.

\fig{[width=\textwidth]{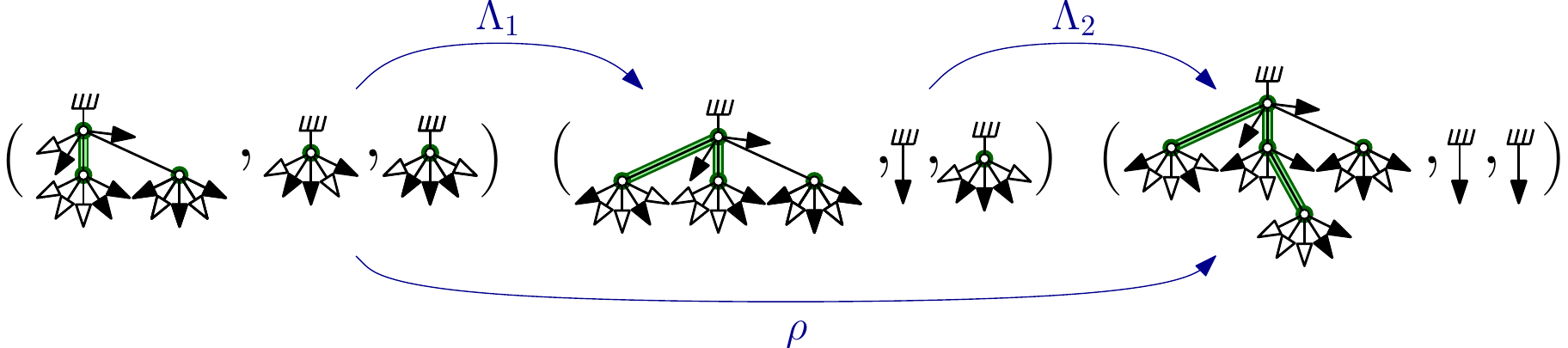}}{Image par $\rho$, avec $q=3$ et $k = 2$.}{representantdelta}

\noindent \textbf{7. Description de $\boldsymbol{\rho^{-1}(y)}$.} Soient $f$ le R-arbre enrichi uniquement constitué d'une feuille et $R$ un R-arbre enrichi différent de $f$. Posons $y = (R,f,\dots,f)$ de sorte que $y \in \bigcap_{\ell = 1}^k E_\ell$. On ne peut appliquer à $y$ qu'un nombre fini de fois les applications $\Delta_{\ell,0}$ et $\Delta_{\ell,1}$, où $\ell \in \ens{1,\dots,k}$. En effet, à chaque fois qu'on applique une de ces deux fonctions, le $\ell$-ième B-arbre maximal qui apparaît dans la décomposition en arbres maximaux de $R$ diminue. 

Plus précisément, définissons par induction $m(B)$ en posant $m(B) = 0$ si $B$ est un bourgeon, et $m(B) = m(B') + 1$ sinon, où $B'$ est le dernier B-arbre maximal qui apparaît dans la décomposition en arbres maximaux de $B$. On peut montrer par récurrence qu'on peut appliquer au plus $m(B_\ell)$ fois 
$\Delta_{\ell,0}$ et $\Delta_{\ell,1}$ à $y$, où $B_\ell$ est le $\ell$-ième arbre maximal de $R$ (et la borne est exacte).

Or $\Delta_{\ell,0}$ et $\Delta_{\ell,1}$ constituent les applications réciproques de $\Lambda_\ell$. Donc tout élément $x \in \rho^{-1}(y)$ peut s'écrire sous la forme 
\begin{equation} \label{xepsilon}
x =   \left(\Delta_{1,\epsilon^1_{p_1}} \circ \dots \circ \Delta_{1,\epsilon^1_{2}} \circ \Delta_{1,\epsilon^1_{1}}\right)  \circ \dots \circ \left(\Delta_{k,\epsilon^k_{p_k}} \circ \dots \circ \Delta_{k,\epsilon^k_{2}} \circ \Delta_{k,\epsilon^k_{1}}\right) (y),
\end{equation}
avec pour tout $\ell \in \ens{1,\dots,k}$ 
\begin{equation} \label{suiteepsilon}
0 \leq p_\ell \leq m(B_\ell), \quad \epsilon^\ell_1 = 1 \quad \textrm{et} \quad  \forall i \in \, \ens{1,\dots,p_\ell}\ \  \epsilon^\ell_i \in \, \ens{0,1}.
\end{equation}

(On rappelle qu'on ne peut pas appliquer $\Delta_{\ell,0}$ à $y$. C'est la raison pour laquelle \mbox{$\epsilon^\ell_1 = 1$}.) Il n'est pas très difficile de voir
 que cette écriture de $x$ est unique \footnote{Démonstration rapide. Considérons deux suites $(\epsilon_i^{\ell})_{i \in \ens{1,\dots,p_\ell}}^{\ell \in \ens{1,\dots,k}}$ et $(\tilde \epsilon_i^{\ell})^{\ell \in \ens{1,\dots,k}}_{i \in \ens{1,\dots,\tilde p_\ell}}$ satisfaisant \eqref{suiteepsilon} avec $p_\ell < \tilde p_\ell$ pour un certain $\ell$, et notons respectivement $x$ et $\tilde x$ les éléments de $\rho^{-1}(y)$ associés à ces deux suites par \eqref{xepsilon}. On voit alors que $\Lambda^{(p_\ell)}(x) \in E_\ell$ et $\Lambda^{(p_\ell)}(\tilde x) \notin E_\ell$ (l'application $\Lambda_\ell$ modifie uniquement le $(\ell+1)$-ième R-arbre du $(k+1)$-uplet donc commute avec $\Lambda_{\ell'}$ avec $\ell' \neq \ell$), d'où $x \neq \tilde x$. Un argument similaire s'applique lorsque $\epsilon_i^\ell \neq \tilde \epsilon_i^\ell$.}  .  En outre, la différence entre le nombre total de composantes dans $x$ et le nombre de composantes dans $R$ est égale à la somme des $\epsilon^\ell_i$,  pour $\ell \in \ens{1,\dots,k}$ et $i \in \ens{1,\dots,p_\ell}$. Notons pour finir que le nombre total de feuilles est préservé par $\rho$ et donc $x$ a exactement $k$ feuilles de plus que $R$. Tout ceci est illustré par la figure \ref{treillil}.

\fig{[width=\textwidth]{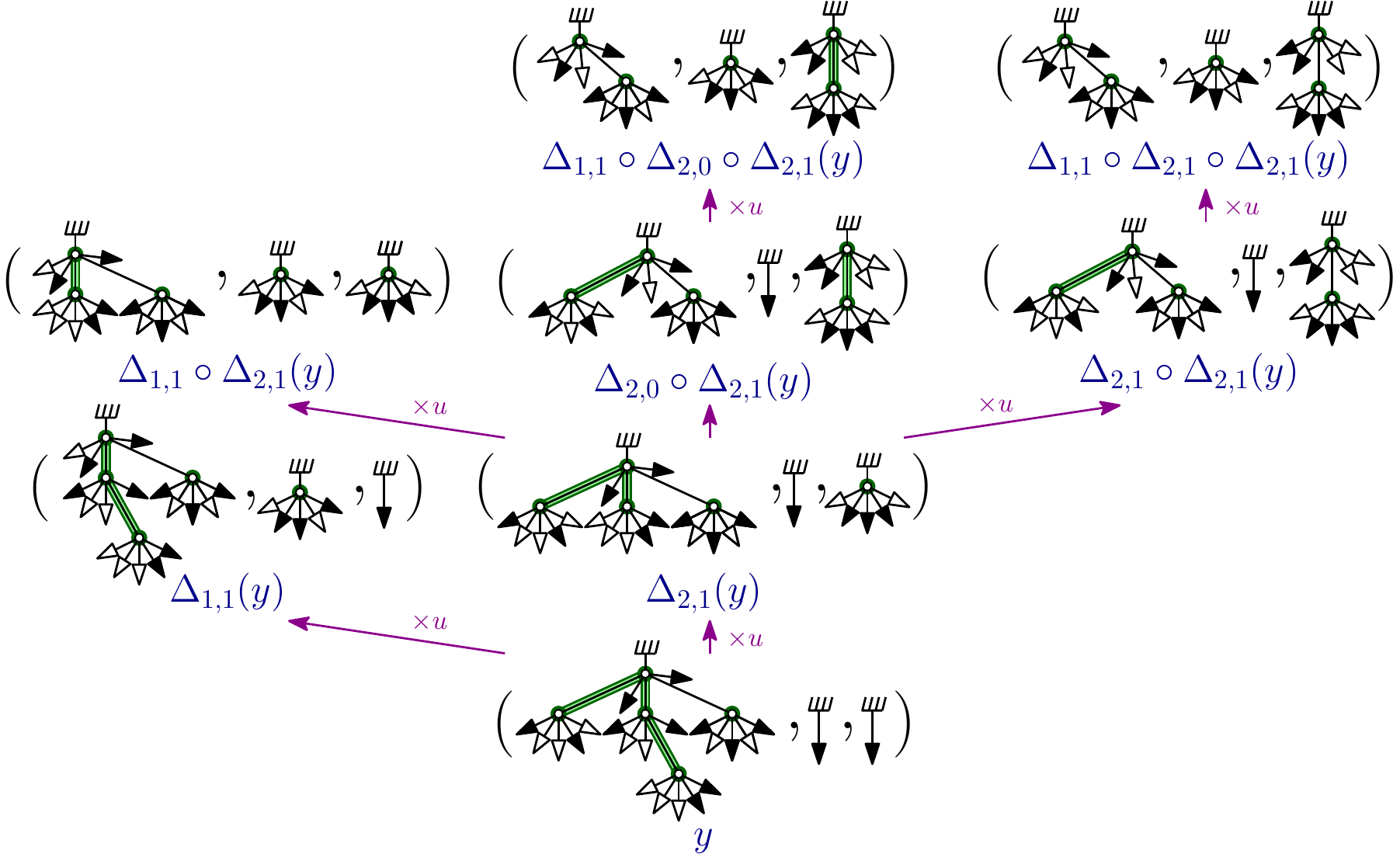}}{Tous les éléments appartenant à $\rho^{-1}(y)$, triés selon le nombre de composantes, où $y$ est le triplet du bas.}{treillil}

\noindent \textbf{8. Preuve de la $\boldsymbol{(u+1)}$-positivité.}
Nous interprétons $(R-z)R^k$ comme la série génératrice des éléments de $E$ :
$$(R-z) R^k   = \sum_{x \in E} z^{\face(x)} u^{\cc(x)},$$
où $\face(x)$ désigne le nombre total de feuilles dans $x$ et $\cc(x)$ le nombre total de composantes dans les forêts. En regroupant selon la valeur de $\rho(x)$, on obtient
$$(R-z) R^k   = \sum_{\substack{R\  \textrm{R-arbre enrichi} \\\textrm{non feuille}}} \,\, \sum_{x \in \rho^{-1}(R,f,\dots,f)}  z^{\face(x)} u^{\cc(x)},
$$
où $f$ désigne le R-arbre enrichi constitué d'une seule feuille. En utilisant les résultats du point 7, l'égalité devient
$$(R-z) R^k   = \sum_{\substack{R\  \textrm{R-arbre enrichi} \\\textrm{non feuille}}} z^{\face(R)+k} u^{\cc(R)}  \prod_{\ell = 1}^k \sum_{\substack{(\epsilon_i^{\ell}) \textrm{ suite} \\ \textrm{vérifiant \eqref{suiteepsilon}} }} u^{\sum_{i=1}^\ell \epsilon_i^\ell},
$$ 
où $B_\ell$ désigne dans \eqref{suiteepsilon} le $\ell$-ième B-arbre maximal de $R$ qu'on notera $B_\ell(R)$. Faisons correspondre à chaque suite $(\epsilon_i^{\ell})$ vérifiant \eqref{suiteepsilon} un mot de $0$ et de $1$ de longueur $m(B_\ell(R))$ en insérant avant $\epsilon_1^\ell$ le nombre adéquat de $0$. Cette correspondance est bijective et préserve la somme, d'où 
$$
\sum_{\substack{(\epsilon_i^{\ell}) \textrm{ suite} \\ \textrm{vérifiant \eqref{suiteepsilon}} }} u^{\sum_{i=1}^\ell \epsilon_i^\ell} = \sum_{\substack{w \textrm{ mot de }0\textrm{ et de }1 \\ \textrm{de longueur }m(B_\ell(R))\textrm{} }} u^{\textrm{somme des chiffres de }w}   
 = (1+u)^{m(B_\ell(R))}.
$$
Donc en posant $m(R) = m(B_1(R)) + \dots + m(B_k(R))$, on voit que
$$(R-z) R^k   = \sum_{\substack{R\  \textrm{R-arbre enrichi} \\\textrm{non feuille}}} z^{\face(R)+k} (1+u)^{m(R)} u^{\cc(R)}.$$
Or les statistiques $\face(R)$ et $m(R)$ ne dépendent pas de la forêt de $R$. On peut alors appliquer la proposition \ref{p:stable}, couplée avec la proposition \ref{Rstable}, pour prouver la $(u+1)$-positivité de $(R-z)R^k$.
\end{proof}

\section{Bijection bourgeonnante pour les cartes forestières}
\label{s:bbforet}

 La section \ref{s:bb} montrait, à travers les travaux d'Olivier Bernardi, comment on pouvait obtenir une carte planaire avec face marquée à partir d'un T-arbre. La présente section introduit la version forestière de cette transformation. Elle jouera le rôle d'une  passerelle entre les deux parties de ce mémoire. En effet, cette bijection donne naissance à une nouvelle activité pour le polynôme de Tutte, l'\textit{activité bourgeonnante}. On se référera pour cela à un chapitre ultérieur (chapitre \ref{c:bloact}).
 
 Cette bijection est l'interprétation combinatoire de l'identité \eqref{forfait}, à savoir 
 $\frac {\partial F} {\partial z} = \theta(t,R,S)$, où
\begin{equation} \label{btheta}
\theta(t,x,y) =  \sum_{i \geq 0} \sum_{j \geq 0}  \left(\frac {2\,t} {2i+j} T'_{2i+j}(t) + T_{2i+j}(t)\right) {2i+j \choose i,i,j}  x^i y^j.
\end{equation}
On rappelle (voir théorème \ref{tfemme} p. \pageref{tfemme} et sa preuve) que cette identité n'est qu'une spécialisation de l'équation~\eqref{eqmp} p.~\pageref{eqmp},  avec $g_k$  égal à la série génératrice des arbres à $k$ pattes enracinés sur une patte et $h_k$ égal à la série génératrice des arbres à $k$ pattes enracinés sur un coin. Bien sûr, nous pourrions nous contenter d'appliquer la bijection de la section~\ref{s:bb} p.~\pageref{s:bb} sur des cartes où les sommets seraient des boîtes noires dans lesquelles se trouveraient des arbres à pattes ; mais il est plus intéressant de voir cette bijection adaptée à des cartes où les arbres sont directement greffés sur les sommets, autrement dit : des cartes forestières. 

%En effet, en plus des propriétés que nous avons déjà vues (comme la transformation feuille/face), nous verrons que la structure de forêt est préservée et que la notion d'arête basculable peut être traduite en terme de cartes.
%
%Cette relation a déjà été expliquée combinatoirement en terme de mobiles dans le cadre de cartes avec prescription des degrés des faces \cite{bg-continuedfractions}. La bijection entre mobiles et cartes avec un sommet marqué de Jérémie Bouttier et Emmanuel Guitter pourraient donc tout à fait être adaptés
% aux cartes forestières, il suffirait juste de pondérer les sommets de manière adéquate et de raisonner dans le dual. Toutefois, notre approche est différente dans la mesure où elle se base sur des arbres bourgeonnants, et non pas sur les mobiles. Elle est en outre parfaitement adaptée à la structure de forêt couvrante. Elle se révélera particulièrement utile lorsque nous nous intéresserons au paramètre $\mu = u - 1$, plutôt que $u$ lui-même (qui, on le rappelle, compte le nombre de composantes dans la forêt). Signalons également que notre approche présente de fortes similarités avec le récent article de Marie Albenque et Dominique Poulalhon \cite{AlPo}.
%

\subsection{T-arbres enrichis}

Un \textit{T-arbre enrichi} est un arbre bourgeonnant\footnote{Pour la définition d'arbre bourgeonnant, on se réfèrera à la section \ref{s:bourgeonnants} ou \ref{s:abe}.} $T$ muni d'une forêt couvrante $F$ tel que :
\begin{enumerate}
\item[(i)] $T$ ne comporte pas de demi-arête racine mais est enraciné sur un coin,
\item[(ii)] la charge totale de $T$ est $0$,
\item[(iii)] tout sous-arbre enraciné sur une arête externe a pour charge $0$ ou $1$.
\end{enumerate}
Comme les précédents arbres enrichis, un T-arbre enrichi n'est qu'un T-arbre (voir sous-section~\ref{ss:tarbre} p.~\pageref{ss:tarbre}) sur lequel on a greffé des arbres à pattes au niveau des sommets. La figure \ref{thetaarbre} montre un exemple de T-arbre enrichi.

Nous pouvons observer que l'ensemble des T-arbres est une classe stable  d'arbres forestiers (voir sous-section \ref{ss:arbresforestiers}). En effet, comme pour les R- et S-arbres enrichis, une arête est basculable si et seulement si le sous-arbre correspondant a pour charge $0$ ou $1$.
 
\fig{[scale=1.2]{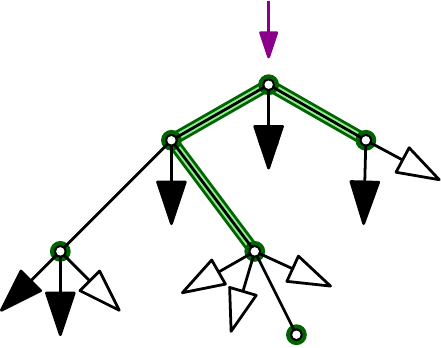}}{Exemple de T-arbre enrichi}{thetaarbre}

\begin{prop} 
%Comme dit plus haut, les T-arbres enrichis sont des T-arbres où pour tout $k$, chaque sommet de degré $k$ est pondéré par un arbre à pattes (enraciné sur une patte pour les sommets non racine, enraciné sur un coin pour le sommet racine).
%Nous appliquons alors la proposition \ref{p:tarbre} p. \pageref{p:tarbre},
%avec $g_k$ égal à $t^{k/2}$ fois la série génératrice des arbre à $k$ pattes enracinés sur une patte et $h_k$   égal à $t^{k/2}$ fois la série génératrice des arbres à $k$ pattes enracinés sur un coin.
%
%Précisons pourquoi la variable $t$ compte le nombre de feuilles plus le nombre d'arêtes : les arêtes internes sont comptés par les séries génératrices par les arbres à pattes, les arêtes externes 
La série génératrice des T-arbres enrichis avec un poids $z$ par feuille, un poids $u$ par composante non racine de la forêt et un poids $t$ par feuille et par arête est égale à $\theta(t,R,S)$, où les séries $R$ et $S$ sont définies par \eqref{bouger} et \eqref{bouse} et $\theta$ par \eqref{btheta}.
\end{prop}
\begin{proof} (Nous choisissons de prouver directement ce résultat plutôt que de passer par la proposition \ref{p:tarbre} ; la variable $t$ s'interprète plus facilement comme cela.)
Pour obtenir un T-arbre enrichi, il suffit de prendre un arbre à pattes enraciné sur un coin (qui sera notre composante racine) puis d'y attacher des R-arbres enrichis, des S-arbres enrichis et des bourgeons au niveau des pattes. Pour que la charge de l'arbre soit nulle, il faut que le nombre de  R-arbres enrichis attachés à l'arbre racine soit égal au nombre de bourgeons attachés à l'arbre racine. 
En utilisant la proposition \ref{beqRS} et le fait que les arbres à $k$ pattes enracinés sur un coin sont comptés par
$
\left(\frac {2\,t} {k} T'_{k}(t) + T_{k}(t)\right)
$ (voir la preuve du théorème \ref{tfemme} p. \pageref{tfemme}), nous concluons que la série génératrice des T-arbres est  $\theta(t,R,S)$.

Toutefois, précisons pourquoi la variable $t$ compte le nombre de feuilles plus le nombre d'arêtes. Remarquons que lorsque nous attachons un R-arbre non feuille ou un S-arbre sur une patte, une arête se crée. Mais la racine d'un R- ou S-arbre est pondérée par $t$ et disparaît lorsqu'on la greffe à une patte. C'est donc elle qui contribue au poids $t$ de la nouvelle arête. Les autres arêtes se trouvent soit dans les R- et S-arbres, soit dans l'arbre racine. Dans tous les cas, elles sont chacune pondérées par $t$.
\end{proof}

La simple interprétation de $\theta(t,R,S)$ en termes d'arbres  bourgeonnants permet d'avoir une minoration sur les coefficients de $\pd F z$, qui nous sera utile par la suite.

\begin{cor} \label{c:minodeF}

Soient $F$ la série génératrice des cartes forestières et $(R,S)$ le couple de séries défini par \eqref{bouger} et \eqref{bouse}. Fixons $u \geq -1$. Notons $F_{f,a}$ (resp. $R_{f,a}$ ) le coefficient de $z^f t^a$ dans la série $F(z,u,t)$ (resp. $R(z,u,t)$). Alors pour $f \geq 2$, on a 
$$ 0 \leq \frac 1 u R_{f,a+1} \leq  f F_{f-1,a}.$$
\end{cor}

\begin{proof} Il existe une injection naturelle $\iota$ de l'ensemble des R-arbres enrichis non réduits à une feuille vers l'ensemble des T-arbres enrichis. Elle consiste à réenraciner tout R-arbre enrichi sur le coin précédant la racine, et à remplacer cette racine par un bourgeon (voir figure \ref{Rtheta}). Il est facile de vérifier que l'arbre obtenu est bien un T-arbre enrichi.

\fig{[scale=1.2]{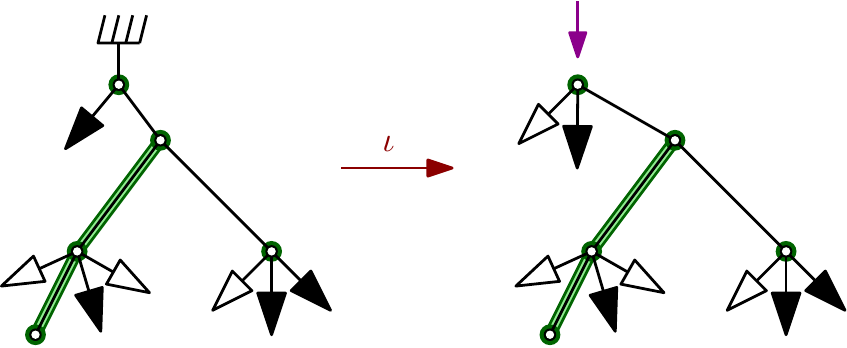}}{Image d'un R-arbre enrichi par l'injection $\iota$}{Rtheta}

En outre, il est facile de voir que l'injection $\iota$ préserve les feuilles, les arêtes et les arêtes basculables. De plus, si $\mu = u + 1$,
la série $R(z,\mu-1,t)/u$ (resp. \mbox{$\theta(t,R(z,\mu-1,t),S(z,\mu-1,t))$}) compte les R-arbres (resp. T-arbres) enrichis munis de leur unique arbre couvrant avec un poids $\mu$ par arête basculable. 
Comme $\iota$ est injective, nous avons
\begin{equation} \label{ineqRT}
0 \leq \frac 1 u R_{f,a+1,b} \,\mu^b \leq T_{f,a,b}\, \mu^b,
\end{equation}
où $R_{f,a,b}$ (resp. $T_{f,a,b}$) désigne le coefficient de $z^f t^a \mu^b$ dans $R(z,\mu-1,t)$ (resp. \mbox{$\theta(t,R(z,\mu-1,t),S(z,\mu-1,t))$}). (L'arbre racine d'un T-arbre enrichi n'est pas pondéré par $u$, d'où la présence du facteur $1/u$. De plus, le décalage de $1$ pour l'exposant en $t$ provient du fait que les R-arbres enrichis comportent une racine pondérée par $t$.)

Mais d'après le théorème \ref{central}, on a $\pd F z = \theta(t,R,S)$. Donc le nombre $T_{f,a,b}$ est également le coefficient de $z^f t^a \mu^b$ dans $\pd F z$. Ainsi, en sommant \eqref{ineqRT} sur tous les $b \geq 0$, nous obtenons l'inégalité voulue.\end{proof}

\subsection{La bijection forestière}

Décrivons rapidement la bijection $\Theta$ entre T-arbres enrichis $(T,F)$ et cartes forestières $(C,G)$ avec une face marquée $\chi$ qui généralise celle\footnote{Nous conseillons le lecteur de la relire avant d'aborder cette fin de chapitre.} de la section \ref{s:bb}. Elle :
\begin{itemize}
\item transforme les feuilles en faces non racine,
\item préserve le nombre de composantes,
\item préserve le nombre d'arêtes internes,
\item préserve les sommets et leur degré (ainsi que celui du sommet racine).
\end{itemize}
De plus, si $(T,F)$ et $(T',F')$ sont deux T-arbres enrichis avec $T = T'$, alors $C=C'$ et $\chi = \chi'$, où $(C,G,\chi)= \Theta(T,F)$ et $(C',G',\chi')= \Theta(T',F')$. 

\noindent \textbf{Remarque 1.} Cette dernière propriété (celle qui stipule que la carte reste la même si seule la forêt couvrante de l'arbre enrichi change) est ce qui nous motive dans l'adaptation de la bijection aux cartes forestières. Nous verrons ci-dessous (remarque 2) pourquoi.

Le processus de clôture $\Theta$ est essentiellement le même que celui décrit sous-section \ref{ss:cloture} : pour $(T,F)$ un T-arbre enrichi, nous relions tout couple bourgeon/feuille qui se suivent le long de la face externe. Nous répétons cela jusqu'à qu'il n'y ait plus de feuilles et de bourgeons. La face $\chi$ marquée est alors la face externe. On pose alors $\Theta(T,F) = (C,G,\chi)$, où $(C,G)$ est la carte forestière obtenue. En réalité, la forêt $F$ ne joue aucune rôle dans cette transformation. Les différentes étapes de cette bijection sont montrées figure \ref{thetaverscartes}.

%Nous réutilisons  le procédé que nous avions employé dans le point 5. de la preuve du théorème \ref{g1rs} p. \pageref{g1rs} (mais formulé de manière différente) : si une feuille suit immédiatement un bourgeon lorsque nous longeons la face externe dans le sens trigonométrique, alors nous relions les deux pour former une arête. Cette arête est tracée de sorte que si on l'oriente dans le sens bourgeon-feuille, la face externe se trouve à la droite de cette arête. Il se peut que le coin racine se trouve entre un bourgeon et une feuille. Cela ne change rien, le bourgeon et la feuille seront reliés et le coin racine sera dans une face non externe. Nous répétons ce processus jusqu'à qu'il n'y ait plus de feuille et de bourgeon. Nous obtenons alors une carte forestière $(C,G)$. Nous marquons la face externe que nous appelons $\chi$. On note $(C,G,\chi) = \Theta(T,F)$. La figure \ref{thetaverscartes} illustre les différentes étapes de cette transformation.

\fig{[width = \textwidth]{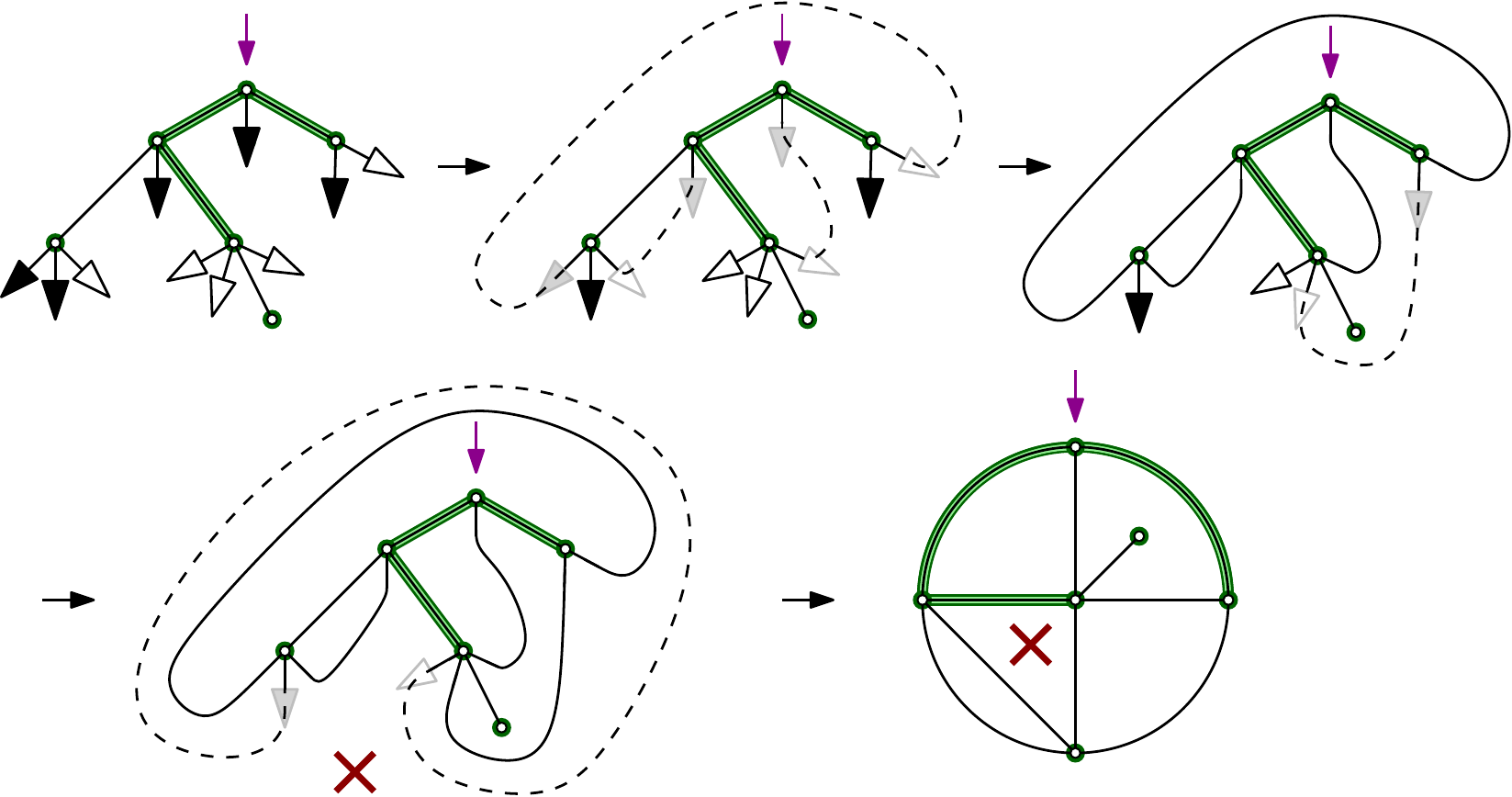}}{Les différentes étapes de la transformation $\Theta$. Les deux dernières cartes sont exactement les mêmes, la dernière est juste dessinée de manière plus esthétique.}{thetaverscartes}

%Notons que la carte obtenue ne dépend pas de l'ordre dans lequel on a choisi de regrouper les couples bourgeon/feuille. En effet, chaque bourgeon est associé de manière unique à une feuille. Pour le voir, il suffit de considérer le chemin de pas montants et descendants comme décrit dans le point 1. de la démonstration du théorème \ref{g1rs} : un bourgeon et une feuille sont associés si et seulement si les pas correspondants se font face.
%

%De plus, nous pouvons voir que chaque feuille de $T$ donne naissance à une arête et une face non racine quand elle est reliée à un bourgeon. Le nombre de faces non racines de $C$ est donc égal au nombre de feuilles dans $T$, et le nombre d'arêtes dans $C$ est égal au nombre total de feuilles et d'arêtes dans $T$. En outre, comme nous l'avions annoncé, les composantes connexes de la forêt, les arêtes internes et les degrés des sommets sont préservés par $\Theta$.
%
%\subsection{La bijection bourgeonnante : des cartes forestières aux $\boldsymbol \theta$-arbres}
%

Le processus d'ouverture $\Theta^*$, application réciproque de $\Theta$, est pratiquement le même que celui de la sous-section \ref{ss:ouverture}, sauf qu'il faut tenir compte en plus de la forêt couvrante : lors de chaque itération, l'ensemble $E$ considéré est l'ensemble des arêtes \textbf{externes} incidentes à la face $\chi$ qui ne sont pas des isthmes. \`A part cela, le reste ne change pas, nous transformons (en prenant certaines précautions) les arêtes de $E$ en couples bourgeon/feuille, selon une orientation dictée par la face $\chi$. On se réfèrera à la sous-section~\ref{ss:ouverture} p.~\pageref{ss:ouverture} pour avoir la description détaillée de cette transformation. Les différentes étapes sont illustrées figure \ref{cartesverstheta}. 

\fig{[width = \textwidth]{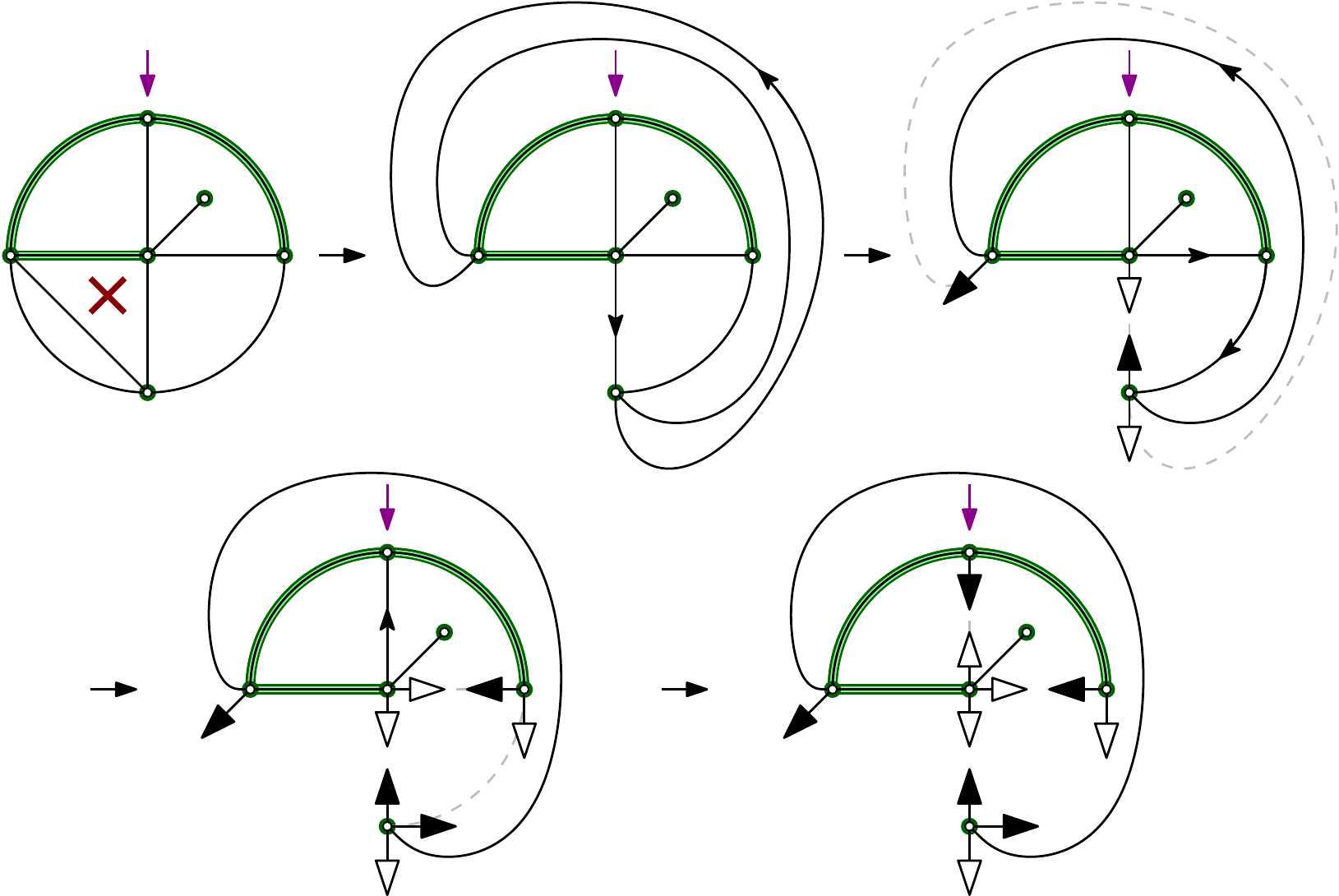}}{Les différentes étapes de la transformation $\Theta^*$.}{cartesverstheta}

\noindent \textbf{Remarque 2.} Soit $C$ une carte planaire. Expliquons de manière combinatoire pourquoi la série génératrice $A(u)$ des forêts couvrantes de $C$ avec un poids $u$ par composante non racine est $(u+1)$-positive. Bien sûr, cette $(u+1)$-positivité peut se justifier en termes du polynôme de Tutte : $A(u)$ est égal à $T_G(u+1,1)$. Mais ce qui est intéressant ici, c'est de voir cette $(u+1)$-positivité à travers le spectre de la bijection bourgeonnante. 

Considérons $\chi$ n'importe quelle face de $C$ (la face racine est un choix judicieux). Pour toute forêt couvrante $G$ de $C$, on peut associer un T-arbre enrichi $(T(G),F(G))$ par l'application $\Theta^*$. Appelons $\mathcal T_C$ l'ensemble $$\mathcal T_C = \enstq {T(G)} {G\textrm{ forêt couvrante de }C}.$$ D'après la propriété décrite plus haut, pour tout  T-arbre enrichi $(T,F)$ avec $T \in \mathcal T_C$, la carte et la face associées à $\Theta(T,F)$ restent $C$ et $\chi$. Par conséquent,  l'ensemble des forêts couvrantes de 
$C$ est en bijection (à travers $\Theta^*$) avec l'ensemble $\mathcal E$ des T-arbres enrichis $(T,F)$ tels que $T \in \mathcal T_C$. Comme $\Theta^*$ préserve le nombre de composantes dans la forêt, $A(u)$ est également la série génératrice de $\mathcal E$. Mais $\mathcal E$ est stable (car la classe des T-arbres enrichis est stable). Par conséquent, d'après la proposition \ref{p:stable}, la série $A(u)$ est $(u+1)$-positive. La figure \ref{bbforet} illustre ce raisonnement.

\fig{[height = 0.5 \textheight]{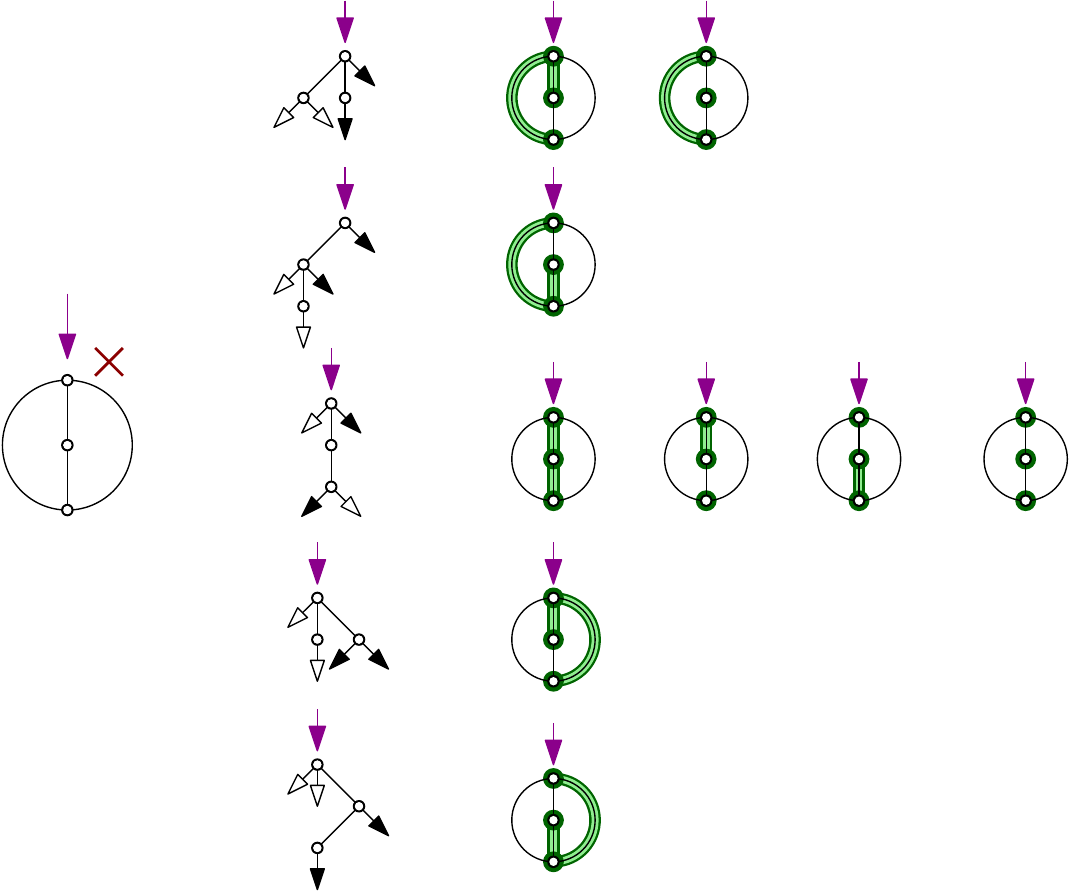}}{ \`A gauche : une carte $C$ où on a marqué la face externe $\chi$. Au milieu : l'ensemble $\mathcal T_C$. \`A droite : Les $9$ forêts couvrantes de $C$ rangées de manière à ce que l'arbre bourgeonnant à gauche de n'importe quelle carte forestière $(C,G)$ soit égal à $\Theta^*(C,G,\chi)$. }{bbforet}

Nous verrons dans le chapitre \ref{c:bloact} que cette propriété de $(u+1)$-positivité peut directement s'interpréter au niveau de la carte $C$. En effet, transposées par la bijection $\Theta$, les arêtes basculables permettent de définir ce qu'on appellera une \textit{activité interne}.

\chapter{Quelques résultats sur les fonctions implicites}
\label{c:implicites}

Le comportement singulier d'une série $Y(z)$ définie par une équation fonctionnelle $H(z,Y(z))=0$ est bien connu lorsque $Y$ est singulier en des points $z$ tels que $H$ est analytique en $(z,Y(z))$ et $\pd H y(z,Y(z)) = 0$. Cela se produit typiquement lorsque $Y$ satisfait un \textit{schéma d'équation fonctionnelle lisse}\footnote{traduction de l'anglais "smooth implicit fonction schema"} (voir \cite[Section VII.4]{flajolet-sedgewick}), auquel cas $Y$ est analytique sur un $\Delta$-domaine avec une singularité de type racine.
Toutefois, nous rencontrerons dans les chapitres \ref{c:asympt} à \ref{c:cubique} des équations fonctionnelles  $H(z,Y(z))=0$ inhabituelles dans la mesure où la fonction implicite $Y$ devient singulière en un point $z$ tel que $H$ est singulier en $(z,Y(z))$. Nous établissons dans ce chapitre des résultats portant sur de telles fonctions.

\section{Fonctions implicites sur l'axe réel positif}

D'après le théorème de Pringsheim \cite[Théo. IV.6 p. 340]{flajolet-sedgewick}, une série à coefficients positifs a une singularité dominante sur l'axe réel positif. Chercher la plus petite singularité réelle positive revêtira donc un intérêt. En effet, même si les séries n'auront pas toutes des coefficients positifs, on pourra se ramener à chaque fois à des séries à coefficients positifs, à transformation affine près (cf. section \ref{s:abe}).

Dans cette section, nous chercherons à voir jusqu'à quel point une fonction implicite $Y$ réelle peut se prolonger sur l'axe réel positif. Nous établissons dans un premier temps un résultat général pour des équations de la forme $H(z,Y(z))=0$. Nous l'appliquerons à la série $\tilde S$ définie par le lemme \ref{l:Stilde} p. \pageref{l:Stilde}.  Nous spécialiserons ce résultat aux équations de la forme $\Omega(Y(z)) = z$. Ce corollaire s'appliquera notamment à la série $R$ du cas eulérien qui satisfait $R/t = z + u \phi(R)$ (voir \eqref{eulR}). 

\begin{theo} \label{t:fi}
Soit $H(x,y)$ une série formelle bivariée réelle, analytique dans un voisinage de $(0,0)$, vérifiant $H(0,0) = 0$ et $\pd H y (0,0) > 0$. Il existe une unique série  formelle $Y$  de variable $z$ qui satisfait $Y(0) = 0$ et $H(z,Y(z)) = 0$. Alors $Y$ a un rayon de convergence non nul, elle définit donc une fonction holomorphe au voisinage de $0$ qu'on notera également $Y$. En outre, il existe $\rho>0$ tel que :
\begin{enumerate}
\item[(a)] $Y$ est analytiquement prolongeable dans un voisinage de $[0,\rho]$ et prend des valeurs réelles sur cet intervalle,
\item[(b)] $H$ est analytiquement prolongeable dans un voisinage de  $\enstq{ (z,Y(z)) } { z \in \, [0,\rho]}$,
\item[(c)]  $H(z,Y(z)) = 0$ pour $z \in [0,\rho]$,
\item[(d)] $\pd H y (z,Y(z)) > 0$ pour $z \in [0,\rho]$.
\end{enumerate}
De plus, si $\tilde \rho$ désigne la borne supérieure ( dans $\R \cup \ens{+\infty}$ ) des valeurs possibles de $\rho$, alors au moins une des propriétés suivantes, illustrées par la figure \ref{implicit}, est vérifiée :
\begin{enumerate}
\item[(i)] $\tilde \rho = + \infty$
\item[(ii)] $\liminf_{z \rightarrow \tilde \rho^-} \pd H y (z,Y(z)) = 0$,
\item[(iii)] $H$ est singulier en $(\tilde \rho,y)$ pour chaque $y \in \left[  \liminf_{z \rightarrow \tilde \rho^-} Y(z) , \limsup_{z \rightarrow \tilde \rho^-} Y(z) \right]$,
\item[(iv)] $\limsup_{z \rightarrow \tilde \rho^-} |Y(z)| = + \infty$.
\end{enumerate}

\fig{[width = \textwidth]{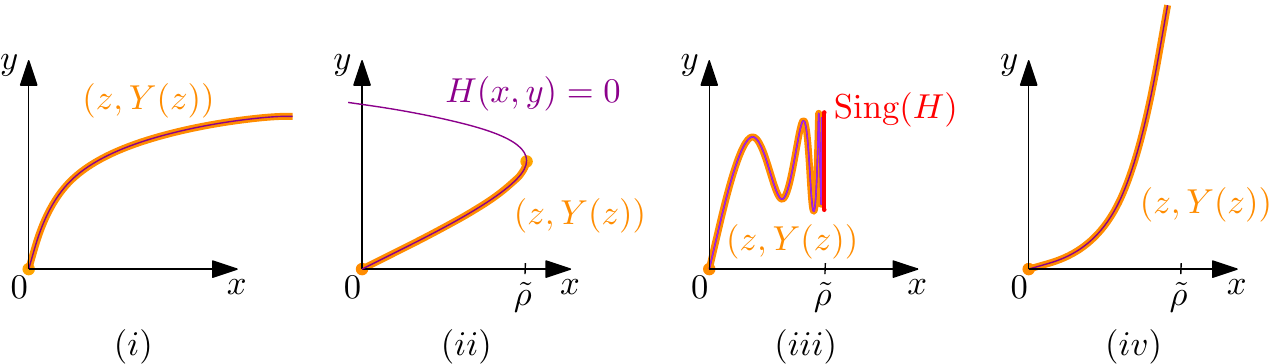}}{Quatre possibilités (non exclusives) pour $\tilde \rho$.}{implicit}

\end{theo}

Le lemme suivant permet de prouver que la série $Y$ reste sur l'axe réel. 

\begin{lem} \label{l:valpos}
Soient $a < 0 < b$ et $Y$ une fonction analytique dans un voisinage de $[a,b]$. Si la série de Taylor de $Y$ en $0$ a des coefficients réels, alors $Y$ prend des valeurs réelles sur $[a,b]$.
\end{lem}
\begin{proof}
Les deux fonctions $z \mapsto Y(z)$ et $z \mapsto \overline {Y\pare{\overline{z}}}$ sont analytiques et coïncident au voisinage de $0$. Par le théorème de prolongement analytique, elles coincident sur un voisinage de $[a,b]$. Par conséquent, $Y(z)$ est réel si $z \in \, [a,b]$.
\end{proof}

\begin{proof}[Démonstration du théorème \ref{t:fi}] L'unicité de $Y$ provient du fait que ses coefficients peuvent être calculés par induction grâce à l'équation $H(z,Y(z)) = 0$ et la condition initiale $Y(0) = 0$ (l'hypothèse $\pd H y(0,0)$ est cruciale ici). Notons que ces coefficients sont bien réels, il sera donc possible d'appliquer le lemme \ref{l:valpos}. Mais avant nous devons prouver que $Y$ a un rayon de convergence non nul. Puisque $\pd H y(0,0) > 0$, la version analytique du théorème des fonctions implicites (voir par exemple \cite[Théo. B.4 p 753]{flajolet-sedgewick} s'applique en $z = 0$ : il existe localement une fonction analytique $\hat Y$ qui est solution de l'équation fonctionnelle $H(z,\hat Y(z)) = 0$ avec $\hat Y(0) = 0$. La série de Taylor de $\hat Y$ en $0$ satisfait la même équation (en tant que série formelle), donc coïncide avec $Y$. On a bien prouvé que le rayon de convergence de $Y$ n'était pas nul.

Considérons maintenant l'ensemble
$$ I = \enstq{ \rho > 0 } {\rho\textrm{ satisfait les conditions (a), (b), (c) et (d)}}.$$
Il s'agit clairement d'un intervalle ouvert de la forme $]0,\tilde \rho[$. En outre, l'ensemble $I$ n'est pas vide puisque les propriétés (a), (b), (c) et (d) sont vérifiées localement en $0$. Supposons qu'aucune des possibilités $(i)$, $(ii)$, $(ii)$ et $(iv)$ ne soit vraie en $\tilde \rho$. En particulier $\tilde \rho$ est fini. Nous voulons arriver à une contradiction en prouvant que $\tilde \rho \in I$.

Comme $(iv)$ n'est pas vérifié,  $Y$ est bornée sur $[0, \tilde \rho[$. Par continuité, l'ensemble des points d'accumulation de $\enstq { Y(z) } { z \in \, [0,\tilde  \rho[}$ est un intervalle, qui coïncide avec $
\left[  \liminf_{z \rightarrow \tilde \rho^-} Y(z) , \limsup_{z \rightarrow \tilde \rho^-} Y(z) \right]$. Pour chaque $y$ dans cet intervalle, le point $(\tilde \rho,y)$ est dans l'adhérence de l'ensemble $\enstq{ (z,Y(z)) } { z \in \, [0,\tilde  \rho[}$ sur lequel $H$ est analytique (voir propriété (b)). Puisque $(iii)$ est supposée fausse, il existe $\tilde y$ dans $
\left[  \liminf_{z \rightarrow \tilde \rho^-} Y(z) , \limsup_{z \rightarrow \tilde \rho^-} Y(z) \right]$ tel que $H$ est analytique en $(\tilde \rho, \tilde y)$. En particulier, $H$ est continue en ce  point, donc $(c)$ implique $H(\tilde \rho,\tilde y) = 0$. Enfin, comme (d) est vraie mais pas (ii), nous avons $\pd H y (\tilde \rho,\tilde y) > 0$.
Ces trois dernières conditions permettent d'appliquer le théorème des fonctions implicites : il existe une fonction analytique $\tilde Y$ définie au voisinage de $\tilde \rho$ telle que $H(z,\tilde Y (z)) = 0$ et $\tilde Y(\tilde \rho) = \tilde y$. Nous souhaitons prouver que $\tilde Y$ est un prolongement analytique de $Y$ en $\tilde \rho$ (en particulier, l'intervalle $\left[  \liminf_{z \rightarrow \tilde \rho^-} Y(z) , \limsup_{z \rightarrow \tilde \rho^-} Y(z) \right]$ est réduit à $\tilde y$).

Puisque $\pd H y (\tilde \rho,\tilde y) > 0$, il existe $\delta > 0$ et un voisinage complexe $V$ de $(\tilde \rho,\tilde y)$ tel que pour tout $(x,y)$ et $(x,y')$ dans $V$,
$$\module{H(x,y)-H(x,y')} \geq \delta \module{y - y'}.$$
Quitte à réduire $V$, nous pouvons également supposer que $\tilde Y (x)$ est bien définie pour $(x,y) \in V$. Comme $(\tilde \rho,\tilde y)$ est un point d'accumulation de  $\enstq { \pare{z,Y(z)} } { z \in \, [0,\tilde  \rho[}$, et comme $Y$ est continue, il existe un segment $[z_0,z_1] \subset ]0,\tilde \rho[$ avec $(z,Y(z)) \in V$ pour chaque $z \in \, [z_0,z_1]$. Par suite, pour $z$ dans cette intervalle,
$$ 0 = \module{H(z,Y(z)) - H(z,\tilde Y (z))} \geq \delta \module{Y(z) - \tilde Y (z)},$$
ce qui montre que les fonctions analytiques $Y$ et $\tilde Y$ coïncident sur $[z_0,z_1]$. La fonction $\tilde Y$ est donc un prolongement analytique de $Y$ en $\tilde \rho$. Autrement dit, la propriété (a) est vraie en $\tilde \rho$. De plus, nous avons (b) par définition de $\tilde y$, (c) par construction de $\tilde Y$ et (d) comme indiqué plus haut. Ainsi, $\tilde \rho$ appartient à $I$, ce qui contredit le fait que $\tilde \rho$ est la borne supérieure de l'intervalle ouvert $I$. Par conséquent, au moins une des propriétés $(i)$, $(ii)$, $(iii)$ et $(iv)$ est vérifiée. \end{proof}

Le corollaire suivant regarde spécifiquement le cas de l'équation fonctionnelle $\Omega(Y(z)) = z$.

\begin{cor}
Soit $\Omega(y)$ une série formelle réelle de rayon non nul avec $\Omega(0) = 0$ et $\Omega'(0) > 0$.  Notons $\omega \in \, ]0,+\infty]$ la plus petite singularité de $\Omega$ sur l'axe réel positif si elle existe, $+\infty$ sinon. Considérons $Y$ l'unique série vérifiant $Y(0) = 0$ et $\Omega(Y(z)) = z$. Alors $Y$ a un rayon de convergence non nul, elle définit donc une fonction analytique au voisinage de $0$ qu'on notera également $Y$. En outre, il existe $\rho \in \, ]0,+\infty[$ tel que :  \label{c:fi}
\begin{enumerate}
\item[(1)] $Y$ est analytique sur un voisinage de $[0,\rho[$  et prend des valeurs réelles sur cet intervalle,
\item[(2)] $Y$ est strictement croissante  sur $[0, \rho[$,
\item[(3)] $Y(z) \in \, [0,\omega[$ pour chaque $z \in \, [0,\rho[$,
\item[(4)] $\Omega(Y(z)) = z$ pour chaque $z \in \, [0,\rho[$,
\item[(5)] $\lim_{z \rightarrow \rho^-} Y(z)  = \tau$ et $\lim_{y \rightarrow \tau^-} \Omega(y) = \rho$, où 
\end{enumerate}
$$\tau = \min \enstq{ y \in \, [0,\omega[} {\Omega'(y) = 0}$$
si cet ensemble est non vide, et $\tau = \omega$ sinon.
\end{cor}

Ce corollaire est énoncé comme un résultat d'existence pour $\rho$, mais le point (5) détermine précisément sa valeur.

\begin{proof} Nous appliquons le théorème \ref{t:fi} à l'équation $H(x,y) = \Omega(y) - x$ (clairement, $H$ est analytique en $(0,0)$, et nous avons $H(0,0) = 0$ et $\pd H y (0,0) = \Omega'(0) > 0$). Nous renommons $\rho$ le nombre $\tilde \rho$ du théorème précédent. 

Le point (1) résulte de (a). D'après les conditions (b) et (c), $\Omega$ a un prolongement analytique sur $Y\pare{[0,\rho[}$ tel que $\Omega(Y(z)) = z$ pour $z \in \, [0,\rho[$ : c'est le point (4). En dérivant cette identité, nous obtenons $Y'(z)\,\Omega'(Y(z)) = 1$, ce qui sachant (d) entraîne (2). Par suite, l'existence d'un prolongement analytique de $\Omega$ sur $Y\pare{[0,\rho[}$ peut se  réécrire (3). La monotonie de $Y$ permet alors de définir $Y(\rho)$ comme $\lim_{z \rightarrow \rho^-} Y(z)$ (qui n'est pas forcément fini).

Déduisons maintenant (5)  des propriétés $(i)$-$(iv)$ du théorème \ref{t:fi}. Nous avons déjà vu (il s'agit du point (3)) que $Y(\rho) \leq \omega$. D'après la condition (d), et par définition de $\tau$, la valeur de $Y(\rho)$ est inférieure ou égale à $\tau$. Supposons $Y(\rho) < \tau$. Alors $\Omega$ est analytique en $Y(\rho)$, et par continuité de $\Omega$ et $Y$, nous avons $\rho = \Omega(Y(\rho)) <  + \infty$ : le point $(i)$ est infirmé.  D'après  $Y(\rho) < \tau$ et par définition de $\tau$, le point $(ii)$ est également infirmé. De même, $(iii)$ et $(iv)$ ne peuvent être vérifiés. Nous aboutissons donc à une contradiction : $Y(\rho) = \tau$. En revenant à (4), nous obtenons $\rho = \Omega(Y(\rho)) = \Omega(\tau)$.
\end{proof}

\section{Inversion de fonctions avec une singularité de la forme $\boldsymbol{z \ln z}$}

Le théorème d'inversion locale est un résultat classique d'analyse complexe : si une fonction $\psi$ est analytique sur un disque $C_s$ centré en $0$ de rayon $s$ avec $\psi(z) \sim z$ lorsque $z \rightarrow 0$, alors il existe $\rho \in \, ]0,s[$, $\rho' > 0$ et une fonction $\Upsilon$ analytique sur un disque $C_{\rho'}$ de rayon $\rho'$ tels que 
$$\forall\,(y,z) \in \, C_{\rho'} \times C_\rho, \quad \psi(z) = y \Longleftrightarrow z = \Upsilon(y). $$
Le but de cette section est d'adapter ce théorème à des fonctions analytiques avec une singularité de la forme $z \ln z$ au voisinage de $0$. Les domaines locaux considérés ne sont plus des disques, mais ils seront de la forme 
$$D_{\rho,\alpha} := \enstq{ r\,e^{i\theta} } {r \in \, ]0,\rho[\textrm{ et } |\theta| < \alpha}.$$

\begin{theo} \label{loginversion}
\textbf{(théorème de log-inversion)} Soit $\psi$ une fonction analytique sur $D_{s,\alpha}$, avec $s > 0$ et $\alpha  \in \, ]\pi/2,\pi]$. Supposons qu'il existe $c > 0$ tel que
$$\psi(z) \sim -c \, z \ln z$$
 quand $z$ tend vers $0$ dans  $D_{s,\alpha}$. Alors pour tout $\beta \in \, ]0,\alpha[$, il existe une fonction  $\Upsilon$ analytique sur $D_{\rho',\beta}$ et à valeurs dans $D_{\rho,\alpha}$, avec $\rho \in \, ]0,s[$ et $\rho' > 0$ (voir figure \ref{liv}) telle que
$$ \forall (y,z) \in \, D_{\rho',\beta} \times D_{\rho,\alpha}, \quad \psi(z) = y \Longleftrightarrow z = \Upsilon(y).$$

En outre, si $y$ tend vers $0$ dans $D_{\rho',\beta}$,
$$\Upsilon(y) \sim - \frac y {c\, \ln y}.$$
\end{theo}

\fig{[scale=1.3]{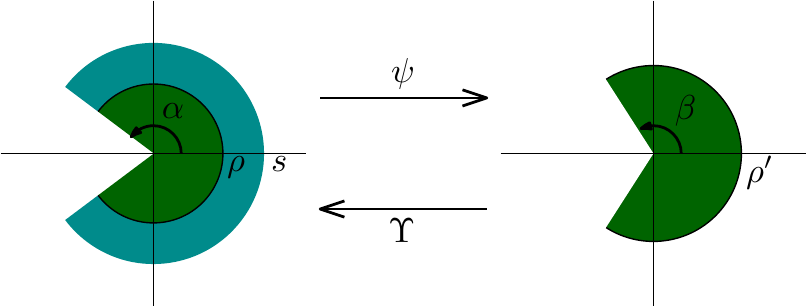}}{Domaines de définition de $\psi$ et $\Upsilon$.}{liv}

La démonstration de ce théorème est plutôt longue et technique. Le plus dur est de prouver que chaque élément de $D_{\rho',\beta}$ admet un unique antécédent dans $D_{\rho,\alpha}$ par $\psi$. Il s'agit de l'énoncé du lemme \ref{l:loginversion}. Une fois que nous avons montré l'existence d'une fonction localement réciproque $\Upsilon$, la preuve touche à sa fin : l'analyticité de $\Upsilon$ n'est qu'une simple application de la version analytique du théorème des fonctions implicites. Afin de prouver le lemme \ref{l:loginversion}, nous devons étudier dans un premier temps l'injectivité et la surjectivité de la fonction $H : z \mapsto - z\, \ln z$ au voisinage de $0$ (sous-section \ref{ss:zlnz}). Nous transférerons par la suite les propriétés de $H$ à la fonction $\psi$ grâce au théorème de Rouché (sous-section \ref{ss:lli}).

\subsection{La fonction $\boldsymbol{ z \mapsto - z\, \ln z}$}
\label{ss:zlnz}

Considérons la fonction
$$ \begin{array}{cccc} H : & \C\, \backslash \, \R^{-} & \longrightarrow & \C \\ \, &
z & \longmapsto & -z\, \ln z
\end{array},$$
où $\ln$ est la détermination principale du logarithme\footnote{Si $z = r\,e^{i\theta}$ avec $r > 0$ et $\theta \in\, ]-\pi,\pi[$, alors $\ln z = \ln r + i \theta$ est la détermination principale du logarithme de $z$. Par ailleurs, l'argument de $z$ se définit comme $\Arg z := \theta$.}. L'allure de la fonction est montrée figure \ref{zlz}. Commençons par exhiber quelques propriétés élémentaires sur $H$.

\fig{[width=0.7\textwidth]{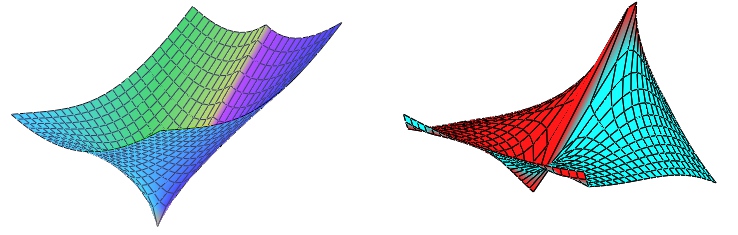}}{\`A gauche : le graphe de $z \mapsto \module { H(z)}$.  \`A droite : le graphe de $z \mapsto \Arg{H(z)}$. }{zlz}

\begin{lem} \label{Hbase}
La fonction $H$ vérifie 
\begin{equation} \label{Hconj}
H(\overline z) = \overline{H(z)}.
\end{equation}
Soit $z = r\,e^{i\theta}$ avec $r > 0$ et $\theta \in\, ]-\pi,\pi[$. Alors
\begin{equation} \label{Hmodule}
|H(z)| = r \sqrt{\ln^2 r + \theta^2}.
\end{equation}
Les arguments de $z$ et $-\ln z$ sont de signes opposés. Si nous supposons en outre $r < 1$, alors $\Arg(-\ln z) \in \, ]-\pi/2,\pi/2[$, d'où 
\begin{equation}
\label{Harg}
\Arg H(z) = \Arg z + \Arg( - \ln z) = \theta + \arctan \left(\frac  \theta {\ln r}\right).
\end{equation}
Si en outre $r \leq 1/\sqrt e$, alors
\begin{equation} \label{Hmajarg}
\left| \Arg H(z) \right| \leq \module \theta.
\end{equation}
En particulier, $H(z) \notin \, \R^-$.
\end{lem}

\fig{[scale=1]{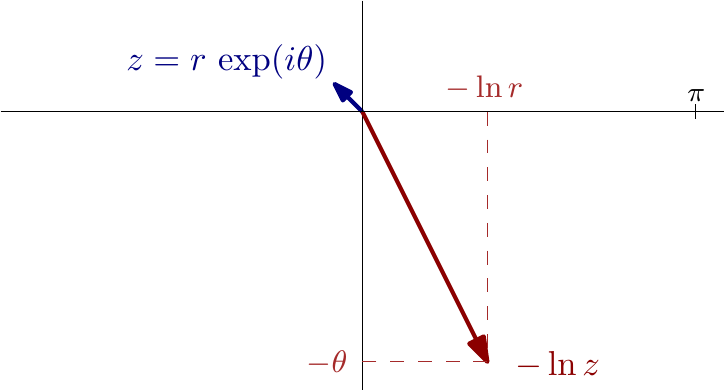}}{Comparaison entre $z$ et $-\ln z$, lorsque $z$ est petit.}{complnz}

\begin{proof}Nous pouvons nous aider de la figure \ref{complnz} pour une meilleure visualisation. Les deux premières identités sont évidentes. La première égalité de \eqref{Harg} est vraie parce que $z$ et $-\ln z$ ont des arguments de signes opposés. La seconde égalité se déduit de $\Arg(-\ln z) \in \, ]-\pi/2,\pi/2[$. Afin de prouver \eqref{Hmajarg}, supposons dans un premier temps que $\theta \geq 0$. L'argument de $- \ln z$ est négatif et la première partie de \eqref{Harg}  prouve que $\Arg H(z) \leq \theta$. De surcroît, $\arctan x \geq x$ si $x \leq 0$, donc d'après la seconde partie de \eqref{Harg},
$$\Arg H(z) \geq \theta + \frac {\theta} {\ln z} \geq \pare{1 - \frac  1 {\ln\pare{1/\sqrt e}}}\theta = - \theta.$$
Le cas $\theta < 0$ est immédiat au vu de \eqref{Hconj}.
\end{proof}

La fonction $H$ n'est pas injective sur $\C$ : par exemple, $H(i) = \pi/2 = H(-i)$. Toutefois, elle est localement injective en 0.

\begin{prop} \label{Hinj}
La fonction $H : z \mapsto -z \ln z$ est injective sur $D_{e^{-1},\pi}$.\end{prop}
\begin{proof} Supposons l'existence de deux complexes $z_1$ et $z_2$ dans $D_{e^{-1},\pi}$ avec $H(z_1) = H(z_2)$. D'après le lemme \ref{Hbase}, le complexe $H(z_1)$ n'est pas un réel négatif. Le passage au logarithme est donc licite : $\ln H(z_1) = \ln H(z_2)$.

Lorsque $z \in \, D_{e^{-1},\pi}$, le lemme nous indique également que $z$ et $-\ln z$ ont des arguments de signes opposés, d'où
\begin{equation} \label{Haddlog}
\ln H(z) = \ln z + \ln(-\ln z).
\end{equation}
Ainsi, puisque $\ln H (z_1) = \ln H (z_2)$,
\begin{equation} \label{lnz12}
\module{\ln z_1 - \ln z_2} = \module{\ln(-\ln(z_1)) - \ln(-\ln(z_2))}.
\end{equation}
Posons $\kappa = - \max \pare{\ln|z_1|,\ln|z_2|}>1$. Les deux complexes $- \ln z_1$ et $- \ln z_2$ se trouvent dans $\enstq z {\Ree(z) \geq  \kappa}$. Cet ensemble est convexe, on peut donc appliquer l'inégalité des accroissements finis (pour les fonctions à valeurs vectorielles, ici complexes) :
$$ \module{\ln(-\ln(z_1)) - \ln(-\ln(z_2))} \leq \module{\ln z_1 - \ln z_2} \sup_{z \in \, [-\ln z_1,- \ln z_2]}\module{\ln'(z)} \leq \frac 1 \kappa \module{\ln z_1 - \ln z_2}. $$
Combiné avec \eqref{lnz12} et $\kappa > 1$, cela donne $\module{\ln z_1 - \ln z_2} = 0$, soit $z_1 = z_2$. \end{proof}

Prouvons maintenant que l'image par $H$ d'un domaine local $D_{\rho,\alpha}$ s'approche d'un domaine local de même forme\footnote{Pour compléter la proposition \ref{Hsurj}, mentionnons la propriété suivante, que nous ne prouverons pas : pour tout $\varepsilon > 0$ et $\alpha \in \  ]0,\pi]$, il existe $\rho > 0$ suffisamment petit pour que $H(D_{\rho,\alpha}) \subseteq D_{-(1+\varepsilon)\rho \ln \rho, \alpha}$.}.

\begin{prop} \label{Hsurj} Pour tout couple $(\alpha,\beta)$ avec $0 < \beta < \alpha \leq \pi$, il existe $\rho > 0$ suffisamment petit pour que 
$$D_{-\rho \ln \rho, \beta} \subseteq H\pare{D_{\rho,\alpha}}.$$
\end{prop}
\begin{proof} Quand $\rho$ est suffisamment petit, nous avons 
\begin{equation} \label{rhopetit}
\arctan \pare{\frac \alpha {\module{\ln \rho}}} \leq \alpha - \beta.
\end{equation}
Nous allons prouver que l'inclusion de la proposition \ref{Hsurj} est vraie pour tout \mbox{$\rho \in \, ]0,e^{-1}[$} vérifiant  \eqref{rhopetit}. Fixons un complexe $r'e^{i\gamma}$ avec $0 < r' < - \rho \ln \rho$ et $|\gamma| < \beta$. Nous voulons prouver l'existence de $r < \rho$ et $\theta \in \  ]-\alpha,\alpha[$ tels que $H\pare{r e^{i \theta}} = r' e^{i \gamma}$. Nous procédons en deux étapes.

\noindent \textbf{(1)  Prouvons qu'il existe une fonction continue $\boldsymbol{\theta : \,]0,\rho[ \rightarrow \, ]-\alpha,\alpha[}$ telle  que }
$$\forall \, r \in \, ]0,\rho[,\quad \Arg H\pare{r\, e^{i \theta(r)}} = \gamma.$$
Fixons dans un premier temps $r \in \, ]0,\rho[$. Pour tout $\theta \in \, ]-\alpha,\alpha[$, posons
$$f(r,\theta) := \Arg H\pare{r e^{i \theta}} = \theta + \arctan \pare {\frac \theta {\ln r}},$$
où la dernière égalité a été obtenue grâce à \eqref{Harg}. La différentiation de $f$ par rapport à $\theta$ donne 
$$\pd f \theta (r, \theta) = 1 + \frac 1 {\pare{1 + \frac {\theta^2} {\ln^2 r}}\ln r} \geq 1 + \frac 1 {\ln r} > 0.$$
Ainsi $f$ est une fonction continue strictement croissante en $\theta$, envoyant $]-\alpha,\alpha[$ sur $]-\alpha - \arctan(\alpha/ \ln r), \alpha + \arctan(\alpha/ \ln r)[$. Comme $r < \rho$ et que $\rho$ vérifie \eqref{rhopetit}, ce dernier intervalle contient $]-\beta,\beta[$, et donc la valeur $\gamma$. Cela montre l'existence et l'unicité (due à la stricte croissance de $f$) de $\theta(r)$ avec  $\Arg H\pare{r\, e^{i \theta(r)}} = \gamma$.

Nous pouvons alors appliquer le théorème des fonctions implicites à l'équation $f(r,\theta) = \gamma$ au voisinage de $(r,\theta(r))$ afin de montrer la continuité de $\theta$ sur $]0,\rho[$.

\noindent \textbf{(2) Prouvons l'existence d'un $\boldsymbol{ r \in \, ]0,\rho[}$ tel que $ \boldsymbol{ \module { H\pare{r\, e^{i \theta(r)}}}  = r' }$.} \\
La fonction
$$ r \longmapsto \module { H\pare{r\, e^{i \theta(r)}}} = r \sqrt{\ln^2 r + \theta(r)^2}$$
est continue sur $]0,\rho[$. Elle tend vers $0$ quand $r$ tend vers $0$ et dépasse $-\rho \ln \rho$ quand $r$ tend vers $\rho$. Comme $0 < r' < - \rho \ln \rho$, nous pouvons appliquer le théorème des valeurs intermédiaires et ainsi montrer qu'il existe $r \in \, ]0,\rho[$ tel que $\module { H\pare{r\, e^{i \theta(r)}}}  = r'.$ \end{proof}

\subsection{Le théorème de log-inversion}
\label{ss:lli}

Mises ensemble, les deux propositions \ref{Hinj} et \ref{Hsurj} montrent que si $\rho$ est suffisamment petit, alors tout point de $D_{-\rho \ln \rho, \beta}$ a un unique antécédent par  $H$ dans $D_{\rho,\alpha}$, où  $0 < \beta < \alpha \leq \pi$. Nous voulons maintenant adapter ce résultat aux fonctions $\psi$ se comportant comme $H$ au voisinage de l'origine.

\begin{lem} \label{l:loginversion}
Soit $\psi$ une fonction analytique sur $D_{s,\alpha}$ avec $s > 0$ et $\alpha \in \, ]0,\pi]$. Supposons que 
$$ \psi(z) \sim H(z) = -z \ln z$$
lorsque $z$ tend vers $0$ dans $D_{s,\alpha}$. Alors pour tout $\beta \in \, ]0,\alpha[$, il existe  $\rho \in \, ]0,s[$ et $\rho' > 0$ tels que tout point de $D_{\rho',\beta}$  a un unique antécédent par $\psi$ dans $D_{\rho,\alpha}$. 
\end{lem}

\begin{proof} Par hypothèse, $\psi(z) - H(z) = o\pare{ z \ln z} = o \pare{- |z| \ln |z|}$ quand $z$ tend vers $0$. Considérons $\rho \in \, ]0,s[$ suffisamment petit pour que tout complexe $z \in \, D_{\rho,\alpha}$ vérifie
\begin{eqnarray}
\module{\psi(z) - H(z)} & < & - \min \pare { \frac 1 2, \sin \pare{ \frac {\alpha -\beta} 4  }  } |z| \ln |z|, \label{42} \\
1 + \frac 1 {\ln |z|} & > & \frac 1 2 + \frac \beta {\alpha + \beta}, \label{43} \\ 
|z \ln z| & \leq & -2 |z| \ln |z|, \label{44}  \\
- \ln \frac {|z|} 8 & \leq & - 2 \ln |z| \label{45} .
\end{eqnarray}
En outre, supposons -- et cela est possible grâce à la proposition \ref{Hsurj} -- que $\rho$ soit suffisamment petit pour que 
\begin{equation} \label{46}
D_{-\frac \rho 8 \ln \frac \rho 8,\beta} \subseteq	H \pare{ D_{\frac \rho 8,\alpha}}.
\end{equation}
Les conditions \eqref{43}-\eqref{45} pourraient être changées par des majorations explicites de $\rho$ (par exemple, \eqref{45} est équivalent à $\rho  \leq 1/8$) mais en l'occurrence il est plus pratique de les laisser sous cette forme.

Fixons maintenant $y_0 \in \, D_{\rho',\beta}$ avec $\rho' = - \frac \rho 8 \ln \frac \rho 8$.  D'après la proposition \ref{Hinj} et \eqref{46}, nous savons que $y_0$ a dans $D_{\rho,\alpha}$ un unique antécédent par $H$, qu'on appellera $z_0$. Nous voulons prouver que $y_0$ en a également un par $\psi$. En d'autres termes, nous souhaitons montrer que les fonctions $H - y_0$ et $\psi - y_0$ ont même nombre de zéros dans $D_{\rho,\alpha}$ (à savoir un seul). Pour cela, nous utilisons le théorème de Rouché  \cite[p. 270]{flajolet-sedgewick}.

\fig{[scale=1]{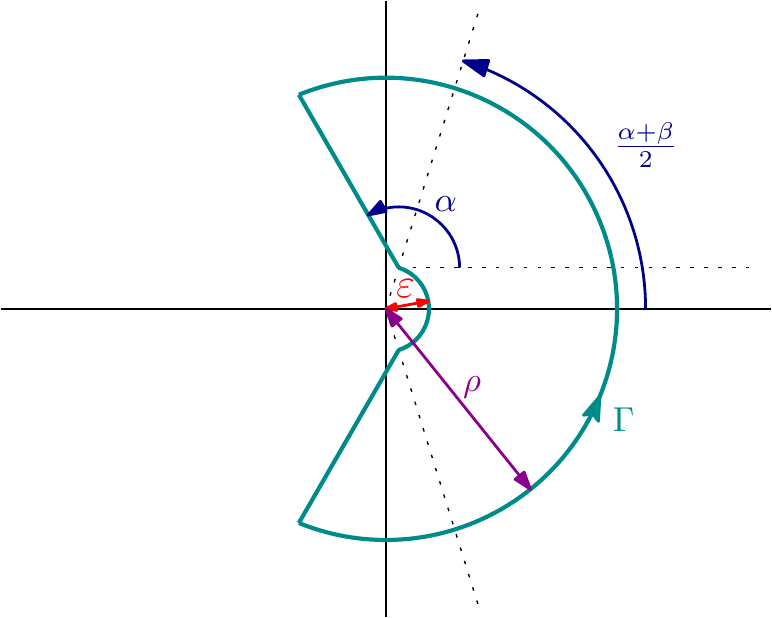}}{Le contour $\Gamma$.}{contour}

Plus précisément, pour $\varepsilon \in \, ]0,|z_0|/8[$, considérons $\Gamma$ le lacet de la figure \ref{contour} et $\Delta$ l'ensemble ouvert délimité par $\Gamma$. Le domaine $\Delta$ converge vers $D_{\rho,\alpha}$ quand $\varepsilon$ tend vers vers $0$. Donc quand $\varepsilon$ sera suffisamment petit, ce domaine contiendra le point $z_0$. De fait, nous devons démontrer que $H - y_0$ et $\psi - y_0$ ont le même nombre de racines dans $\Delta$, pour tout $\varepsilon$ suffisamment petit. D'après le théorème de Rouché\footnote{Les fonctions $H - y_0$ et $\psi - y_0$ sont holomorphes et donc n'ont pas de pôles.}, il suffit de prouver que $\module{\psi - H} < \module{H-y_0}$ sur le lacet $\Gamma$. Décomposons $\Gamma$ en trois parties (non disjointes) :
$$\Gamma_1 = \Gamma \cap \enstq z {\module z = \rho}, \quad \Gamma_2 = \Gamma \cap \enstq z {\module z = \varepsilon}\quad\textrm{ et }\quad \Gamma_3 = \Gamma \cap \enstq z {\module {\Arg z} > \frac {\alpha+\beta} 2 }.$$

Pour l'étude de $\Gamma_1$ et de $\Gamma_2$, nous utiliserons le résultat élémentaire suivant.

\begin{lem} \label{elcomplexe}
 Si $z$ et $z'$ sont deux nombres complexes non négatifs avec $8 \module{z'} \leq \module z \leq \rho$, alors
$$\module{z \ln z - z' \ln z'} \geq - \frac 1 2 \module z \ln \module z.$$
\end{lem}
\begin{proof}
Nous avons les minorations suivantes :
\begin{align*}
\module{z \ln z - z' \ln z'} & \geq \module{z \ln z} - \module{z' \ln z'}     \\
\  & \geq  - \module z \ln \module z + 2 \module z' \ln \module z' && \textrm{d'après \eqref{Hmodule} et \eqref{44}}, \\
\  & \geq  - \module z \ln \module z + \frac{ \module z'} 4 \ln \frac {\module z'} 8  && \textrm{car }\module {z'} \leq \frac {\module{z}} 8, \\
\ & \geq - \frac 1 2 \module z \ln \module z && \textrm{d'après }\eqref{45}.
\end{align*}
L'inégalité est ainsi prouvée.
\end{proof}
Soit $z \in \, \Gamma_1$. D'après \eqref{46}, nous savons que $|z_0| \leq \rho / 8$. Nous pouvons donc appliquer le lemme \ref{elcomplexe} aux nombres $z$ et $z_0$. Grâce à \eqref{42}, on trouve alors
$$\module{H(z) - y_0} \geq  - \frac 1 2 \module z \ln \module z > \module{\psi(z)-H(z)}.$$
Considérons maintenant $z \in \, \Gamma_2$. Comme $\varepsilon < |z_0|/8$, le lemme \ref{elcomplexe} s'applique aux nombres $z_0$ et $z$, où cette fois $z$ joue le rôle de $z'$ et $z_0$ le rôle de $z$. Alors, d'après \eqref{42}, 
$$\module{ y_0 - H(z)} \geq  - \frac 1 2 \module {z_0} \ln \module {z_0} \geq  - \frac 1 2 \module z \ln \module z > \module{\psi(z)-H(z)}.$$
Il nous reste plus qu'à étudier le cas $z \in \, \Gamma_3$. Nous souhaitons dans un premier temps montrer que si  $z \in \, \Gamma_3$, alors
\begin{equation} \label{47}
\module {\Arg H(z)} \geq \beta + \frac{\alpha - \beta} 4.
\end{equation}
Nous pouvons restreindre la preuve de cette inégalité au cas où $ \Arg z \geq 0$ puisque $H(\overline z) = \overline{H(z)}$ (voir lemme \ref{Hbase}). Nous avons alors
\begin{align*}
\Arg H(z) & = \Arg z + \arctan\pare{ \frac {\Arg z}  {\ln \module z}  } && \textrm{d'après le lemme \eqref{Hbase},} \\
\ & \geq \pare { 1 + \frac 1 {\ln \module z} } \Arg z && \textrm{car } \arctan x \geq x, \\
\ & > \pare{\frac 1 2 + \frac \beta {\alpha + \beta}} \Arg z && \textrm{d'après \eqref{43}}, \\
\ & \geq \beta + \frac{\alpha - \beta} 4 && \textrm{car } \Arg z \geq \frac{\alpha +  \beta} 2.
\end{align*}
Ainsi \eqref{47} est vérifié par tout point $z \in \, \Gamma_3$. Mais puisque $\module {\Arg y_0} < \beta$, nous avons 
\begin{equation} \label{48}
\module {\Arg H(z) - \Arg y_0} > \frac{\alpha - \beta} 4.
\end{equation}
Nous avons alors besoin d'un dernier résultat élémentaire :\begin{lem} Soient deux nombres complexes $a$ et $b$ non négatifs et $\gamma > 0$. \\ Si $\module{\Arg a - \Arg b} \geq \gamma$, alors $|a-b| \geq \module a \sin \gamma$.
\end{lem}
\begin{proof} Si $\module{\Arg a - \Arg b} \leq \frac \pi 2$, alors
$$\module{a-b} = \module{ \module a e^{i\pare{\Arg a - \Arg b}} - \module b} \geq \module {\Ima\pare{\module a  e^{i\pare{\Arg a - \Arg b}} } } \geq \module a \sin \gamma.$$
Si $\module{\Arg a - \Arg b} \geq \frac \pi 2$, alors
\begin{multline*}
\module{a-b} = \module{ \module a  - \module b e^{i\pare{\Arg a - \Arg b}}} \geq \module{\Ree \pare{ \module a  - \module b e^{i\pare{\Arg a - \Arg b}}  }} \\
= \module a  - \module b \cos\pare{ \Arg a - \Arg b  } \geq \module a \geq \module a \sin \gamma.
\end{multline*}
L'implication est ainsi prouvée dans les deux différents cas.
\end{proof}
Nous appliquons le précédent lemme à l'inégalité  \eqref{48} avec $\gamma = (\alpha - \beta)/4$ de sorte que
\begin{align*}
\module {\Arg H(z) - \Arg y_0} & \geq  \module{H(z)} \sin\pare{\frac {\alpha - \beta} 4} \\
\ & \geq - \sin\pare{\frac {\alpha - \beta} 4} \module z \ln \module z && \textrm{d'après \eqref{Hmodule},} \\
\ & > \module{\psi(z) - H(z)} && \textrm{d'après \eqref{42}}.
\end{align*}
Nous avons finalement prouvé que $\module{\psi - H} < \module{H - y_0}$ sur tout le lacet $\Gamma$. Par conséquent, $\psi - y_0$ a, comme $H-y_0$, une unique racine dans $D_{\rho,\alpha}$.\end{proof}

Nous pouvons maintenant prouver le théorème \ref{loginversion} dans son intégralité.

\begin{proof}[Démonstration du théorème \ref{loginversion}.] Quitte à écrire $\psi = c \psi_1$ et $\Upsilon(y) = \Upsilon_1(y/c)$, nous pouvons supposer sans perte de généralité que $c = 1$. Choisissons alors $\rho$ et $\rho'$ comme dans le lemme \ref{l:loginversion}. 
En outre, pour $y_0 \in D_{\rho',\beta}$, définissons $\Upsilon(y_0)$ comme l'unique point $z_0$ de $D_{\rho,\alpha}$ avec $\psi(z_0) = y_0$. 

Nous pouvons appliquer le théorème des fonctions implicites à l'équation fonctionnelle $\psi(\tilde \Upsilon(y)) = y$ au voisinage de $(y_0,z_0)$. En effet, la fonction $\psi$ est analytique en $z_0$ et localement injective (d'après le lemme \ref{l:loginversion}). Par conséquent, il existe une fonction $\tilde \Upsilon$ analytique au voisinage de $y_0$ telle que $\tilde \Upsilon(y_0) = z_0$ et $\psi(\tilde \Upsilon(y)) = y$ dans ce même voisinage.

Mais l'injectivité de $\psi$ sur $D_{\rho,\alpha}$ oblige $\Upsilon$ et $\tilde \Upsilon$ à coïncider sur un voisinage de $y_0$. Ainsi, $\Upsilon$ est analytique en $y_0$, et donc, de manière générale, sur tout $D_{\rho',\beta}$.

Terminons cette preuve par le comportement singulier de $\Upsilon$ au voisinage de $0$. L'équation $\psi(\Upsilon(y)) = y$, couplée avec l'approximation $\psi(z) \sim - z \ln z$, implique que $\Upsilon(y)$ tend vers $0$ lorsque $y$ s'approche de $0$. Dans un voisinage de $0$ (intersecté avec  $D_{\rho',\beta}$), nous avons alors
$$ y \sim - \Upsilon(y) \ln \Upsilon(y).$$
Le passage au logarithme dans l'équivalence ci-dessus est licite car les deux membres tendent vers $0$. Donc, d'après \eqref{Haddlog}, nous avons
$$\ln y \sim \ln  \Upsilon (y) + \ln \pare{- \ln  \Upsilon (y)}   \sim \ln \Upsilon (y).$$
En combinant les deux dernières équivalences, nous obtenons $\Upsilon(y) \sim - y / \ln y$.
\end{proof}

%\begin{minipage}{0.22 \textwidth} \small \begin{multline*}\theta(x) = 2q \times \\\sum_{i \geq 1} \frac{((2q-1)i)!}{ (i-1)! ((q-1)i+1)!^2 }   \end{multline*}\end{minipage} & \begin{minipage}{0.22 \textwidth}\small \begin{multline*} \theta(t,x) = 2 \times \\ \sum_{\substack{i \geq 1 \\k \geq 0 }}   \frac{ (2i+k-1)!(i+k)!      } {i!^2(i-1)!k!(k+1)!} x^i t^k  \end{multline*}\end{minipage}   & \begin{minipage}{0.22 \textwidth} \small \begin{multline*} \theta(x,y) = 3 \times \\ \sum_{\substack{i \geq 0 \\ j \geq 0 \\ 2i+j \geq 3}} \frac {(4i+2j-4)!} {i!^2j!(2i+j-3)!} x^i y^j,   \end{multline*}\end{minipage}

\chapter{Comportement asymptotique des cartes forestières $\boldsymbol 4$-eulériennes et apériodiques}
\chaptermark{Asymptotique des cartes forestières $\boldsymbol 4$-eulériennes et apériodiques }
\label{c:asympt}

\section{Présentation des théorèmes}
\label{s:theoremes}

Ce chapitre étudie le comportement asymptotique des cartes forestières 
avec certaines contraintes sur la  pondération  des degrés des sommets.
%  Nous montrons alors plusieurs propriétés universelles qui ne dépendent pas de la pondération choisie (voir théorème \ref{theoasymp}). 

\noindent Plus précisément, nous allons considérer des cartes 
telles que :
\begin{itemize}
\item les sommets ont des degrés pairs (autrement dit les cartes sont eulériennes). Nous rappelons que les équations caractérisant la série génératrice des cartes forestières sont plus simples dans ce cadre-là\footnote{cf. section \ref{seulerien} p. \pageref{seulerien}}.
\item les sommets ont des degrés différents de $2$. Comme nous allons énumérer les cartes selon le nombre de faces, cette hypothèse est nécessaire pour qu'il n'y ait qu'un nombre fini de cartes avec un nombre de faces fixé. Une carte qui satisfait ces deux dernières contraintes est dite \textit{$4$-eulérienne}.
\item il existe deux cartes dont le nombre de faces ne diffère que d'un. On parle alors de cartes \textit{apériodiques}. En termes de contraintes sur les degrés des sommets, cela revient à exiger \eqref{contraintedegre}. (Cette assertion se démontre via la formule d'Euler.) Le chapitre \ref{c:2qreg} montre comment s'en sortir quand cette hypothèse n'est plus vraie, dans le cadre des cartes régulières.
\end{itemize}
Donnons deux exemples de telles classes de cartes : l'ensemble des cartes tétravalentes et l'ensemble des cartes $4$-eulériennes non pondérées (cf. section  \ref{spanorama} p. \pageref{spanorama}).

Comme dit plus haut, nous allons énumérer les cartes forestières selon le nombre de faces. Nous abandonnons de fait la variable $t$ des arêtes des précédents chapitres. Autrement dit, nous substituons $t$ par $1$. De plus, la variable $u$ qui compte les composantes non racine de la forêt devient un paramètre réel supérieur à $-1$. Par exemple, si $u = 0$, nous énumérons les cartes munies d'un arbre couvrant. Les poids $d_k$ sur les sommets de degré $k$ sont eux des paramètres réels positifs que nous adapterons selon la classe étudiée.

Considérons donc $F(z) \equiv F(z,u) = \sum_{n \geq 3} f_n(u) z^n$ la série génératrice des cartes forestières avec un poids $d_{2k}$ par sommet de degré $2k$ et un poids $u \in \, [-1,+\infty[$ par composante non racine, comptées selon le nombre de faces. En d'autres termes,
$$f_n(u) = \sum_{\substack{(C,F)\textrm{ carte forestière}\\\textrm{à }n\textrm{ faces}}} u^{\cc(F)-1} \prod_{k \geq 2} d_{2k}^{\ \ \som_{2k}(C)},$$
où $\cc(F)$ désigne le nombre de composantes connexes de $F$ et $\som_{2k}(C)$ le nombre de sommets de degré $2k$ dans $C$. Le nombre $f_n(\mu - 1)$ possède de nombreuses interprétations combinatoires intéressantes, comme nous le rappelle la sous-section \ref{ss:modele} p. \pageref{ss:modele}. Nous allons étudier le comportement asymptotique de $f_n(u)$ où $u \geq -1$. Chaque théorème de ce chapitre sera illustré en l'appliquant au cas tétravalent, ce qui j'espère facilitera la lecture. 

Mais avant d'énoncer les théorèmes, rappelons la proposition \ref{teulerien} p. \pageref{teulerien} adaptée au cas $t=1$. Pour $i \geq 2$, définissons les nombres $T_{2i}$ et $T^\bigstar_{2i}$ comme 
$$T_{2i} = \sum_{\substack{A\textrm{ arbre}\\ \textrm{à }2i\textrm{ pattes}    }} \prod_{k \geq 2} d_{2k}^{\ \ \som_{2k}(A)}, \quad T^{\bigstar}_{2i} = \sum_{\substack{A\textrm{ arbre à}\\ \textrm{}2i\textrm{ pattes avec} \\ \textrm{une arête marquée}  }} \prod_{k \geq 2} d_{2k}^{\ \ \som_{2k}(A)}$$
(ici les arbres sont enracinés sur une patte)
% On pose $T^\bigstar_{2k} = \pd {T_{2k}} {t} (1)$ et $T^\bigstar(y) = \sum_{k \geq 2} T^\bigstar_{2k} y^{k}.$
et les séries $\phi(x)$ et $\theta(x)$  comme
\begin{equation} \label{eulphi}
\phi(x) = \sum_{i \geq 1} T_{2i} { 2i-1 \choose i } x^i,
\quad \theta(x) =  \sum_{i \geq 2} \left(\frac {1} {i} T^\bigstar_{2i} + T_{2i}\right) {2i \choose i}  x^i.
\end{equation}
Il existe une unique série formelle $R$ en $z$ et $u$ qui est solution de l'équation
 \begin{equation} \label{neulR}
R = z + u\,\phi(R).
\end{equation}
La série génératrice $F(z)$ des cartes forestières satisfait alors
\begin{equation} \label{neulF}
\frac {\partial F} {\partial z} = \theta(R).
\end{equation}

Commençons par décrire la fonction qui à tout $u \in \, [-1,+\infty[$ associe le rayon de convergence $\rho_u$ de $F(z)$ en $z$.

%\begin{theo} \label{theoasymp}
%Fixons $u \geq -1$. Soit $D(y)= \sum_{k \geq 1} d_{2k+2}\,y^k$ une série entière aux coefficients positifs avec un rayon de convergence $\delta$ qui vérifie les cinq propriétés suivantes.
%\begin{enumerate}
%\item[(I)]La suite $(d_{2k})$ satisfait  l'hypothèse d'\textit{apériodicité}
%$$\pgcd\left( k \geq 1 \ |\  d_{2k+2} \neq 0 \right) =1.$$
%\item[(II)] Il existe deux réels $\kappa$ et $\sigma $ avec $0 < \kappa \sigma ^2 < \delta$ et $ 1 < \sigma $ tels que :
%\begin{equation} \label{eqcar}
%\sigma D(\kappa\,\sigma ^2) + 1 = \sigma ,\quad 2\,\kappa\,\sigma ^3\,D'(\kappa \sigma^2) = 1.
%\end{equation}
%\item[(III)] $D$ est prolongeable analytiquement sur tout $\C\  \backslash\  [\delta,+\infty[$ et  le système d'équations \eqref{eqcar} n'admet qu'une seule solution sur ce domaine.
%\item[(IV)] Si une suite $(z_n)$ à valeurs dans $\C\  \backslash\  [\delta,+\infty[$ converge vers $z \in [\delta,+\infty[$ de sorte que $D(z_n)$ converge, alors
%$$\lim z_n (1 - D(z_n))^2 \ \in \ ]\kappa,+\infty[.$$
%\item[(V)]  Toute suite $(z_n)$ à valeurs dans $\C\  \backslash\  [\delta,+\infty[$ telle que $D(z_n)$ converge vers $1$ est bornée.
%\end{enumerate}

\begin{theo} \label{theorayon} 
Soient $u \geq -1$ et $(d_{2k})_{k \geq 2}$ une suite de poids positifs qui vérifie la condition d'apériodicité :
\begin{equation} \label{contraintedegre}
\pgcd\left( k \geq 1 \ |\  d_{2k+2} \neq 0 \right) =1.
\end{equation}
De plus, supposons que sa série génératrice $D(y)= \sum_{k \geq 1} d_{2k+2}\,y^k$ a un rayon de convergence $\delta$ non nul (possiblement infini) et que $D$ diverge en  $\delta$.

Alors il existe un unique couple de réels $(\kappa,\sigma)$ avec $0 < \kappa \sigma ^2 < \delta$ et $ 1 < \sigma $ tels que :
\begin{equation} \label{eqcar}
\sigma D(\kappa\,\sigma ^2) + 1 = \sigma ,\quad 2\,\kappa\,\sigma ^3\,D'(\kappa \sigma^2) = 1.
\end{equation}
Les séries $\phi$ et $\theta$ définies par \eqref{eulphi} ont pour rayon de convergence $\kappa / 4$. La série génératrice $F(z)$ des cartes forestières eulériennes avec un poids $d_{2k}$ par sommet de degré $2k$ a elle pour rayon de convergence
\begin{equation} \label{deathrow}
\rho_u = \tau - u \phi(\tau),
\end{equation}
où $\tau$ est défini de manière unique par  
$$\left\{ \begin{array}{ll}
\tau = \kappa/4 & \textrm{si }u \leq 0, \\
1 - u \phi'(\tau) = 0 & \textrm{si }u > 0. 
\end{array} \right.
$$
En particulier, $\rho_u$ est une fonction affine de $u$ sur $[-1,0]$. De manière générale, la fonction $u \mapsto \rho_u$ est décroissante, analytique sur $[-1,+\infty[$, à part en $0$ où elle est infiniment dérivable. En effet, nous avons
\begin{equation} \label{exprho}
\rho_u \underset{u \rightarrow 0^+}= \frac \kappa 4 - u \phi\left( \frac \kappa 4 \right) + O\left(\exp\left( -c u ^{-1} \right)\right),
\end{equation} 
avec $c > 0$.
\end{theo}

\noindent \textbf{Application au cas tétravalent. }Les poids associées aux cartes tétravalentes sont $d_4 = 1$ et $d_k = 0$ pour $k \neq 4$. La série $D(y)$ associée est donc $y$. Elle diverge bien en $+\infty$. Le système \eqref{eqcar} s'écrit  
$$ \kappa\,\sigma ^3 + 1 = \sigma ,\quad 2\,\kappa\,\sigma ^3 = 1. $$
Il admet pour unique solution $\kappa = 4/27$ et $\sigma = 3/2$. La courbe de $u \mapsto \rho_u$ est  montrée figure \ref{roux}. 

\fig{[scale=0.3]{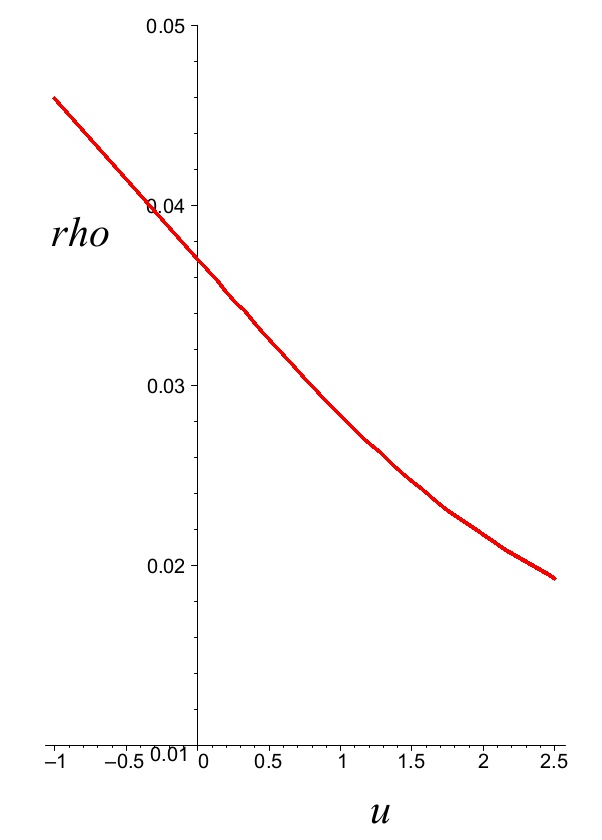}}{Le rayon de convergence de $F(z,u)$ comme fonction de $u \geq -1$, dans le cas tétravalent.}{roux}

Nous pouvons également calculer dans le cas tétravalent la valeur de $\phi$ en $\kappa/4 = 1/27$ comme nous le verrons dans la section suivante\footnote{Il s'agit d'ailleurs d'un des rares cas où nous puissions le faire.}. La valeur de $\rho_u$ sur cet intervalle est alors
\begin{equation} \label{rho4}
\rho_u= \frac 1 {27} -u\phi\left(\frac 1 {27}\right)= \frac {1+u} {27}
- u \frac{\sqrt 3}{12\pi}.
\end{equation}
Notons que le rayon de convergence en $-1$ vaut $\rho_{-1} = \sqrt 3 / (12\pi)$ et est transcendant. Il correspond à la série génératrice des cartes tétravalentes munies d'un arbre couvrant sans arêtes internes actives. Le rayon de convergence d'une série holonome étant forcément algébrique sur $\Q$, nous en déduisons la non-holonomie de la précédente série, qui est généralisé par le corollaire \ref{cor:pasholonome} ci-dessous.

Décrivons maintenant le comportement asymptotique du nombre de cartes forestières.

\begin{theo} \label{theoasymp}
Fixons $u \geq -1$. Considérons $(d_{2k})_{k \geq 2}$  une suite de poids positifs  qui vérifie la condition d'apériodicité \eqref{contraintedegre} et telle que $\sum_{k \geq 2} d_{2k}\,y^k$ diverge en son rayon de convergence supposé non nul. Pour $n \geq 3$, notons $f_n(u)$ le coefficient en $z^n$ de $F(z,u)$ précédemment défini, et $\rho_u$ le rayon de convergence de $F(z,u)$ (voir théorème précédent).

Il existe une constante positive $c_u$ tel que 
$$f_n(u) \sim \left\{ \begin{array}{ll}
c_u \rho_u^{-n}n^{-3} & \textrm{si }u=0, \\
c_u \rho_u^{-n}n^{-5/2} & \textrm{si }u>0.
\end{array} \right.$$
Lorsque $u \in \, [-1,0[$, la série génératrice $F$ des cartes forestières est analytique sur un domaine ouvert en forme de poussin (voir figure \ref{domainepoussin}) et le comportement singulier de $\pd {^2 F} {z^2}$ en $\rho_u$ est 
$$\pd {^2 F} {z^2}(z) \underset{z \rightarrow \rho_u} =  -  \frac{ d} u +  \frac { c_u } {\ln\pare{1 - z/\rho_u}} + o\pare{\frac 1 {\ln\pare{1 - z/\rho_u}}} ,$$
où $c_u$ et $d$ sont deux constantes strictement positives. 

\fig{[scale=1.3]{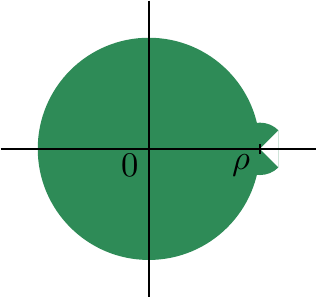}}{Domaine poussin.}{domainepoussin}
\end{theo}

Le  théorème  précédent montre ce qu'on appelle une \textit{transition de phase} en $0$ : le comportement asymptotique de $F(z,u)$ n'est pas le même pour $u$ avant/en/après $0$. L'exposant $-5/2$ que nous trouvons pour $u > 0$ est typique des cartes planaires (voir par exemple les tables $1$ et $2$ de \cite{banderier-maps}). Quant au cas $u=0$, il est intimement lié à la fonction $\theta$ (qui peut s'interpréter comme la série génératrice des arbres à pattes enracinés sur un coin avec une bicoloration équilibrée des pattes) puisque $\pd F z (z,0) = \theta(z)$. L'exposant $-3$ est cohérent avec les résultats de Mullin \cite{mullin-boisees} dans la mesure où la formule de Stirling appliquée aux formules closes trouvées par Mullin sur des cartes régulières donne le même exposant.

Le cas $u \in \, [-1,0[$ est plus surprenant (et plus délicat à étudier). Le comportement singulier de $\pdd 2 F z$ en  $1/{\ln\pare{1 - z/\rho_u}}$ est vraiment inhabituel pour des cartes planaires. Nous ne connaissons pas d'autres études où un tel comportement peut être observé. En particulier, ce comportement 
 en  $1/{\ln\pare{1 - z/\rho_u}}$ est incompatible avec le fait que $F$ soit holonome (voir \cite[p. 520 et 582]{flajolet-sedgewick}). Nous déduisons donc le corollaire suivant. 

\begin{cor} \label{cor:pasholonome}
Pour $u \in \, [-1,0[$, la série génératrice $F(z,u)$ des cartes forestières eulériennes avec un poids $d_{2k}$ pour chaque sommet de degré $2k$ tel que la suite $(d_{2k})$ vérifie les hypothèses du théorème \ref{theoasymp} n'est pas holonome. Ceci est également vrai quand $u$ est une variable indéterminée.
\end{cor}

Pour revenir au théorème précédent, l'analyticité de $F$ n'est stipulée que sur un domaine en forme de poussin, et non pas sur un $\Delta$-domaine. En effet, sans contrainte supplémentaire, nous ne savons pas prouver que $F$ est analytique sur le disque de rayon $\rho_u$ (privée du point $\rho_u$). Cette hypothèse est pourtant nécessaire pour déduire le comportement asymptotique de $f_n(u)$ pour $u$ négatif via le théorème de transfert. Le théorème suivant introduit une nouvelle hypothèse (vérifiée par les cartes tétravalentes et 4-eulériennes) qui résout ces problèmes.

\begin{theo} \label{theobis}
 Reprenons les notations et les hypothèses du théorème \ref{theoasymp}. Supposons en outre que les séries $\phi$ et $\theta$ définies par \eqref{eulphi} sont analytiques sur un disque centré en $0$ de rayon $2 \phi(\kappa/4) + \kappa/4$ privé de la demi-droite $[\kappa/4,+\infty[$, où $\kappa/4$ est le rayon de convergence de $\phi$.

Alors la série génératrice $F$ des cartes forestières est analytique sur un $\Delta$-domaine\footnote{voir théorème~\ref{theotransfert} p.~ \pageref{theotransfert} pour la définition de $\Delta$-domaine} de rayon $\rho_u$ pour tout $u \in \, [-1,+\infty[$ préalablement fixé, et
\begin{equation}
f_n(u) \sim \left\{ \begin{array}{ll}
 c_u \rho_u^{-n}n^{-3} (\ln n)^{-2} & \textrm{si }u \in [-1,0[, \\
c_u \rho_u^{-n}n^{-3} & \textrm{si }u=0, \\
c_u \rho_u^{-n}n^{-5/2} & \textrm{si }u>0.
\end{array} \right.
\label{fnuasympt}
\end{equation}
\end{theo}

\noindent \textbf{Application au cas tétravalent. }Si on se réfère à la figure \ref{recapitulatif} p. \pageref{recapitulatif}, la fonction $\phi$ est égale à 
$$  \phi(x) = \sum_{i \geq 2}
 \frac{(3i-3)!}{(i-1)!^2i!} x^i.$$
Nous pouvons alors voir que 
\begin{equation} \label{phihypergeom}
\phi(x) = x \pare{\,_2F_1 \pare{  \frac 1 3, \frac 2 3 ; 2 ; 27 x   } - 1 },
\end{equation}
 où $_2F_1(a,b;c;z)$ désigne la \textit{série hypergéométrique} standard de paramètres $a$, $b$ et $c$ :
\begin{equation} \label{hypergeom}
_2F_1(a,b;c;z)=\sum_{n \geq 0} \frac{(a)_n(b)_n}{(c)_n} \frac{x^n} {n!} ;
\end{equation}
ici $(a)_n$ est la factorielle montante $a(a+1)\dots(a+n-1)$. Sachant qu'une série hypergéométrique standard est analytique sur $\C \backslash \, [1,+\infty[$, nous en déduisons que $\phi$ satisfait les hypothèses du théorème précédent. D'après la relation \eqref{thetaphi4} p. \pageref{thetaphi4}, il en est de même pour la série $\theta$. Nous en déduisons les équivalents asymptotiques sur $f_n(u)$.

\section{Les fonctions $\phi$ et $\theta$}

Le système fonctionnel \eqref{neulR} \eqref{neulF}  qui régit la série $F$  fait intervenir deux séries  $\phi$ et $\theta$. Ces séries sont intimement liées aux nombres $T_{2i}$ qui comptent les arbres à pattes selon le degré des sommets. Le comportement asymptotique des nombres $T_{2i}$ a déjà été étudié dans \cite[Théo. 5]{drmota-trees} mais la preuve comporte une erreur lorsque la période n'est pas $1$ (ici elle est de $2$).  Nous déterminons dans  un premier lemme le véritable comportement de ces nombres.

\begin{lem} \label{l:asympt}
Soit $D(y)= \sum_{k\geq 1} d_{2k+2}y^k$ de rayon de convergence $\delta$ vérifiant les hypothèses du théorème \ref{theorayon}. 

Il existe un unique couple de réels $(\kappa,\sigma)$ avec $0 < \kappa \sigma ^2 < \delta$ et $ 1 < \sigma $ satisfaisant \eqref{eqcar} :
$$\sigma D(\kappa\,\sigma ^2) + 1 = \sigma ,\quad 2\,\kappa\,\sigma ^3\,D'(\kappa \sigma^2) = 1.$$
La série génératrice $T(y) = \sum_{k \geq 2} T_{2k} y^{k}$ des arbres à pattes est apériodique\footnote{dans le sens de \cite[Définition IV.5 p.266]{flajolet-sedgewick} : il existe trois indices $i < j < k$ avec $T_{2i}T_{2j}T_{2k}\neq 0$ et $\textrm {pgcd}(j-i,k-i) = 1$.} et son rayon de convergence est $\kappa$. De plus, $T$ est analytique sur un $\Delta$-domaine de rayon $\kappa$. Le développement singulier de $T$ en $\kappa$, localement valide sur le domaine, vaut :
\begin{equation} \label{devT}
T(y) \underset{y \rightarrow \kappa}= \kappa(\sigma  - 1) - \gamma \sqrt{1 - y/\kappa} + \kappa a (1- y/\kappa) + O\pare{( 1- y/\kappa )^{3/2}},
\end{equation}
où $\gamma = \kappa \sigma / \sqrt{3+4\kappa^2 \sigma^5D''(\kappa \sigma ^2)}$ et $a$ est une constante. Le théorème de transfert donne alors un équivalent des nombres $T_{2n}$ :
$$T_{2n} \sim \frac{\gamma}{2 \sqrt{\pi}} \kappa^{-n} n^{- \frac 3 2}.$$
La dérivée de $T$ admet pour développement singulier en $\kappa$ :
\begin{equation} \label{devTp}
T'(y) \underset{y \rightarrow \kappa}= \frac {\gamma} {2 \kappa} \frac 1 {\sqrt{1 - y/\kappa}} +  a +  O\pare{ \sqrt{1- y/\kappa}}.
\end{equation}
En outre, la série génératrice $T^\bigstar = \sum_{k \geq 2} T^\bigstar_{2k} y^{k}$ des arbres à pattes avec une arête distinguée est elle aussi analytique sur un $\Delta$-domaine de rayon $\kappa$ et
\begin{equation} \label{devstarr}
T^\bigstar(y) \underset{y \rightarrow \kappa} = \gamma (\sigma - 1) \frac 1 {\sqrt{1 - y/\kappa}}  + a^\bigstar + O\pare{ \sqrt{1- y/\kappa}},
\end{equation}
où $a^\bigstar$ désigne une constante.
\end{lem}

\begin{proof}Commençons par montrer l'apériodicité de la série $T$. \'Etant donné que \mbox{$\pgcd\left( k \ |\  d_{2k+2} \neq 0 \right) =1$}, le théorème de Bézout indique qu'il existe $\ell$ entiers $k_1,\dots,k_\ell$, avec $d_{2k_j+2} \neq 0$ pour  $j \in \, \ens{1,\dots,\ell}$, et $2\ell$ entiers positifs $a_1,b_1,\dots,a_\ell,b_\ell$ tels que
$$a_1k_1 + a_2k_2 + \dots + a_\ell k_\ell = 1 + b_1k_1 + b_2k_2 + \dots + b_\ell k_\ell.$$
Or un arbre à pattes qui est constitué de $a_1$ sommets de degré $2k_1+2$, \dots , de $a_\ell$ sommets de degré $2k_\ell+2$, compte $2(a_1 k_1 + \dots + a_\ell k_\ell)$ pattes  et a un poids non nul relativement à $\left(d_{2k}\right)$. Donc 
$T_{a_1k_1 +  \dots + a_\ell k_\ell} \neq 0.$ De même, $T_{b_1k_1 + b_2k_2 + \dots + b_\ell k_\ell} \neq 0$. Alors, en prenant un entier $m$ suffisamment grand tel que $T_m \neq 0$, on remarque que
$$\pgcd( (a_1k_1 +  \dots + a_\ell k_\ell) - (b_1k_1 + \dots + b_\ell k_\ell), m - (b_1k_1 + \dots + b_\ell k_\ell)) =  1.$$
La série $T$ est bien apériodique.

On va maintenant chercher à appliquer la définition VII.4 et le théorème VII.3 de \cite[p. 467] {flajolet-sedgewick}.  En posant $\tilde T = T/y$, l'équation \eqref{deftl} p.\pageref{deftl} qui caractérise les nombres $T_{2k}$ s'écrit
$$y \tilde T(y^2) =  \sum_{k \geq 2} d_{2k} (y\,\tilde T(y^2) + y)^{2k-1}.$$ En d'autres termes, en posant $G(z,w) = (w+1) D(z(w+1)^2)$, on obtient l'équation fonctionnelle $\tilde T = G(y,\tilde T)$. 

Commençons par montrer que \eqref{eqcar} admet une unique solution $(\kappa,\sigma)$ telle que $0 < \kappa \sigma ^2 < \delta$ et $ 1 < \sigma $. Comme $D\left(x\right) + 2 \, x \, D'\left(x\right)$ vaut $0$ en $x=0$, diverge en $+ \infty$ quand $x$ tend vers $\delta^-$ et est strictement croissante en $x$, il existe un unique $\xi \in ]0,\delta[$ tel que \mbox{$D\left(\xi\right) + 2 \, \xi \, D'\left(\xi\right)= 1$.} Donc tout couple $(\kappa,\sigma)$ satisfaisant \eqref{eqcar} vérifie $\xi = \kappa \sigma^2$. Comme $\sigma = (1 - D(\xi))^{-1}$ (d'après la définition de $\xi$, on a $D(\xi) < 1$), l'unicité du couple est prouvée. Pour l'existence, il reste à vérifier que $\sigma > 1$, mais cela est vrai car $\sigma = 1 + \sigma D(\xi)$.

Vérifions maintenant les trois conditions de la définition VII.4 de \cite[p. 467]{flajolet-sedgewick}. La condition $\boldsymbol{(I_1)}$  est vérifiée car $G(z,w)$ est analytique sur un domaine de la forme $|z| < \kappa + \varepsilon$ et $|w| < \sigma - 1 + \varepsilon$, où $\varepsilon$ est un réel strictement positif suffisamment petit. La condition $\boldsymbol{(I_2)}$ est triviale. Quant au système caractéristique, il vaut
$$ (s+1) \,D\left(r(s+1)^2\right) = s, \quad  D\left(r(s+1)^2\right) + 2 \, r \, (s+1)^2 \, D'\left(r(s+1)^2\right) = 1,$$
ce qui équivaut à \eqref{eqcar} avec $\kappa = r$ et $\sigma = s + 1$. La condition $\boldsymbol{(I_3)}$ est donc satisfaite. Le théorème VII.3 montre alors que $\tilde T$ est analytique sur 
%une boule\footnote{L'énoncé du théorème en lui-même ne précise l'analyticité de $T$ que sur un $\Delta$-domaine, mais en regardant en détail la preuve dudit théorème, et notamment le lemme VII.3, on s'aperçoit que l'analycité et le développement asymptotique en racine restent vrais sur un voisinage de $\kappa$ privé de la demi-droite $[\kappa,+\infty[$.} centrée en $0$ de rayon strictement supérieur à $r$ à laquelle on a enlevé la demi-droite $[\kappa,+\infty[$. 
un $\Delta$-domaine de rayon $\kappa$.
Le développement singulier de $\tilde T$ sur ce domaine est
\begin{equation} \label{til}
\tilde T(y) \underset{y \rightarrow \kappa}= (\sigma - 1) - \sqrt{\frac{ 2\kappa \pd G z (\kappa,\sigma-1)}{\pd {^2G} {w^2}(\kappa,\sigma-1) }} \sqrt{1 - \frac y \kappa} + a \pare{1 - \frac y \kappa} + O\left( \pare{1 - \frac y \kappa}^{3/2}\right),
\end{equation}
où $a$ est une constante dont nous omettons l'expression\footnote{Cette constante n'apparaît pas en tant que telle dans l'énoncé du théorème VII.3 mais il est précisé au début de la page 469 que le développement singulier de $R$ peut s'écrire en termes de puissances de $ \sqrt{1 - y/\kappa}$.}. 
Ce dernier développement est bien équivalent à \eqref{devT} ; le coefficient $\gamma$ devant la racine a été notamment simplifié grâce à la relation $2\,\kappa\,\sigma^3\,D'(\kappa \sigma^2) = 1$.

Le développement singulier de $T'$ est obtenu en dérivant celui de $T$. Cette opération est bien licite, comme nous l'assure \cite[Théorème VI.8 p. 419]{flajolet-sedgewick}.

Intéressons-nous maintenant à $T^\bigstar$. Les arbres à pattes avec une arête distinguée (qui sont comptés par $T^\bigstar$, on le rappelle) sont en bijection évidente avec les couples constitués d'un arbre à pattes avec une patte non racine distinguée et  d'un autre arbre à pattes (sans autre distinction que sa racine).  En effet, dans un sens il suffit de couper l'arête distinguée pour trouver deux arbres (celui qui contient la racine a un double-marquage au niveau des pattes), et dans l'autre sens il suffit de recoller les deux arbres au niveau de leurs racines et distinguer l'arête ainsi formée. En termes de séries génératrices, cela prouve  
$$T^\bigstar = \frac 1 y \times T \times \pare{2 y T' - T},$$
où le facteur $1/y$ provient de la perte de deux pattes due au recollement. Comme $T$ est analytique sur un $\Delta$-domaine de rayon $\kappa$, nous en déduisons l'analyticité de $T^\bigstar$ sur ce même $\Delta$-domaine. Quant au développement singulier  \eqref{devstarr} de $T^\bigstar$, il  résulte de  l'identité précédente, dans laquelle on a injecté les développements de $T$ et $T'$ (voir \eqref{devT} et \eqref{devTp}).
\end{proof}

\noindent \textbf{Remarque 1. }Nous aurions pu calculer les valeurs exactes de $a$ et $a^{\bigstar}$ mais elles ne nous sont pas utiles.

\noindent \textbf{Application au cas tétravalent. } Le nombre $T_{2k}$ d'arbres tétravalents à $2k$ pattes admet ici une forme close (voir par exemple \cite[Théorème 5.3.10]{stanley-vol-2}) :
$$T_{2k} = \frac{(3k-3)!}{(k-1)! (2k-1)!}.$$ 
Vérifions que cela est bien cohérent avec l'équivalent asymptotique du précédent lemme. D'un côté, on a  $\kappa = 4/27$ et $\sigma = 3/2$, d'où $\gamma = 2 \sqrt{3}/27$. De l'autre côté, la formule de Stirling donne 
$$ T_{2k} \sim \frac{\sqrt 3} {27 \sqrt \pi } \, \pare{\frac {27} 4}^n n^{-3/2},$$
ce qui est bien égal à $\gamma \kappa^{-n} n^{-3/2} / (2 \sqrt \pi)$. Nous pouvons également calculer avec \textit{maple} une formule explicite pour la série cubique  $T$ et vérifier les différents développements singuliers.

Nous déduisons maintenant les développements singuliers de $\phi$ et $\theta$ à partir du lemme précédent.

\begin{lem} \label{l:thetaphi}
Reprenons les hypothèses et les notations du lemme \ref{l:asympt}. La série $\phi$ définie par \eqref{eulphi} converge en son rayon de convergence qui est $\kappa/4$. De plus, elle est analytique sur un $\Delta$-domaine de rayon $\kappa/4$ et
\begin{equation} \label{defi}
\phi(z) \underset{z \rightarrow \kappa/4}= \phi\left( \frac \kappa 4 \right)  + \frac \gamma {4\pi}\,\left(1 - \frac{4z}{\kappa} \right)\,\ln\left(1- \frac{4z}{\kappa}\right) + \kappa\,b\, \left(1-\frac{4z}{\kappa}\right) + O \left(\pare{1-\frac{4z}{\kappa}}^{\frac 3 2}\right),
\end{equation}
où $b$ est une constante et $\gamma = \kappa \sigma / \sqrt{3+4\kappa^2 \sigma^5D''(\kappa \sigma ^2)}$.
Quant à la dérivée de $\phi$, elle admet comme développement singulier :
\begin{equation} \label{depi}
\phi'(z) \underset{z \rightarrow \kappa/4}= -\frac \gamma {\pi \kappa}\,\ln \left(1 - \frac{4z}{\kappa} \right)  + b + O \left(\pare{1-\frac{4z}{\kappa}}^{\frac 1 2}\right).
\end{equation}
En outre, la fonction $\theta$ définie par \eqref{eulphi} a elle aussi $\kappa/4$ pour rayon de convergence et est analytique sur un $\Delta$-domaine de rayon $\kappa/4$. Sa dérivée admet pour développement
%\begin{equation} \label{devtata}
%\theta(z) - \theta\pare{\frac \kappa 4} \underset{z \rightarrow \kappa/4} \sim - \frac {\gamma} {2 \, \pi} \pare{2\sigma - 1} \left(1 - \frac{4z}{\kappa} \right) \ln \left(1 - \frac{4z}{\kappa} \right),
%\end{equation}
%et sa dérivée
\begin{equation} \label{devthetap}
\theta'(z) \underset{z \rightarrow \kappa/4}= - \frac {2\, \gamma} {\kappa \, \pi} \pare{2\sigma - 1} \ln \left(1 - \frac{4z}{\kappa} \right) + b^\bigstar + O \left(\pare{1-\frac{4z}{\kappa}}^{\frac 1 2}\right),
\end{equation}
où $b^\bigstar$ est une constante.
\end{lem}

\begin{proof}En utilisant le lemme \ref{l:asympt} pour $T_{2i}$ et la formule de Stirling  pour ${2i-1 \choose i} $, nous voyons que le $i$-ième coefficient de $\phi$ est équivalent à $i^{-2} (4/\kappa)^i$ à une constante multiplicative près. Donc le rayon de convergence de $\phi$ est $\kappa/4$ et $\phi$ converge en ce point.

Notons $$B(z) = \sum_{i \geq 1} {2i-1 \choose i} z^i = \frac 1 {2\sqrt{1-4z}} - \frac 1 2$$ de sorte que $\phi$ soit le produit de Hadamard de $T$ et de $B$ (que l'on notera $T \odot B$). 
%Or les singularités du produit de Hadamard de deux fonctions $f$ et $g$ sont incluses dans l'ensemble produit des singularités de $f$ et des singularités de $g$ :
%$$\textrm{Sing}(f \odot g) \subseteq \ens{xy\ |\ x \in \, \textrm{Sing}(f),\ y \in \, \textrm{Sing}(g)}.$$
%Donc 
%$$\textrm{Sing}\left(\phi\left( \frac {4z} \kappa \right)\right) \subseteq \ens{xy\ |\ x \in \, \textrm{Sing}\left(T\left( \frac {z} \kappa \right)\right),\ y \in \, \textrm{Sing}\left(T\left( 4z \right)\right) }.$$
%Mais il existe un $\varepsilon > 0$ tel que 
%$$\ens{xy\ |\ x \in \, \textrm{Sing}\left(T\left( \frac {z} \kappa \right)\right),\ y \in \, \textrm{Sing}\left(T\left( 4z \right)\right) }  \subseteq \ens{xy\ | |x| > 1 + \varepsilon\textrm{ ou }
%x \in [1,+\infty[,\ y \in [\frac 1,+\infty}$$
 Appliquons la première partie du théorème VI.11 p. 423 de \cite{flajolet-sedgewick} : la série $\phi(\kappa z/4) = T(\kappa z) \odot B(z/4)$ est analytique sur un $\Delta$-domaine de rayon $1$. De plus, l'estimation \eqref{devT} implique
$$\phi\left(\frac{\kappa z} 4 \right) \underset{z \rightarrow 1}= \left(\kappa(\sigma - 1) - \gamma \sqrt{1 - z} + \kappa a (1-z) + \eta(z)\right) \odot \frac 1 {2\sqrt{1-z}},$$
où $\eta(z) = O( (1- z)^{3/2} )$. D'une part, 
%$$r(t-1) \odot \frac 1 {2\sqrt{1-z}} = \frac 1 2 r(v-1).$$
%D'autre part, 
on remarque que 
$$ - \gamma \sqrt{1 - z} \odot \frac 1 {2\sqrt{1-z}} = - \frac \gamma 2 \ _2F_1\left(- \frac {1} 2, \frac 1 2 ; 1 ; z \right),$$
où $_2F_1(a,b;c;z)$ désigne la série hypergéométrique standard de paramètres $a$, $b$ et $c$ (voir équation \eqref{hypergeom}). La série $_2F_1\left(- \frac {1} 2, \frac 1 2 ; 1 ; z \right)$ peut être analytiquement prolongée sur $\C \backslash [1,+\infty[$ et son comportement quand $z$ tend vers $1$ dans ce domaine est donné par \cite[\'Equation (15.3.11)]{AS} :
$$- \frac \gamma 2 \, _2F_1\left(- \frac {1} 2, \frac 1 2 ; 1 ; z \right) \hspace{-2pt} = \hspace{-2pt}
\frac {-\gamma} \pi + \frac \gamma {4\pi}(1 - z)\ln(1-z) + \frac {\gamma ( 1 - \ln 16)} {4\pi} (1 - z) + O( (1-z)^2 \ln (1-z)).$$
D'autre part, la seconde partie du théorème VI.11 de \cite{flajolet-sedgewick} nous indique que 
\begin{equation}
 \eta(z) \odot  \frac 1 {2\sqrt{1-z}} = k_1 + k_2(1-z) + O( (1-z)^2 \ln (1-z)), \label{brouille}
\end{equation}
où $ctke_1$ et $k_2$ sont deux constantes.
En recombinant tout ce qui a été dit, nous obtenons \eqref{defi}. 

Le développement singulier de $\phi'$ est obtenu en dérivant celui de $\phi$ (voir \cite[Théo. VI.8 p. 419]{flajolet-sedgewick}).

Quant à $\theta$, on remarque que $\theta' = 2 \phi' + 2 T^\bigstar \odot B$, où  $T^\bigstar = \sum_{k \geq 2} T^\bigstar_{2k} y^{k}$ est la série génératrice des arbres à pattes avec une arête distinguée. Nous appliquons alors la même méthode que pour $\phi$ (avec entre autres le théorème VI.11). Nous prouvons de cette manière que $\theta'$ (et donc $\theta$) a pour rayon de convergence $\kappa/4$, qu'elle est analytique sur un $\Delta$-domaine de rayon $\kappa/4$ et que le développement singulier de $\theta'$ est \eqref{devthetap}.
Pour trouver cette dernière expression, nous avons notamment utilisé la relation
$$\frac 1 {\sqrt{1-z}} \odot \frac 1 {\sqrt{1-z}} = \  _2F_1 \left( \frac {1} 2, \frac 1 2 ; 1 ; z \right) = - \frac{\ln(1 - z)} \pi + k + o(1),$$
où la dernière égalité est déduite de \cite[\'Equation (15.3.10)]{AS}. 
%
%Enfin, le développement \eqref{devtata} de $\theta$ est obtenu en intégrant celui de $\theta'$. Cela est bien possible grâce au théorème VI.9 de \cite{flajolet-sedgewick}.
\end{proof}

\noindent \textbf{Remarque 2.} Le cas $u=0$ des théorèmes de la section \ref{s:theoremes} peut être déduit du précédent lemme. En effet, l'identité \eqref{neulF} se lit $\pd F z (z,0) = \theta(z)$ quand $u=0$. Donc $F$ est la primitive de $\theta$, son rayon de convergence est $\kappa/4$ et elle est analytique sur un $\Delta$-domaine de rayon $\kappa/4$. En appliquant le théorème de transfert de \cite{flajolet-sedgewick} à \eqref{devthetap}, nous déduisons le comportement asymptotique du $n$-ième coefficient de $\theta'(z)$, à savoir $(n+2)(n+1)f_{n+2}(0)$. Nous déduisons alors l'équivalent asymptotique
$$f_n(0) \sim \frac{\kappa \gamma}{8 \pi} (2 \sigma - 1) \pare{\frac 4 \kappa}^n n^{-3}.$$

\noindent \textbf{Remarque 3.} Contrairement au lemme \ref{l:asympt}, nous ne sommes pas capables de calculer la valeur exacte de $\phi(\kappa/4)$, $b$ et $b^\bigstar$ (à part dans le cas tétravalent). En effet, même si nous connaissions le développement asymptotique exact de $T$ et de $T^\bigstar$, le produit de Hadamard "brouille" toutes les constantes associées à des termes analytiques comme $1$ ou $1-z$ (voir par exemple  \eqref{brouille}). % \na{parler de $_{p+1}F_p$ ? }
% Dans le même ordre d'idée, aucune formule close n'est connue pour la valeur d'une série hypergéométrique $_{p+1}F_p$ en $z=1$ avec $p \geq 2$, pour la majorité des paramètres associés, alors qu'elle peut se voir comme produit de Hadamard 

\noindent \textbf{Application au cas tétravalent. } Les coefficients des séries $\phi$ et $\theta$ sont explicites (voir figure \ref{recapitulatif}). En appliquant \cite[\'Equation (15.3.11)]{AS} à l'identité \eqref{phihypergeom} qui lie $\phi$ à une série hypergéométrique, nous trouvons le développement
\begin{equation}
\phi\left(\frac{1 }{27}- \varepsilon\right) =
 \frac{\sqrt{3}}{12\pi} - \frac{1}{27} +
 \frac{\sqrt{3}}{2\pi}\,\varepsilon\ln{\varepsilon} + 
\left( 1 - \frac{\sqrt 3}{2 \pi}\right) \, \varepsilon + O(\varepsilon^2
 \ln{\varepsilon}), \label{devphi4} 
\end{equation}
 avec des valeurs explicites pour $\phi(\frac 1 {27})$ et $b$. Nous pouvons vérifier que ce développement est bien cohérent avec \eqref{defi}. De même, nous trouvons 
$$  
  \theta'\left(\frac{1 }{27}- \varepsilon\right) = 
 - \frac{2\sqrt{3}}{ \pi} \ln{\varepsilon} 
- \frac{9\sqrt{3}}{ \pi}  + O(\varepsilon \ln{\varepsilon}). 
$$

\section{Quand $\boldsymbol{ u > 0 }$}

\begin{prop} Fixons $u > 0$. Considérons $D(y)= \sum_{k\geq 1} d_{2k+2}y^k$ de rayon de convergence $\delta$ vérifiant les hypothèses du théorème \ref{theorayon} et définissons $\kappa$ et $\sigma $ les  deux nombres satisfaisant \eqref{eqcar} avec $0 < \kappa \sigma ^2 < \delta$ et $ 1 < \sigma $. 

La série $R$ est apériodique et satisfait les hypothèses du schéma d'équation fonctionnelle lisse de \cite[Def VII.4]{flajolet-sedgewick}. Son rayon de convergence est donné par \eqref{deathrow}, il est décroissant et analytique selon le paramètre $u$, et il satisfait \eqref{exprho} au voisinage de $u = 0$. La série $R$ est analytique sur un $\Delta$-domaine de rayon $\rho = \rho_u$, avec une singularité de type racine en $\rho$ :
\begin{equation}\label{per}
R(z) = \tau - c\, \sqrt{1 - z/\rho} + O\left(1 - z/\rho\right),
\end{equation}
où $\tau$ est défini par le théorème \ref{theorayon} et $c = \sqrt{\frac{ 2\rho} {u \phi''(\tau)}}$ avec $\phi$ donné par \eqref{eulphi}.
La série $\pd F z$ est également analytique sur un $\Delta$-domaine de rayon $\rho$ avec une singularité de type racine :
\begin{equation} \label{posdevF}
\pd F z (z)= \theta(\tau) -c \,\theta'(\tau)  \, \sqrt{1 - z/\rho} + O\left(1 - z/\rho\right),
\end{equation}
où $c$ est le coefficient décrit plus haut, et $\theta$ définie par \eqref{eulphi}. Par conséquent, le $n$-ième coefficient de $F$ satisfait, quand $n$ tend vers $ +\infty$,
\begin{equation} \label{pequif}
f_n(u) \sim \theta'(\tau) \sqrt{ \frac{\rho^3} {2 \pi u \phi''(\tau) }} \rho^{ -n} n^{- \frac 5 2}.
\end{equation}
\end{prop}

Avec cette proposition, nous retrouvons les résultats du cas $u > 0$ des théorèmes de la section \ref{s:theoremes}.

\begin{proof}
Les résultats concernant la série $R$ proviennent de nouveau de la définition VII.4 et du théorème VII.3 de \cite{flajolet-sedgewick}.  Posons $G(z,w) = z + u\,\phi(w)$ afin de retrouver la notation du livre. (On rappelle d'après \eqref{neulR} que $R = G(z,R)$.) La fonction $G$ est analytique sur $\C \times \ens{|w| < \kappa/4}$. Le système caractéristique est satisfait ici par le couple $(\rho,\tau)$ où $\tau = \tau_u$ est l'unique élément de $]0,\kappa/4[$ tel que $\pd G w(\rho,\tau) = u \phi'(\tau) = 1$ et $\rho = \rho_u = \tau - u \phi(\tau)$. L'existence de $\tau$ est garantie par le théorème des valeurs intermédiaires : $\phi'(0) = 0$ et $\lim_{x \mapsto \kappa/4} \phi'(x) = + \infty$ (voir équation \eqref{depi}). Son unicité est due à la stricte croissance de $\phi'$. En outre, comme $R = z + u \phi(R)$, le $n$-ième coefficient de $R$ est minoré par celui de $u \phi(z)$, à savoir $u T_{2n}\,{2n - 1 \choose n}$.
Or la série génératrice des nombres $T_{2n}$ est apériodique (voir lemme \ref{l:asympt}), donc $R$ est également apériodique. Les hypothèses du théorème VII.3 sont bien vérifiées. Donc d'après ledit théorème, $R$ est analytique sur un $\Delta$-domaine de rayon $\rho$ et admet \eqref{per} pour comportement singulier.

Considérons maintenant la série $\pd F z$ grâce à la relation $\pd F z = \theta(R)$. Puisque les coefficients de $R$ sont positifs et $R(\rho,u) = \tau < \kappa/4$, il existe un $\Delta$-domaine de rayon $\rho$ dans lequel $R$ est analytique et strictement majoré en module par $\kappa/4$. Puisque le rayon de convergence de $\theta$ est également $\kappa/4$ (voir lemme \ref{l:thetaphi}), la série $\pd F z=\theta(R)$ est également analytique sur ce domaine, et le comportement singulier \eqref{posdevF} près de $\rho$ provient d'un développement de Taylor de la série $\theta$ en $\tau$. Nous appliquons alors le théorème de transfert \cite[Théo. VI.4 p. 393]{flajolet-sedgewick} afin d'obtenir le comportement asymptotique du $n$-ième coefficient de $\pd F z$ en $z$, à savoir $(n+1)f_{n+1}(u)$. L'équivalent \eqref{pequif} de $f_n(u)$ s'en déduit.

Il nous reste à étudier la fonction $u \mapsto \rho_u$. En appliquant la version analytique du théorème des fonctions implicites à l'équation $ u \phi'(\tau) = 1$, nous prouvons que la fonction $u \mapsto \tau_u$ est analytique sur $]0,+\infty[$. Au vu de \eqref{deathrow}, la fonction $u \mapsto \rho_u$ est également analytique sur cet intervalle. Sa décroissance en $u$ provient du fait que $f_n(u)$ est strictement croissante en $u$ (on utilise par exemple  la formule de Hadamards).  Cherchons maintenant un équivalent asymptotique de $\rho_u$ lorsque $u \rightarrow 0^+$. Rappelons que $\phi'(\tau_u) = 1/u$. De fait, $\tau_u$ tend vers $\kappa/4$ lorsque $u$ s'approche de $0$ et le développement \eqref{depi} de $\phi'$ nous donne l'estimation
$$\ln(\kappa/4 - \tau_u) = -\frac{ \kappa \pi} {\gamma u} + O(1)$$
de sorte que 
$$\kappa/4 - \tau_u = O\left(\exp\left( -\frac{ \kappa \pi} {\gamma} u ^{-1} \right)\right).$$
Mais $\rho_u = \tau_u - u \phi(\tau_u)$, donc en utilisant le développement asymptotique \eqref{defi} de $\phi$, nous obtenons l'approximation \eqref{exprho} avec  $c = \kappa \pi /\gamma$.
\end{proof}

\noindent \textbf{Application au cas tétravalent. }L'allure générale de la courbe $z \mapsto R(z,1)$ est montrée figure \ref{R4}.

\fig{[scale=0.35]{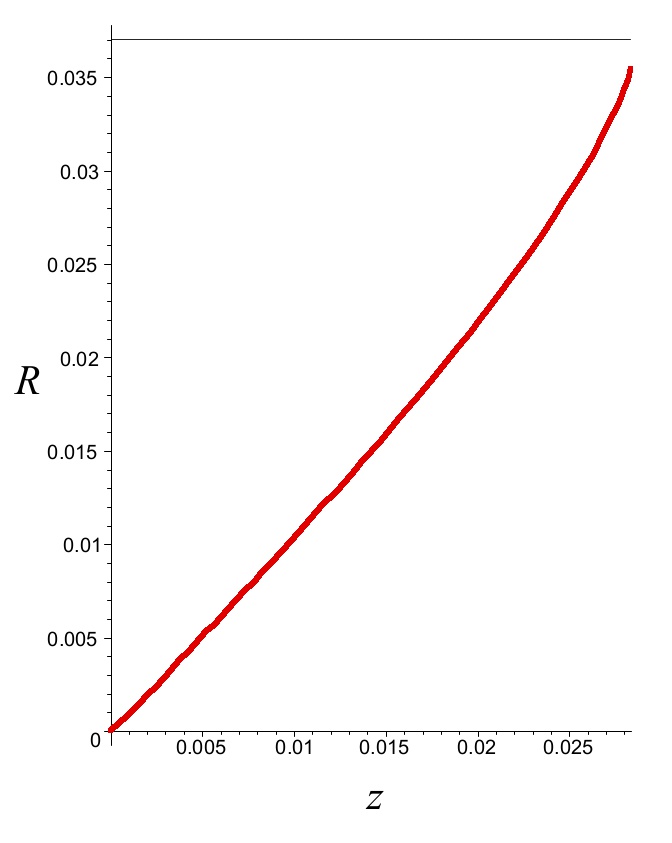} }{Courbe de $R(z,u)$, pour $z \in \, [0,\rho_1[$ et $u = 1$. La fonction $R$ n'atteint pas la singularité dominante de $\phi$ (qui est ici $1/27 \simeq 0.037$).}{R4}

\section{Quand $\boldsymbol{ u < 0 }$ : hypothèses faibles}

\begin{prop} \label{p:neg}
Soit $u \in [-1,0[$. Considérons $D(y)= \sum_{k\geq 1} d_{2k+2}y^k$ de rayon de convergence $\delta$ vérifiant les hypothèses du théorème \ref{theorayon} et définissons $\kappa$ et $\sigma $ les  deux nombres satisfaisant \eqref{eqcar} avec $0 < \kappa \sigma ^2 < \delta$ et $ 1 < \sigma $. 

Les séries $R$ et $F$ ont toutes les deux un rayon de convergence égal à $\rho = \rho_u$, où $\rho$ est défini par \eqref{deathrow}. Elles sont analytiques sur un domaine poussin (voir figure \ref{domainepoussin}) et elles admettent les développements singuliers
\begin{equation} \label{ner}
R(z)  \underset{z \rightarrow \rho}= \frac \kappa 4 -   \frac  {\pi\,\rho\sqrt{3 + 4 \kappa^2 \sigma^5 D''(\kappa\sigma^2)}} {\sigma\,u} \frac{1 - z/\rho} {\ln \pare {1 - z/\rho}} + o\pare{\frac{1 - z/\rho} {\ln \pare {1 - z/\rho}}},
\end{equation}
\begin{equation} \label{nef}
\pd {^2 F} {z^2}(z)  \underset{z \rightarrow \rho}= - \frac{ 2(2 \sigma -1)} u + \frac{\pi \sqrt{3 + 4 \kappa^2 \sigma^5 D''(\kappa\sigma^2)} }{u^2 \sigma} \frac { 4 \sigma - 2 - u \lambda }  {\ln\pare{1 - z/\rho}} + o\pare{\frac 1 {\ln\pare{1 - z/\rho}}},
\end{equation}
où  $\lambda$ est la valeur en $\kappa/4$ de la série 
\begin{equation} \label{explambda}
 \sum_{i \geq 0}  \pare{ (\sigma-1)2i\,T_{2i} - T_{2i}^\bigstar} {2i \choose i} z^i.
 \end{equation}
De plus, nous avons  $4 \sigma - 2 - u \lambda > 0$  : le premier terme non analytique de \eqref{nef} se comporte  bien comme $1/\ln(1-z/\rho)$
.
\end{prop}

Cette proposition établit le cas $u <0$ des théorèmes \ref{theorayon} et \ref{theoasymp}.

\noindent \textbf{Remarque. }Nous conjecturons que la constante $\lambda$ est strictement positive, mais nous n'avons pas besoin de cette hypothèse.

\begin{proof} 
Nous commençons comme précédemment par la série $R$. Réécrivons l'équation \eqref{neulR} sous la forme $\Omega(R) = z$ avec $\Omega(y) = y - u \phi(y)$. Comme $\Omega(0) = 0$ et $\Omega'(0) = 1 > 0$, nous pouvons appliquer le corollaire \ref{c:fi} p. \pageref{c:fi} dans lequel $R$ joue le rôle de $Y$. Soient $\tau$ et $\rho$ définis dans le corollaire. Le rayon de convergence de $\Omega$ est le même que celui de $\phi$, à savoir $\omega = \kappa/4$. \'Etant donné que $u < 0$, la dérivée $\Omega'(y) = 1 - u \phi'(y)$ ne s'annule pas sur $[0,\kappa/4[$, ce qui implique que $\tau = \omega = \kappa/4$. La propriété $(5)$ du corollaire \ref{c:fi} donne alors \eqref{deathrow}, c'est-à-dire $\rho = \kappa/4 - u \phi(\kappa/4)$.

Ce même corollaire montre que $R$ admet un prolongement analytique sur $[0,\rho[$, intervalle sur lequel $R$ croît (de $0$ vers $\omega$) et sur lequel l'équation $R = z + u \phi(R)$ est vérifiée.

D'après le corollaire \ref{c:RSupos} p. \pageref{c:RSupos}, la série $(R-z)/u$, qu'on notera $\mathcal R$, a des coefficients positifs en $z$. Cette série est, comme $R$, analytique sur $[0,\rho[$. Le théorème de Pringsheim implique alors que le rayon de convergence de $\mathcal R$ doit être au moins égal à $\rho$. Il en est donc de même pour le rayon de convergence de $R$. Montrons maintenant que le comportement de $R$ au voisinage de $\rho$ est singulier -- nous prouverons de fait que le rayon de convergence de $R$ est  $\rho$.

Déduisons le comportement de $R$ en $\rho$ du théorème  de log-inversion (théorème \ref{loginversion} p.\pageref{loginversion}). Posons 
$$\psi(z) := \rho + \frac \kappa 4(z-1) + u \phi \pare{\frac \kappa 4(1 - z)}.$$
D'après le lemme \ref{l:thetaphi}, les séries $\phi$ et $\theta$ sont analytiques sur un $\Delta$-domaine de rayon $\kappa/4$. En particulier, il existe $s>0$ et $\alpha \in \, ]\pi/2,\pi]$ tels que $\phi$ et $\theta$ sont analytiques sur $\kappa/4 - \kappa D_{s,\alpha}/4$, où
$$D_{s,\alpha} = \enstq{ r\,e^{i\theta} } {r \in \, ]0,s[\textrm{ et } |\theta| < \alpha}.$$
Donc $\psi$ est analytique sur $D_{s,\alpha}$. En utilisant \eqref{defi} et l'égalité $\rho = \kappa/4 - u \phi(\kappa/4)$, on voit que $\psi 
(z) \sim - c \, z \ln z$ quand $z$ s'approche de 0 dans $D_{s,\alpha}$, où 
$$c = - u \frac {\kappa\,\sigma} {4\,\pi\,\sqrt{3 + 4 \kappa^2 \sigma^5 D''(\kappa\sigma^2)}}.$$
Nous pouvons alors appliquer le théorème de log-inversion à $\psi$ et $\beta = \pi/4 + \alpha/2$ (de sorte que $\beta \in \, ]\pi/2,\alpha[$) : il existe $s > r > 0$, $r' > 0$ et une fonction $\Upsilon$ analytique sur $D_{r',\beta} = \ens{ |z| < r'\textrm{ et }|\Arg z|< \beta}$ vers $D_{s,\beta}$ tels que $\psi(\Upsilon(y)) = y$. De plus, $\Upsilon(y)$ est l'unique antécédent de $y$ par $\psi$ qui se trouve dans $D_{r,\alpha}$. Or l'équation $R = z + u \phi(R)$ peut se réécrire 
$$\psi\pare{1 - \frac 4 \kappa R(\rho-y) } = y.$$
De plus, quand $y$ tend vers $0^+$ par valeurs (réelles) positives, \mbox{$1 -  4  R(\rho-y) / \kappa$} tend  aussi vers $0$ par valeurs positives (on rappelle que $R$ est croissante). En d'autres termes, $1 -  4  R(\rho-y) / \kappa$  reste dans $D_{r,\alpha}$ quand \mbox{$y>0$} est suffisamment petit. Ainsi, d'après l'unicité décrite plus haut, les fonctions $\Upsilon$ et \mbox{$1 -  4  R(\rho-y) / \kappa$} coïncident sur les petites valeurs positives de $y$ et donc, par prolongement analytique, sur tout $D_{r',\beta}$. Si nous revenons aux variables d'origine, cela signifie que $R$ est analytique sur $\rho-D_{r',\beta}$ et est à valeurs dans $\kappa/4-\kappa D_{r,\alpha} /4$ sur ce domaine. Du reste, comme $\Upsilon \sim - y/(c \ln y)$ (il s'agit de la dernière parie du théorème de log-inversion), nous obtenons l'approximation \eqref{ner}, qui est valide sur $\rho-D_{r',\beta}$.

En conclusion, $R$ est analytique sur le domaine ouvert en forme de poussin de la figure \ref{domainepoussin}.

Prouvons qu'il en est de même pour $F$. D'après \eqref{neulF}, nous savons que $\pd F z = \theta(R)$. La série $R$ croît strictement sur $[0,\rho[$, d'où $0 = R(0) \leq R(z) < R(\rho) =\kappa/4$ pour $z$ dans cet intervalle. Mais le rayon de convergence de $\theta$ est $\kappa/4$. Donc $\pd F z$ est analytique sur $[0,\rho[$, ce qui implique par le théorème de Pringsheim que le rayon de convergence de $F$ est au moins égal à $\rho$. De plus, $R(z)$ est à valeurs dans $\kappa/4-\kappa D_{r,\alpha} /4$ quand $z$ appartient à $\rho-D_{r',\beta}$ et $\theta$ est analytique sur $\kappa/4-\kappa D_{r,\alpha} /4$. Donc $\pd F z$ est analytique sur $\rho-D_{r',\beta}$. Au final, on a bien prouvé que $F$ était analytique sur le domaine poussin ouvert de la figure \ref{domainepoussin}.

\'Etudions maintenant le comportement singulier de $F$ autour de $\rho$. Appliquer \eqref{ner} à $\pd F z = \theta(R)$ serait une première idée, mais cela ne donne pas immédiatement le comportement singulier de $\pd F z$ en fonction de \mbox{$1- z/\rho$}. Le plus rapide est de travailler directement avec $\pd {^2 F} {z^2}$. En effet,
\begin{equation} \label{relsec}
\pd {^2 F} {z^2} = \pd R z\, \theta'(R(z)) = \frac{\theta'(R(z))}{1-u\phi'(R(z))}.
\end{equation}
(Nous avons utilisé l'identité
\begin{equation} \label{derR}
(1 - u \phi'(R)) \, \pd R z  = 1
\end{equation}
qui provient de la dérivation de l'identité \eqref{neulR}.)
Or en utilisant le lemme \ref{l:thetaphi},
on remarque que
\begin{multline} \label{thetapsurphip}
\frac {\theta'\pare{\frac \kappa 4 (1-\varepsilon)}} {1 - u \phi'\pare{\frac \kappa 4 (1-\varepsilon)}} = \frac {- \frac {2 \gamma} { \kappa \pi } (2\sigma - 1) \ln \varepsilon + b^\bigstar + O\pare{\varepsilon^{\frac 1 2}}}
{1 + u \frac \gamma {\kappa \pi} \ln \varepsilon  - b\,u + O\pare{\varepsilon^{\frac 1 2}}} \\
= - \frac{ 2(2 \sigma -1)} u + \frac{ \kappa \pi}{u^2 \gamma} \pare{ 4 \sigma - 2 - u \pare{ 2 (2 \sigma -1) b - b^\bigstar } } \frac 1 {\ln \varepsilon} + O \pare{\frac {\varepsilon^{\frac 1 2}} {\ln \varepsilon}}.
  \end{multline}
Mais $\ln\pare{1-4R(z)/\kappa} \sim \ln\pare{1 - z/\rho}$, d'après \eqref{ner}. Donc en remplaçant $\varepsilon$ par \mbox{$1 - 4R(z)/\kappa$} dans le développement limité du dessus et en utilisant \eqref{relsec}, nous obtenons \eqref{nef}, où la constante $\lambda$ est égale à $2(2\sigma - 1)b-b^\bigstar$, soit la valeur en $\kappa/4$ de la série $2(2\sigma -1)\phi'-\theta'$ (cf. lemme \ref{l:thetaphi}), ce qui se réécrit \eqref{explambda}.

Il nous reste à prouver que $4 \sigma - 2 - u \lambda$  est strictement positif. Cela revient à montrer que $c_u > 0$, où
$$ \pdd 2 F z  = - \frac{ 2(2 \sigma -1)} u + c_u \frac 1 {\ln\pare{1 - z/\rho}} + o \pare{\frac 1 {\ln\pare{1 - z/\rho}}}.$$
 Nous utilisons pour cela le corollaire \ref{c:minodeF} p. \pageref{c:minodeF} (on rappelle que $t=1$ ici). En sommant sur $a$, nous voyons que 
$$\frac {n (n-1)} u r_{n-2}(u) \leq n(n-1)(n-2) f_{n-3}(u),$$
où $r_n(u)$ est le $n$-ième coefficient de $R$ en $z$. En termes de séries génératrices, cela signifie que $\ \pd {^2 R} {z^2} \leq u \pd {^3 F} {z^3}$. Cette inégalité reste vraie au voisinage de $\rho$.  
Cherchons donc le développement singulier de $ u\pd {^3 F} {z^3} - \pd {^2 R} {z^2}$ en $z = \rho$. 

Notons une petite subtilité technique : le théorème VI.8 p. 419 de \cite{flajolet-sedgewick}, qui permet de dériver proprement des développements asymptotiques, ne peut pas être appliqué (voire adapté) à \eqref{ner} et \eqref{nef}. En effet, il n'est pas prouvé\footnote{Je conjecture même que cela est faux de manière générale. En effet, cela impliquerait que la dérivée d'une fonction $f(\varepsilon)$ qui est analytique sur un voisinage de $0$ privé de $]-\infty,0]$ et qui converge en $0$ est un $o(1/\varepsilon \ln \varepsilon)$ quand $\varepsilon$ tend vers $0$. Je ne vois aucune raison pourquoi cela serait vrai.} que la dérivée d'un $o\pare{\ln(1-z)^{-1}}$ est un $o \pare{(1-z)^{-1} \ln(1-z)^{-2}}$.

Pour le comportement de $\pdd 2 R z$, utilisons la relation
$$ \pdd 2 R z = \pd \, z \pare{\frac 1 {1-u\,\phi'(R)}} = \frac{ - u \phi''(R)} { \pare{1-u\,\phi'(R)}^3    }.$$
Le développement de $\phi''(z)$ peut être obtenu en dérivant la relation \eqref{depi} (ici le théorème VI.8 de \cite{flajolet-sedgewick} s'applique). En composant ce développement et \eqref{depi}   par le développement \eqref{ner} de $R$, nous obtenons après calcul
\begin{equation} \label{pddR}
\pdd 2 R z  \sim  - \frac{\pi \kappa}{u \gamma \rho} \frac{1}{\pare {1 - z/\rho} \ln^2 \pare {1 - z/\rho}},
\end{equation}
comme nous pouvions nous y attendre au vu de  \eqref{ner}.

Pour le comportement de $\pdd 3 F z$, nous utilisons l'identité
\begin{equation} \label{pdd3F}
\pdd 3 F z = \pd \, z \pare{ \frac  {\theta'(R)} {1 - u \phi'(R)}  }
= \frac {\theta''(R) + u \phi''(R)  \frac{\theta'(R)}{1-u\phi'(R)} } {\pare{1 - u \phi'(R)}^2}.
\end{equation}
Il est important de trouver un équivalent asymptotique  de $\theta''(z) + u  \phi''(z)  \frac{\theta'(z)}{1-u\phi'(z)}$ en $z = \kappa /4$ avant d'y injecter le développement \eqref{ner} de $R$. (En effet, les termes dominants des  deux termes s'annulent.) En utilisant le développement \eqref{thetapsurphip}, qui peut se réécrire 
$$\frac {\theta'\pare{\frac \kappa 4 (1-\varepsilon)}} {1 - u \phi'\pare{\frac \kappa 4 (1-\varepsilon)}} = 
 - \frac{ 2(2 \sigma -1)} u +  \frac {c_u} {\ln \varepsilon} + O \pare{\frac {\varepsilon^{\frac 1 2}} {\ln \varepsilon}},$$    et les dérivées des estimées \eqref{depi} et \eqref{devthetap}, 
nous trouvons 
$$\theta''(z) +  u \phi''(z)  \frac{\theta'(z)}{1-u\phi'(z)} \sim  c_u {4 u \gamma}{\kappa^2 \pi}  \frac 1 {\left(1 - \frac{4z}{\kappa} \right) \ln \left(1 - \frac{4z}{\kappa} \right)}. $$
Nous utilisons alors cette relation avec \eqref{pdd3F}, \eqref{depi} et \eqref{ner} et nous retrouvons :
\begin{equation} \label{Ftierceneg} \pdd 3 F z \sim   \frac{c_u}{\rho \pare {1 - z/\rho} \ln^2 \pare {1 - z/\rho}}. 
\end{equation}
(Nous pouvons maintenant comprendre pourquoi nous avons dérivé $F$ trois fois plutôt qu'une ou deux :  son développement singulier a un terme dominant proportionnel à $c_u$.)  En utilisant l'inégalité 
$ u\pd {^3 F} {z^3} \geq \pd {^2 R} {z^2}$ et \eqref{pddR}, nous voyons que $ u c_u \geq -\pi \kappa / (u \gamma)$, ce qui montre bien que $c_u$ est strictement positif.
\end{proof}

\noindent \textbf{Application au cas tétravalent. } La série \eqref{explambda} vaut après calcul 
$$\sum_{i \geq 1} \frac{4 (3i)!} {i!^2(i-1)!} z^i,$$
qui est bien une série à coefficients positifs.
La constante $\lambda$, qui est la valeur de cette série en $z= \frac 1 {27}$, est égale $\frac {9 \sqrt{3}} {\pi} - 4$. En utilisant \eqref{rho4}  et $\sigma = 3/2$, nous remarquons la (surprenante) simplification 
$$4\sigma - 2 - u \lambda = 27 \rho.$$ Le développement asymptotique de $ \pdd 2 F z$ s'écrit alors
$$
\pdd 2 F z + \frac 4 u \sim   
\frac{72 \sqrt 3 \pi \rho}{u^2 \ln (1-z/\rho)}.
$$
La courbe de la fonction $z \mapsto R(z,-1/2)$ est montrée figure \ref{R4valent02}. Nous pouvons constater qu'elle n'a pas du tout la même forme que la fonction  $z \mapsto R(z,-1/2)$ tracée figure \ref{R4}.

\fig{[scale=0.35]{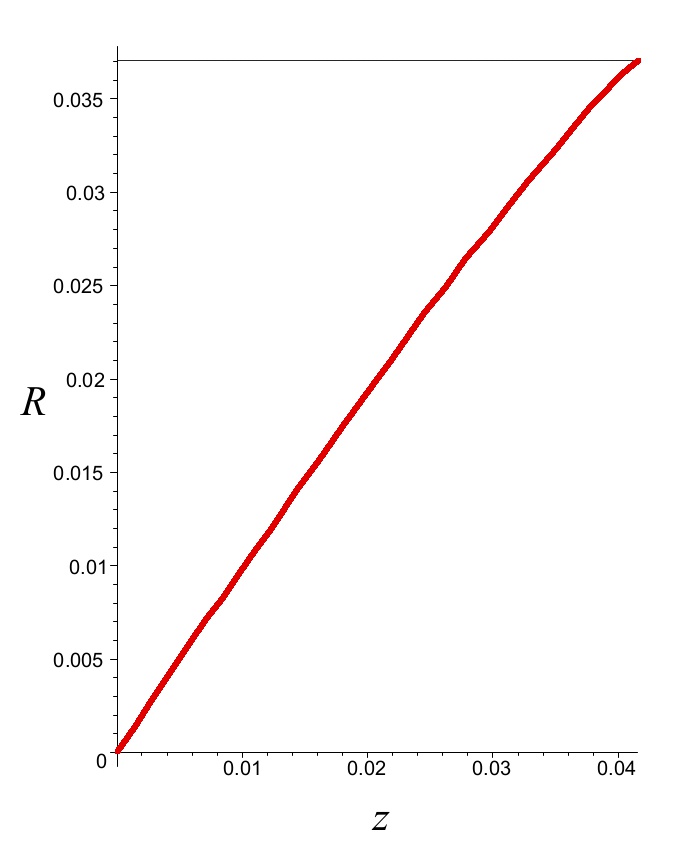} }{Courbe de $R(z,u)$, pour $z \in \, [0,\rho_{-1/2}[$ et $u = -1/2$. La valeur de $R$ en $\rho_{-1/2}$ vaut $1/27$, qui est la singularité dominante de $\phi$.}{R4valent02}

\noindent \textbf{Remarque. } Même sans hypothèses supplémentaires, nous pouvons appliquer le théorème taubérien de Hardy, Littlewood et Karamata (voir \cite[Théo. VI.13 p. 435]{flajolet-sedgewick}) à $\pdd 3 F z (\rho\,z)$ grâce à la relation \eqref{Ftierceneg}. En effet, $z \mapsto 1/{\ln^2(1-z)}$ est une fonction à variations lentes et les coefficients  de $\pdd 3 F z (\rho\,z)$, qui sont égaux à $\rho^{n} n(n+1)(n+2)f_{n+2}(u)$, sont positifs. Nous pouvons alors en déduire l'équivalent asymptotique 
$$\sum_{k = 0}^n \rho^{k} k(k+1)(k+2)f_{k+2}(u) \sim c_u n^{-1} (\ln n)^{-2},$$
ce qui est cohérent avec l'équivalent \eqref{fnuasympt}, obtenu par des hypothèses plus fortes. Malheureusement, cela ne nous permet pas de trouver un équivalent asymptotique sur $f_n(u)$ seul\footnote{L'hypothèse de croissance de $\rho^{-k} k(k+1)(k+2)f_{k+2}(u)$ en $k$, qui nous aurait permis de conclure, est manifestement fausse et semble difficilement adaptable.}.

\section{Quand $\boldsymbol{ u < 0 }$ : hypothèses fortes}

\begin{prop} \label{p:lfort}
Soit $u \in [-1,0[$. Supposons que la série génératrice des poids $D(y)= \sum_{k\geq 1} d_{2k+2}y^k$, de rayon de convergence $\delta$, vérifie les hypothèses du théorème \ref{theorayon}. Supposons en outre que les séries  $\phi$  et $\theta$ définies par \eqref{eulphi} et sont analytiques sur un disque centré en $0$ de rayon $2 \phi(\kappa/4) + \kappa/4$ privée de la demi-droite $[\kappa/4,+\infty[$, où $\kappa$ et $\sigma $ sont les  deux nombres satisfaisant \eqref{eqcar} avec $0 < \kappa \sigma ^2 < \delta$ et $ 1 < \sigma $.  

Alors les séries $R$ et $\pd F z$ sont analytiques sur un $\Delta$-domaine de rayon $\rho$. Le théorème de transfert permet alors de déduire 
$$f_n(u) \sim \frac{\pi \sqrt{3 + 4 \kappa^2 \sigma^5 D''(\kappa\sigma^2)} }{u^2 \sigma} \pare{ 4 \sigma - 2 - u \lambda }\,  \rho^{-n+2} n^{-3} \pare{\ln n}^{-2}, $$
où $\lambda$ est la constante définie par la proposition \ref{p:neg}.
\end{prop}

Cette proposition synthétise ce qui nous restait à démontrer pour pouvoir établir le théorème \ref{theobis}. Le point le plus délicat est l'analyticité de $R$ sur un $\Delta$-domaine : l'analyticité de $\pd F z$ se déduit de la relation $\pd F z = \theta(R)$ et l'équivalent asymptotique de $f_n(u)$ se déduit du théorème de transfert via \eqref{nef}.

\noindent \textbf{Remarque. }L'hypothèse d'analyticité de $\phi$ et $\theta$ sur un domaine qui dépasse le disque fermé de rayon $\rho$ n'est pas superflue. En effet, nous pouvons montrer à partir de \eqref{derR} que $R$ est concave sur $[-\rho,\rho]$ et que donc $\module{R(-\rho)} > R(\rho) = \kappa/4$. La fonction $R$ prend donc des valeurs qui sortent du disque de convergence de $\phi$, il est donc important de contrôler le domaine d'analyticité de $\phi$ et de $\theta$ !

Commençons par prouver le lemme suivant.

\begin{lem} Soit $u \in [-1,0[$. Sous les hypothèses de la proposition précédente, $R(z)$ est bien défini pour tout $z \neq \rho$ de module inférieur ou égal à $\rho$, et appartient au disque de centre $0$  et de rayon $2 \phi(\kappa/4) + \kappa/4$ privé de la demi-droite $[\kappa/4,+\infty[$.
\end{lem}
\begin{proof} Notons $\mathcal R = (R-z)/u$. 
Le corollaire \ref{c:RSupos} p. \pageref{c:RSupos} montre que $\mathcal R$ est une série à coefficients positifs. Ainsi $\mathcal R(z)$ converge absolument pour tout $z$ avec $\module z  \leq \rho$ car 
$$  \mathcal R(\rho) = (R(\rho)-\rho)/ u = \phi(\kappa/4) < + \infty.$$
Il en est donc de même pour $R$ et nous avons
$$\module{R(z)} \leq   \module{z} - u \module{\mathcal R(z)} \leq  \rho - u \phi(\kappa/4) = \kappa/4 - 2 u \phi(\kappa/4) \leq \kappa/4 + 2 \phi(\kappa/4) $$
(on a utilisé \eqref{deathrow} et $u \geq -1$). 
Montrons maintenant que $R(z)$ n'appartient pas à $[\kappa/4,+\infty[$. 

Supposons dans un premier temps que $z \in [0,\rho[$. Nous avons vu dans la preuve de la proposition \ref{p:neg} que $R$ était croissante sur $[0,\rho[$, donc $R(z) \leq R(\rho) = \kappa/4$.

Supposons ensuite que $z \in [-\rho,0[$. Sur cet intervalle, la fonction $R$ est réelle (cf. lemme \ref{l:valpos} p. \pageref{l:valpos}) et continue. De plus, $R(0) = 0$. Donc si $z \in [\kappa/4,+\infty[$, alors il existe  $t \in \, [-\rho,0[$ tel que $R(t) = \kappa/4$. Considérons $t$ maximal pour cette propriété. Il existe alors un voisinage complexe de l'intervalle $]t,0]$ sur lequel $R$ est à valeurs dans un disque centré en $0$ de rayon $2 \phi(\kappa/4) + \kappa/4$ privée de la demi-droite $[\kappa/4,+\infty[$. Comme $\phi$ est analytique sur ce domaine, la relation $R' = (1-u \phi'(R))^{-1}$ y est valide. Mais $R'(0) = 1$ et $R(t) = \kappa/4 > R(0) = 0$, donc $R'$ doit s'annuler sur $]t,0[$, mais cela contredit le fait que $R' = (1-u \phi'(R))^{-1} \neq 0.$

Enfin, supposons que $z$ ne soit pas réel. Nous allons prouver que $R(z)$ n'est pas réel non plus. D'une part,
\begin{equation} \label{dfgh2}
|\Ima R(z)| = |\Ima (z + u \mathcal R(z))| \geq |\Ima z| + u \, |\Ima \mathcal R(z)|.
\end{equation}
D'autre part, comme $ \mathcal R (\Ree z)$ est réel,
$$
|\Ima \mathcal R (z)| = |\Ima (\mathcal R(z) -  \mathcal R (\Ree z))| \leq |\mathcal R(z) -  \mathcal R (\Ree z)|.
$$
Donc d'après l'inégalité des accroissements finis\footnote{L'inégalité des accroissements finis se généralise aux fonctions holomorphes, contrairement au théorème des accroissements finis (avec égalité) -- par exemple si $\alpha(t) = e^{it}$, alors $\alpha(0) = \alpha(2 \pi)$ mais $\alpha'$ ne s'annule jamais.},
on a 
\begin{equation} \label{fghj2}
|\Ima \mathcal R (z)| < |z - \Ree z| \max_{y \in [\Ree z, z]} | \mathcal R'(y) | \leq | \Ima z | \max_{|y| \leq \rho} | \mathcal R'(y)|,
\end{equation}
où l'inégalité stricte provient du fait que $\mathcal R'$ n'est pas constante sur $[\Ree z, z]$. Mais $\mathcal R'$ est une série entière à coefficients positifs, donc pour $|y| \leq \rho$,
\begin{equation}\label{ghjk2}
|\mathcal R'(y)| \leq \mathcal R'(\rho) = \frac 1 u(R'(\rho) - 1)
= \frac 1 u \left(\lim_{t \rightarrow \rho} \frac 1 {1 - u \phi'(R(t))} - 1\right) = -\frac 1 u,
\end{equation}
car $\phi'(z)$ diverge vers $+\infty$ quand $z$ tend vers $\kappa/4$
 (voir \eqref{depi}). En revenant sur \eqref{fghj2},  nous trouvons \mbox{$|\Ima \mathcal R(z)| < -|\Ima z|/u$}, ce qui, couplé à \eqref{dfgh2}, donne $|\Ima R(z)| > 0$.
\end{proof}

Nous pouvons maintenant prouver  la proposition principale de cette partie.

\begin{proof}[Démonstration de la proposition \ref{p:lfort}.] D'après la proposition \ref{p:neg}, la série $R$ est analytique sur un domaine en forme de poussin. Pour prouver la $\Delta$-analyticité, il nous reste donc à montrer que pour tout $\mu \neq \rho$ de module $\rho$, la fonction $R$ est analytique en $\mu$.

D'après le lemme précédent, on sait que $R$ est prolongeable par continuité en $\mu$ et $\phi$ est analytique en $R(\mu)$. Par ailleurs, la relation $R = z + u \phi(R)$ se prolonge par analyticité sur tout le disque ouvert centré en $0$ de rayon $\rho$. Par continuité, nous avons $R(\mu) = \mu + u \phi(R(\mu)).$

 Nous voulons appliquer le théorème des fonctions implicites en $(\mu,R(\mu))$ à l'équation $R = z + u \phi(R)$. Il faut montrer au préalable que $u\phi'(R(\mu)) \neq 1$. Supposons le contraire. D'après \eqref{derR}, cela veut dire que $R'$, et donc $\mathcal R' = (R' - 1)/u$, diverge vers $+ \infty$ en module quand $z$ s'approche de $\mu$. Mais $\mathcal R'$ a des coefficients positifs et \eqref{ghjk2} montre que $\mathcal R'(\rho) = -1/u < + \infty$. Donc $\mathcal R'$ reste borné au voisinage de $\mu$, ce qui mène à une contradiction. Donc le théorème des fonctions implicites s'applique : $R$ est bien analytique en $\mu$.
 
La fonction $R$ est donc analytique sur un $\Delta$-domaine de rayon $\rho$ et d'après le lemme précédent, son image est incluse dans le domaine d'analyticité de $\theta$. Par conséquent, $\pd F z = \theta(R)$ est analytique sur le même $\Delta$-domaine. Nous pouvons donc appliquer le théorème de transfert \cite[Théo. VI.4 p. 393]{flajolet-sedgewick} à $\pdd 2 F z$ via \eqref{nef}. Nous obtenons l'équivalent asymptotique voulu.
\end{proof}

\noindent \textbf{Application au cas tétravalent. }Nous obtenons l'équivalent
$$f_n(u) \sim \frac{72 \sqrt 3 \pi \rho}{u^2} \rho^{-n+3} n^{-3} (\ln n)^{-2}.$$

\section{Exemple récapitulatif : le cas $4$-eulérien}

Nous voulons appliquer les résultats précédents afin d'obtenir le comportement asymptotique des cartes forestières $4$-eulériennes.

La série génératrice des poids vaut $D(y) = y/(1-y)$. Cette série diverge bien en son rayon $\delta=1$. L'hypothèse d'apériodicité est trivialement vérifiée car le coefficient de $y^1$ dans $D$ n'est pas nul.  

Le système \eqref{eqcar} se traduit en
$$ \kappa\,\sigma ^3 =  (\sigma - 1) (1 - \kappa \sigma^2) ,\quad 2\,\kappa\,\sigma ^3 = (1 - \kappa \sigma^2)^2$$
et admet pour unique\footnote{avec l'hypothèse $\kappa \sigma^2 < 1$} solution  $\kappa = (71 - 17 \sqrt{17})/8$ et $\sigma = (7 + \sqrt{17} )/8$.

Remarquons au passage que le lemme \ref{l:asympt} nous donne l'équivalent asymptotique non trivial\footnote{Dans une première version, j'avais calculé cet équivalent asymptotique grâce à la méthode de Bender \cite{bendermethod}. Mais cette approche était beaucoup plus calculatoire -- cela m'a tout de même permis de confirmer le résultat.} 
$$T_{2n} \sim \pare{\frac 1 8 - \frac {\sqrt{17}} {34}} \sqrt{17 + 9 \sqrt{17}} \sqrt{\frac 2 \pi}  \pare{ \frac{ 71 + 17 \sqrt{17}} 4  }^n n^{-\frac 3 2}$$
sachant que $T_{2n}$ vaut d'après le lemme \ref{l:4euln} p. \pageref{l:4euln}
$$T_{2n} = \sum_{k=0}^{n-2} \frac 1 {2n-1}  { 2n + k - 1  \choose k + 1 } { n - 2 \choose k }.$$

 Afin de prouver les hypothèses du théorème \ref{theobis}, rappelons les fonctions $\phi$ et $\theta$ associées aux cartes forestières $4$-eulériennes (voir figure \ref{recapitulatif}) :
$$
\theta(x) =  2  \sum_{\substack{ i \geq 2 \\ i - 2  \geq k \geq 0  }}  \frac{ (i+k)(2i+k-1)!(i-2)!} {i!^2(i-k-2)!k!(k+1)!} x^i,$$
$$
 \phi(x) =  \sum_{\substack{ i \geq 2 \\  i - 2  \geq k \geq 0   }}  \frac{ (2i+k-1)!(i-2)!      } {i!(i-1)!(i-k-2)!k!(k+1)!} x^i.
$$
Le lemme \ref{l:thetaphi} montre que $(71-17\sqrt{17})/32$ est le rayon de convergence de $\theta$ et $\phi$.
Nous allons prouver que ces deux séries  sont analytiques sur $\C \backslash \, \left[(71-17\sqrt{17})/32,+\infty\right[$. Le plus simple pour cela est d'étudier les équations différentielles linéaires associées. En effet, la série $\phi$ satisfait l'équation
\begin{multline*}
\pare{1728\,{x}^{4}-8068\,{x}^{3}+1991\,{x}^{2}-50\,x} \phi'''(x) + \pare{3456\,{x}^{3}-8868\,{x}^{2}+3550
\,x-50
} \phi''(x) \\ + \pare{-432\,{x}^{2}+408\,x+360} \phi'(x) + \pare{432\,x+240}\phi(x)+ 300+540\,x = 0.
\end{multline*}
(Cette équation a été devinée puis vérifiée par \textit{maple}.)
D'après un théorème classique sur les équations différentielles linéaires (voir par exemple \cite[Théo 2.2]{wasow}), la fonction $\phi$ est analytique sur tout point $x$ tel que $$1728\,{x}^{4}-8068\,{x}^{3}+1991\,{x}^{2}-50\,x \neq 0.$$ Or les racines de ce  polynôme sont $0,{\frac {25}{108}},{\frac {71}{32}}-{\frac {17}{32}}\,\sqrt 
{17},{\frac {71}{32}}+{\frac {17}{32}}\,\sqrt {17}$. Comme $\phi$ a un rayon de convergence non nul,  elle est analytique en $0$. Ses singularités potentielles se trouvent donc au niveau des trois autres racines. Nous prouvons ainsi que $\phi$ est analytique sur $\C \backslash \, \left[(71-17\sqrt{17})/32,+\infty\right[$ puisque $(71-17\sqrt{17})/32$ est la plus petite des trois racines. 
Quant à l'analyticité de $\theta$ sur le même domaine, il suffit d'observer que
\begin{multline*}\theta(x) = \pare{\frac{34} {27} x^3 - \frac{1207} {216} x^2 + \frac{17} {108} x } \phi'''(x) + \pare{\frac{34} 9 x^2 - \frac{1207} {108} x + \frac{17} {108}} \phi''(x) \\ 
+ \pare{\frac x {18} - \frac 4 3} \phi'(x) - \frac 5 9 \phi(x) - 10 x^3 - \frac x 2 - \frac{17} {18}.\end{multline*}

Nous avons ainsi prouvé toutes les hypothèses des théorèmes de la section \ref{s:theoremes}. En particulier, nous pouvons en déduire que le nombre pondéré $f_n(u)$ des cartes forestières $4$-eulériennes à $n$ faces et avec un poids $u$ par composante non racine satisfait
\begin{equation} \label{fnu4eul}
f_n(u) \sim \left\{ \begin{array}{ll}
 c_u \rho_u^{-n}n^{-3} (\ln n)^{-2} & \textrm{si }u \in [-1,0[, \\
c_u \rho_u^{-n}n^{-3} & \textrm{si }u=0, \\
c_u \rho_u^{-n}n^{-5/2} & \textrm{si }u>0.
\end{array} \right.
\end{equation}
(Les valeurs de $c_u$ que nous donnent les théorèmes ne sont pas forcément intéressantes ici.)

\noindent \textbf{Remarque. }Nous pouvons également pondérer chaque arête des cartes forestières $4$-eulériennes par un réel $t>0$ plus forcément égal à $1$. De fait $t$ joue le rôle d'un paramètre modulable au même titre que $u$. Il s'avère que les théorèmes s'appliquent de la même manière. En effet la nouvelle série génératrice des poids est $D(y) = (ty)/(1-ty)$. Elle diverge également en son rayon. Pour prouver l'analyticité de $\phi$ et $\theta$ sur $\C \backslash [\kappa/4,+\infty[$ (ici $\kappa$ vaut ${\frac {27}{8}}\,{t}^{2}+\frac 9 2 \,t+1-\frac 1  8\,\sqrt {t \left( 9\,t+8 \right) 
^{3}}
$), nous devons recalculer les équations différentielles et procéder à la même analyse que plus haut. Au final, le même comportement asymptotique que \eqref{fnu4eul} est obtenu.

Nous pouvons déduire de  cela le comportement asymptotique des cartes forestières \textbf{eulériennes} (on autorise à nouveau les sommets de degré $2$) avec un poids $u$ par composante connexe et un poids $t$ par arête, où $-1 \leq  u$ et $0 < t < \min\pare{1,(u+1)^{-1}}$ (cette dernière inégalité est nécessaire pour que les coefficients en $z$ de la série génératrice des cartes forestières eulériennes soient finis). En effet, d'après la proposition \ref{correspondanceeul} p. \pageref{correspondanceeul}, la série génératrice des cartes forestières eulériennes et la série génératrice des cartes forestières 4-eulériennes ne diffèrent que d'un changement de variable affine (si on fixe $t$ et $u$), le comportement \eqref{fnu4eul} est donc encore valable.

\section{Grandes cartes aléatoires munies d'une forêt ou d'un arbre couvrant}

Nous établissons dans cette ultime section quelques résultats sur les grandes cartes aléatoires (toujours $4$-eulériennes et apériodiques) 
équipées d'une forêt ou d'un arbre couvrants. Dans chaque cas, nous définissons une distribution de Boltzmann sur les cartes à $n$ faces qui fait intervenir un paramètre $u$. Il prend en compte le nombre de composantes dans la forêt couvrante, ou le nombre d'arêtes internes actives dans l'arbre couvrant (de manière équivalente, il s'agit du niveau d'une configuration du tas de sable dans le dual, comme nous l'avons expliqué sous-section \ref{ss:modele} p. \pageref{ss:modele}). L'effet de la transition de phase en $0$ est toujours visible.

\subsection{Nombre et taille des composantes dans les cartes forestières}

\newcommand{\pc}{\mathbb P_c}

Soient $n \in \N$ et $u \in \,[0,+\infty[$. Définissons la distribution de probabilité  $\pc$ sur les cartes forestières $(C,F)$ avec un poids $d_{2k}$ par sommet de degré $2k$ et $n$ faces par
$$\pc(C,F) = \frac{u^{\cc(F) - 1}}{f_n(u)},$$
où $\cc(F)$ désigne le nombre de composantes de $F$ et où $f_n(u)$ compte le nombre de cartes forestières avec un poids $u$ par composante non racine. Appelons $C_n$ la variable aléatoire sous cette distribution qui compte le nombre de composantes de $F$, et $S_n$ celle qui prend pour valeur le nombre de demi-arêtes incidentes à la composante racine (il s'agit ici du paramètre de taille le plus naturel que nous pouvons mettre sur les arbres). 

Quand $u=0$, seules les cartes avec un arbre couvrant ont une probabilité non nulle, d'où $C_n=1$ et $S_n = 2(n-1)$ (où $n-1$ est le nombre d'arêtes externes pour n'importe quel arbre couvrant d'une carte à $n$ faces).

Regardons ce qui change quand $u > 0$.
\begin{prop} \label{p:CnSn}
Supposons $u > 0$. Considérons une suite de poids $(d_{2k})$ satisfaisant les hypothèses du théorème \ref{theorayon}. Avec la distribution de probabilité $\pc$, nous avons quand $n$ tend vers $+\infty$
$$\mathbb E_c(C_n) \sim \frac{u \phi(\tau)}{\tau - u \phi(\tau)},
$$
où $\phi$ est la série définie par \eqref{eulphi} et $\tau$ est l'unique solution  réelle positive inférieure au rayon de convergence de $\phi$ de  l'équation $u \phi'(y) = 1$.

La taille $S_n$ de la composante racine admet une loi limite discrète. En effet, pour $k \geq 1$,
\begin{equation} \label{pcsn}
\lim_{n \rightarrow + \infty} \pc(S_n = k) =  \frac {k \tau^{k-1}}{\theta'(\tau)} [z^k] \theta(z),
\end{equation}
où la série $\theta$ est définie par \eqref{eulphi} (le nombre $[z^k] \theta(z)$ désigne son $k$-ième coefficient).
\end{prop}
\begin{proof} Nous avons
\begin{equation} \label{EcCn}
\mathbb E_c (C_n - 1) = \sum_{(C,F)} \pare{\cc(F) - 1} \frac{u^{\cc(F) - 1}}{f_n(u)} = u \frac{f_n'(u)}{f_n(u)} = u \frac{[z^{n-1}] \pd {^2 F}{z \partial u} } {[z^{n-1}] \pd F z }.
\end{equation}
Or \eqref{neulR} implique que $\pd R u = \phi(R) + u \pd R u \phi'(R)$. Donc en utilisant \eqref{neulF}, nous obtenons l'identité
$$ \pd {^2 F}{z \partial u}  = \frac{\phi(R)\theta'(R)}{1-u \phi'(R)}.
$$
Nous procédons maintenant à l'analyse de singularités. D'après le théorème \ref{theorayon}, les fonctions $\phi$ et $\theta$ sont analytiques en $\tau = R(\rho)$, où $\tau$ est l'unique élément de $]0,\rho[$ satisfaisait $u \phi'(\tau) = 1$. Par un développement de Taylor de $\phi'$ en $\tau$, nous obtenons
$$\pd {^2 F}{z \partial u} \sim \frac{\phi(\tau)\theta'(\tau)}{u\phi''(\tau) (\tau - R)} \sim \frac{\phi(\tau)\theta'(\tau)}{u\phi''(\tau) c \sqrt{1 - z/\rho}},
$$
où la dernière équivalence provient de  \eqref{per}, avec $c = \sqrt{\frac{ 2\rho} {u \phi''(\tau)}}$. L'équivalence asymptotique de $\mathbb E_c (C_n)$ résulte alors du théorème de transfert (nous rappelons que \eqref{pequif} donne un équivalent de $f_n(u)$ et que $\rho = \tau - u \phi(\tau)$).

Pour étudier $S_n$, nous ajoutons une variable $x$ à la série génératrice $F(z,u)$ qui compte la moitié du nombre des demi-arêtes incidentes à la composante racine (qui est forcément pair). La relation \eqref{femme} p. \pageref{femme} devient alors
$$F(z,u,t)= \overline M( z ; 0, 0 , 0, g_4(1), 0,g_6(1), \dots; 0, 0 , 0, x^2 h_4(1/u), 0, x^3 h_6(1)/u, \dots).$$
(On rappelle que $g_2(t) = g_{2k+1}(t) = h_2(t) = h_{2k+1}(t) = 0$ avec $k \geq 0$ pour les cartes $4$-eulériennes.) Grâce à la proposition \ref{prop:mp} p. \pageref{prop:mp}, la relation \eqref{neulF} devient
$$\pd F z (z,u,x) = \theta(x R).$$
(La série $R$ n'a pas changé, elle ne comporte toujours pas de variable $x$.) Nous pouvons exprimer $\pc(S_n=k)$ en termes de $\pd F z$ :
$$\pc (S_n = k ) = \frac{[z^{n-1}x^k]\pd F z(z,u,x)}{[z^{n-1}]\pd F z(z,u,1)} = \frac{[z^{n-1}x^k]\theta(xR)}{[z^{n-1}]\theta(R)}.
$$
Nous appliquons maintenant la proposition IX.1 de \cite[p. 629]{flajolet-sedgewick} : l'expression \eqref{pcsn} en résulte.
\end{proof}

\subsection{Nombre d'arêtes internes actives dans les cartes munies d'un arbre couvrant}

Soient $n \in \N$ et $u \in [-1,+\infty[$. Considérons la distribution sur les cartes  $C$   avec $n$ faces munies d'un arbre couvrant $T$ (et un poids $d_{2k}$ par sommet de degré $2k$) de probabilité
$$\mathbb P_i (C,T) = \frac{ (u+1)^{\inte(C,T)} }{f_n(u)},$$
où $\inte(C,T)$ est le nombre d'arêtes internes actives dans $(C,T)$. L'équation \eqref{Fintact} montre que cela est bien une distribution de probabilité. On note $I_n$ la variable aléatoire qui compte les arêtes internes actives sous cette distribution. Comme nous l'avons montré à la sous-section \ref{ss:modele}, $I_n$  décrit également le niveau d'une configuration récurrente du tas de sable pour une carte aléatoire duale.

\begin{prop}Soit $(d_{2k})$ une suite de poids qui satisfait les hypothèses du théorème \ref{theobis} (hypothèses dites fortes).  Le nombre moyen d'arêtes internes actives connaît une  transition de phase (très légère) en $u = 0$. En effet, quand $n$ tend vers $+ \infty$,
nous avons 
$$\mathbb E_i (I_n) \sim \kappa_u \, n,$$
avec
$$\kappa_u = \frac{(1+u)\,\phi(\tau)}{\tau - u \phi(\tau)}$$
où $\phi$ est définie par \eqref{eulphi} et $\tau$ est définie  dans le théorème \ref{theorayon}. La fonction $\kappa_u$ est analytique en la variable $u$ sur l'intervalle $[-1,+\infty[$ sauf en $u=0$, où elle est infiniment dérivable : quand $u \rightarrow 0$,
$$\kappa_u = \frac{(1+u)\phi(\kappa/4)}{\kappa/4 - u \phi(\kappa/4)} + O\pare{\exp\pare{- \frac c u}}, $$
où $\kappa/4$ est le rayon de convergence de $\phi$.
\end{prop}
\begin{proof}
Nous avons
$$\mathbb E_i(I_n) = \sum_{(C,T)} \inte(C,T) \frac{ (u+1)^{\inte(C,T)} }{f_n(u)} = (u+1) \frac{f'_n(u)}{f_n(u)} = (u+1)\frac{[z^{n-1}] \pd {^2 F}{z \partial u} } {n f_n(u) }.$$
En comparant avec \eqref{EcCn}, nous remarquons que pour $u > 0$, $$\mathbb E_i(I_n) = (1 + 1/u) \mathbb E_c(C_n).$$ Ainsi l'équivalence $\mathbb E_i(I_n) \sim \kappa_u n$ quand $u > 0$ provient de la proposition \ref{p:CnSn}. De plus, le développement asymptotique de $\kappa_u$ autour de $0^+$ est déduit du développement \eqref{exprho} de $\tau$. 

Considérons maintenant $u \in [-1,0[$. La série $\pd {^2 F}{z \partial u}$ vaut 
$$ \pd {^2 F}{z \partial u} = \frac {\phi(R)\theta'(R)}{1-u \phi'(R)} = \phi(R) \pdd 2 F z = \frac{R - z} u \pdd 2 F z.$$
En injectant les développements  \eqref{ner} et \eqref{nef} de $R$ et $F$ et en utilisant le théorème de transfert, nous trouvons
$$[z^{n-1}] \pd {^2 F}{z \partial u} \sim \frac{\kappa/4 - \rho} u n^2 f_{n+1}(u).$$
D'après le théorème \ref{theobis}, nous déduisons $\mathbb E_i(I_n) \sim (1+u) (\kappa/4 - \rho) n /\rho$. Comme $\rho$ vaut $ \kappa/4 - u \phi(\kappa/4)$, nous obtenons le résultat annoncé pour $u < 0$.

Il nous reste le cas $u=0$. Nous avons $R = z$. C'est pourquoi 
$n f_n(u) \sim [z^{n-1}]\theta(z)$ et $\pd {^2 F}{z \partial u}(z,0) = \phi(z) \theta'(z)$. Grâce au lemme \ref{l:thetaphi}, nous pouvons appliquer le théorème de transfert, et ainsi trouver le résultat  asymptotique annoncé pour $u=0$.
\end{proof}

\chapter{Comportement asymptotique des cartes forestières régulières eulériennes }
\label{c:2qreg}

Le chapitre précédent étudiait le comportement asymptotique des cartes forestières $4$-eulériennes sous condition d'apériodicité. Peut-on prouver le même phénomène de transition de phase en $u=0$ avec les mêmes régimes asymptotiques si on enlève cette hypothèse d'apériodicité ? Nous allons répondre à cette question par la positive dans le cadre particulier des cartes régulières.
La méthode globale reste la même, mais à cause de la périodicité l'étude sera centrée sur $R^p$ plutôt que sur $R$, où $p$ est un entier supérieur à $2$. Les équations n'auront pas tout à fait la même forme et le difficile théorème de $(u+1)$-positivité  de la page \pageref{Pestpositif} joue un rôle important ici.

Comme dans le chapitre précédent, nous substituons la variable des arêtes $t$ par $1$ (c'est d'autant plus justifié ici qu'il y a une redondance entre la variable des arêtes et  celle des faces) et la variable $u$ qui compte les composantes non racine de la forêt couvrante est un réel fixé supérieur ou égal à $-1$. 

Comme le montre le théorème suivant, nous observons également une transition de phase en $u=0$ pour les cartes forestières régulières eulériennes avec les mêmes régimes asymptotiques, confirmant le caractère universel du comportement observé dans le chapitre précédent.

\begin{theo} Soient $u \geq -1$ et $p$ un entier naturel non nul. Appelons $f_n(u)$ le $n$-ième coefficient de la série génératrice $F(z,u)$ des cartes forestières $2(p+1)$-régulières comptées selon le nombre de faces ($z$) et de composantes non racine ($u$).

Si $m$ n'est pas de la forme $m=pn+2$, alors $f_m(u) = 0$. Sinon il existe deux constantes positives $c_u$ et $\rho_u$ telles que
$$
f_{pn+2}(u) \sim \left\{ \begin{array}{ll}
 c_u \rho_u^{-n}n^{-3} (\ln n)^{-2} & \textrm{si }u \in [-1,0[, \\
c_u \rho_u^{-n}n^{-3} & \textrm{si }u=0, \\
c_u \rho_u^{-n}n^{-5/2} & \textrm{si }u>0.
\end{array} \right.
$$
\end{theo}

La démonstration de ce théorème sera répartie sur tout ce chapitre.

\noindent \textbf{Remarque 1. }Le cas $u=0$ se démontre assez simplement dans ce cas précis. En effet, la relation \eqref{Fz0reg} p. \pageref{Fz0reg} s'écrit ici 
$$F(z,0) = \sum_{n \geq 1} \frac{2(p+1)((2p+1)n)!} {(n-1)! (1 + pn)! (2 + pn)!} z^{2 + pn}.$$
Une simple application de la formule de Stirling montre que
$$f_{pn+2}(0) \sim \frac{p+1}{\pi p^4} \sqrt{2p+1} \pare{\frac {\pare{2p+1}^{2p+1}} {p^{2p}}}^n n^{-3}.$$

\noindent \textbf{Remarque 2. }Comme dans le cas apériodique, le comportement de $f_n(u)$ pour $u$ négatif en $\rho_u^{-n}n^{-3} (\ln n)^{-2}$ est incompatible avec  l'holonomie de $f_n(u)$, d'où le corollaire suivant.

\begin{cor} Pour entier naturel $p$ non nul, la série génératrice $F(z,u)$ des cartes forestières $2(p+1)$-régulières n'est pas holonome.
\end{cor}

\section{Transformations des équations}
Soit $p \geq 1$. D'après la proposition \ref{regulierecentral} p. \pageref{regulierecentral}, la série génératrice $F$ des cartes forestières $2(p+1)$-régulières satisfait le système
$$\pd F z (z,u) = \theta(R),\quad R = z + u \phi(R),$$
où les séries $\theta$ et $\phi$ sont définies par  :
$$\theta(x) = 2 (p+1) \,\sum_{i \geq 1}  \frac{((2p+1)i)!}{ (i-1)! (pi+1)!^2 } x^{pi+1}$$ et \begin{equation}
\phi(x) = \sum_{i \geq 1} \frac{((2p+1)i)!}{ i! (pi)! (pi+1)!  }  x^{pi+1}. \label{phi2q}
\end{equation}

Comme nous pouvions le prévoir, nous voyons apparaître une période $p$ dans chacune des séries. La solution la plus pratique pour contourner ce problème est de considérer un nouveau jeu de variables $v = x^p$ et $\zeta = z^p$ (lorsque cela est possible) afin de rendre les séries apériodiques. 

Remarquons que si $m$ n'est pas de la forme $pn+1$, avec $n \geq 0$, alors le $m$-ième coefficient de $R$ est nul. Cela peut se prouver par induction grâce à la relation $R = z + u \phi(R)$. Cela implique en outre que $R^p$ est une série formelle apériodique en la variable $\zeta = z^p$  (contrairement à $R$). 
%En effet, si $r_n(u)$ désigne le $(pn+1)$-ième coefficient en $z$ de $R$, alors
%$$R^p(z) = \left(\sum_{n \geq 0} r_n(u) z^{pn+1}\right)^p = \sum_{n_1,\dots,n_p \geq 0} \left( \prod_{i=1}^p r_{n_i}(u) \right) z^{p\left(1 + \sum_{i=1}^p n_i\right)}.$$
On définit alors la série $P$ comme 
\begin{equation} \label{defP} P (\zeta) = R^p(\zeta^{\frac 1 p}).\end{equation}
On peut alors poser
\begin{equation} \label{defpsi}
\psi(v) = \sum_{i \geq 1} \frac{((2p+1)i)!}{ i! (pi)! (pi+1)!  }  v^{i} = v^{-1/p} \phi\pare{v^{1/p}}
\end{equation}
de sorte que 
\begin{equation} \label{rpsip}
R = z + u R\,\psi(P(z^p)).
\end{equation}
En élevant \eqref{rpsip} à la puissance $p$ après avoir isolé $z$, nous obtenons une équation satisfaite par la série $P$ seule :
\begin{equation} \label{ep}
P (1 - u \psi(P))^p  = \zeta.
\end{equation}

Pour résumer, nous nous sommes ramenés à une série $P$ apériodique qui satisfait une équation fonctionnelle légèrement plus complexe que celle du chapitre précédent. Néanmoins l'approche employée sera similaire.

\section{Série  $\boldsymbol \psi$ et cons\oe{}urs}

Cette section étudie la série $\psi$ et autres séries apparentées en les traduisant en termes de séries hypergéométriques.

Rappelons que $\psi$ est définie par \eqref{defpsi}. Grâce à la formule de Stirling, nous pouvons voir que les coefficients de cette série est asymptotiquement équivalent à $\lambda^{-i}/i^2$ à une constante multiplicative près, où
\begin{equation} \label{defkappa}
\lambda = \frac {p^{2p}}{(2p+1)^{2p+1}}.
\end{equation}
Par conséquent la série $\psi$ admet $\lambda$ comme rayon de convergence, où elle converge. Par contre, sa dérivée $\psi'$ diverge en $\lambda$.

% L'étude du comportement singulier de $\psi$ est bien plus complexe que celle des séries du cas tétravalent. 
Appelons 
 $_{m+1}F_m \left(a_1,a_2,\dots,a_{m+1}; b_1,b_2,\dots,b_{m}; t \right)$  la série hypergéométrique de paramètres $(a_i)_{i = 1,\dots,m+1}$ et $(b_i)_{i = 1,\dots,m}$, à savoir
$$ _{m+1}F_m \left(a_1,a_2,\dots,a_{m+1}; b_1,b_2,\dots,b_{m+1}; t \right) = \sum_{n \geq 0} \frac  {\prod_{i = 1}^{m+1} (a_i)_n} {\prod_{i = 1}^{m} (b_i)_n} \frac{t^n} {n!},$$
où $(a)_n$ désigne la factorielle ascendante $a\,(a+1)\dots(a+n-1)$. Le comportement singulier de telles fonctions a été étudiée par Bühring dans \cite{buhring}.
 Observons que\footnote{On peut même voir $\psi(v)$ comme une série hypergéométrique de la forme $_{2p}F_{2p-1}$ en simplifiant par $(1)_n$ au numérateur et au dénominateur, mais cette simplification n'est pas nécessaire, elle complique au contraire les calculs.}
$$\psi(v) =\  _{2p+1}\hspace{-1pt}F_{2p} \left(\frac 1 {2p+1},\frac 2 {2p+1},\dots,1; \frac 1 p, \frac 2 p, \frac 2 p,  \frac 3 p,  \frac 3 p, \dots, 1,1,1 + \frac 1  p ; \frac v {\lambda}\right) - 1.$$
La série $\psi(v)$ est donc analytiquement prolongeable sur le plan complexe privé de la demi droite $[\lambda,+ \infty[$ et son comportement en $v = \lambda$ se traduit grâce à \cite[Théorème 3]{buhring} par
$$
\psi(\lambda - \varepsilon) \underset{\varepsilon \rightarrow 0}= \psi(\lambda) + \frac {\prod_{i = 1}^{p} \Gamma(i/p) \prod_{i = 2}^{p+1} \Gamma(i/p)} {\prod_{i = 1}^{2p+1} \Gamma(i/(2p+1))}  \varepsilon \ln \varepsilon + O(\varepsilon).
$$
Or la formule dite de multiplication concernant la fonction $\Gamma$ indique que
$$\Gamma\left(z\right)\Gamma\left(z + \frac 1 m\right) \cdots \Gamma\left(z + \frac {m-1} m\right) = (2 \pi)^{(m-1)/2} \, m^{1/2 - mz} \, \Gamma(mz).$$
Quelques calculs donnent alors le développement :
\begin{equation} \label{devpsi}
\psi(\lambda - \varepsilon) \underset{\varepsilon \rightarrow 0}= \psi(\lambda) + \frac 1 { 2 \pi} \frac {\sqrt{2p+1}} {p^2} \varepsilon \ln \varepsilon + O(\varepsilon).
\end{equation}
\textbf{Remarque.} \`A notre connaissance, il n'existe pas d'expression simple pour $\psi(\lambda)$ (nous savons seulement qu'il est fini et positif), mis à part pour le cas tétravalent.

Nous aurons également besoin de connaître le comportement singulier en $\lambda$ des séries
\begin{equation} \label{stars}
\theta^*(v) = 2(p+1)\,\sum_{i \geq 1}  \frac{((2q-1)i)!}{ (i-1)! (pi)! (pi+1)! } v^{i}, \quad \psi^*(v) = \sum_{i \geq 1} \frac{((2p+1)i)!}{ i! (pi)!^2  }  v^{i}.
\end{equation}
De manière similaire, nous obtenons les approximations (toujours d'après \cite{buhring})
\begin{equation} \label{st1}
\theta^*(\lambda-\varepsilon)  \underset{\varepsilon \rightarrow 0}= - \frac {(p+1) \sqrt{2p+1}} {\pi p^2} \ln \varepsilon + A + O(\varepsilon \ln \varepsilon),
\end{equation}
\begin{equation} \label{st2}
\psi^*(\lambda-\varepsilon) \underset{\varepsilon \rightarrow 0}= - \frac {\sqrt{2p+1}} {2 \pi p} \ln \varepsilon + B + O(\varepsilon \ln \varepsilon),
\end{equation}
où $A$ et $B$ sont deux réels également sans expression simple connue.

\section{Quand $\boldsymbol{ u > 0 }$}

\begin{prop} Fixons $u > 0$. La série $P$ définie par \eqref{defP} est analytique sur un $\Delta$-domaine de rayon $\rho$, où $\rho$ est le rayon de convergence de $P$. Le développement singulier de $P$ en $\rho$ est égal à
\begin{equation} \label{qdevr}
P(\zeta)  \underset{z \rightarrow \rho}= \tau - \gamma \sqrt{1 - \frac{\zeta} \rho } + O\left(1 - \frac{\zeta} \rho \right)
\end{equation}
où  $\tau$ et $\gamma$ sont deux réels positifs.

La série $F^*(\zeta) = {\zeta^{- \frac 1 p} } \pd F z (\zeta^{\frac 1 p})$ est elle aussi analytique sur un $\Delta$-domaine de rayon $\rho$ avec une singularité de type racine.
Nous déduisons par le théorème de transfert l'équivalent asymptotique
\begin{equation} \label{fpn}
f_{pn+2}(u) \sim \frac 1 {2 p \sqrt \pi} \,\gamma'\, \rho^{-n} \, n^{-\frac 5 2},
\end{equation}
où $\gamma'$ est une constante positive.
\end{prop}

\begin{proof}
Nous commençons par la série $P$. On rappelle que cette série satisfait l'équation \eqref{ep}. Nous souhaitons appliquer le théorème VI.6 de \cite[p. 404-405]{flajolet-sedgewick} à la fonction 
$$ \omega : v \mapsto (1 - u \psi(v))^{-p}.$$
Elle satisfait trivialement la condition $(\boldsymbol{H_1})$. Quant à la condition  $(\boldsymbol{H_2})$, il faut vérifier qu'il existe un unique $\tau$ réel positif dans le disque de convergence ouvert  de $\omega$ tel que $\omega(\tau)-\tau\omega'(\tau)=0$. Soit $\lambda'$ le rayon de convergence de la série $\omega$. S'il existe $t \in \, ]0,\lambda[$ tel que $1 - u \psi(t) = 0$ (dans ce cas $t$ est unique car $\psi$ est strictement croissante), alors $\lambda'=t$. Sinon, $\lambda'=\lambda$.

En outre, un calcul rapide montre que 
$$\omega(v)-v\omega'(v)= \left(1 - u\,(\psi(v) + p\,v\,\psi'(v))\, \right)
(1-u\,\psi(v))^{-p-1}.$$
Posons $g(v) = \psi(v) + p\,v\,\psi'(v)$. Le $i$-ième coefficient de $g(v)$, à savoir $\frac{((2p+1)i)!}{ i! (pi)! (pi)!}$, est  équivalent à $\lambda^{-i}/i$. La série $g$ diverge donc en $\lambda$. Comme cette série est strictement croissante, il existe un unique $\tau \in \  ]0,\lambda[$ tel que $g(\tau) = 1/u$. Mais $g > \psi$ sur l'intervalle $]0,\lambda[$, d'où l'inégalité $1 - u\psi(\tau) > 0$. Par conséquent $\tau$ est  strictement inférieur à $\lambda'$ et $\omega(\tau)-\tau\omega'(\tau)=0$. La condition $(\boldsymbol{H_2})$ est bien vérifiée, on peut appliquer le théorème VI.6 de \cite{flajolet-sedgewick}.

Le réel $\rho = \tau/\omega(\tau) = \tau (1 - u \psi(\tau))^p $ est alors le rayon de convergence de $P(\zeta)$ et le développement asymptotique de $P(\zeta)$ au voisinage de $\rho$ est bien décrit par \eqref{qdevr} avec 
$$ \gamma = \sqrt 2\frac{1-u \psi(\tau)} {\sqrt{up( (p+1)u\psi'(\tau)^2 + (1- u \psi(\tau)) \psi''(\tau))} }.$$ De plus, comme $\omega$ est apériodique, la fonction $P(\zeta)$ est analytique sur un $\Delta$-domaine\footnote{Ce n'est pas écrit explicitement dans l'énoncé du théorème VI.6 mais l'analyticité sur un $\Delta$-domaine est mentionnée p. 406.} de rayon $\rho$. 

Continuons avec $\pd F z = \theta(R)$. Remarquons dans un premier temps que \eqref{rpsip} permet d'exprimer $R$ en fonction de $P$. Par conséquent, l'égalité  $\pd F z = \theta(R)$ se réécrit sous la forme $F^*(\zeta) = h(P(\zeta))$ où 
\begin{equation} \label{fstar}  \pd F z (z)  = z F^*(z^p)
\end{equation}
et \begin{equation} \label{fache} h(v) =   {v^{-\frac 1 p} }\,\theta(v^{\frac 1 p})\,(1 - u \psi(v))^{-p}.
\end{equation}
D'une part, la série $h$ est analytique sur un disque ouvert de centre $0$ et de rayon $\lambda'$. D'autre part, il existe un $\Delta$-domaine de rayon $\rho$ dans lequel $P$ est analytique et strictement majoré (en module) par $\lambda'$ car $P$ est à coefficients positifs et $P(\rho) = \tau < \lambda'$. Par conséquent,  $F^* = h(P)$ est analytique sur le même $\Delta$-domaine et un développement de Taylor de $h$ en $\tau$ donne le comportement singulier de $F^*$ en $\rho$ : 
$$F^*(\zeta) = h(\tau) - \gamma \,h'(\tau)  \sqrt{1 - \frac{\zeta} \rho } + O\left(1 - \frac{\zeta} \rho \right).$$
On peut alors appliquer le théorème de transfert à $F^*$ \cite[p. 393]{flajolet-sedgewick} pour obtenir le comportement asymptotique de son $n$-ième coefficient, qui est $(pn+2)f_{pn+2}$. L'équivalent obtenu est donné par \eqref{fpn} avec $\gamma' = \gamma\, h'(\tau)$. \end{proof}

\section{Quand $\boldsymbol{ u < 0 }$}

\begin{prop} Fixons $ u \in \, [-1,0[$. La série $P$ définie par \eqref{defP} a pour rayon de convergence \mbox{$\rho = \lambda(1 - u \psi(\lambda))^p$}, où $\psi$ est définie par \eqref{defpsi} et  $\lambda$  par \eqref{defkappa}. Elle est analytique sur un $\Delta$-domaine de rayon $\rho$ et son développement singulier est 
\begin{equation} \label{Pneg}
P(\zeta) - \lambda \underset{\zeta \rightarrow \rho }\sim -\frac{2\, \pi \, p \,\rho}{\lambda \sqrt{2p+1} \,  u} \frac{1 - \zeta/\rho}{\ln(1 - \zeta/\rho)}. 
\end{equation}
De même, la série $F^{**}(\zeta) = \pdd 2 F z (\zeta^{1/p})$ a pour rayon de convergence $\rho$, elle est analytique sur un  $\Delta$-domaine de rayon $\rho$ et elle satisfait localement
\begin{equation} \label{devattendu}
F^{**}(\zeta) +  \frac {2(p+1)}{pu} \underset{\zeta \rightarrow \rho }\sim \frac { 4 \pi (p+1) \left( 1 - u \psi(\lambda) \right)}  {\sqrt{2p+1} u^2}  \frac 1 {\ln \left(1 - \zeta/\rho\right)}.
\end{equation}
Nous en déduisons par le théorème de transfert l'équivalent asymptotique
\begin{equation} \label{devfinal}
f_{pn+2}(u) \sim \frac { 4 \pi (p+1) \left( 1 - u \psi(\lambda) \right)}  { \sqrt{2p+1} p^2 u^2}  \frac{ \rho^{-n}} {n^{3} \ln^2 n }.
\end{equation}
\end{prop}
\begin{proof}
Nous commençons encore une fois par l'étude de la série $P(\zeta)$. L'équation \eqref{ep} se lit \mbox{$\Omega(P) = \zeta$}, où $\Omega(v) = v (1 - u \psi(v))^p$. On voit que $\Omega(0) = 0$ et $\Omega'(0) = 1 > 0$ de sorte que le corollaire \ref{c:fi} p. \pageref{c:fi} peut être appliqué, où $\zeta$ joue le rôle de $z$ et $P$ celui de $Y$.
 Soient $\omega$, $\tau$ et $\rho$ définis dans le corollaire. Le réel $\omega$ est le rayon de convergence de $\psi$, désigné par $\lambda$ et donné par \eqref{defkappa}. De plus, comme $-u > 0$, la fonction $\Omega$ est clairement strictement croissante sur $[0,\lambda[$. Donc $\Omega'$ reste strictement positive sur cet intervalle. On a alors $\tau = \lambda$. De plus, le point $(5)$ du corollaire montre que 
$$\rho = \Omega(\lambda) = \lambda(1 - u \psi(\lambda))^p.$$
Toujours d'après le corollaire \ref{c:fi}, la fonction $P$ admet un prolongement analytique  le long de l'intervalle $[0,\rho[$, et elle y est croissante. De plus, l'équation \mbox{$P (1 - u \psi(P))^p = \zeta$} est vraie sur cet intervalle.

D'après le théorème \ref{Pestpositif} p. \pageref{Pestpositif}, les coefficients de la série $(R^p-z^p)/u$ sont positifs (la série $R$  du théorème \ref{Pestpositif} avec $q = p + 1$ est bien égale à la série $R$ de ce chapitre  -- on peut notamment comparer \eqref{uren} et \eqref{phi2q} pour s'en convaincre). D'après la relation~\eqref{defP}, la série $(P-\zeta)/u$ admet également des \mbox{coef\-ficients} positifs. Posons $$\mathcal P(\zeta) = \frac {P(\zeta)-\zeta} u.$$ Comme $P$, la série $\mathcal P$ admet un prolongement analytique sur $[0,\rho[$. Donc d'après le théorème de Pringsheim, le rayon de convergence de $\mathcal P$, qui est le même que celui de $P$, est au moins égal à $\rho$.

\textbf{\`{A} partir de ce point, nous reprenons les points de démonstration du chapitre précédent. De nombreux passages ont été recopiés (mis à part à la fin de la démonstration).}

Nous nous intéressons maintenant  au comportement de $P$ au voisinage de $\rho$. S'il est singulier, nous aurons montré que $\rho$ est effectivement le rayon de convergence de $P$. L'approximation \eqref{devpsi} de $\psi$ en $\lambda$ nous indique que
$$
\Omega(\lambda - \varepsilon) = 
%(\lambda-\varepsilon) \left( 1 - u \psi(\lambda - \varepsilon) \right)^p  = 
(\lambda-\varepsilon) \left( 1 - u \psi(\lambda) - \frac u { 2 \pi} \frac {\sqrt{2p+1}} {p^2} \varepsilon \ln \varepsilon + O(\varepsilon) \right)^p 
 = \rho - \frac {u \lambda \sqrt{2p+1}} {2 \pi p} \varepsilon \ln \varepsilon + O(\varepsilon).
$$
On peut donc appliquer le théorème de log-inversion \ref{loginversion} p. \pageref{loginversion} avec $\psi(z) = \rho - \Omega(\lambda - z)$,
$c = - (u \lambda \sqrt{2p+1})/ (2 \pi p) > 0$, $\alpha = \pi$
et $\beta = 3 \pi / 4$ : il existe deux réels strictement positifs $r$ et $r'$ et une fonction $\Upsilon$ analytique sur le domaine \mbox{$D_{r',\alpha} = \ens{|z| < r'\textrm{ et } |\Arg z| < 3 \pi / 4 }$} telle que 
$$\rho - \Omega(\lambda - \Upsilon(y)) = y.$$
De plus, $\Upsilon(y)$ est le seul antécédent de $y$ par la fonction $z \mapsto\rho - \Omega (\lambda - z)$ qui se situe dans \mbox{$D_{r,\pi} = \ens{|z| < r\textrm{ et } |\Arg z| < \pi }$}. Or l'équation \eqref{ep} se lit $$\rho - \Omega(\lambda - (\lambda - P(\rho - y))) = y$$ et est vérifiée pour les petites valeurs positives de $y$. Donc pour ces mêmes valeurs de $y$, nous avons $\Upsilon(y) = \lambda - P(\rho - y)$. Si on réécrit cette égalité selon la variable $\zeta$, on obtient 
\begin{equation}\label{Pupsilon}
P(\zeta) = \lambda - \Upsilon(\rho - \zeta),\end{equation}
qui est vraie lorsque $\zeta$ est réel et proche de $\rho^-$. Ainsi $P$ est analytiquement prolongeable  sur $\ens{ |\zeta - \rho| < r'\textrm{ et } |\Arg (\zeta-\rho)| > \pi / 4}$. La conclusion du \ref{loginversion} nous donne le développement asymptotique \eqref{Pneg}.
Par ailleurs, nous avons prouvé que $\rho$ est le rayon de convergence de $P$.

Nous voulons maintenant prouver que $R$ est analytique sur un $\Delta$-domaine de rayon $\rho$. Pour cela, il nous reste à prouver que $P$ n'a pas de singularité autre que $\rho$ sur le cercle de convergence. Considérons $\mu$ tel que $|\mu| = \rho$ et $\mu \neq \rho$. La série \mbox{$\mathcal P = (P - \zeta)/u$} converge en $\mu$ car $\mathcal P$ a des coefficients positifs et $\mathcal P(\rho)$ est fini. Par conséquent, $P$ converge également en $\mu$. 

Prolongeons l'égalité $\Omega(P) = \zeta$ sur le disque de convergence afin d'appliquer le théorème des fonctions implicites en $\mu$. La série $\Omega$ est analytique sur le plan complexe privé de la demi-droite $[\rho,+\infty[$. Donc si on prouve le lemme suivant, la fonction $\Omega(P)$ sera définie et analytique sur le disque de convergence et nous pourrons effectivement prolonger l'égalité précédente. 

\begin{lem} \label{Pevite}
Pour tout $\zeta \in\, \ens{ |\zeta| \leq \rho}$ différent de $\rho$, nous avons $P(\zeta) \notin [\lambda,+ \infty[$.
\end{lem}
\begin{proof}La propriété est vraie sur l'intervalle $[0,\rho[$ car $P$ est croissante sur cet intervalle et vaut $\lambda$ en $\rho$. Sur $[-\rho,0[$, la fonction $P$ est réelle (voir lemme \ref{l:valpos} p. \pageref{l:valpos}) et continue. De plus, $P(0) = $. Donc si $P([-\rho,0[)$ n'est pas inclus dans $[\lambda,+ \infty[$, alors il existe $t \in \, [-\rho,0[$ tel que $P(t) = \lambda$. Considérons $t$ maximal pour cette propriété. Ainsi, sur un voisinage complexe de l'intervalle $]t,0]$, la fonction $P$ est à valeurs dans $\C \backslash [\lambda,+\infty[$ de sorte que la relation $\Omega(P) = \zeta$ y est vraie. En particulier, nous pouvons la différentier :
$$P'(\zeta) = \frac 1 {\Omega'(P(\zeta))}.$$
Comme $\Omega'$ est analytique sur $]t,0]$, on a donc $P' \neq 0$ sur cet intervalle. En outre, \mbox{$P'(0) = 1$}. Or $P(t) = \lambda > P(0) = 0$. Donc $P'$ doit s'annuler sur $]t,0[$, ce qui contredit ce que nous avons dit ci-dessus. 

Maintenant supposons que $\zeta$ n'est pas réel. Nous allons prouver que $P(\zeta)$ n'est pas réel non plus. D'une part,
\begin{equation} \label{dfgh}
|\Ima P(\zeta)| = |\Ima (\zeta + u \mathcal P(\zeta))| \geq |\Ima \zeta| + u \, |\Ima \mathcal P(\zeta)|.
\end{equation}
D'autre part, comme $ \mathcal P (\Ree \zeta)$ est réel,
$$
|\Ima \mathcal P (\zeta)| = |\Ima (\mathcal P(\zeta) -  \mathcal P (\Ree \zeta))| \leq |\mathcal P(\zeta) -  \mathcal P (\Ree \zeta)|.
$$
Donc d'après l'inégalité des accroissements finis,
on a 
\begin{equation} \label{fghj}
|\Ima \mathcal P (\zeta)| < |\zeta - \Ree \zeta| \max_{y \in [\Ree \zeta, \zeta]} | \mathcal P'(y) | \leq | \Ima z | \max_{|y| \leq \rho} | \mathcal P'(y)|,
\end{equation}
où l'inégalité stricte provient du fait que $\mathcal P'$ n'est pas constante sur $[\Ree \zeta, \zeta]$. Mais $\mathcal P'$ est une série entière à coefficients positifs, donc pour $|y| \leq \rho$,
\begin{equation}\label{ghjk}
|\mathcal P'(y)| \leq \mathcal P'(\rho) = \frac 1 u(P'(\rho) - 1)
= \frac 1 u \left(\lim_{t \rightarrow \rho} \frac 1 {\Omega'(P(t))} - 1\right) = -\frac 1 u,
\end{equation}
car $\Omega'(y)$ diverge vers $+\infty$ quand $y$ tend vers $\lambda$.
(En effet, $\psi'$ apparaît dans la dérivée de $\Omega'$ et cette série diverge en $\lambda$.) En revenant à \eqref{fghj},  nous trouvons \mbox{$|\Ima \mathcal P(\zeta)| < -|\Ima \zeta|/u$}, ce qui, combiné à \eqref{dfgh}, donne $|\Ima P(\zeta)| > 0$.
\end{proof}

Ainsi l'égalité $\Omega(P) = \zeta$ est vraie sur tout le disque de rayon $\rho$. Par différentiation, on a également 
$P'\Omega'(P) = 1$. Revenons à notre point $\mu \neq \rho$ de module $\rho$. D'après \eqref{ghjk}, la série $\mathcal P'$ converge en  $\mu$. Donc $\Omega'(P(\mu)) = 1/P'(\mu) \neq 0$. De plus, on sait que $\Omega$ est analytique en $P(\mu)$.  On peut alors appliquer le théorème des fonctions implicites à l'équation $\Omega(P(\zeta))=\zeta$ au point $(\zeta,P(\zeta))$. Ainsi $P$ est analytique en $\mu$. 

En conclusion, nous avons prouvé que $P$ était analytique sur un $\Delta$-domaine de rayon $\rho$ qu'on notera $D$. Quitte à considérer un domaine $D$ plus petit, nous pouvons supposer que pour tout $\zeta \in D$, $P(\zeta)$ n'appartient pas à $[\lambda,+\infty[$. Cela provient d'une part du lemme \ref{Pevite} pour tout point de $D$ qui n'est pas dans un voisinage de $\rho$ et d'autre part de la relation \eqref{Pupsilon} pour tout point de $D$ au voisinage de $\rho$ (en effet la fonction $\Upsilon$ évite les réels négatifs).

\'Etudions maintenant la série $F$. Rappelons que $F^* = h(P)$ où $F^*$ et $h$ sont respectivement définis par \eqref{fstar} et \eqref{fache}. La fonction $h$ est analytique sur $\C$ privé de la demi-droite $[\lambda,+\infty[$ et des points $v$ tels que $\psi(v)=1/u$. Or, sur le domaine $D$, la fonction $P$ évite $[\lambda,+\infty[$. Par ailleurs, si $\psi(P(\zeta)) = 1/u$, alors $\zeta = P(\zeta) (1 - u  \psi(P(\zeta)))^p=0$. Mais $\psi(P(0)) = 0$, donc il n'y a pas de solution à $\psi(P(\zeta)) = 1/u$. Par conséquent, l'image de $P$ se trouve dans le domaine d'analyticité de $h$ : la fonction $F^* = h(P)$ est donc analytique sur le $\Delta$-domaine $D$.

Réécrivons \eqref{fstar} sous la forme $z F^*(z^p) = F'(z)$. On a alors $F''(z) = F^*(z^p)+ p z^p {F^*}'(z^p)$.
Posons $$F^{**} (\zeta) = F''(\zeta^{\frac 1 p}) = F^*(\zeta)+ p \, \zeta \, {F^*}'(\zeta).$$ 
Ainsi $F^{**}$ est également analytique sur $D$. De plus, en dérivant la relation $F' = \theta(R)$ et en remplaçant $R'$ par $1/(1-u\phi'(R))$ (car $(1 - u \phi'(R)) R' = 1$ puisque $R = z + u \phi(R)$), nous obtenons
$$F^{**}(\zeta) = \frac {\theta^*(P(\zeta))} {1 - u \psi^*(P(\zeta))},$$
où $\theta^*$ et $\psi^*$ sont définies par \eqref{stars}. Nous déduisons de \eqref{st1} et \eqref{st2} le développement asymptotique
\begin{multline} \label{prdevss}
\frac {\theta^*(\lambda-\varepsilon)} {1 - u \psi^*(\lambda - \varepsilon)} =  \frac{ -\frac {(p+1) \sqrt{2p+1}} {\pi p^2} \ln \epsilon + A + O(\varepsilon \ln \varepsilon)} { 1 + u
\frac {\sqrt{2p+1}} {2 \pi p} \ln \epsilon - u B + O(\varepsilon \ln \varepsilon)} \\ = \frac{-2(p+1)}{pu} + \frac { 2 \pi} { \sqrt{2p+1} u^2} \left( 2(p+1) - u ( 2(p+1)B - pA) \right) \frac 1 {\ln \varepsilon} + O\left(\frac 1 {\ln^2 \varepsilon} \right).
\end{multline}
Nous pouvons simplifier l'expression $2(p+1)B - pA$. En effet, nous remarquons que
$$2(p+1)\psi^*(v) - p \theta^*(v) = 2(p+1) \psi(v).$$
En passant à la limite, on trouve
$$2(p+1)B - pA = 2(p+1) \psi(\lambda).$$
Ainsi \eqref{prdevss} se réécrit 
$$
\frac {\theta^*(\lambda-\varepsilon)} {1 - u \psi^*(\lambda - \varepsilon)} = \frac{-2(p+1)}{pu} + \frac { 4 \pi (p+1)} { \sqrt{2p+1} u^2} \left( 1 - u \psi(\lambda) \right) \frac 1 {\ln \varepsilon} + O\left(\frac 1 {\ln^2 \varepsilon} \right).
$$
En substituant $\varepsilon$ par $\lambda - P(\zeta)$ et en utilisant \eqref{Pneg}, on trouve \eqref{devattendu}, le développement de $F^{**}$ attendu. Finalement, comme $F^{**}$ est analytique sur $D$, nous pouvons appliquer le théorème de transfert \cite[p. 393]{flajolet-sedgewick} pour obtenir un équivalent asymptotique du $n$-ième coefficient de $F^{**}$, qui est $(pn+2)(pn+1)f_{2+pn}(u)$. On obtient alors \eqref{devfinal}.
\end{proof}

\chapter{Comportement asymptotique des cartes forestières cubiques}
\label{c:cubique}

\section{Résultat principal}

Nous nous intéressons dans cette section au comportement singulier de la série $F(z,u)$ qui compte les cartes forestières cubiques selon le nombre de faces et le nombre de composantes, ainsi qu'au comportement asymptotique de ses coefficients. Cette étude s'inscrit logiquement dans la continuité des chapitres précédents : si on enlève la contrainte sur le caractère eulérien des cartes, le cas cubique constitue un des exemples les plus simples à étudier.

Les résultats trouvés sont qualitativement cohérents avec le cas eulérien, ce qui accrédite la thèse d'"universalité" du comportement asymptotique. Toutefois, quand $u$ est négatif, ils sont moins complets que dans les chapitres précédents, à cause des difficultés apportées par l'apparition de la série $S$. En effet, comme indiqué dans la figure \ref{recapitulatif}, le système fonctionnel  satisfait par $F$ est maintenant
\begin{equation} \label{nsystem}
\pd F z = \theta(R,S), \quad R = z + u \phi_1(R,S), \quad S = u \phi_2(R,S),
\end{equation}
où
\begin{equation}
\theta(x,y) = 3 \sum_{i \geq 0} \sum_{\substack{j \geq 0 \\ 2i+j \geq 3}} \frac {(4i+2j-4)!} {i!^2j!(2i+j-3)!} x^i y^j,
\label{ctheta}
\end{equation}
\begin{equation} \label{cphi1}
\phi_1(x,y) =  \sum_{i \geq 0} \sum_{\substack{ j \geq 0 \\ 2i+j \geq 3}} \frac {(4i+2j-4)!} {i!(i-1)!j!(2i+j-2)!} x^i y^j,
\end{equation}
\begin{equation}
\phi_2(x,y) = \sum_{i \geq 0} \sum_{\substack{ j \geq 0 \\ 2i+j \geq 2}} \frac {(4i+2j-2)!} {i!^2j!(2i+j-1)!} x^i y^j.
\label{cphi2}
\end{equation}
Quand $u < 0$, le seul résultat que nous avons pu établir est le comportement de $\pd F z$ au voisinage de la singularité dominante sur l'axe réel. En particulier nous n'avons pas prouvé que $F$ est analytique sur un $\Delta$-domaine. Par conséquent, nous n'avons pas obtenu le régime asymptotique des coefficients de $F$ quand $u < 0$.

\begin{theo}\label{t:cubique}
Fixons $u \geq -1$. La rayon de convergence de la série génératrice $F(z,u)$ des cartes forestières cubiques  vaut 
$$\rho_u = \tau - u \phi_1(\tau,\sigma)$$
où le couple $(\tau,\sigma)$ satisfait
$
\sigma= u \phi_2(\tau, \sigma)
$ 
et
$$\left\{
\begin{array}{ll}
   \tau =  (1-4\sigma)^2 & \hbox{   si } u\le 0,
\\
 \left(1-u\pd {\phi_1} x (\tau, \sigma)\right)
\left(1-u\pd {\phi_2} y(\tau, \sigma)\right)
= u^2 \, \pd {\phi_1} y(\tau,\sigma)\pd {\phi_2} x(\tau, \sigma) & \hbox{  si } u>0,
\end{array}\right. 
$$
 sachant que les séries $\phi_1$ et $\phi_2$ sont données par \eqref{cphi1} et \eqref{cphi2}. 
 En particulier, $\rho_u$ est une fonction algébrique de $u$ sur $[-1,0]$:
\begin{equation}
\label{rho-3}
\rho_u= \frac{3(1-u^2)^2\pi^4+96u^2\pi^2(1-u^2)+512u^4+ 16u\sqrt2 \left( \pi^2(1-u^2)+8u^2\right)^{3/2}}{192\pi^4(1+u)^3}.
\end{equation}

  Notons $f_n(u)$  le coefficient en $z^n$ de $F(z,u)$. Il existe une constante strictement positive $c_u$ telle que 
$$
f_n(u) \sim \left\{\begin{array}{lll}
\displaystyle c_u\, {\rho_u^{-n}}{n^{-3} }  & \hbox{ si }  u = 0,\\
\displaystyle  c_u \,{\rho_u^{-n}}{n^{- 5 /2}}  & \hbox{ si }  u > 0.
\end{array}\right.
$$
Pour $u \in [-1,0]$, la série $\pd F z$  admet le développement singulier 
\begin{equation}
\label{exp-3}
\pd F z(z) \underset{ z\rightarrow \rho_u^-}= \pd F z (\rho_u) + \alpha (\rho_u-z) + 
\beta\,
\frac{  \rho_u-z}{\ln(\rho_u-z)}\left(1+ o(1)\right),
\end{equation}
où 
$$
\beta =\frac{4u-3\sqrt2 \sqrt{\pi^2(1-u^2)+8u^2}}{2u^2} <0.
$$
\end{theo}

\noindent \textbf{Remarque 1.} Comme dans le cas eulérien, le comportement singulier de $\pd F z$ pour des valeurs négatives de $u$ n'est pas compatible avec l'holonomie de la série.

\begin{cor} Pour $u \in \,[-1,0[$,  la série génératrice $F(z,u)$ des cartes forestières cubiques n'est pas holonome. Ceci est également vrai quand $u$ est une variable indéterminée.
\end{cor}

\noindent \textbf{Remarque 2.} Rappelons qu'il existe une formule close pour $f_n(0)$  donnée par \eqref{Fz0reg} p.~\pageref{Fz0reg}. Par conséquent, le comportement asymptotique décrit dans le précédent théorème pour $u=0$ n'est qu'une application de la formule de Stirling. De fait, la suite du chapitre ne traitera que des cas $u > 0$ et $u < 0$.

\noindent \textbf{Remarque 3.} Le rayon de convergence  de la série énumérant les cartes cubiques munies d'un arbre couvrant sans arêtes internes actives vaut $\rho_{-1}  = \pi^2/384$, un beau nombre transcendant.

\section{Les séries $\boldsymbol \phi_1$, $\boldsymbol \phi_2$, $\boldsymbol \psi_1$ et $\boldsymbol \psi_2$}

Rappelons qu'il existe une réécriture agréable des séries $\phi_1$ et $\phi_2$ en termes de séries monovariées (voir équations \eqref{phi1cubique} p. \pageref{phi1cubique}) :
\begin{equation*}
\phi_1(x,y) = (1-4y)^{3/2}\,\psi_1\left(\frac x {(1-4y)^2}\right) - x,
\end{equation*}
\begin{equation}
\phi_2(x,y) = \sqrt{1-4y}\,\psi_2\left(\frac x {(1-4y)^2}\right) + \frac 1 4 \, (1 - \sqrt{1 - 4y})^2, 
\label{nphienpsi}
\end{equation}
où $\psi_1$ et $\psi_2$ sont les deux séries hypergéométriques définies par \eqref{psicubique} p. \pageref{psicubique}. En termes des fonctions hypergéométriques standards $\:_2F_1$, la fonction $\psi_1$ s'écrit 
$$\psi_1(z) = z \, _ 2F_1\pare{\frac 1 4, \frac 3 4 ; 2 ; 64 z}.$$
Ainsi d'après \cite[\'Equation (15.3.11)]{AS},  la fonction $\psi_1$ est analytiquement prolongeable sur $\C \backslash [1/64,+\infty[$ et son comportement singulier en $1/64$ est
\begin{equation}
\psi_1\left(\frac{1} {64}-\varepsilon\right) = 
{\frac {\sqrt {2}}{24\, \pi }}
+  {\frac {\sqrt {2}  }{2\, \pi }}\, \varepsilon \,\ln   \varepsilon
-  {\frac {\sqrt {2}  }{2\, \pi }} \varepsilon
+ O \left( {\varepsilon}^{2} \ln \varepsilon \right).
\label{devpsi1}
\end{equation}
Quant à $\psi_2$, rappelons la relation \eqref{relpsicubique} p. \pageref{relpsicubique} 
$$(1-64z) \psi_1'(z)+ 4\psi_1(z)+ 2\psi_2(z)=1,$$
qui prouve que $\psi_2$ est également analytique sur $\C \backslash [1/64,+\infty[$.
En couplant cette identité avec le développement de $\psi_1$, nous obtenons
\begin{equation}
\psi_2\left(\frac{1} {64}-\varepsilon\right) \underset{\varepsilon \rightarrow 0}= 
\frac 1 2 -{\frac {\sqrt {2}}{\pi }}
+ {\frac {4\sqrt {2} }{ \pi }}\, \varepsilon\,\ln 
  \varepsilon 
+ {\frac {12\,\sqrt {2} \, }{\pi }}  \varepsilon
+ O \left( {\varepsilon}^{2} \ln \varepsilon \right).
\label{devpsi2}
\end{equation}

Revenons à $\phi_1$ et $\phi_2$. Le rayon de convergence des séries $\sqrt{1-4y}$ et $(1-4y)^{-2}$ est $1/4$ et celui des séries $\psi_1$ et $\psi_2$ est $1/64$. Ainsi d'après \eqref{nphienpsi}, les séries $\phi_1$ et $\phi_2$ convergent absolument pour $|y| < 1/4$ et $64|x| < (1-4|y|)^2$ (voir le graphique gauche de la figure \ref{absolucubique}). Nous savons également que les séries $\psi_1$ et $\psi_2$ sont analytiques sur $\C \,\backslash \,[1/64,+\infty[$, donc $\phi_1$ et $\phi_2$ se prolongent analytiquement sur 
$$ \enstq{(x,y)} {y \in \C \, \backslash \, [1/4,+\infty[ \textrm{ et }x/(1-4y)^2 \in \C \, \backslash \, [1/64,+\infty[}$$ (voir le graphique droit de la figure \ref{absolucubique}). Il peut exister un domaine d'analyticité plus grand, mais dans tous les cas $\phi_1$ et $\phi_2$ sont singulières sur chaque point de la \textit{parabole critique} $\enstq {(x,y)} {64x = (1 - 4y)^2}$ puisque $\psi_1'(t)$ et $\psi_2'(t)$ tendent vers $+\infty$ quand \mbox{$t \rightarrow 1/64$}.

\begin{figure} \begin{center}
 \begin{tabular}{c}\includegraphics[scale=0.4]{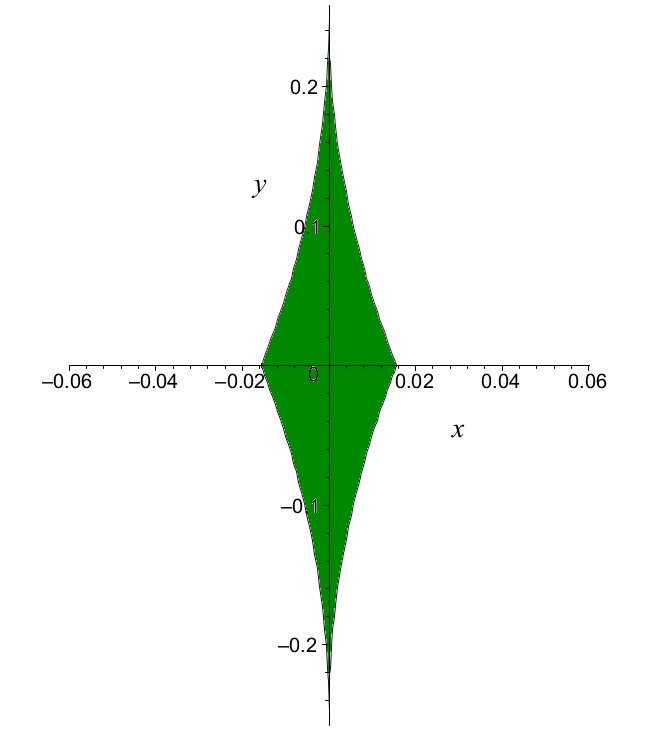}\end{tabular}
\hskip 10mm
 \begin{tabular}{c}\includegraphics[scale=0.4]{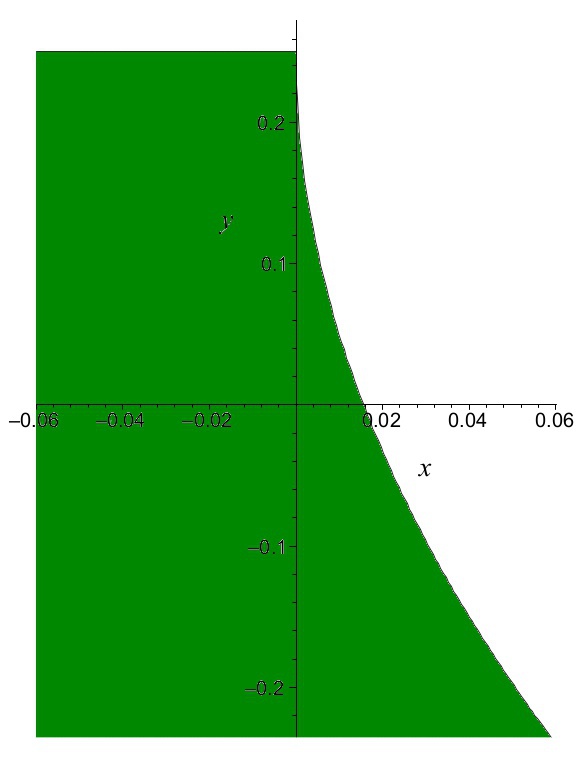}\end{tabular}
 \end{center}
 \caption{ \`A gauche : un domaine de convergence absolue dans le plan réel pour les séries $\phi_1$ et $\phi_2$. \`A droite : un domaine où les deux séries se prolongent analytiquement.}
 \label{absolucubique}
 \end{figure}

\section{Quand $\boldsymbol{u > 0}$}

\begin{prop}
Supposons $u > 0$. Les séries $R$, $S$ et $\pd F z$ ont le même rayon de convergence noté $\rho_u$, qui satisfait les conditions du théorème \ref{t:cubique}. Ces trois séries sont analytiques dans un $\Delta$-domaine de rayon $\rho_u$ et ont une singularité de type racine en $\rho_u$. En particulier, il existe une constante positive $c_u$ telle que
$$f_n(u) \sim c_u \,\rho_u^{-n}  \,n^{-5/2}.$$
\end{prop}

\begin{proof}Les séries $R$ et $S$ sont définies de manière couplée (voir le système \eqref{nsystem}). Ceci apporte une difficulté supplémentaire : l'analyse des systèmes fonctionnels de la forme $R_i = F_i(R_1,\dots,R_m)$ est la plupart du temps délicate, même quand les séries $F_i$ ont des coefficients positifs et le système ne fait intervenir que deux fonctions inconnues. En particulier, le rapport entre la valeur du rayon de convergence et les solutions de ce qu'on appelle \textit{système caractéristique} est subtil (voir \cite{drmota-systems,burris}). Toutefois  le système fonctionnel \eqref{nsystem} que nous devons étudier présente une particularité : l'équation qui définit $S$ (à savoir $S = u \phi_2(R,S)$) ne fait pas intervenir la variable $z$ explicitement, si bien que nous pouvons prudemment procéder en deux étapes grâce à une série intermédiaire. Cette série, que nous notons $\tilde S (z) = \tilde S (z,u)$, est définie comme l'unique série formelle en $z$ telle que $\tilde S(0,u) = 0$ et 
$$\tilde S = u \phi_2 \pare{ z,\tilde S}.$$
Nous allons d'abord étudier $\tilde S$, puis passer à la série $R$, qui est définie par
$$ R = z + u \phi_1 \pare{R,\tilde S (R)}. $$
Par unicité du couple $(R,S)$ (voir la proposition \ref{regulierecentral} p.~\pageref{regulierecentral}), cette série correspond bien à la série $R$ définie par \eqref{nsystem} et $S = \tilde S (R)$. 

Nous pouvons prouver que $\tilde S$ suit le schéma d'équation fonctionnelle lisse de \cite{flajolet-sedgewick}, mais il est plus rapide d'appliquer directement le théorème \ref{t:fi} p.~\pageref{t:fi} de ce mémoire, où $\tilde S$ joue le rôle de $Y$. La série $H(x,y) = y - u \phi_2(x,y)$ satisfait les hypothèses de ce théorème. Nous pouvons donc définir $\tilde \rho$ comme dans le théorème. Par positivité des coefficients de $\tilde S$, les points $\pare{z,\tilde S (z)}$ forment, lorsque $z$ va de $0$ à $\tilde \rho$, une courbe croissante dans le plan réel qui commence au point $(0,0)$.  D'après la condition (b), cette courbe  ne peut pas traverser la parabole critique $64x = (1 - 4y)^2$ car $ {\phi_2} $ y est singulière.
Cela élimine les possibilités $(i)$ et $(iv)$. De plus, quand $(x,y)$ approche la parabole, \mbox{$\pd H y = 1 - u \pd  {\phi_2} y$} tend vers $- \infty$. La condition $(d)$ est donc incompatible avec la condition $(iii)$. La courbe $\pare{z,\tilde S (z)}_{z \in [0,\rho]}$ se termine donc avant avoir atteint la parabole. En outre, seule la condition $(ii)$ peut être vraie : $\pd H y \pare{\tilde \rho,\tilde S \pare{\tilde \rho}} = 0$, soit de manière équivalente
\begin{equation}
1 = u \pd {\phi_2} y \pare{\tilde \rho, \tilde S \pare{\tilde \rho}}.
\label{cubcritique}
\end{equation}
(La limite inférieure de la condition $(ii)$ est ici une vraie limite à cause de la positivité des coefficients de $\phi_2$ et $\tilde S$.) D'après (a), le rayon de convergence de $\tilde S$ est au moins $\tilde \rho$. Or la relation $ \tilde S = u \phi_2 \pare{ z,\tilde S}$ une fois dérivée donne pour $z \in [0, \tilde \rho [$
\begin{equation}
\pd {\tilde S} z (z) = u \pd {\phi_2} x \pare{z,\tilde S(z)} \pare{ 1 - u \pd {\phi_2} y \pare{z,\tilde S(z)} }^{-1}.
\label{pdSz}
\end{equation}
En utilisant \eqref{cubcritique},  nous voyons que $\pd {\tilde S} z (z,u)$ diverge vers $+ \infty$ quand $z \rightarrow \tilde \rho$. Donc $\tilde S$ a pour rayon de convergence $\tilde \rho$. Le graphique gauche de la figure \ref{fig:Stilde-pos} illustre ce comportement.

\newcommand{\tr}{\tilde \rho}
\newcommand{\tS}{{\tilde S}}
\begin{figure}[h!]
\vskip -10mm
\begin{center}
\begin{tabular}{c}\includegraphics[scale=0.3]{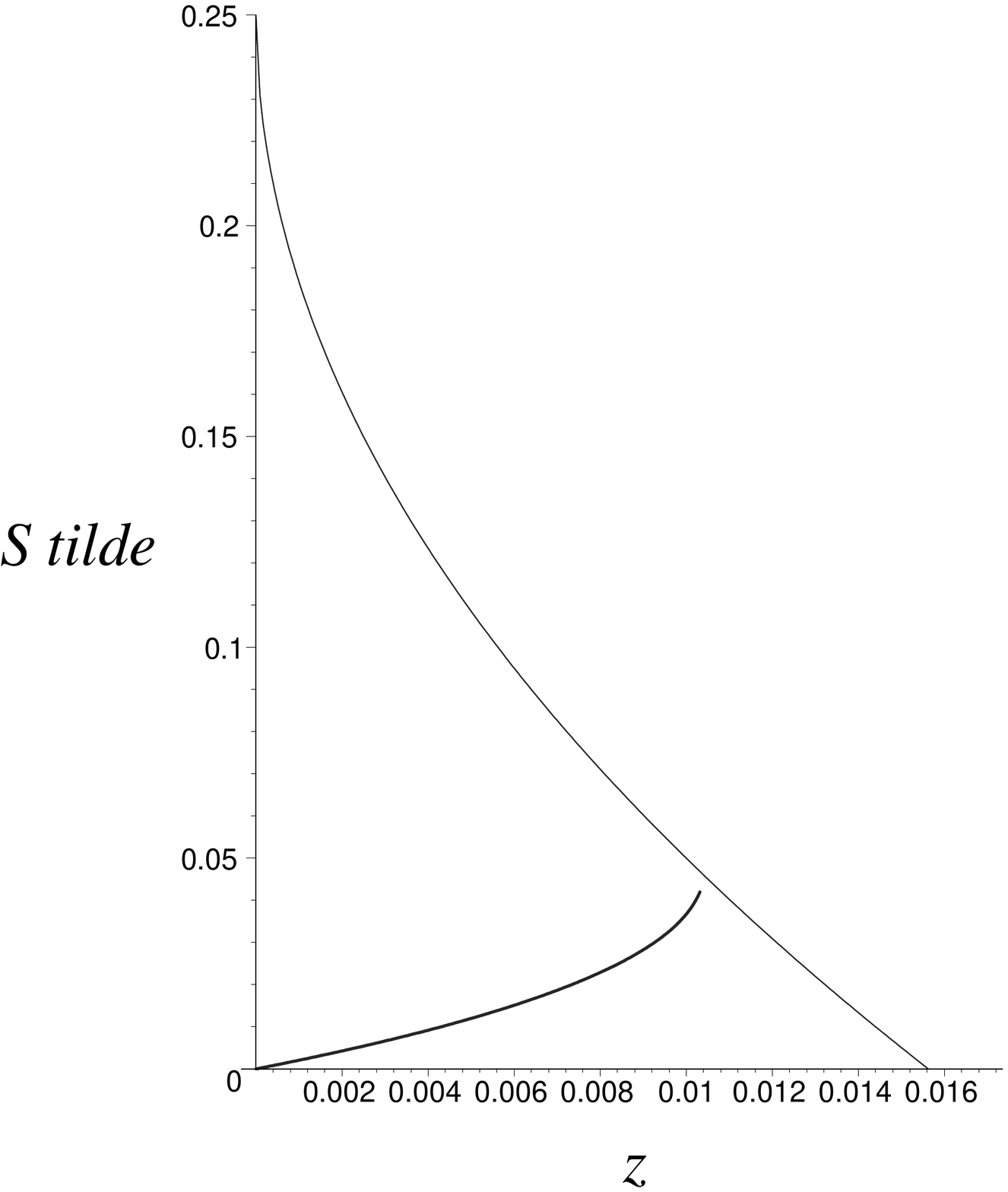}\end{tabular}
 \begin{tabular}{c}\includegraphics[scale=0.3]{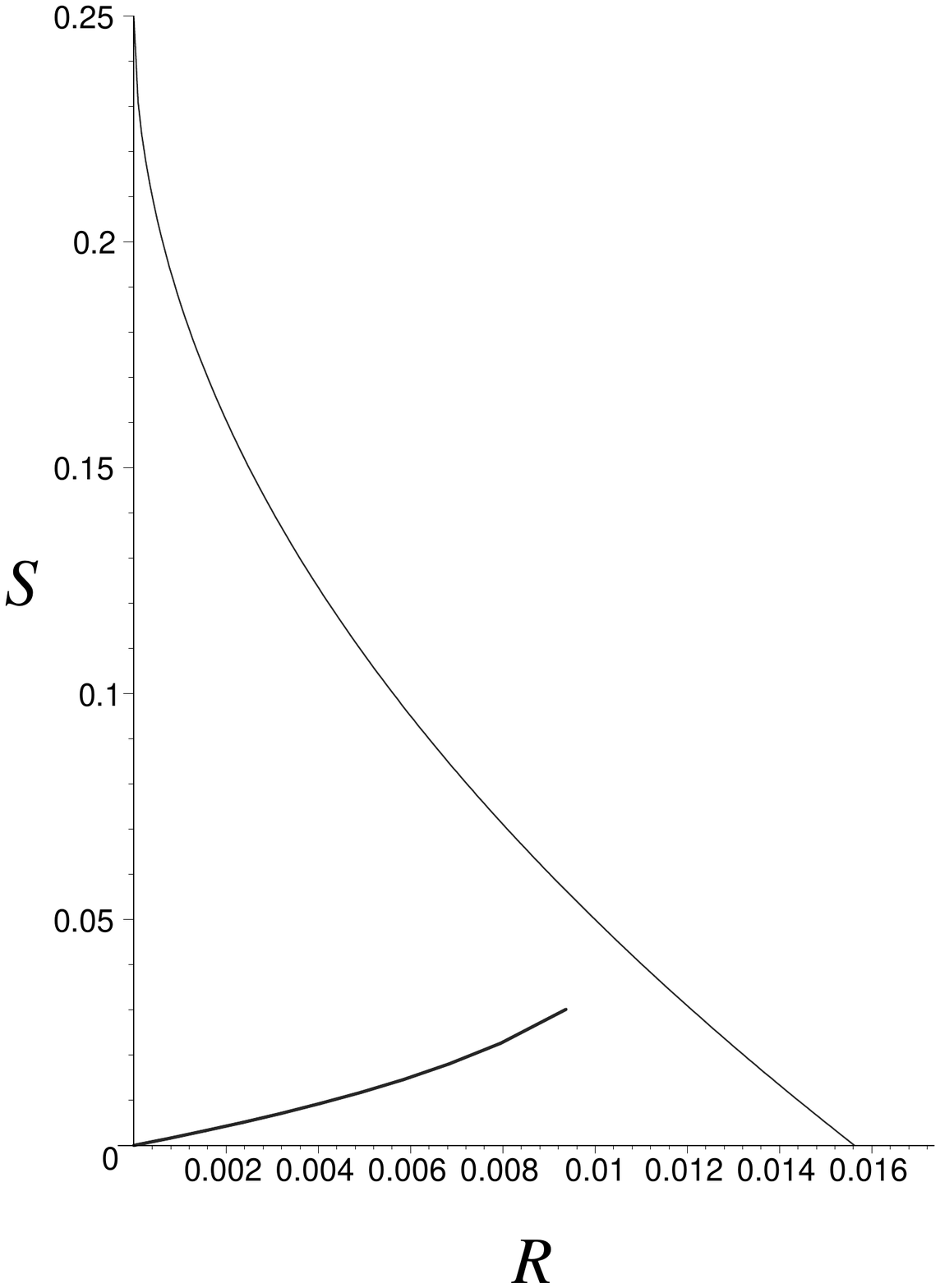}\end{tabular}
\end{center}
\vskip -10mm
\caption{\`A gauche : courbe $\pare{z,\tilde S(z)}$ pour $u=1$ and $z\in \, \left[0,
    \tr \right]$. Elle reste en dessous de la parabole 
    $64z=\pare{1-4\tilde S}^2$.  
Cette courbe a été obtenue en développant la série
 $\tilde S(z)$  à l'ordre $80$ (c'est pourquoi la divergence de $\pd \tS z$ en $\tr$ n'est pas très claire sur l'image). \`A droite : courbe $\pare{R(z), S(z)}$  pour $u=1$
et $z \in \, [0,\rho]$ (nous avons $\rho \simeq 0.0098$). Cette courbe suit celle de $\pare{z,\tilde S(z)}$, mais s'arrête au point $(R(\rho),S(\rho))$, pour lequel
$R(\rho)<\tr$.}
% Equation~\eqref{cubic-char} holds at $\rho\simeq 0.0098$.
\label{fig:Stilde-pos}
\end{figure}

Maintenant considérons l'équation $R = z + u \phi_1 \pare{R,\tilde S (R)}$. Nous voulons prouver à partir d'elle que $R$ satisfait le schéma d'équation fonctionnelle lisse de \cite[p. 467-468]{flajolet-sedgewick}. Avec la notation du livre, $G(z,w) = z + u \phi_1 \pare{w, \tS (w)}$. D'après ce qu'on a montré sur $\tilde S$, la fonction $G$ est analytique sur $\C \times \enstq w {\module w < \tr}$. L'équation caractéristique $\pd G w (\rho, \tau) = 1$ ne fait pas intervenir $\rho$ et s'écrit
\begin{equation}
 1= u\, \left(\pd {\phi_1} x  \pare{\tau ,\tilde S(\tau )} + \, \pd {\tilde S} z (\tau )\, \pd
 {\phi_1} y  \pare{\tau ,\tilde S(\tau ) } \right).
 \label{ceqcar}
\end{equation}
Le membre de droite de cette équation croît  de $0$ à $+ \infty$ quand $\tau$ décrit $\left[0,\tr\right]$  (on rappelle que $\pd {\tilde S} z \pare{\tr} = + \infty$). Par conséquent, l'équation \eqref{ceqcar} détermine $\tau \in \, ]0,\tr[$ de manière unique. L'équation $\tau = G(\rho,\tau)$ donne alors la valeur de $\rho$, à savoir
$$\rho = \tau - u \phi_1 \pare{\tau, \tilde S (\tau)}.$$
Posons $\sigma = \tS(\tau)$. En couplant les identités $\tilde S = u \phi_2 \pare{ z,\tilde S}$, \eqref{pdSz} et \eqref{ceqcar}, nous retrouvons les propriétés de $\tau$ et $\sigma$ énoncées dans le théorème \ref{t:cubique}.

Les premiers termes du développement de $R$ en $0$ 
$$R=z+2u(2u+3)z^2+4u(42u^2+63u+10u^3+35)z^3+ O(z^4)$$
montrent que cette série est apériodique.
Par le théorème VII.3 de \cite[p. 468]{flajolet-sedgewick}, la série $R$ a pour rayon de convergence $\rho$, et elle est analytique sur un $\Delta$-domaine de rayon $\rho$. Elle atteint la valeur $\tau$ en $\rho$, et y présente une singularité de type racine. Par composition avec la série $\tS$ (on rappelle que $\tr$, le rayon de $\tS$, est strictement supérieur à $\tau$), les mêmes propriétés sont vérifiées par $S= \tS(R)$. Au vu de l'identité  \eqref{F-cubic} p.~\pageref{F-cubic}, à savoir
$$
\pd F z = 2 \, \frac z u +\frac S u - \pare{1+\frac 1 u}(2R+S^2),$$ ces propriétés sont également vraies pour $\pd F z$ : la preuve est ainsi terminée.

Le graphique droit de la figure \ref{fig:Stilde-pos} illustre le comportement de $R$ et $S$ sur l'intervalle $[0,\rho]$.
\end{proof}

\section{Quand $\boldsymbol{u < 0}$}

\begin{prop} Fixons $u \in [-1,0[$. Les séries $R$, $S$ et $\pd F z$ ont le même rayon de convergence $\rho_u$, qui satisfait les conditions du théorème \ref{t:cubique}. Quand $z \rightarrow \rho_u^-$, ces trois séries admettent un développement de la forme \eqref{exp-3}, avec $\beta > 0$.
\end{prop}

\newcommand{\tr}{\tilde \rho}
\newcommand{\tS}{{\tilde S}}

\begin{proof} La théorie de la section \ref{s:abe} p.~\pageref{s:abe} nous sera très utile dans cette preuve. En effet, nous y avons montré que les séries $(R-z)/u$ et $S/u$ sont $(u+1)$-positives
 (corollaire \ref{c:RSupos}),
  ainsi que la série $\tilde S/u$, où $\tS$ est définie comme précédemment par $\tS(0,u) = 0$ et $\tS = u \phi_2 \pare{z,\tS}$  (lemme \ref{l:Stilde}). De même, nous savons que $\pd F z$ a des coefficients positifs (voir la sous-section \ref{ss:modele} pour les différentes interprétations). D'après le théorème de Pringsheim, le rayon de convergence de chacune des séries $\pd F z$, $R$, $S$, $\tilde S$ correspond à leur plus petite singularité positive réelle. Nous utiliserons souvent cette propriété dans cette démonstration, sans référence systématique au théorème de Pringsheim.
  
Comme dans le cas $u > 0$, nous procédons en deux étapes : d'abord $\tilde S$, puis $R$.  Commençons donc par appliquer le théorème \ref{t:fi} p.~\pageref{t:fi} à la série $\tilde S$ avec $H(x,y) = y -u \phi_2(x,y)$. Nous voulons infirmer les possibilités $(i)$, $(ii)$ et $(iv)$. 

\noindent \textbf{Condition $\boldsymbol{(i)}$.} La $(u+1)$-positivité de $\tilde S/u$ fait que le rayon de convergence de cette série est au plus égal au rayon de $\tilde S(z,-1)$. D'après le lemme \ref{l:Stilde}, la série $- \tilde S(z,-1)$ est une série à coefficients entiers qui n'a pas un degré borné en $z$. Son rayon de convergence est donc au plus égal à $1$. Ceci est donc vrai pour $\tilde S(z,u)$ ; cette série ne peut donc être analytiquement prolongeable sur $[0,+\infty[$.

\noindent \textbf{Condition $\boldsymbol{(ii)}$.} D'après le lemme \ref{l:Stilde}, la série $\pd {\phi_2} y \pare{z,\tS(z)}$ a des coefficients positifs. Comme son coefficient constant est nul, la fonction $z \mapsto 1 - u \pd {\phi_2} y \pare{z,\tS(z)}$ croît sur $\left[0,\tr\right[$, avec pour valeur initiale $1$. La fonction $\pd H y (z,\tilde S)$ ne peut donc pas approcher $0$.

\noindent \textbf{Condition $\boldsymbol{(iv)}$.} Comme ses coefficients sont négatifs, la série $\tilde S$ décroît sur $\left[0,\tr\right[$. Supposons que $\tS$ tende vers $- \infty$. La valeur de $\tr$ étant finie, cela implique 
$$\lim_{z \rightarrow \tr^-} \psi_2 \pare{\frac z {(1-4\tS(z))^2}} = \psi_2(0) = 0.$$
Or la relation $\tS = u \phi_2 \pare{z,\tS}$ exprimée en termes de $\psi_2$ donne (voir \eqref{nphienpsi})
$$\tilde S = u \sqrt{1-4 \tS} \, \psi_2 \pare{\frac z {(1-4\tS)^2}} + \frac u 4 \, \pare{1 - \sqrt{1 - 4 \tS}}^2.$$
En passant à la limite, nous trouvons
$$\pare{1 + u} \, \tilde S \underset{z \rightarrow \tr^-} = - \frac u 2 \sqrt{1 - 4 \tilde S} + o \pare{\sqrt{1 - 4 \tilde S}},$$
ce qui aboutit à une contradiction, les termes $\tilde S$ et $\sqrt{1 - 4 \tS}$ ne pouvant pas être du même ordre.

La seule possibilité restante est donc $(iii)$ : la série $\phi_2$ est singulière en $\pare{\tr,\tS\pare \tr}$. Or les singularités de $\phi_2$ d'abscisse minimale dans le quart de plan $[0,+\infty[ \, \times \, ]-\infty,0]$  se trouvent sur la parabole critique, ce qui  implique 
$$64 \tr = \pare{1- 4 \tS \pare \tr }^2.$$
Le rayon de $\tS$ est au moins égal à $\tr$ (la valeur explicite de $\tr$ sera donnée plus tard). La figure \ref{Sneg} montre le graphe de la fonction $\tilde S$ avec $ u = -1/2$. Nous pouvons prouver que $\tr$ est le rayon de convergence de $\tS$, mais cela ne nous sera pas utile.

\fig{[scale = 0.3]{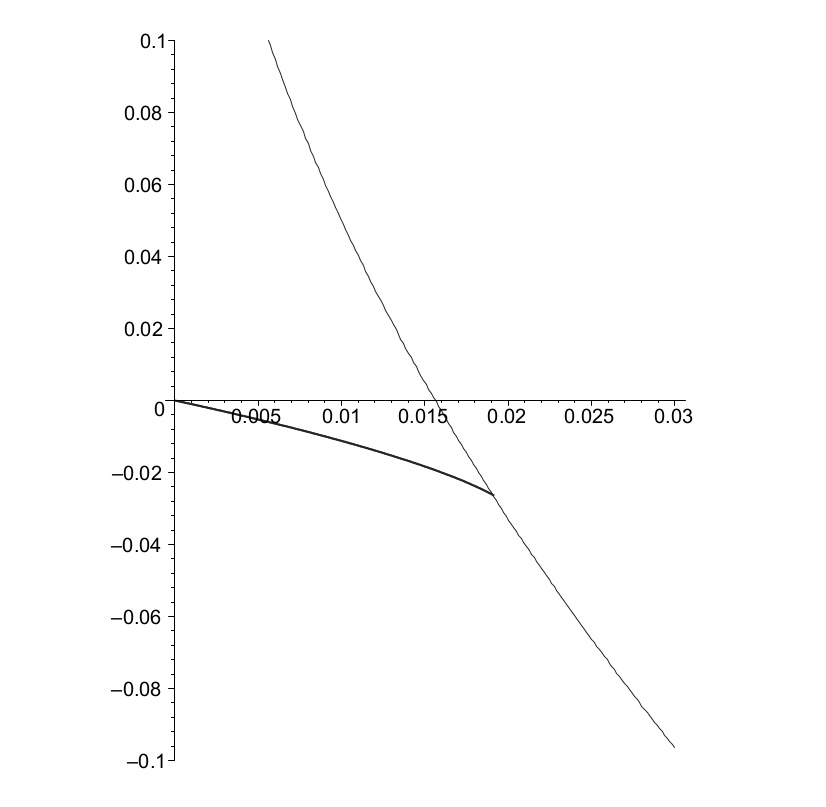}}{Courbe de  $\pare{z,\tilde S(z)}$ pour $u=-1/2$ and $z\in \, \left[0,
    \tr \right]$. Elle finit sur la parabole $64z = \pare{1 - 4 \tS}^2$. Il s'agit également de la courbe $\pare{R(s),S(z)}$ avec $z \in [0,\rho[$.}{Sneg}

Considérons maintenant $R$. Nous appliquons le corollaire \ref{c:fi} p. \pageref{c:fi} avec $\Omega(y) = y - u \phi_1\pare{y,\tS\pare y}$ de sorte que la fonction $Y$ du corollaire soit égale à $R$. Nous avons vu que pour $y \in \left[0,\tr\right[$, le point $\pare{y,\tS\pare y}$ n'appartient pas à la parabole critique. Donc la première singularité réelle positive de $\Omega$, notée $\omega$ dans le corollaire, est supérieure ou égale à $\tr$. Définissons $\tau$ et $\rho$ comme dans le corollaire \ref{c:fi}. 

Prouvons que $\Omega'(\tau) \neq 0$. En dérivant la relation $\Omega(R(z)) = z$, nous obtenons $\Omega'(R(z)) = 1 / R'(z)$. D'une part, $R' \geq 0$ car $R$ est croissante. D'autre part, $R'$ est bornée par $1$ sur $]0,\rho[$ puisque $(R'-1)/u$ a des coefficients positifs et $u$ est négatif. Par conséquent, $\Omega'(\tau) > 0$. On a donc $\tau = \omega \geq  \tr$.

Comme $R$ croît de $0$ à $\omega$ sur $[0,\rho]$, il existe un unique $\hat \rho$ dans cet intervalle tel que $R \pare{\hat \rho} = \tr$. Puisque le rayon de $\tS$ est au moins égal à $\tr$, celui de la série $S = \tS (R)$ est au moins égal à $\hat \rho$. La courbe de $\pare{R(z),S(z)}$ pour $z \in \, \left[0,\hat \rho \right]$ coïncide avec celle de $\pare{z,\tS(z)}$ pour $z \in \, \left[0,\tilde \rho \right]$, comme l'illustre la figure \ref{Sneg}.

Le reste de la preuve consiste à  trouver des développements singuliers de $R$ et $S$ au voisinage de $\hat \rho$ à partir de \eqref{nsystem}. Ce faisant, nous prouverons que $R$ et $S$ sont singuliers en $\hat \rho$, ce qui implique que $\hat \rho$  est le rayon de convergence de $R$ et $S$, noté $\rho_u$ dans le théorème \ref{t:cubique}.

\newcommand{\hr}{{\hat \rho}}

Adoptons la notation suivante $x = \hat \rho - z$, $r = \tilde \rho - R(z)$, $s = S\pare {\hat \rho} - S(z)$ et 
\begin{equation}
\label{defvareps}
\varepsilon = \frac 1 {64} - \frac{R(z)}{\pare{1 - 4 S(z)}^2}.
\end{equation}
Les quantités $x$, $r$, $s$ et $\varepsilon$ tendent vers $0$  quand $z$ approche $\hr$. Commençons par développer au voisinage de $z=\hat \rho$ la relation
$$S = u \sqrt{1-4 S} \, \psi_2 \pare{\frac {R} {(1-4S)^2}} + \frac u 4 \, \pare{1 - \sqrt{1 - 4 S}}^2,$$
obtenue en combinant l'identité $S = u \phi_2(R,S)$ et \eqref{nphienpsi}. En utilisant  le développement \eqref{devpsi2} de $\psi_2$ au voisinage de $1/64$, nous trouvons
\begin{equation}
\label{exp1}
a_1+b_1s+c_1\varepsilon \ln \varepsilon + d_1\varepsilon = O\pare{\varepsilon^2 \ln \varepsilon } +
O\pare{s^2}+ O\pare{s\,\varepsilon \ln \varepsilon},
\end{equation}
avec
$$
a_1=\frac{1+u}4 \delta^2 - \frac{u\sqrt 2} \pi \delta+ \frac{u-1}4,
$$
$$
b_1= - \frac{2u \sqrt 2}{\pi \delta}+1+u, \quad 
c_1= \frac{4\sqrt 2}\pi u \delta , \quad 
d_1=3c_1,
$$
et $\delta=\sqrt{1-4\tilde S\pare{\tr}}$. En particulier, $a_1$ doit s'annuler, ce qui donne la valeur de $\delta$ :
\begin{equation}
\label{delta-c}
\delta = \sqrt{1-4\tilde  S\pare{\tr}}= \frac{2\sqrt 2 u +
  \sqrt{\pi^2(1-u^2)+8u^2}}{\pi(1+u)}. 
\end{equation}
(Le choix du signe moins avant $\sqrt{\pi^2(1-u^2)+8u^2}$ aurait mené à une valeur négative, ce qui aurait contredit la positivité de $\delta = \sqrt{1-4\tilde  S\pare{\tr}}$.) Notons que $u$ peut s'exprimer de manière rationnelle en fonction de $\delta$ :
$$
u= -  \frac{\pi(\delta^2-1)}{\pi \delta^2-4 \sqrt 2 \delta+\pi}.
$$
Nous remplaçons désormais toute occurrence de $u$ par cette expression, ce qui nous évitera de manipuler des coefficients algébriques.

Revenons à \eqref{exp1}. \'Etant donné que $\delta > 0$ et  $u \in [-1,0[$, nous avons $b_1 > 0$. Donc $s = O\pare{\varepsilon \ln \varepsilon}$, ce qui permet de réécrire \eqref{exp1} comme
\begin{equation}
\label{exp2}
b_1s+c_1\varepsilon \ln \varepsilon + d_1\varepsilon = O\pare{\varepsilon^2 \ln^2 \varepsilon}.
\end{equation}
Développons maintenant \eqref{defvareps}. En remplaçant $\tr$ par $ \delta^4/64$, nous obtenons
\begin{equation}
\label{exp3}
 b_2 s + d_2\varepsilon + e_2 r = O \pare{ \varepsilon^2 \ln^2 \varepsilon},
\end{equation}
avec $ b_2=-8 \delta^2, d_2= 64 \delta^4 , e_2=-64$.
Enfin, l'équation $R = z + u \phi_1(R,S)$ se développe grâce à \eqref{phi1enpsi} et \eqref{devpsi1} en 
\begin{equation}
a_3+b_3 \, s + c_3 \varepsilon \ln \varepsilon + d_3 \,  \varepsilon + e_3 r + f_3 x
=O\pare{\varepsilon^2 \ln ^2 \varepsilon },
\label{exp4}
\end{equation}
où, en particulier,
$$ 
a_3=96\pare{\pi \delta^2 -4\sqrt 2 \delta+\pi}\hr+\delta^3 \pare{2\sqrt 2\delta^2-3\pi\delta
+4\sqrt 2}.
$$
Au vu de \eqref{exp4}, le nombre $a_3$ doit être nul. Nous pouvons donc exprimer $\hat \rho$ en termes de $\delta$. En y injectant la relation \eqref{delta-c}, nous obtenons une expression explicite de $\hat \rho$ qui coïncide avec \eqref{rho-3}. Nous ne donnons pas les expressions de $b_3$, $c_3$, $d_3$, $e_3$ et $f_3$, rationnelles en $\delta$, qui sont faciles à calculer. Mentionnons juste que $f_3 \neq 0$.

Maintenant, grâce à ~\eqref{exp2},~\eqref{exp3} and~\eqref{exp4}, nous obtenons respectivement pour  $s$, $r$ et enfin $x$ des expressions en fonction de $\varepsilon$ de la forme
\begin{eqnarray}
s&= & c_4\,\varepsilon  \ln \varepsilon + d_4\,\varepsilon + O\pare{\varepsilon^2 \ln ^2
\varepsilon },
\label{s-eps}
\\
r&=&  c_5\,\varepsilon  \ln \varepsilon + d_5\,\varepsilon +O\pare{\varepsilon^2 \ln ^2
\varepsilon },
 \label{r-eps}\\
x&=&  c_6\,\varepsilon  \ln \varepsilon + d_6\,\varepsilon+ O\pare{\varepsilon^2 \ln ^2
\varepsilon }.
\label{x-eps}
\end{eqnarray}
En particulier, $c_6 \neq 0$ et l'équation qui précède donne $x \sim \, c_6 \, \varepsilon \, \ln \varepsilon$. Ainsi $\ln x \sim \ln \varepsilon$ et donc
$$\varepsilon = \frac x {c_6 \ln x} \pare{1 + o(1)}.$$
De plus, \eqref{x-eps} permet d'exprimer $\varepsilon \ln \varepsilon$ comme une combinaison linéaire de $x$ et $\varepsilon$ (plus des termes négligeables). En injectant cette combinaison linéaire dans \eqref{s-eps} et \eqref{r-eps}, puis en remplaçant $\varepsilon$ par l'expression donnée ci-dessus, nous obtenons
\begin{eqnarray*}
s&= & \frac{c_4}{c_6} x + \frac{d_4c_6-c_4d_6}{c_6^2} \frac x{\ln x} (1+o(1)),
\\
r&=&  \frac{c_5}{c_6} x + \frac{d_5c_6-c_5d_6}{c_6^2} \frac x{\ln x}
(1+o(1)) .
\end{eqnarray*}
Nous pouvons écrire ces développements de manière explicite :
\begin{eqnarray*}
S(z)& \underset{z \rightarrow \hat \rho}=& \frac {1-\delta^2}4+
\frac{4\pi}{\delta\sqrt{\pi^2(1-u^2)+8u^2}}
(\hr -z)
- \frac{2\sqrt 2\pi}{u\delta}\frac{
  \hr -z}{\ln     (\hr -z)}\left(1+o(1)\right),
\\
R(z)&\underset{z \rightarrow \hat \rho}=& \tr   - \frac{\pi \delta}{2\sqrt{\pi^2(1-u^2)+8u^2}}(\hr-z) - 
\frac{\sqrt 2 \pi \delta}{4u} \frac{
  \hr -z}{\ln(\hr-z)}\left(1+o(1)\right).
\end{eqnarray*}
Comme attendu, $R$ et $S$ sont singulières en $\hr$, ce qui montre que $\hr$ est le rayon de convergence commun à ces deux séries.

Pour conclure avec $\pd F z$, nous utilisons la relation  \eqref{F-cubic} p.~\pageref{F-cubic}
$$
\pd F z = 2 \, \frac z u +\frac S u - \pare{1+\frac 1 u}(2R+S^2)$$
afin d'obtenir un développement asymptotique de $\pd F z$ au voisinage de $\rho$. Ce dernier coïncide avec \eqref{exp-3}. Le coefficient $\beta$ de $(\rho - z) / \ln (\rho -z)$ ne s'annule pas sur $[-1,0[$, c'est pourquoi le rayon de convergence de la série $\pd F z$ est également $\rho$.
\end{proof}

\noindent \textbf{Remarque.} Contrairement aux deux derniers chapitres, démontrer l'analyticité de $R$ et $S$ sur un $\Delta$-domaine semble excessivement difficile. En effet, nous pouvons montrer (par des arguments de concavité) que $R$ prend sur $[-\rho,0]$ des valeurs en dehors du disque de convergence de $\tilde S$. Donc pour prouver l'analyticité de $R$ (et $S$) sur le cercle de convergence (privé de $\rho$), il faudrait être capable de contrôler le domaine d'analyticité de $\tilde S$, mais celui-ci semble difficilement accessible.

\chapter{Perspectives : les cartes animalières}
\label{c:perspectives}

Modèle classique de la physique statistique, la  \textit{percolation} constitue un pôle de recherche très actif, notamment en théorie de probabilités (voir par exemple \cite{stauffer}). Elle est en général étudiée sur des réseaux réguliers, mais plusieurs chercheurs se sont intéressés à ce phénomène sur des cartes planaires \cite{angel-perco,perco2}. Si ce chapitre est consacré à la percolation sur les cartes planaires (ou plus précisément la \textit{percolation de lien}), c'est parce qu'elle est le prolongement naturel de l'étude des cartes forestières. Les  résultats que nous énonçons dans ce chapitre aspirent à des travaux plus poussés sur cette notion (comme par exemple l'étude du comportement asymptotique).

\section{Définition}

Soit $C$ une carte planaire enracinée. Un \textit{animal de lien} (ou un \textit{animal} tout court) est un sous-graphe connexe de $C$, pas forcément couvrant, qui contient l'arête racine\footnote{Il n'existe pas de véritable convention sur la position de la racine, nous en avons donc choisi une qui mènera aux équations les plus simples. En particulier, cela veut dire que l'animal n'est pas vide.}.
Une \textit{carte animalière} $(C,A)$ est une carte planaire enracinée $C$ équipée d'un animal de lien $A$. Un exemple est montré figure \ref{animal}.

\fig{[scale=1.2]{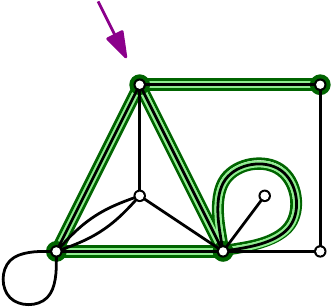}}{Exemple de carte animalière.}{animal}

Notons $A(z,t,b,u)$ la série génératrice des cartes animalières $(C,A)$ avec un poids $z$ par face de $C$, un poids $t$ par arête de $C$, un poids $b$ par arête interne\footnote{la lettre $b$ correspond à "bond" en anglais, lien en français} (i.e. arête qui appartient à l'animal) et un poids $u$ par face de l'animal\footnote{Un animal est une carte planaire, donc comporte des faces. Elles s'obtiennent en supprimant toutes les arêtes externes. Cette statistique n'est pas forcément très naturelle, mais elle s'insère facilement dans le cadre de cette étude. En outre, elle détermine  le nombre de sommets dans l'animal grâce au paramètre $s$ et la formule d'Euler.}. En outre, nous pondérons chaque \textbf{face} de degré $k$ par une variable $d_k$. Tout comme les cartes forestières, cette dernière pondération permet d'obtenir différentes classes de cartes. Par exemple, pour $d_3 = 1$ et $d_k = 0$ pour $k \neq 3$, nous retrouvons la classe des triangulations, et les premiers termes de la série $A$ sont :
$$A(z,t,b,u) = \pare{3\,b\,u + b\,u^2 + 4\,b^2\,u + 4\,b^2\,u^2  + 4 \, b^3 \, u^2} \, z^2 \,t^3  + O\pare{t^4}.$$

\section{\'Equations}

Nous voulons exprimer la série $A$ en termes de séries génératrices plus connues. Pour cela, nous devons au préalable définir la série génératrice des cartes avec une face racine simple. Une carte a une \textit{face racine simple} si les sommets autour de la face racine sont tous distincts. Cela signifie que la face racine ne comporte ni de croisement en huit, ni d'isthme. La figure \ref{cartesfacesimple} illustre cette définition. Nous notons $C_{k}(z,t)$  la série génératrice des cartes planaires avec une face racine simple de degré $k$, où $z$ compte le nombre de faces et $t$ le nombre d'arêtes. Nous pondérons également par $d_k$ chaque face de degré $k$. Nous donnons une caractérisation des séries $C_{k}(z,t)$ dans le lemme \ref{l:cegetem}.

\fig{[scale = 1]{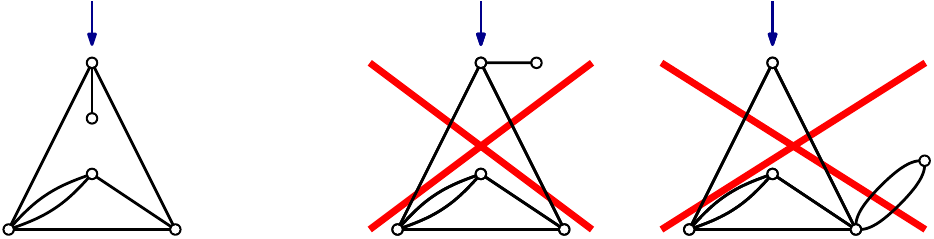}}{La première carte a une face racine simple, les deux autres non.}{cartesfacesimple}

\begin{theo} \label{amselle} Notons $C_{k}(z,t)$ la série génératrice des cartes planaires avec une face racine simple de degré $k$ et $M(u;g_1,g_2,\dots)$ la série génératrice des cartes planaires avec un poids $u$ par face et un poids $g_k$ par face de degré $k$.

La série génératrice $A(z,t,b,u)$ des cartes animalières est 
\begin{equation}
A(z,t,b,u) = M \pare{u; \frac 1 z \pare{\frac  b t}^{1/2}  C_1(z,t), \frac 1 z \pare{\frac  b t}^{1} C_2(z,t), \frac 1 z \pare{\frac  b t}^{3/2} C_3(z,t), \dots  }.
\label{mamzelle}
\end{equation}
\end{theo}

\begin{proof}
Nous décrivons une bijection $\Phi$ entre cartes animalières $(C,A)$ et cartes planaires $C'$ avec une collection de cartes $(C_f)_{f\textrm{ face de }C'}$ à face racine simple telle que pour toute face $f$ de $C'$, le degré de la face racine de $C_f$ correspond au degré de $f$.

Soit $(C,A)$ une carte animalière.  Commençons par découper cette carte le long de chaque face de l'animal $A$. Nous voyons apparaître à l'intérieur de chacune des faces $f$ de $A$ une carte planaire $C_f$ à bord, ce bord étant simple. Posons $C' = A$. 

Soit $f$ une face de $C'$. Nous rajoutons le long du bord de $C_f$ une face appelée \textit{face bordière}\footnote{Nous voulons éviter le terme de  "face externe", puisque la face bordière de la carte associée à la face externe de $A$ n'est pas externe...}. Par construction, le degré de la face bordière est égal au degré de $f$.  En outre, la carte $C_f$ est naturellement enracinée sur un coin incident à la face bordière ; expliquons pourquoi : nous pouvons ordonner de manière canonique les coins de $C'$ (comme dans la preuve du théorème \ref{tfemme} p. \pageref{tfemme}, cet ordre peut avoir une origine combinatoire ou pas). Considérons alors le plus petit coin incident à $f$ pour cet ordre-là. Les deux demi-arêtes formant ce coin existent également dans $C_f$ et sont incidentes à la face bordière ; nous pouvons donc enraciner $C_f$ sur le coin formé par ces deux demi-arêtes. 

Nous posons alors $\Phi(C,A) = \pare{C',(C_f)}$. La figure \ref{decompoanimal} donne un exemple de calcul d'une telle image.

\fig{[width=\textwidth]{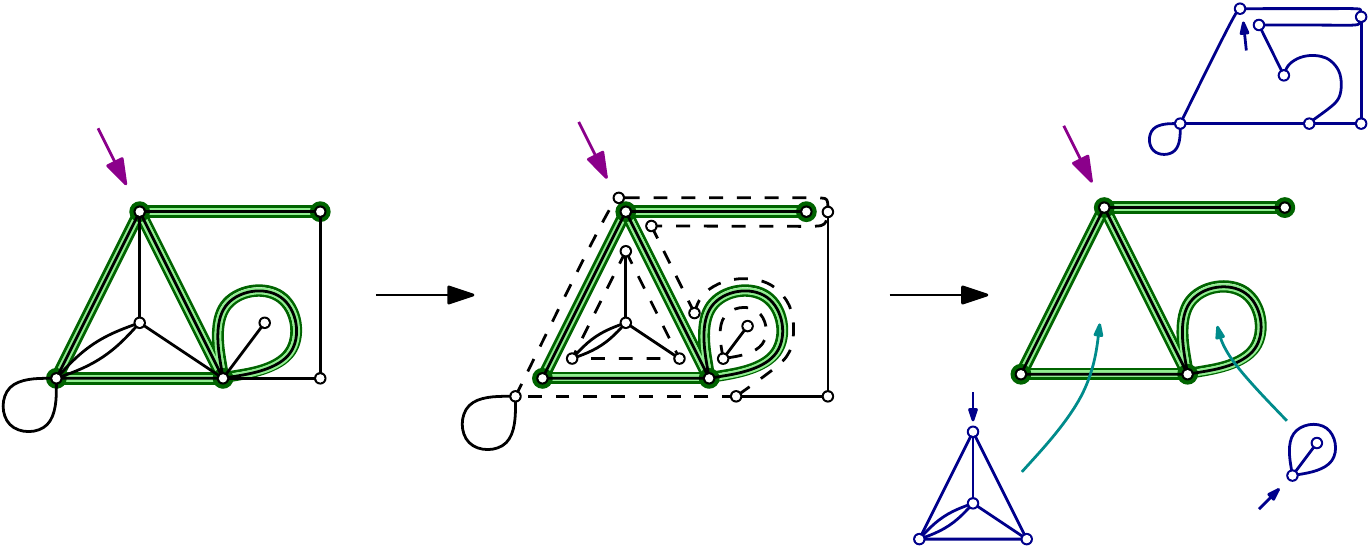}}{Les deux étapes dans le calcul de $\Phi$.}{decompoanimal}

Pour reconstruire la carte $C$ à partir de la donnée de $\pare{C',(C_f)}$, il suffit d'insérer à l'intérieur de chaque face $f$ de $C'$ la carte $C_f$, puis de recoller les arêtes de la face bordière aux arêtes de la face $f$, en veillant à ce que la racine de $C_f$ coïncide avec le plus petit coin de $C'$ incident à $f$ pour l'ordre précédemment choisi. Les arêtes de $C'$ forment alors l'animal $A$. L'enracinement de $C$ nécessite une ultime précaution : \textit{a priori} la racine de $C$ peut se trouver n'importe où entre les deux demi-arêtes qui constituaient la racine de $C'$. Mais par définition, l'arête racine de $C$ doit appartenir à l'animal ; il ne reste donc qu'une seule possibilité pour l'enracinement de $C$. Au final, nous avons bien prouvé que $\Phi$ était une bijection.

Pour conclure, il suffit d'observer que :
\begin{itemize}
\item le nombre de faces de $C$ est égal au nombre total de faces non racine dans toutes les cartes $C_f$,
\item le nombre d'arêtes internes de $C$ est  la moitié du nombre d'arêtes incidentes à la face racine d'une des cartes $C_f$,
\item en comptant toutes les arêtes appartenant à une des cartes $C_f$, nous comptons chaque arête externe de $C$ une fois, chaque arête interne deux fois ; le nombre d'arêtes dans $C$ est donc égal au nombre de telles arêtes, moins le nombre d'arêtes internes,
\item le nombre de faces de $A$ est égal au nombre de faces de $C'$ (ces deux cartes sont les mêmes !).
\end{itemize}
La relation \eqref{mamzelle} résulte alors de la méthode symbolique décrite en introduction.
\end{proof}

\noindent \textbf{Remarque. } Ce théorème est à comparer avec le théorème \ref{tfemme} p. \pageref{tfemme}. De manière duale, $C_{k}(z,t)$ compte les cartes planaires telles que le sommet racine a pour degré $k$ et n'est pas un point d'articulation, avec un poids $z$ par sommet et un poids $t$ par arête. Par une découpe au niveau du sommet racine, cela revient à considérer les cartes planaires à $k$ pattes,  enracinées sur une patte, où chacune des $k$ pattes est incidente à la face racine. Par conséquent, les cartes animalières peuvent donc se voir comme des cartes planaires dans lesquelles on a inséré au niveau de chaque sommet une carte à pattes (de manière à ce que les pattes et les demi-arêtes incidentes coïncident). La figure \ref{animalcartesgonflee} illustre cette correspondance directement par dualité. De ce point de vue, les cartes animalières constituent l'extension naturelle des cartes forestières. En effet, d'après le théorème \ref{tfemme}, les cartes forestières sont des cartes planaires sur lesquelles on a greffé sur chaque sommet un arbre à pattes.

\fig{[width = \textwidth]{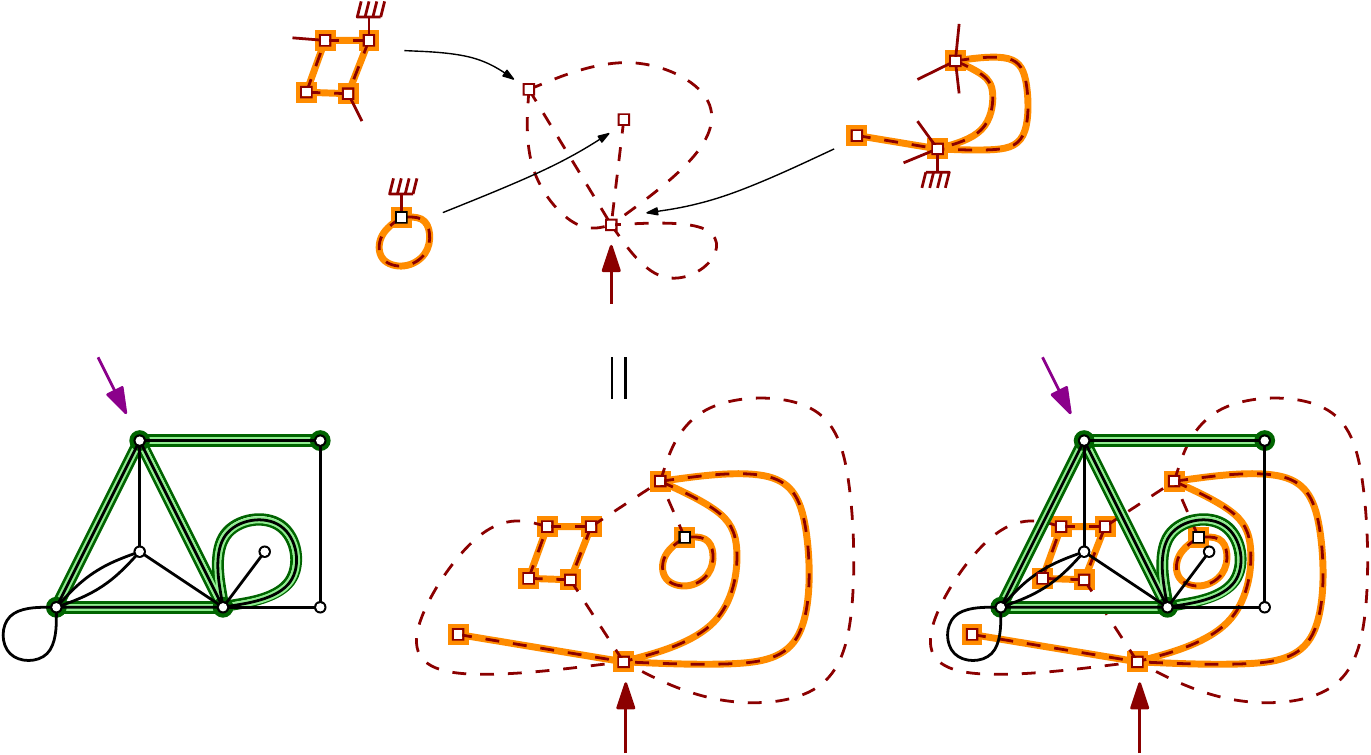}}{\`A gauche : une carte animalière. Au milieu : la carte duale correspondante, une carte planaire dans laquelle on a greffé une carte à pattes sur chaque sommet. \`A droite : la superposition des deux premières cartes. }{animalcartesgonflee}

Passons maintenant aux équations proprement dites. Commençons par décrire la série génératrice des cartes planaires avec une face racine simple.

\begin{lem} \label{l:cegetem}
 Désignons par $C(z,t,x) = \sum_{k \geq 1} C_k(z,t) x^k$,  la série génératrice des cartes planaires avec une face racine simple et $M \pare{z,t,x}$ la série génératrice des cartes planaires. La variable $z$ compte les faces, $t$ compte les arêtes et $x$ compte le degré de la face racine. Nous pondérons également chaque face de degré $k$ par $d_k$. Alors
\begin{equation} \label{CetM}
M\pare{z,t,x} = z  + C\pare{z,t, \frac x z \, M\pare{z,t,x}}.
\end{equation}
On rappelle que la série $M = M(z,t,x)$  satisfait 
\begin{equation}
\pd M z = u \sum_{i \geq 0} \sum_{j \geq 0} x^{2i+j} d_{2i+j} {2i+j \choose i,i,j} {\hat R}^i {\hat S}^j,
\label{Mfaceexterne}
\end{equation}
où $\hat R$ et $\hat S$ sont deux séries de variables $z$ et $t$ telles que
$$\hat R = tz + t \sum_{i \geq 1} \sum_{j \geq 0} d_{2i+j} { 2i+j-1 \choose i-1,i,j } {\hat R}^i {\hat S}^j \quad
\textrm{et} \quad
\hat  S = t \sum_{i \geq 0} \sum_{j \geq 0} d_{2i+j+1} { 2i+j \choose i,i,j } {\hat R}^i {\hat S}^j.
$$
\end{lem}
\begin{proof} Le système fonctionnel liant $\pd M z$, $\hat R$ et $\hat S$ est une conséquence directe de la proposition \ref{prop:mp} p.~\pageref{prop:mp}. 

\fig{[scale=1.1]{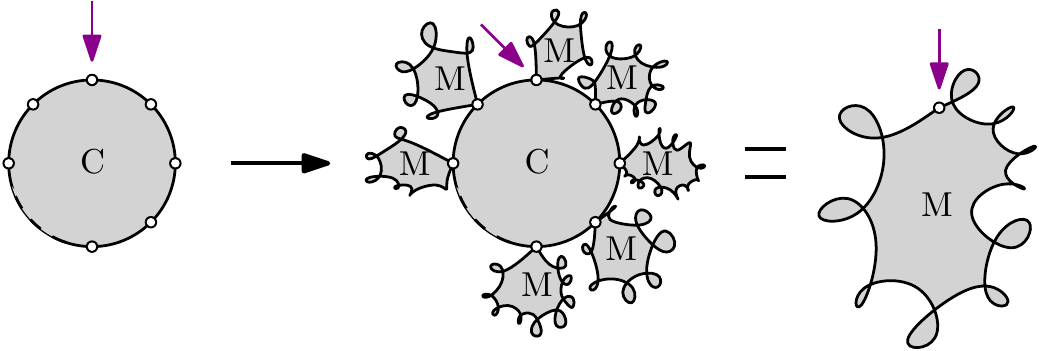}}{Interprétation combinatoire de l'identité \eqref{CetM}.}{entreCetM}

Quant à l'équation \eqref{CetM}, elle se déduit d'une construction combinatoire élémentaire -- précisons laquelle. Soit $C$ une carte planaire à face racine simple. Greffons sur chaque coin de $C$ incident à la face racine une carte planaire (potentiellement vide). Nous réenracinons la carte globale juste après la carte que nous avons greffée sur le coin racine. Nous obtenons ainsi une carte planaire sans restriction sur la face racine, avec au moins une arête (voir figure \ref{entreCetM}). Pour retrouver $C$, il suffit  de considérer la plus grosse sous-carte à face racine simple qui suit la racine. Cette décomposition est bien biunivoque. (Le terme $z$ est la carte est à un seule sommet que nous excluons des cartes à face racine simple.)
\end{proof}

Nous donnons maintenant une caractérisation de la série génératrice des cartes animalières.

\begin{theo} 
Il existe un unique couple $(R,S)$ de séries formelles en $z$, $b$, $u$, $t$, $d_1,d_2,d_3,\dots$  qui satisfait
$$
R = \frac b t + \,\frac {bu} {tz}\, \sum_{i \geq 1} \sum_{j \geq 0} C_{2i+j}(z,t) { 2i+j-1 \choose i-1,i,j } R^i S^j
$$
$$
S = \frac {bu} {tz} \, \sum_{i \geq 0} \sum_{j \geq 0} C_{2i+j+1}(z,t) { 2i+j \choose i,i,j } R^i S^j.
$$
où $C_{k}(z,t)$ désigne la série génératrice des cartes à face racine simple de degré $k$, caractérisée par le lemme \ref{l:cegetem}.

Alors la série génératrice $A(z,t,b,u)$ des cartes animalières, avec un poids $z$ par face, un poids $t$ par arête, un poids $b$ par arête interne et un poids $u$ par face de l'animal  est décrite par
\begin{multline*}
A(z,t,b,u) = \frac{tz} b R + \frac{tz} b S^2 - z^2 \\
- 2 \frac  {tu}{b}  S \sum_{\substack{i \geq 2 \\ j \geq 0}}  C_{2i+j-1}(z,t) { 2i - j - 2 \choose i,i-2,j} R^i S^j - \frac{ t u}  {b} \sum_{\substack{i \geq 3 \\ j \geq 0}}  C_{2i+j-2}(z,t) { 2i - j - 3 \choose i,i-3,j} R^i S^j.
\end{multline*}
\end{theo}

\begin{proof}
Il suffit de combiner le théorème \ref{cartesgk} et le théorème \ref{amselle}. La variable $g_k$ est remplacée par $(b/t)^{k/2}C_k(z,t)/z$ pour chaque $k \geq 1$. Les équations de ce théorème sont alors obtenues en effectuant la substitution bijective $(R,S) \mapsto (t R / b,\sqrt{ t / b}S)$ (nous avons procédé à cette substitution pour que $R$ et $S$ soient des séries formelles en $b$ et $t$). L'unicité du couple $(R,S)$ provient du lemme~ \ref{caracteresse}. \end{proof}

\section{Algébricité différentielle}

Grâce au chapitre \ref{c:ed}, l'algébricité différentielle de $A(z,t,b,u)$ devient facile à démontrer sous certaines hypothèses.

\begin{theo} Supposons que la suite des poids $(d_k)$ vérifie l'une ou l'autre des deux conditions suivantes :
\begin{enumerate}
\item[(i)] la suite $(d_k)$ est nulle à partir d'un certain rang,
\item[(ii)] la suite $(d_k)$ vaut $1$ sur une union finie de suites arithmétique et éventuellement sur un sous-ensemble fini de $\N$, et elle vaut $0$ sur le sous-ensemble restant.
\end{enumerate}
Alors la série génératrice des cartes animalières $A(z,t,b,u)$ est différentiellement algébrique en $z$ et en $t$. (Rappel : la variable $z$ compte les faces et $t$ les arêtes.)
\end{theo}
\begin{proof} D'après le théorème \ref{t:dalg} p. \pageref{t:dalg} appliqué à l'identité du théorème \ref{amselle}, il suffit de montrer que la série $C(z,t,x) = \sum_{k \geq 1} C_k(z,t) x^k$ est holonome, où $C_k(z,t)$ désigne la série génératrice des cartes planaires à face racine simple de degré $k$, pour $k \geq 0$. (Ici on attribue les mêmes poids au sommet racine qu'aux sommets non racine.)

En réalité, nous allons prouver que la série $C$ est algébrique (i.e. est annulée par un polynôme). Pour cela, nous nous basons sur le résultat suivant : si une des deux hypothèses $(i)$ ou $(ii)$ est vraie, alors la série génératrice $M(z,t,x)$ des cartes planaires, avec un poids $d_k$ pour chaque face de degré $k$, comptées par les faces ($z$), les arêtes ($t$) et le degré de la face racine ($x$), est algébrique.  Pour la condition $(i)$, l'algébricité de $M$ est montrée dans un article de Mireille Bousquet-Mélou et Arnaud Jehanne \cite[Corollaire 7]{jehanne}. L'algébricité sous la condition $(ii)$ est quant à elle montrée dans un papier de Bender et Canfield \cite[Corollaire 1]{bender-canfield}.

Il existe donc un polynôme $P$ non nul à coefficients entiers tels que $$P(z,t,x,M(z,t,x))=0.$$ Nous pouvons légèrement modifier ce polynôme pour qu'en fait
$$P\pare{z,t,\frac x z \,M(z,t,x),M(z,t,x)} = 0.$$
(Il suffit de remplacer toute occurrence de $x$ dans $P$ par $ x \, z/M(z,t,x)$ et multiplier par un exposant suffisamment grand de $M(z,t,x)$ de sorte que $P$ reste un polynôme.)
D'après le lemme  \ref{l:cegetem}, on a alors
$$P\pare{z,t,\frac x z \, M(z,t,x),z + C\pare{z,t, \frac x z \,M(z,t,x)}} = 0.$$
En posant $y = x \, M(z,t,x) / z$, nous obtenons
$$P\pare{z,t,y,z + C\pare{z,t,y}} = 0,$$
ce qui prouve l'algébricité de $C$.
\end{proof}

\noindent \textbf{Remarque.} Ce résultat est théorique, l'obtention d'une équation différentielle explicite pour $A$ semble difficilement réalisable.

\section{Les perspectives}

Le système fonctionnel ayant été établi, nous aimerions pouvoir l'analyser et en déduire le comportement asymptotique des cartes animalières. Des réponses devraient être apportées aux questions suivantes :
\begin{itemize}
\item Y a-t-il une transition de phase ?
 \item Quelle est la taille moyenne de l'animal (selon les arêtes/les sommets) ?
 \item La série génératrice est-elle  holonome ?
\end{itemize}

Une des difficultés de la démarche semble être le niveau d'imbrication entre les différentes séries : d'abord $\hat R$ et $\hat S$, puis $M$, puis $C$, puis $R$ et $S$ et enfin $A$... En outre, contrairement aux arbres avec un nombre de pattes fixé, il semble ne pas y avoir de formules closes simples pour le nombre de cartes à face racine simple de degré fixé, même dans les classes de cartes les plus simples.

  Une autre question intéressante serait de savoir s'il existe un phénomène de $(u+1)$-positivité pour les cartes animalières. En effet, nous avons vu dans la section~\ref{s:bbforet} p.~\pageref{s:bbforet} (lire plus particulièrement la remarque 2) que la $(u+1)$-positivité de la série génératrice des cartes forestières pouvait être déduite de la bijection entre arbres bourgeonnants forestiers et cartes forestières. Je pense que cette approche peut être adaptée aux cartes animalières et que la $(u+1)$-positivité (ou quelque chose d'approchant) de la série $A$ peut être montrée. Si cela s'avère vrai, il serait intéressant (mais sûrement très difficile !) d'étudier le comportement asymptotique de telles cartes pour des valeurs négatives de $u$. Et pourquoi pas espérer un régime asymptotique à base de logarithmes, comme pour les cartes forestières ?

\selectlanguage{english}

\part{A general notion of activity for the Tutte polynomial}

\chapter{Activities for the Tutte polynomial}
\label{c:activities}

\subsection*{Introduction}

We mentioned in Section \ref{s:poltutte} that the Tutte polynomial of a graph $G$ can be defined as the generating function of the spanning trees $T$ counted according to the numbers $i(T)$ and $e(T)$ of internal and external active edges:
\begin{equation}\label{first}
T_G(x,y) = \sum_{T\textrm{ spanning tree}} x^{i(T)}y^{e(T)}.
\end{equation}
Such a description appeared for the first time in the founding paper of Tutte \cite{tutte54}. Tutte's notion of activity required to linearly order the edge set. More recently, Bernardi \cite{bernardi-tutte} gave a new notion of activity which was this time based on an embedding of the graph. 
The two notions are not equivalent, but they both satisfy \eqref{first}. 
One can also cite the notion of external activity introduced by Gessel and Sagan \cite{GesselSagan} involving the depth-first search algorithm and requiring a linear ordering on the vertex set.

The purpose of this part is to unify all these notions of activity.
We thereby define a new notion of activity, called \textit{$\Delta$-activity}. Its definition is based on a new combinatorial object named \textit{decision tree}, to which the letter $\Delta$ refers.  We show that each of the previous activities is a particular case of $\Delta$-activity. Moreover, we see that the $\Delta$-activity enjoys most of the properties that were true for the other activities, like Crapo's property of partition of the subgraphs into intervals \cite{crapo}.

\section{Preliminary definitions}
\label{DNM}
%**************************************************************%
%**************************************************************%

\subsection{Sets}

The set of non-negative integers is denoted by $\N$. We denote by $|A|$ the cardinality of any set $A$. When we say that a set $A$ is the disjoint union of some subsets $A_1,\dots,A_k$, this means that $A$ is the union of these subsets and that $A_1,\dots,A_k$ are pairwise disjoint. We then write $A = \biguplus_{i=1}^k A_i$. For any pair of sets $A$, $B$, we denote by $\diffs A B$ the symmetric difference of $A$ and $B$\footnote{We do not use the notation $A \, \triangle \, B$ because the triangle $\triangle$ can be easily mistaken for the letter $\Delta$, very much used in this manuscript.}. Let us recall that the symmetric difference is commutative and associative. 

\subsection{Graphs, subgraphs, intervals} \label{sss:interval}
For a graph $G$, the set of vertices is denoted by $V(G)$ and the set of edges by $E(G)$. 
A \textit{spanning subgraph} of $G$ is a graph $S$ such that $V(S)=V(G)$ and $E(S) \subseteq E(G)$. Unless otherwise indicated, all subgraphs will be spanning in this thesis. A subgraph is completely determined by its edge set, therefore we identify the subgraph with its edge set. For instance, given $A$ a set of edges, we will allow ourselves to write $S \subseteq A$ if $S$ is a subgraph of $G$ only made of edges that belong to $A$. For any subgraph $S$, we denote by $\overline S$ the complement subgraph of $S$ in $G$, i.e. the subgraph $\overline S$ such that $V(\overline S) = V(G)$ and $E(\overline S) = E(G) \backslash E(S)$.
If $S$ is a subgraph of $G$ and $A \subseteq E(S)$, we write $S \backslash A$ to denote the subgraph of $G$ with edge set $E(S) \backslash A$. 
Given a subgraph $S$ of $G$, an edge $e \in E(G)$ is said to be \emph{internal} if $e \in S$, \emph{external} otherwise.

The collection of all subgraphs of $G$ can be ordered via inclusion to obtain a boolean lattice. A \textit{subgraph interval} $I$ denotes an interval for this lattice, meaning that there exist two subgraphs $H^-$ and $H^+$ such that $I$ is the set of subgraphs $S$ satisfying $$H^- \subseteq S \subseteq H^+.$$ In this case, we write $I = [H^-,H^+]$.

\subsection{Cycles and cocycles}

A \emph{path} is an alternating sequence $(v_1,e_1,v_2,\dots,e_k,v_{k+1})$ of vertices and edges such that the endpoints of $e_j$ are $v_j$ and $v_{j+1}$ for all $j \in \ens{1,\dots,k}$. A path with no repeated vertices and edges, with the possible exception $v_1=v_{k+1}$, is called a \emph{simple} path. A path is said to be \emph{closed} when $v_1=v_{k+1}$. A \emph{cycle} is the set of edges of from a simple and closed path. For instance, a loop can be seen as a cycle reduced to a singleton. A graph (or a subgraph) with no cycle is said to be \emph{acyclic}.
%For instance, the singleton set consisting of a loop is a cycle.

A \emph{cut} of a graph $G$ is a set of edges $K$ such that the endpoints of each edge of $K$ are in two distinct connected components of $\overline K$. A \emph{cocycle} is a cut which is minimal for inclusion. (Therefore the deletion of a cocycle exactly increases the number of connected components by one.) An \emph{isthmus} is an edge whose deletion increases the number of connected components. (In other terms, an isthmus is a cocycle reduced to a singleton.) Note that a subgraph $S$ of $G$ is connected if and only if the set $\overline S$ does not contain any cocycle of $G$.

An edge which is neither an isthmus, nor a loop, is said to be \emph{standard}.

\noindent \textbf{Example. }Consider the graph of Figure \ref{exg}. The set $\ens{a,c,d}$ is a cycle but $\ens{a,c,d,g,i,j}$ is not (because the associated path is not simple). The set $\ens{b,h}$ is a cocycle but  $\ens{b,h,i,j}$ is not.

\fig{[scale=1.2]{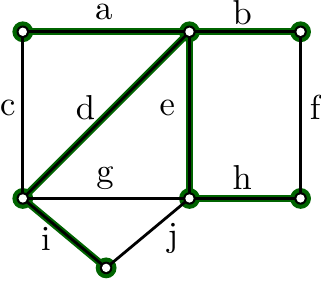}}{A graph equipped with a spanning tree in bold edges}{exg}

A \emph{spanning tree} is a (spanning) subgraph that is a tree, that is to say a connected and acyclic graph. 
Fix $T$ a spanning tree of a graph $G$. The \emph{fundamental cycle} of an external edge $e$ is the only cycle contained in the subgraph $T \cup \{e\}$, i.e. the cycle made of $e$ and the unique path in $T$ linking the endpoints of $e$. Similarly, the \emph{fundamental cocycle} of an internal edge $e$ is the only cocycle contained in $\overline T \cup \sing e$, i.e. the cocycle made of  the edges having exactly one endpoint in each of the two subtrees obtained from $T$ by removing $e$.

\label{found}

\noindent \textbf{Example. }Consider the spanning tree from  Figure \ref{exg}. The fundamental cycle of $j$ is $\ens{d,e,i,j}$, the fundamental cocycle of $e$ is $\ens{e,f,g,j}$.

\subsection{Embeddings}

We now define combinatorial maps as Robert Cori and Toni Machi did in \cite{cori-these,cori-machi}. This extends the notion of planar maps from Section \ref{s:introcartes} p. \pageref{s:introcartes}. A \emph{map} $M = (H,\sigma,\alpha)$ is a finite set of half-edges $H$, a permutation $\sigma$ of $H$ and an involution without fixed point $\alpha$ on $H$ such that the group generated by $\sigma$ and $\alpha$ acts transitively on $H$. A map is \textit{rooted} when one of its half-edges, called the \textit{root}, is distinguished. In this thesis, \emph{all the maps are rooted}.

For every map $M = (H,\sigma,\alpha)$, we define its \emph{underlying graph} as follows: We form a vertex for each cycle of $\sigma$, and an edge for each cycle of $\alpha$. A vertex is incident with an edge if the corresponding cycles have a non-empty intersection. Observes that such a graph is always connected since $\sigma$ and $\alpha$ act transitively on $H$.

\begin{figure}[h!]
\begin{center}
\includegraphics[scale=1]{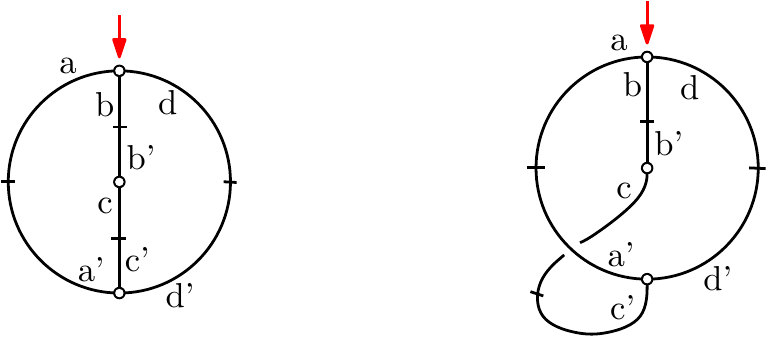}
\end{center}
\caption{Two different embeddings of the same graph.}
\label{fig:embe}
\end{figure}

A \emph{combinatorial embedding} of a connected graph $G$ is a map $M$ such that the underlying graph of $M$ is isomorphic to $G$. Every embedding of $G$ is in correspondence with a \emph{rotation system} of $G$, that is to say a choice of a circular ordering of the half-edges around each vertex of $G$. The rotation system actually corresponds to the above permutation $\sigma$. The permutation $\alpha$ is automatically given by the matching of the half-edges. When an embedding of $G$ is considered, we write the edges of $G$ as pairs of half-edges (for instance $e = \ar h$). 

A map $M = (H,\sigma,\alpha)$ can be graphically represented in the following way: We draw the underlying graph of $M$ in such a manner that the cyclic ordering of the half-edges in counterclockwise order around each vertex matches the corresponding cycle in $\sigma$. If we want to indicate the root, we add an arrow just before\footnote{in counterclockwise order} the root half-edge, pointing to the root vertex. In other terms, if $h_0$ denotes the root, we put an arrow between the half-edges $h_0$ and $\sigma^{-1}(h_0)$. This construction allows us to see the planar maps of Section~\ref{s:introcartes} p.~\pageref{s:introcartes} as combinatorial maps.

For example, Figure \ref{fig:embe} shows two different maps with the same underlying graph.  The left map corresponds to $(H,\sigma,\alpha)$ and the right map to $(H,\sigma',\alpha)$, where $H = \ens{a,a',b,b',c,c',d,d'}$,  $\sigma = (a\,b\,d)(a'\,d'\,c')(b'\,c)$, $\sigma' = (a\,b\,d)(a'\,c'\,d')(b'\,c)$ and \mbox{$\alpha = (a\,a') (b\,b') (c\,c') (d\,d')$} (the permutations are written in cyclic notation). Each of these maps is rooted on the half-edge $a$.

Given a map $M = (H,\sigma,\alpha)$, we say that a half-edge $h_2$ \emph{immediately follows} another half-edge $h_1$ when $h_2 = \sigma \circ \alpha (h_1)$. In Figure \ref{fig:embe}, the half-edge that immediately follows $a$ for the left map is $d'$, while it is $c'$ for the right map.

\subsection{Edge contraction and deletion}

\emph{Edge deletion} is the operation that removes an edge $e$ from the edge set of a graph $G$ but leaves its endpoints unchanged. The resulting graph is denoted by $\delete G e$.

\emph{Edge contraction} is the operation that removes an edge from a graph by merging the two vertices it previously connected. More precisely, given an edge $e$ with endpoints $v$ and $w$ in a graph $G$, the contraction of $e$ yields a new graph $G'$ where $w$ has been removed from the vertex set and any edge which was incident to $w$ in $G$ is now incident to $v$ in $G'$.
% $$E(G')= \left(E(G) \backslash E_w \right) \cup  \ens{ \ens{u,v} | \ens{u,w} \in E(G) \textrm{ and } u \neq w} \cup  \ens{ \ens{v,v} | \ens{w,w} \in E(G)},$$
%where $E_w$ is the (multi)set of edges incident to $w$. 
The resulting graph $G'$ is denoted by $\contract G e$. 

We can also extend the operations of deletion and contraction to maps: Let us consider  an embedding $M = (H,\sigma,\alpha)$ of a graph $G$ and $e = \ens{h_1,h_2}$ an edge of $G$. Let $H'$ be $H \backslash \ens{h_1,h_2}$ and $\alpha'$ the involution $\alpha$ restricted to $H'$. If $e$ is not an isthmus, we define $\delete M e$ as the map $(H',\sigma_d,\alpha')$ where
$$
\sigma_d(h) = \left\{ 
\begin{array}{ll} 
\sigma \circ \sigma \circ \sigma(h) & \textrm{if }(\sigma(h) = h_1\textrm{ and }\sigma(h_1)=h_2)\textrm{ or } (\sigma(h) = h_2\textrm{ and }\sigma(h_2)=h_1), \\
\sigma \circ \sigma(h) & \textrm{if }(\sigma(h) = h_1\textrm{ and }\sigma(h_1) \neq h_2)\textrm{ or } (\sigma(h) = h_2\textrm{ and }\sigma(h_2) \neq h_1), \\
\sigma(h) & \textrm{otherwise}. \\
\end{array}
\right. \hspace*{-0.1cm}
$$
Similarly, if $e$ is not a loop, we define $\contract M e$ as the map $(H',\sigma_c,\alpha')$ where
$$
\sigma_c(h) = \left\{ 
\begin{array}{ll} 
\sigma \circ \sigma (h) & \textrm{if }(\sigma(h) = h_1\textrm{ and }\sigma(h_2)=h_2)\textrm{ or } (\sigma(h) = h_2\textrm{ and }\sigma(h_1)=h_1), \\
\sigma \circ \alpha \circ \sigma (h) & \textrm{if }(\sigma(h) = h_1\textrm{ and }\sigma(h_2) \neq h_2)\textrm{ or } (\sigma(h) = h_2\textrm{ and }\sigma(h_1) \neq h_1), \\
\sigma(h) & \textrm{otherwise}. \\
\end{array}
\right. 
$$
We do not want to define the deletion of an isthmus or the contraction of a loop since it could bring about the disconnection of the map\footnote{Yes, even for a contraction. Take for instance the map $(H,\sigma,\alpha)$ with $H = \ens{a,a',b,b',c,c'}$, $\sigma = (a\,b\,a'\,c) (b')(c')$ and $\alpha = (a\,a') (b\,b') (c\, c')$. If $e = \ar a$, we can check that $\sigma_c = (b)(b')(c)(c')$.}.

One can check that $\delete M e$ (resp. $\contract M e$) is indeed an embedding of $\delete G e$ (resp. $\contract G e$). If $h_1$ is the root of the map $M$, we choose $\sigma_d(h_1)$ (resp. $\sigma_c(h_1)$) as the root of $\delete M e$ (resp. $\contract M e$).

% A graph $H$ is the \emph{minor} of $G$ can be formed from $G$ by contracting and deleting edges\footnote{In the standard definition, one can also delete isolated vertices, but this is not important here.}. In this case we write $G \rightsquigarrow H$.

%Here is a basic property about the minors that we are going to use:
%
%\begin{prop} Let $G$ and $H$ be two graphs with $G \rightsquigarrow H$. 
%\begin{itemize}
%\item If $C$ is a cycle of $G$ such that no edge of $C$ has been deleted in $H$, then $C \cap H$ forms a closed path of $H$.
%\item If $C$ is a cocycle of $G$  such that no edge of $C$ has been contracted in $H$, then $C \cap H$ is a cut for $H$.
%\end{itemize}
%In particular, if $e$ is an edge of $H$ which is a loop (resp. an isthmus) in $G$, then $e$ is a loop (resp. an isthmus) in $H$.
%\end{prop}
%\begin{proof} 1. Let $\Gamma = (v_1,e_1,v_2,\dots,e_{k-1},v_k)$ be a closed path of $G$. One easily checks that every of the following points holds:
%
%(a) If the edge $e$ does not belong to $\Gamma$, then $\Gamma \cap \delete G e = (v_1,e_1,v_2,\dots,e_{k-1},v_k)$ is still a closed path of $G$.
%
%(b) If the edge $e$ does not belong to $\Gamma$, then $\Gamma \cap \contract G e = (v_1,e_1,v_2,\dots,e_{k-1},v_k)$ is a closed path of $G$, even if some repetitions of vertices can occur in $\Gamma \cap \contract G e$.
%
%(c) If the edge $e=e_j$ belongs to $\Gamma$, then $\Gamma \cap \contract G e = (v_1,\dots,v_j=v_{j+1},\dots,e_{k-1},v_k)$ is a closed path of $G$.
%
%\na{Incomplet}
%
%\end{proof}

\section{Activities}

%%%%%%%%%%%%%%%%%%%%%%%%%%%%%%%%%%%%%%%%%%%%%%%%%
\subsection{Back to the Tutte polynomial}
%%%%%%%%%%%%%%%%%%%%%%%%%%%%%%%%%%%%%%%%%%%%%%%%%

Let us  recall the definition to the Tutte polynomial and provide some additional details.

Two parameters are important to define the Tutte polynomial. The first one is the \emph{number of connected components} of a subgraph $S$, denoted by $\cc(S)$. Recall that by convention each subgraph is spanning. This implies that the subgraph of $G$ with no edge has $|V(G)|$ connected components. The second parameter is the \emph{cyclomatic number} of $S$, denoted by $\cycl(S)$. It can be defined as 
\begin{equation} \cycl(S) = \cc(S) + |S| - |V(G)|. \end{equation}
It equals the minimal number of edges that we need to remove from $S$ to obtain an acyclic graph. In particular, $\cycl(S)=0$ if and only if $S$ is a forest.

\begin{defie} The \emph{Tutte polynomial} of a graph $G$ is 
\begin{equation} T_G(x,y) = \sum_{S\textrm{ subgraph of }G}(x-1)^{\cc(S)-\cc(G)}(y-1)^{\cycl(S)},
\label{eq:deftutte}
\end{equation}
where $\cc(S)$ and $\cycl(S)$ respectively denote the number of connected components of $S$ and the cyclomatic number of $S$.
\end{defie}

For example, consider the graph of Figure \ref{fig:extut}. Let us list all its subgraphs with their contributions to the Tutte polynomial: there are one subgraph with no edge (contribution \mbox{$(x-1)^2$}), four subgraphs with one edge (contribution $4\,(x-1)$), five acyclic subgraphs with two edges (contribution $5$), the subgraph $\ens{a,d}$ (contribution $(x-1)\,(y-1)$), four subgraphs with three edges (contribution $4\,(y-1)$) and the whole graph (contribution $(y-1)^2$). Thus the Tutte polynomial of this graph is $$(x-1)^2 + 4 \, (x-1) +  5 + (x-1)\,(y-1) + 4\,(y-1) + (y-1)^2,$$ which can be rewritten as $x^2 + x + x\,y + y + y^2$.

One can easily deduce from \eqref{eq:deftutte} that if the graph $G$ is the disjoint union of two graphs $G = G_1 \uplus G_2$, then the Tutte polynomial of $G$ is the product of the two other Tutte polynomials: \mbox{$T_G(x,y) = T_{G_1}(x,y) \times T_{G_2}(x,y)$}. We can \textit{de facto} restrict our study to connected graphs. \emph{From now on, all the graphs $G$ we consider are connected.}
We also assume that our graphs have at least one edge. (Otherwise, the Tutte polynomial is equal to $1$.)  

Let us  recall the relation of induction satisfied by the Tutte polynomial, due to Tutte himself \cite{tutte54}.

\begin{prop} \label{eq:ind} Let $G$ be a graph and $e$ be one of its edges. The Tutte polynomial of $G$ satisfies:
$$
T_G(x,y) = \left\{ \begin{array}{ll} T_{\contract G e}(x,y) + T_{\delete G e}(x,y) & \textrm{if }e\textrm{ is standard,} \\
x\,T_{\contract G e}(x,y) & \textrm{if }e\textrm{ is an isthmus,} \\
y\, T_{\delete G e}(x,y) & \textrm{if }e\textrm{ is a loop.} 
 \end{array} \right.
$$
\end{prop} 

Since the Tutte polynomial of a graph with one edge is equal to $x$ or $y$, the previous proposition implies by induction the following property.

\begin{core}The Tutte polynomial of any graph has non-negative integer coefficients in $x$ and $y$.
\end{core} 

It is natural to ask if a combinatorial interpretation exists for these coefficients. Tutte found in 1954 a characterization of his polynomial in terms of an "activity" based on a total ordering of the edges. Some decades later, Bernardi gave a similar characterization with a notion of activity this time related to an embedding of the graph. The precise definitions will be given in Section~\ref{sec:activity}.

\subsection{Activities}

Let us formalize the notion of activity.  (The following definitions are not conventional, as the activity of a spanning tree usually denotes the number of active edges.)
% However we reckon them to be relevant. 

An \emph{activity} is a function that maps spanning trees of $G$ on subsets of $E(G)$.
We say that an activity $\psi$ is \emph{Tutte-descriptive} if the Tutte polynomial of $G$ is equal to 
\begin{equation}
T_G(x,y) = \sum_{T\textrm{ spanning tree of }G} x^{|\mathcal I(T)|}\,y^{|\mathcal E(T)|}, 
\label{eq:compact} 
\end{equation}
where $\mathcal I (T)= \psi(T) \cap T$ and $\mathcal E(T) = \psi(T) \cap \overline T$.

An \emph{internal activity} (resp. \emph{external activity}) is a function that maps any spanning tree $T$ of $G$ onto a subset of $T$ (resp. a subset of $\overline T$). We say that an internal activity $\mathcal I$ (resp. external activity $\mathcal E$) can be \emph{extended into a Tutte-descriptive activity} if there exists a Tutte-descriptive activity $\psi$ such that $\psi(T) \cap T = \mathcal I(T)$ (resp. $\psi(T) \cap \overline T = \mathcal E(T)$) for any spanning tree $T$ of $G$. 

The objective of the second part of this thesis is to introduce several families of Tutte-descriptive activities, and to describe a general framework from which each of these activities can be deduced.  

% The purpose of this paper is to yield some examples of edge activity, complete or incomplete, and to give a general framework for these activities. 

%%%%%%%%%%%%%%%%%%%%%%%%%%%%%%%%%%%%%%%%%%%%%%%%%%%%
\section{Three families of activities}
\label{sec:activity}
%%%%%%%%%%%%%%%%%%%%%%%%%%%%%%%%%%%%%%%%%%%%%%%%%%%%

In this section, we introduce three families of activities, all of which are already known. All these families are (pairwise) not equivalent, meaning that no family is included in another. We will prove in Chapter \ref{sec:spec}  that each of these activities is Tutte-descriptive, or can be extended into a Tutte-descriptive activities.

%------------------------------------------------%
\subsection{Ordering activity (Tutte)} 
\label{ss:ord}

The first family of activity was defined by Tutte in \cite{tutte54}, as follows.

Consider $G$ a graph. We equip it with a linear ordering on the edge set. An external (resp. internal) edge of a spanning tree $T$ is said to be \emph{ordering-active} if it is minimal in its fundamental cycle (resp. cocycle). The \emph{ordering activity} is the function that sends every spanning tree $T$ onto the set of its ordering-active edges. This activity naturally depends on the chosen ordering of the edges.

\begin{figure}[h!]
\begin{center}
\includegraphics[scale=1]{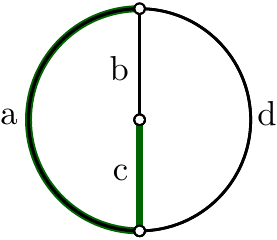}
\end{center}
\caption{A graph with spanning tree $T = \sing {a,c}$.}
\label{fig:extut}
\end{figure}

\noindent \textbf{Example.} Consider the graph of Figure \ref{fig:extut}. The edges are ordered alphabetically, that is to say $a < b < c < d$. With this ordering, the spanning tree  $T = \sing {a,c}$ induces only one internal active edge, $a$, and no external active edge. Indeed, $b$ is not externally active since its fundamental cycle is $\{a,b,c\}$. Similarly, $c$ is not internally active because $b$ belongs to its fundamental cocycle. 

Tutte established that the ordering activity is always Tutte-descriptive. In particular, the sum \eqref{eq:compact}, where $\mathcal I(T)$ and $\mathcal E(T)$ denote the sets of internal and external ordering-active edges of a spanning tree $T$, does not depend on the chosen linear ordering, although the activity clearly does.

%------------------------------------------------%
\subsection{Embedding activity (Bernardi)}
\label{ss:bernardi}

Bernardi defined in \cite{bernardi-tutte} other activities that are well adapted to the notion of maps.
We consider $M_G = (H,\sigma,\alpha)$ an embedding of a graph $G$, that we root on a half-edge denoted by $h_0$. To each spanning tree $T$, we associate a \emph{motion function} $t$ on the set $H$ of half-edges by setting
\begin{equation} 
\label{motionfunction}
t(h) = \left\{ \begin{array}{ll} \sigma(h) & \textrm{if }h\textrm{ is external,} \\ \sigma \circ \alpha (h)  & \textrm{if }h\textrm{ is internal.}    \end{array} \right. 
\end{equation}
If the notation is ambiguous, we will write $t(h;T)$.

\begin{figure}[h!]
\begin{center}
\includegraphics[scale=1]{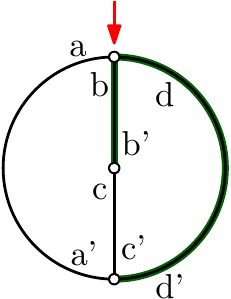} \hspace{3cm} \includegraphics[scale=1]{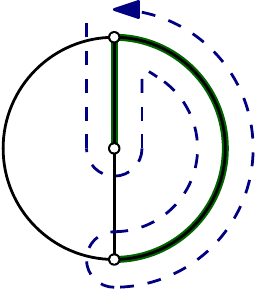}
\end{center}

\caption{Example of an embedded graph equipped with a spanning tree and representation of the corresponding tour.}
\label{fig:exber}
\end{figure}

The motion function characterizes the \emph{tour} of a spanning tree. In informal terms, a tour is a counterclockwise walk around the tree that follows internal edges and crosses external edges (see Figure \ref{fig:exber}). Bernardi proved the following result \cite{bernardi-tutte}.
\begin{prop} 
\label{tourcyclique}
For each spanning tree, the motion function is a cyclic permutation of the half-edges.
\end{prop}
Thus, we can define a linear order on the set $H$ of half-edges, called the ($M_G$,$T$)-order, by setting 
$h_0 < t(h_0) < t^2(h_0) < \dots < t^{|H| - 1}(h_0)$, where $h_0$ is the root of $M_G$. This order can be then transposed on the set of edges: We say that $e = \ar{h_1} < e' = \ar{h_2}$ when $\min(h_1,h_1') < \min(h_2,h'_2)$.

Then, an external (resp. internal) edge is said to be $(M_G,T)$-\emph{active} if it is minimal for the $(M_G,T)$-order in its fundamental cycle (resp. cocycle). The \emph{embedding activity} is the function that associates with a spanning tree $T$ the set of $(M_G,T)$-active edges of $G$.

\noindent \textbf{Example.} Take the embedded graph from Figure \ref{fig:exber}, that we will denote $M_G$, rooted on $a$ and equipped with the spanning tree $T=\sing {\ar b,\ar d}$. The motion function for this spanning tree is the cycle $(a,b,c,b',d,c',a',d')$.  So the half-edges are sorted for the $(M_G,T)$-order as follows: $a < b < c < b' < d < c' < a' < d'$. Thus, the $(M_G,T)$-order for the edges is $\ar a  < \ar b  < \ar c < \ar d $. There is one external active edge, $\ar a $, and one internal active edge, $\ar b$. The edge $\ar c$ (resp. $\ar d$) is not active since $\ar b$ (resp. $\ar a$) is in its fundamental cycle (resp. cocycle). 

It was proven by Bernardi that any embedding activity is Tutte-descriptive, whatever the chosen embedding is.

%------------------------------------------------%
\subsection{DFS activity (Gessel, Sagan)}
\label{ss:dfs}
%------------------------------------------------%

Gessel and Sagan described in \cite{GesselSagan} a notion of external activity  based on the Depth First Search (DFS in abbreviate). 

Consider a graph $G$. We assume that $V(G) = \ens{1,\dots,n}.$ Moreover, we assume that $G$ does not have multiple edges. This allows us to avoid technical details, which can be easily included if needed. Every edge will be denoted by the pair of integers that corresponds to its endpoints, for example $e=\ens{1,2}$.

Given a (not necessarily connected) graph $H$, Algorithm \ref{DFS} computes the \textit{(greatest-neighbor) DFS forest} of $H$, denoted by $\mathcal F(H)$.

\begin{algorithm}[h!]
\caption{DFS forest of a graph}
\label{DFS}
\begin{algorithmic}[5]
\Require  graph $H$.
\Ensure $\mathcal F(H)$, spanning forest of $H$.
\State $\mathcal F(H) \leftarrow \emptyset$;
\While {there is a unvisited vertex} 
	\State \textbf{mark} the least unvisited vertex of $H$ \textbf{as visited}; \label{linedfs}
	\While {there is a visited vertex with unvisited neighbors} 
		\State $v \leftarrow$ the most recently visited such vertex;
		\While{$v$ has a unvisited neighbor} 
			\State $u \leftarrow$ the greatest unvisited neighbor of $v$;
			\State $e \leftarrow \ens{u,v}$;
			\State \textbf{mark} $u$ \textbf{as visited};
			\State $v \leftarrow u$; 
			\State \textbf{add} $e$ in $\mathcal F(H)$;
		\EndWhile
	\EndWhile 
\EndWhile
\State \Return $\mathcal F(H)$
\end{algorithmic}
\end{algorithm}

\noindent \textbf{Informal description.}  We begin by the least vertex. We proceed to the DFS of the graph $H$ that favors the largest neighbors. Each time we move from a vertex to another, we add the edge that we have followed to the DFS forest. When we have visited all the vertices of a connected component, we reset the process, starting from the least unvisited vertex. 

Here we only apply the algorithm to a subgraph $H$  of $G$. An example is shown in Fi\-gure~\ref{fig:dfs}: during the DFS of the subgraph on the left, the vertices will be visited in the order $1,4,6,2,3,5$. The resulting DFS forest is represented on the right. Note that $S$ and $\mathcal F (S)$ have the same number of connected components. 

Given a spanning forest $F$ of $G$, we say that an external edge is \emph{DFS-active} if $$\mathcal F\left(F \cup \ens{e}\right) = F.$$
The \textit{external DFS activity} is the function that sends every spanning forest onto the set of its external DFS-active edges.

For instance, let us go back to the spanning forest $F$ of Figure \ref{fig:dfs} (right). It has two external DFS-active edges: $\ens{1,2}$ and $\ens{5,5}$. On the contrary, the edge $e=\ens{1,6}$ is not DFS-active since $\mathcal F\left(F \cup \ens{e}\right)$ is equal to \mbox{$\ens{\ens{1,6},\ens{6,4},\ens{4,2},\ens{3,5}}$}.

\begin{figure}[h!]
\begin{center}
\includegraphics[scale=1]{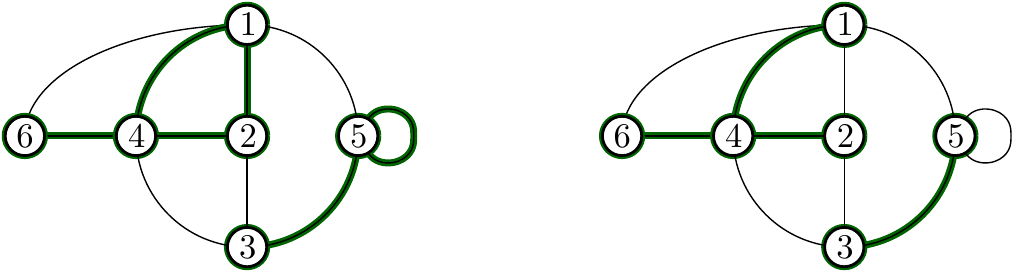}
\end{center}
\caption{Left: a subgraph indicated by thick green edges. Right: Its DFS forest.}
\label{fig:dfs}
\end{figure}

 There is a natural question about Algorithm \ref{DFS}: Given a spanning forest $F$ of $G$, can we describe the set of subgraphs $S$ such that $\mathcal F (S)$ equals $F$? The notion of external DFS activity answers this question.
Indeed, for any subgraph $S$ and for any spanning forest $F$ of $G$, we have $\mathcal F(S) = F$ if and only if \mbox{$S \in [F, F \cup \mathcal E(F)]$}, where $\mathcal E(F)$ denotes the set of external DFS-active edges (see Subsection \ref{sss:interval} p. \pageref{sss:interval} for the definition of intervals).

Let us mention this alternative characterization of externally DFS-active edges (Lemma 3.2 from \cite{GesselSagan}).
\begin{prop} \label{lemGS}
Let $F$ be a spanning forest of $G$. An external edge $e$ is DFS-active if and only if:
\begin{itemize}
\item $e$ is a loop, or
\item $e = \ens{u,v}$ where $v$ is a descendant\footnote{that is to say a vertex in the same component of $u$, visited after $u$} of $u$, and $w > v$, where $w$ is the child of $u$ on the unique path in $F$ linking $u$ and $v$. (We also say that $(w,v)$ is an inversion.)
\end{itemize}
The last case is depicted in Figure \ref{descendant}.
\end{prop}

\fig{[scale = 1.5]{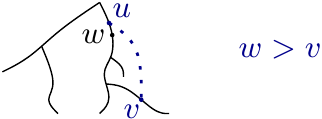}}{A non-loop DFS-active external edge.}{descendant}

Gessel and Sagan proved that the Tutte polynomial of $G$ satisfies:
\begin{equation} \label{DFStut}
T_G(x,y) = \sum_{F\textrm{ spanning forest of }G} (x-1)^{\cc(F)-1} y^{|\mathcal E(F)|},
\end{equation}
where $\mathcal E$ denotes the external DFS activity. We are going to show that the external DFS activity, restricted to spanning trees, can be extended into a Tutte-descriptive activity (cf. Prop. \ref{dfstut}). 

In the same paper, Gessel and Sagan defined a notion of \emph{external activity with respect to NFS} (NFS for Neighbors-first search). We will not detail it but this activity also falls within the scope of this work.
Also note that a notion of edge activity based on Breadth-First Search is conceivable.

\chapter{$\boldsymbol \Delta$-activity}
\label{c:delta}

We are going to introduce a meta-family of Tutte-descriptive activities, named $\Delta$-activities. This family includes all the Tutte-descriptive activities we have seen so far.

%A whole theory will be developed around this notion, going from several characterizations to combinatorial interpretations.

%********************************************************************%

\label{s:alg}
%********************************************************************%

%%%%%%%%%%%%%%%%%%%%%%%%%%%%%%%%%%%%%%%%%%%%%%%%%%%%%%%%%%%%%%%%%%%%%
\section{Decision trees and decision functions}
\label{subsec:dectree}
%%%%%%%%%%%%%%%%%%%%%%%%%%%%%%%%%%%%%%%%%%%%%%%%%%%%%%%%%%%%%%%%%%%%

All the families of activities we saw depended on a parameter: linear order for the ordering activities, embedding for the embedding activities... Similarly, the $\Delta$-activities will depend on a object, which is in a certain sense more general, named \emph{decision tree}.

Consider a graph $G$. A \emph{decision tree} is a perfect binary tree\footnote{A \textit{perfect binary tree} is a binary tree in which every node has $0$ or $2$ children and all leaves are on the same level. Sometimes perfect trees are called full trees.} $\Delta$ with a labelling $V(\Delta) \rightarrow E(G)$
  %(It means that each node of $\Delta$ is labeled with an edge of $G$.) 
%Furthermore, the edges linking a node and its 4 children are labeled by \textbf{S$_i$} (like Standard internal), \textbf{S$_e$} (like Standard external), \textbf{I} (like Isthmus) and \textbf{L} (like Loop), from left to right. 
such that along every path starting from the root and ending on a leaf, the labels of the nodes form a permutation of $E(G)$. 
In particular, the depth of every leaf is $|E(G)|$. An example of decision tree is shown in Figure \ref{fig:ex}.

\begin{figure}[h!]
\begin{center}
\begin{minipage}{0.2 \textwidth}\includegraphics[width= \textwidth]{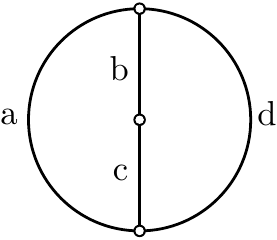}\end{minipage} \hspace{0.1 \textwidth} \begin{minipage}{0.65 \textwidth}\includegraphics[width= \textwidth]{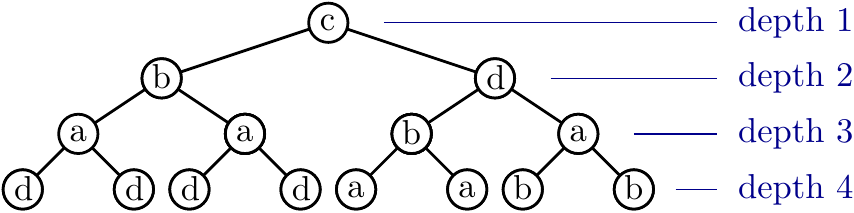}\end{minipage}
\end{center}
\caption{Example of a graph and an associated decision tree.}
\label{fig:ex}
\end{figure}

A \emph{direction} denotes either left or right, which we will write $\ell$ and $r$. Each node of the decision tree bijectively corresponds to a sequence $(d_1,\dots,d_k)$ of directions with $0 \leq k < |E(G)|$: the root node maps to the empty sequence, its children respectively map to the sequences $(\ell)$ and $(r)$, its grand-children map to $(\ell,\ell)$, $(\ell,r)$, $(r,\ell)$ and $(r,r)$, and so on. 
If a node $n$ maps to the sequence $(d_1,\dots,d_k)$, then the label of $n$ will be denoted by $\Delta(d_1,\dots,d_k)$. 
For instance, given the decision tree of Figure \ref{fig:ex}, we have $\Delta(r,\ell) = b$ and $\Delta(\ell,\ell,r) = d$.
By convention, the left or right child of a leaf in $\Delta$ is  null.

This  function $\Delta$ thus maps every sequence $(d_1,\dots,d_k)$ of directions to an edge of $G$ such that 
\begin{equation} \forall\, i < k \ \ \Delta(d_1,\dots,d_i) \neq \Delta(d_1,\dots,d_k),\end{equation}
for every sequence $(d_1,\dots,d_k)$ of directions. Such a function is called a
\emph{decision function}.
The decision function and the decision tree are the same object under different forms. Indeed, given a decision function, it is not difficult to label a decision tree accordingly. Therefore, when a formal definition of a decision tree is needed, we can give the decision function instead.

\section{Algorithm}
%%%%%%%%%%%%%%%%%%%%%%%%%%%%%%%%%%%%%%%%%%%%%%%%%%%%%%

The $\Delta$-activities can be defined in several manners. In this section, we are going to describe an algorithm that compute these activities.

We fix $G$ a connected graph with $m$ edges and $\Delta$ a decision tree for $G$. Given a subgraph $S$ of $G$ (not necessarily a spanning tree), Algorithm \ref{type} outputs a partition ($TypeSe$, $TypeL$, $TypeSi$, $TypeI$) of $E(G)$. In other terms, the algorithm assigns to each edge a type, denoted by \bSe\, (for Standard External), \bL\, (for Loop), \bSi\, (for Standard Internal) or \bI\, (for Isthmus). Thus, if an edge belongs to $TypeSe$ (resp. $TypeL$, $TypeSi$, $TypeI$), we say that this edge has $\Delta$-\emph{type} \bSe\, (resp. \bL, \bSi, \bI) for $S$. If there is no ambiguity, we simply write \emph{type}.

\begin{algorithm}[h!]
\caption{How types are assigned to edges.}
\label{type}
\begin{algorithmic}[5]
% \Algsetup{indent=2em}
\Require $S$ subgraph of $G$.
\Ensure A partition ($TypeSe$, $TypeL$, $TypeSi$, $TypeI$) of $E(G)$.
\State $m \leftarrow $ number of edges in $G$;
\State $TypeSe \leftarrow \emptyset$; $TypeL \leftarrow \emptyset$; $TypeSi \leftarrow \emptyset$; $TypeI \leftarrow \emptyset$; 
\State $n \leftarrow$ root of $\Delta$;
\State $H \leftarrow G$;
\For {$k$ from $1$ to $m$}

	\State $e_{k} \leftarrow$ label of $n$;
	\If{$e_k$ is standard in $H$ \textbf{and} $e_k \notin S$ }

			\State $H \leftarrow \delete H e$;  
			%\Comment{deletion}
			\State \textbf{add} $e_k$ in $TypeSe$;
			\State $n \leftarrow$ left child of $n$;
	\EndIf
	\If{$e_k$ is a loop in $H$}
		\State {\color{darkgray} $H \leftarrow \delete H e$;} 			    
		\label{suppressionfacultative} 
		\Comment{optional (see Prop. \ref{variants})}
		\State \textbf{add }$e_k$ in $TypeL$;
		\State $n \leftarrow$ left child of $n$;
	\EndIf 
	\If{$e_k$ is standard in $H$ \textbf{and} $e_k \in S$}
		 	\State $H \leftarrow \contract H e$;
			\State \textbf{add} $e_k$ in $TypeSi$;
			\State $n \leftarrow$ right child of $n$;
	\EndIf 
	\If{$e_k$ is an isthmus in $H$}
		\State {\color{darkgray} $H \leftarrow \contract H e$ ;}
		  \label{contractionfacultative} 
		\Comment{optional (see Prop. \ref{variants})}
		\State \textbf{add }$e_k$ in $TypeI$;
		\State $n \leftarrow$ right child of $n$;
	\EndIf 
\EndFor
\State \Return ($TypeSe$, $TypeL$, $TypeSi$, $TypeI$)
\end{algorithmic}
\end{algorithm}

\noindent \textbf{Informal description.} We start from the edge that labels the root of $\Delta$. If this edge is standard external or a loop, we assign it the type \bSe\, or \bL, the edge is deleted and we go to the left subtree of $\Delta$. If this edge is standard internal or an isthmus, we assign the type \bSi\, or \bI, the edge is contracted and we go to the right subtree of $\Delta$.  We repeat the process until the graph has no more edge. Figure \ref{schema} illustrates this description.

\fig{[scale=1.4]{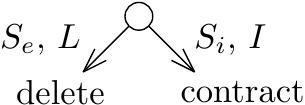}}{Diagram representing a step of Algorithm \ref{type}.}{schema}

We say that an edge is $\Delta$-\emph{active} (or \textit{active})  for $S$ if it has $\Delta$-type \bL\, or \bI. The $\Delta$-activity is the function denoted by $\Act$ that maps each spanning tree onto its set of $\Delta$-active edges.   Warning: an edge with type \bL\, or \bI\, is not necessarily a loop or an isthmus in $S$.  It is an edge that is a loop or an isthmus at some point in Algorithm \ref{type}.

% Let $\Act(S)$  denote the set of active edges for a subgraph $S$ of $G$. \\

\noindent \textbf{Example.} Consider $G$ the graph and $\Delta$ the decision tree of Figure \ref{fig:ex}. The run of the algorithm for $S = \{a,d\}$ is illustrated by Figure \ref{ex:run} (top). We have $\Act(S) = \ens{b,d}$. 

\begin{figure}[h!]

\begin{center}
\hrule \  \\
$H$  at each iteration when   line \ref{suppressionfacultative} and line \ref{contractionfacultative} are present: \\
\vspace{0.1cm}

\includegraphics[scale=0.82]{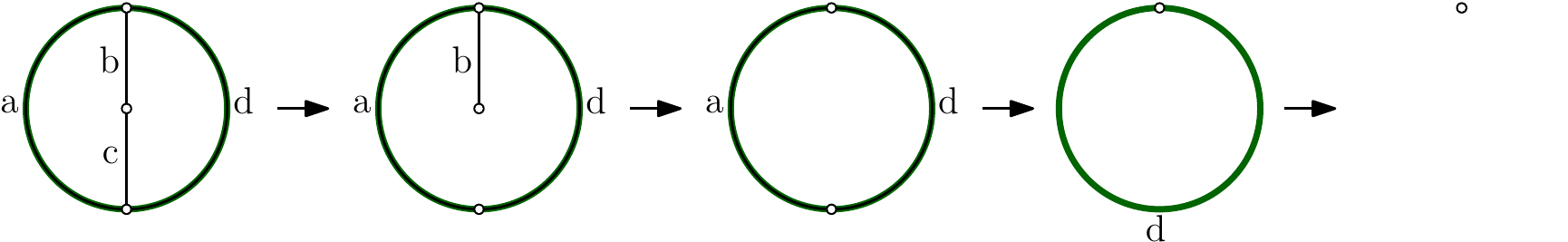} \\
\vspace{0.1cm}
\hrule \  \\
\vspace{0.1cm}

$H$ at each iteration when line \ref{suppressionfacultative} and line \ref{contractionfacultative} are missing: \\
\vspace{0.1cm}
\includegraphics[scale=0.82]{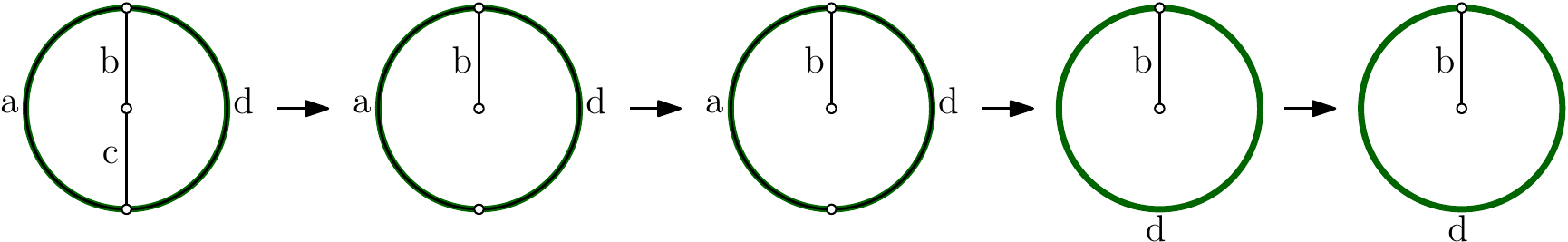}
\vspace{0.1cm}
\hrule
\end{center}

\caption{Run of two versions of Algorithm \ref{type} with $S = \{a,d\}$. One has assigned to $c$ the type \bSe, to $b$ the type \bI, to $a$ the type \bSi, to $d$ the type \bL, in this order. }
\label{ex:run}
\end{figure}

Algorithm \ref{type} is a generalization of the matroid resolution algorithm from Gordon and Traldi  \cite{gordon-traldi} that computes the ordering activity. In Gordon and Traldi's version, the edges $(e_k)$ are considered in a fix order that does not depend on the subgraph. This constitutes a noticeable difference with Algorithm \ref{type} where the sequence $(e_k)$ is given by the decision tree.  \\

%constraint on the order of the resolution prevents to recover most of the Tutte-descriptive activities, like the embedding activities. 

%In order to know in which order the algorithm has visited the edges of $G$,

A sequence of edges $(e_1,\dots,e_m)$ and a sequence of types $(t_1,\dots,t_m)$  can be naturally assigned to each subgraph $S$, where $m$ is the number of edges in $G$, where $e_k$ is the $k$-th edge visited by Algorithm \ref{type} and $t_k$ is the type of $e_k$. This pair of sequences is called the \emph{history} of $S$ and is denoted by $e_1 \ua{t_1} e_2 \ua{t_2} \cdots  \ua{t_{m-1}} e_m \ua{t_{m}}$. For instance, given $G$ and $\Delta$ as in Figure \ref{fig:ex}, the history of $\{a,d\}$ is $c \ua{\mSe} b \ua{\bI} a \ua{\mSi} d \ua{\bL}$.

We can easily prove by induction that for every history $e_1 \ua{t_1} e_2 \ua{t_2} \cdots  \ua{t_{m-1}} e_m \ua{t_{m}}$, the relation
\begin{equation}
\label{ekdelta}
e_{k+1} = \Delta(d(t_1),\dots,d(t_k))
\end{equation}
holds for every $k \in \ens{0,\dots,m-1}$, where $d$ denotes the map $\ens{\mSe,\bL,\mSi,\bI} \rightarrow \ens{\ell,r}$ defined as
\begin{equation} 
d(t) = \left\{\begin{array}{cl} \ell &   \textrm{if }t = \mSe\,\textrm{ or }t =\bL, \\ r &  \textrm{if }t = \mSi\,\textrm{ or }t =\bI. \end{array} \right.\end{equation}

%Here is a straight consequence of Prop. \ref{coincide}.
% \begin{lem} LNote that two histories which coincide on their types must be equal ( of Prop. \ref{coincide}). \\
% \end{lem}

% Unless otherwise indicated, we consider the version of the algorithm in which the lines \ref{suppressionfacultative} and \ref{contractionfacultative} are included.

%%%%%%%%%%%%%%%%%%%%%%%%%%%%%%%%%%%%%%%%%%%%%%%%%%%%%%%%%%%%%%%%%%%%%%%
\section{Description of the Tutte polynomial with $\boldsymbol \Delta$-activity}
%%%%%%%%%%%%%%%%%%%%%%%%%%%%%%%%%%%%%%%%%%%%%%%%%%%%%%%%%%%%%%%%%%%%%%%

%We begin by a proposition which allows us to come closer to the Tutte's terminology.

We are now ready to state the main theorem.

\begin{theoe} \label{charact}
Let $G$ be a connected graph and $\Delta$ a decision tree for $G$. 
The $\Delta$-activity is Tutte-descriptive. In other terms, the Tutte polynomial of $G$ is equal to
\begin{equation}
\label{celuila}
T_G(x,y) = \sum_{T\textrm{ spanning tree of }G} x^{|\mathcal I(T)|}y^{|\mathcal E(T)|},
\end{equation}
where $\mathcal I(T)$ and $\mathcal E(T)$ are respectively the sets of internal $\Delta$-active and external $\Delta$-active edges of the spanning tree $T$.\end{theoe}

Observe that this theorem holds for every decision tree $\Delta$, although the notion of $\Delta$-activity depends on the chosen decision tree.

Before giving a proof of this theorem, let us remark that when the subgraph is a spanning tree, the edges of type \bI/\bL\, coincide with the internal/external active edges.

\begin{prop}
\label{intextact}
Let $G$ be a connected graph and $\Delta$ a decision tree for $G$. 
If the considered subgraph $S=T$ is a spanning tree of $G$, then each edge of type \bI\, is internal and each edge of type \bL\, is external.
\end{prop}
\begin{proof}
It is not difficult to see that at each step of the algorithm, $T \cap H$ is a spanning tree of $H$. Thus, when the algorithm assigns to an edge $e$ the type \bI\, (resp. \bL), $e$ is an isthmus (resp. a loop) in $H$, so it must be inside $T$ (resp. outside $T$).
\end{proof}

 In other terms, for every spanning tree $T$, the edges of type \bI\, are precisely the internal $\Delta$-active edges  and the edges of type \bL\, the external $\Delta$-active edges. Also note that each edge of type \bSe\, is external and each edge of type \bSi\, is internal, even if the subgraph is not a spanning tree.
 
 %  The sets of internally active and externally active edges of a spanning tree $T$ are respectively denoted by $\Intact(T)$ and $\Extact(T)$. 

\begin{proof}[Proof of Theorem \ref{charact}] Let $\tut$ be the polynomial 
$$\tut(x,y)=\sum_ {T\textrm{ spanning tree of }G} x^{|\mathcal I(T)|}y^{|\mathcal E(T)|}.$$
Let us prove by induction on the number of edges that the equality $T_G = \tut$ is true for any graph $G$ and decision tree $\Delta$.
More precisely, we show that $\tut$ satisfies the same relation of induction as the Tutte polynomial (see Proposition~\ref{eq:ind} p. \pageref{eq:ind}).

 Assume that $G$ is reduced to a graph with one edge $e$. Then $e$ is a loop or an isthmus, and $\Delta$ is reduced to a leaf labelled by $e$. We easily check that    $T_G(x,y)=y=\tut(x,y)$ when $e$ is a loop and $T_G(x,y)=x=\tut(x,y)$ when $e$ is an isthmus.

 Now assume that $G$ has at least two edges. We denote by $n_1$ the root node of $\Delta$ and $e_1$ its label.

\begin{leme} \label{respect}
Let $T$ be a spanning tree of $G$. If $e_1$ is external for $T$, define the graph $G'$ as $\delete G {e_1}$
and $\Delta'$ as the subtree rooted on the left child of $n_1$. If $e_1$ is internal, define the graph $G'$ as $\contract G {e_1}$ and $\Delta'$ as the subtree rooted on the right child of $n_1$.
Then the $\Delta$-type in $G$ for $T$ and the $\Delta'$-type in $G'$ for $T \backslash e_1$ of every edge in $G'$ are the same. 
\end{leme}

\begin{proof} 
After the first iteration of Algorithm \ref{type} with graph $G$, decision tree $\Delta$ and input $T$, one can check that the graph $H$ is equal to $G'$, and the node $n$ is the root node of $\Delta'$. (Indeed, by Proposition \ref{intextact}, an edge in a spanning tree is external if and only if it has type \bSe\, or \bL. So we go to the left subtree of $\Delta$ at the first iteration if and only if $e_1$ is external.) These are exactly the values of $H$ and $n$ at the beginning of Algorithm \ref{type} with graph $G'$, decision tree $\Delta'$ and input $T \backslash e_1$. Therefore, in both cases, the algorithm will evolve in the same way from this point on: the edges will be visited in the same order and the types will be identically assigned.
%Taken Algorithm \ref{type} with graph $G'$ (resp. $G$), decision tree $\Delta'$ (resp. $\Delta$) and  input $S  \backslash e_1$ (resp.  $S$), we define the notations:
%\begin{itemize}
%\item $H_j$ (resp. $H'_j$) for the graph $H$ just before the $j$-th iteration,
%\item $n'_j$ (resp. $n_j$) for the node $n_j$, 
%\item $e'_1 \ua{t'_1}  \dots  \ua{t'_{m-1}} e'_m \ua{t'_{m}}$  (resp. $e_1 \ua{t_1}  \dots  \ua{t_{m-1}} e_m \ua{t_{m}}$) for the corresponding history.
%\end{itemize}
%
%Observe that $H_1 = G' = H'_0$. Moreover, the root node $n'_1$ of $\Delta'$ is the left child of $n_1$ in $\Delta$ if $e_1$ has type \bSe\, or \bL, the right child if $e_1$ has type \bSi\, or \bI. Therefore $n_2 = n'_1$.
%Now assume $H'_j = H_{j+1}$ and $n'_j = n_{j+1}$. The equality on the nodes implies the equality on the labels: $e'_j = e_{j+1}$. Since $H'_j = H_{j+1}$, the edge $e'_j = e_{j+1}$ has the same status (standard internal, isthmus, standard external,  loop) in both cases, hence $t'_j = t_{j+1}$,  $H'_{j+1} = H_{j+2}$ and $n'_{j+1} = n_{j+2}$. Then we conclude by proceeding an induction.
\end{proof}

We study now three  cases:

\noindent \textbf{(1) $\boldsymbol{e_1}$ is a standard edge.} Then we can partition the set of spanning trees of $G$ into two disjoint sets: the set of spanning trees \textit{not} containing the edge $e_1$, denoted by $\mathbb T_1$, and the set of spanning trees  containing $e_1$, denoted by $\mathbb T_2$. Let us consider $\Delta_1$ (resp. $\Delta_2$) the subtree rooted on the left child (resp. right child) of $n_1$.

The map $\phi : T \mapsto \delete T {e_1}$ is a bijection from $\mathbb T_1$ to the spanning trees of $\delete G  {e_1}$. (Actually $\phi$ is nothing else than the identity map.) Moreover, for every $T \in \mathbb T_1$, the edge $e_1$ has type \bSe. So according to Lemma \ref{respect}, the bijection $\phi$ preserves the internal and external active edges (for $G$ and $\Delta$ on the one hand, for $G'$ and $\Delta_1$ on the other hand), hence
$$\mathcal T_{\delete G {e_1},\Delta_1} = \sum_{T \in \mathbb T_1} x^{|\mathcal I(T)|}y^{|\mathcal E(T)|}.$$
Similarly, one can prove that 
$$\mathcal T_{\contract G {e_1},\Delta_2} = \sum_{T \in \mathbb T_2} x^{|\mathcal I(T)|}y^{|\mathcal E(T)|}.$$
But we have $\mathcal T_{\delete G {e_1},\Delta_1} = T_{\delete G {e_1}}$ and $\mathcal T_{\contract G {e_1},\Delta_2} = T_{\contract G {e_1}}$ by the induction hypothesis.
So with Proposition \ref{eq:ind}, we get
\begin{eqnarray*} T_G(x,y) & = & T_{\delete G {e_1}}(x,y) + T_{\contract G {e_1}}(x,y) \\ \ & = & \sum_{T \in \mathbb T_1} x^{|\mathcal I(T)|}y^{|\mathcal E(T)|} + \sum_{T \in \mathbb T_2} x^{|\mathcal I(T)|}y^{|\mathcal E(T)|} = \tut(x,y).
\end{eqnarray*}

\noindent \textbf{(2) $\boldsymbol{e_1}$ is an isthmus.} Then $e_1$ must be internal for every spanning tree of $G$ and so the map $T \mapsto \contract T {e_1}$ is a bijection from the spanning trees of $G$ to the spanning trees of $\contract G {e_1}$. 
Moreover, for each spanning tree $T$ of $G$, the edge $e_1$ has $\Delta$-type \bI\, because it is an isthmus. Let $\Delta'$ denote the subtree of $\Delta$ rooted on the right child of $n_1$. By Lemma \ref{respect}, the $\Delta$-type for $T$ and the $\Delta'$-type for $\contract T {e_1}$ are identical  for every edge different from $e_1$.
So from the properties above, we have $$\tut(x,y) = x\, \mathcal T_{\contract G {e_1},\Delta'}(x,y).$$ We conclude by induction and Proposition \ref{eq:ind}.

\noindent \textbf{(3) $\boldsymbol{e_1}$ is a loop.} Exactly the dual demonstration of the case (2).
\end{proof}

%%%%%%%%%%%%%%%%%%%%%%%%%%%%%%%%%%%%%%%%%%%%%%%%%%%%%%%%%%%%%%%%%%%%%%%
\section{Variants of the algorithm}
%%%%%%%%%%%%%%%%%%%%%%%%%%%%%%%%%%%%%%%%%%%%%%%%%%%%%%%%%%%%%%%%%%%%%%%

Algorithm \ref{type} is rather flexible. As stated inside the pseudo-code, it admits variants that lead to the same partition of edges.

\begin{prop} \label{variants}
Removing line \ref{suppressionfacultative} or/and line  \ref{contractionfacultative} from Algorithm \ref{type} does not change its output. In other terms, given a subgraph $S$ of $G$, the types of the edges remain identical if we choose not to delete the edges of type \bL\, or/and not to contract the edges of type \bI.
\end{prop}

\begin{proof} When a graph $G'$ is obtained by deletions and contractions of some edges from a graph $G$, we write $G \rightsquigarrow G'$. In this case, note that if $e$ is a loop (resp. an isthmus) in $G$ and an edge in $G'$, then $e$ is still a loop (resp. an isthmus) for $G'$.
Let $A$ be the variant of Algorithm \ref{type} where the lines   \ref{suppressionfacultative} and \ref{contractionfacultative} are missing and  $A'$ a variant where these lines are potentially present.
The graphs $H$ in the algorithms $A$ and  $A'$ just before the $j$-th iteration will be respectively  denoted by  $H_j$ and  $H'_j$.
% Observe that for $j \leq k$, we have $\tilde{G_j} \rightsquigarrow \tilde{G_k}$ and $\tilde{G_k} \rightsquigarrow \hat{G_k}$.
Given a subgraph $S$ of $G$, let $e_1 \ua{t_1}  \cdots  \ua{t_{m-1}} e_m \ua{t_{m}}$ (resp. $e'_1 \ua{t'_1}  \cdots  \ua{t'_{m-1}} e'_m \ua{t'_{m}}$) be the history of $S$ for the algorithm $A$ (resp. $A'$). Let us prove by induction on $k$ that for each $1 \leq j \leq k$, we have $e_j = e'_j$ and $t_j = t'_j$. 

Assume that the induction hypothesis is true for $k-1$. The edges $e_k$ and $e'_k$ are identical since they are equal to $\Delta(d(t_1),\dots,d(t_{k-1}))$ (see Equation \eqref{ekdelta}). 
Let us prove the equivalence 
$$e_k\textrm{ loop in } H_k \Leftrightarrow e_k\textrm{ loop in } H'_k.$$
 The left-to-right implication is obvious since $H_k \rightsquigarrow H'_k$.
Conversely, if $e_k$ is not a loop in $H_k$, then there exists a cocycle $C$ in $H_k$ including $e_k$. Let us show that for every edge $x$ of $H_k$, we have $x \notin H'_k \Rightarrow x \notin C$, which implies that $C$ is included in $H'_k$ and so $e_k$ is not a loop in $H'_k$. Let $e$ be an edge of $H_k \backslash H'_k$. It must be an edge of the form $e_j$ with $j < k$ which has type \bI\, or \bL. So this edge was an isthmus or a loop in $H_j$, hence also in $H_k$ since $H_j \rightsquigarrow H_k$. Therefore $e$ cannot appear in a non-singleton cocycle of $H_k$, and in particular in $C$.

%$C$ is also in $H'_k$, which implies that $e_k$ is not a loop in $\widehat{H_k}$. Let $e$ be an edge that If an edge of $\widetilde{H_k}$ is not present in $\widehat{H_k}$, then it is an edge of the form $e_j$ with $j < k$ which has type \bL\, or \bI. But such an edge was a loop or an isthmus in $\widetilde{H_j}$, so it is the same in $\widetilde{H_k}$ since $\widetilde{H_j} \rightsquigarrow \widetilde{H_k}$. Therefore an edge of the form $e_j$ with $j < k$ which have type \bL\, or \bI\, cannot appear in a non-singleton cocycle of $\widetilde{H_k}$ and in particular in $C$. Thus, $C$ is included in $\widehat{H_k}$.

Similarly we prove 
$$e_k\textrm{ isthmus in }H_k \Leftrightarrow e_k\textrm{ isthmus in }H'_k.$$

So if $e_k$ is respectively a loop, an isthmus, a standard external edge, a standard internal edge in  $H_k$, then it will be a loop, an isthmus, a standard external edge, a standard internal edge in $H'_k$; hence $t_k = t'_k$.\end{proof}

Each version of Algorithm \ref{type} can be of interest: For implementation, it would be better to perform a minimum number  of operations and consequently choose the variant where the edges of type \bL\, or \bI\, remain untouched. From a theoretical point of view, deleting each edge of type \bL\, and contracting each edge of type \bI\, can facilitate the proofs. For example, with this version, it is easy to see that $e_m$ has necessarily type \bL\, or \bI\, since at the last iteration $e_m$ is the only edge of $H$, so must be an isthmus or a loop.

Moreover, Algorithm \ref{type} can  be adapted differently depending on the context. For instance, Algorithm \ref{algint} allows us to compute the internal $\Delta$-active edges of a spanning tree, the external active edges being omitted. This algorithm will be useful to connect the so-called \textit{blossoming activities} with $\Delta$-activities.

\begin{algorithm}[h!]
\caption{Another way to compute the set of internal active edges.}
\label{algint}
\begin{algorithmic}[5]
% \Algsetup{indent=2em}
\Require $T$ spanning tree of $G$.
\Ensure A subset $IntAct$ of $E(G)$.
\State $m \leftarrow $ number of edges in $G$; $IntAct \leftarrow \emptyset$; 
\State $n \leftarrow$ root of $\Delta$; $H' \leftarrow G$
\For {$k$ from $1$ to $m$}

	\State $e_{k} \leftarrow$ label of $n$;
	\If{$e_k \notin T$}
			\State $H' \leftarrow \delete {H'} e$;  
			\State $n \leftarrow$ left child of $n$;
	\Else
   			\State $n \leftarrow$ right child of $n$;
	\EndIf

	\If{$e_k$ is an isthmus in $H'$}
		\State \textbf{add }$e_k$ in $IntAct$;
	\EndIf 
\EndFor
\State \Return $IntAct$
\end{algorithmic}
\end{algorithm}

\noindent \textbf{Informal description.} The principle is the same as in Algorithm \ref{type}, except that the internal edges are never contracted.

\begin{prop} \label{prop:algint}
For any spanning tree $T$ of $G$, Algorithm \ref{algint} outputs the set of edges of type~\bI.
\end{prop}
\begin{proof}
On the  one hand, we consider the version of Algorithm \ref{type} where the edges of type \bI\, are not contracted and where the edges of type $\bL$ are deleted. Thus, the graph $H$ at the $k$-th iteration is obtained by contracting all theedges $e_i$ of type \bSi\, and by deleting all external edges $e_i$, with $i \in \ens{1,\dots,k-1}$\footnote{By Proposition \ref{intextact}, an edge is external  in a spanning tree if and only if it has type \bSe\, or \bL.}. On the other hand, in Algorithm \ref{algint}, the graph $H'$ at the $k$-th iteration is obtained by deleting all the external edges of the form $e_i$, with $i \in \ens{1,\dots,k-1}$. So $H'$ differs from $H$ at this moment only by some edge contractions (which do not involve $e_k$). Therefore, $e_k$ is an isthmus in $H'$ if and only if $e_k$ is an isthmus in $H$. This means that $e_k$ belongs to the output of Algorithm \ref{algint} if and only if $e_k$ has type \bI.
\end{proof}

%********************************************************************%
%\chapter {Edge ordering and $\boldsymbol \Delta$-activity}
%\label{s:ord}
%********************************************************************%

%%%%%%%%%%%%%%%%%%%%%%%%%%%%%%%%%%%%%%%%%%%%%%%%%%%%%%%%%%%%%%%%%%%%%%%%%%%
\section{$\boldsymbol{(\Delta,S)}$-ordering}
%%%%%%%%%%%%%%%%%%%%%%%%%%%%%%%%%%%%%%%%%%%%%%%%%%%%%%%%%%%%%%%%%%%%%%%%%
It turns out that we also can define the notion of $\Delta$-activity by using fundamental cycles and cocycles, as Tutte and Bernardi did (see p. \pageref{found} for the definition of fundamental cycle/cocycle). This is the purpose of this section.

Consider a graph $G$ with decision tree $\Delta$. Let $e_1 \ua{t_1} e_2 \ua{t_2} \cdots  \ua{t_{m-1}} e_m \ua{t_{m}}$ denote the history of a subgraph $S$ of $G$. We define \textit{the $(\Delta,S)$-ordering} on the edge set of $G$ by setting:
$$e_1 < e_2 < \dots < e_m.$$
It corresponds to the visit order of the edges in Algorithm \ref{type}. This notion will be especially used for the subgraphs $S=T$ that are spanning trees.

Given a decision tree $\Delta$ and a spanning tree $T$ of a graph $G$, the $(\Delta,T)$-ordering can be easily constructed from $\Delta$ and $T$: Start with the root node of $\Delta$. If the edge label is external, go  down to the left child. If the edge is internal, go down to the right child. Repeat this operation until joining a leaf node. The sequence of node labels gives the $(\Delta,T)$-ordering.

The explanation is simple. It relies on 
%the combination of equation
$\eqref{ekdelta}$ and the fact that for all spanning trees, an edge is external if and only it has type \bSe\, or \bL (see Proposition \ref{intextact}).

\noindent \textbf{Example.} Consider the graph and the decision tree $\Delta$ from Figure \ref{fig:ex}, with the spanning tree $T=\ens{a,c}$. By following the path corresponding to the sequence of directions $(r,\ell,\ell)$, we can see that the $(\Delta,T)$-ordering is equal to $c < d < b < a$.

%%%%%%%%%%%%%%%%%%%%%%%%%%%%%%%%%%%%%%%%%%%%%%%%%%%%%%%%%%%%%%%%%%%%%%%%%
%\section{Fundamental cycles and cocycles}
%%%%%%%%%%%%%%%%%%%%%%%%%%%%%%%%%%%%%%%%%%%%%%%%%%%%%%%%%%%%%%%%%%%%%%%%%

The following proposition is the first step to prove the link between the $\Delta$-activity and some activities from Section \ref{sec:activity}.
\begin{prop} 
\label{maximal}
Consider a graph $G$ with a decision tree $\Delta$ and a spanning tree $T$. An external (resp. internal) edge $e$ is $\Delta$-active if and only if it is maximal for the $(\Delta,T)$-ordering in its fundamental cycle (resp. cocycle).
\end{prop}

\noindent \textbf{Remark. }The active edges are here characterized by maximality in their  fundamental cycles/cocycles, and not by minimality as in Tutte's or Bernardi's works. We cannot simply reverse the order so that  "maximal" becomes "minimal". Indeed, we would need the reverse order to correspond to a $(\Delta,T)$-ordering, which seldom happens (except for the ordering activity). This is related to the notion of tree-compatibility, which is treated in the next subsection.

Before proving this proposition, let us state a lemma that will be used on several occasion in this part of the thesis. 
\begin{leme}
\label{cyclecocycle}
For each subgraph $S$ of $G$, we have the following properties:
\begin{enumerate}
	\item For each edge $e$ of type \bL, there exists a cycle in $G$ which consists of $e$ and edges of type \bSi. Moreover, $e$ is maximal in this cycle for the $(\Delta,S)$-ordering.
	\item For each edge $e$ of type \bI, there exists a cocycle of $G$ which consists of $e$ and edges of type \bSe. Moreover, $e$ is maximal in this cocycle for the $(\Delta,S)$-ordering.
\end{enumerate}
\end{leme}
\begin{proof} Here we consider the variant of Algorithm \ref{type} where the edges of type \bI\, are \textit{not} contracted and the edges of type \bL\, are \textit{not} deleted.

(1) Suppose that $e=e_k$, i.e. $e$ was visited at the $k$-th iteration of the algorithm. Then let us prove by a decreasing induction on $j \in \{k,\dots,1\}$ that the graph $H$ at the $j$-th iteration contains a path linking the endpoints of $e$ only made of edges of type \bSi\, and visited before $e$. 
For the base case $j=k$, $e_k$ is a loop, so the empty path matches. Then assume that the induction hypothesis is true for some $j \in \{k,\dots,2\}$. The only operation which could ensure that the endpoints of $e$ can be linked with edges of type \bSi\, after the $(j-1)$-th iteration, but not before, is edge contraction. But a contracted edge must have type \bSi. So, whatever the type of $e_{j-1}$ is, there will be always, at the $(j-1)$-th iteration, a path made of edges of type \bSi\, and visited before $e$ linking the endpoints of $e$. So the induction step holds: The resulting path with the edge $e$ forms the expected cycle.

(2) Similar to point (1), in a dual way.
\end{proof}

We can now tackle the proof of Proposition \ref{maximal}.

\begin{proof}[Proof of Proposition \ref{maximal}] 1. Assume that the external edge $e_j$ is maximal in its fundamental cycle $C$. We  use a version of Algorithm \ref{type} where edges of type \bI\, are contracted. Except $e_j$, the fundamental cycle $C$ of $e_j$ is made of internal edges, so edges with type \bSi\, or \bI. Thus, at the $j$-th iteration, the edges other than $e$ have been contracted, which implies that $e_j$ is a loop. Therefore $e_j$ has type $\bL$. In other terms, it is an external active edge.

2. Assume that an external edge $e_j$ is active, that is to say it has type \bL. By Lemma \ref{cyclecocycle}, there exists a cycle $C$ made of $e$ and edges of type \bSi. Since the edges of type \bSi\, are internal, $C$ is the fundamental cycle of $e$. Lemma \ref{cyclecocycle} also claims that $e$ has been visited last in $C$: it means that $e$ is maximal in $C$ for the $(\Delta,T)$-ordering.

The case where $e$ is internal can be processed in the same way.
\end{proof}

%%%%%%%%%%%%%%%%%%%%%%%%%%%%%%%%%%%%%%%%%%%%%%%%%%%%%%%%%%%%%%%%%%%%%%%
\section{Order map and tree-compatibility}
%%%%%%%%%%%%%%%%%%%%%%%%%%%%%%%%%%%%%%%%%%%%%%%%%%%%%%%%%%%%%%%%%%%%%%%%%

If we want to prove that a specific activity (for example an embedding activity) is a $\Delta$-activity, we need to define the adequate decision tree. However, building this decision tree can be rather tricky. In this subsection, we are going to give a combinatorial condition that guarantees the existence of such a decision tree. 

Fix a graph $G$. An \emph{order map} is a map from the set of spanning trees of $G$ onto the set of total orders of $E(G)$ (or equivalently the permutations of $E(G)$). Given an order map $\phi$ and a spanning tree $T$, we denote by $\phi_k(T)$ the $k$-th smallest edge in this order.

An order map $\phi$ is said to be \emph{tree-compatible} if we can construct a decision tree $\Delta_\phi$ such that for every spanning tree $T$, the $(\Delta_\phi,T)$-ordering coincides with $\phi(T)$. In other terms, a tree-compatible order map is such that for each $k \in \ens{1,\dots,|E(G)|}$, the edge $e_k$, defined by Algorithm \ref{type} with input $T$ for some decision tree $\Delta$, is equal to $\phi_k(T)$. 

\noindent \textbf{Example}: Consider the graph from Figure \ref{fig:ex} with this order map :
\begin{equation} \label{exom} \phi : \begin{array}{ccc}
\ens{a,b} & \mapsto & c < b < a < d \\
\ens{a,c} & \mapsto & c < d < b < a \\
\ens{b,c} & \mapsto & c < d < b < a \\
\ens{b,d} & \mapsto & c < b < a < d \\
\ens{c,d} & \mapsto & c < d < a < b
\end{array} .
\end{equation}
The order map $\phi$ is tree-compatible: If we denote by $\Delta_\phi$  the decision tree of Figure~\ref{fig:ex}, then we can easily check that $\phi(T)$ and the
$(\Delta_\phi,T)$-ordering are the same for all spanning trees $T$.
We can notice that the decision tree is not unique. For example, the same decision tree fits if we replace the leftmost subtree with three nodes, namely \begin{tabular}{l}\includegraphics[height=20px]{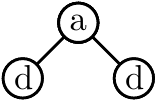}\end{tabular}, by \begin{tabular}{l}\includegraphics[height=20px]{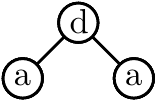}\end{tabular}.

% In other terms, if $T$ and $T'$ have the same internal and external edges inside $\ens{\phi_1(T),\dots,\phi_k(T)}$, then the $k+1$ smallest edges for $\phi(T)$ and for $\phi(T')$ are the same.

The following theorem gives a characterization of tree-compatible order maps.

\begin{theoe} \label{treecompatible}
An order map $\phi$ is tree-compatible if and only if 
%Recall that $\phi_j(T)$ the $j$-th smallest edge for $\phi(T)$. 
for all spanning trees $T$ and $T'$ and for every $k \in \ens{0,\dots,|E(G)|-1}$ the implication 
\begin{multline} \label{tc}
T \cap \ens{\phi_1(T),\dots,\phi_k(T)} = T' \cap \ens{\phi_1(T),\dots,\phi_k(T)} \\ \Longrightarrow \ \forall j \in \ens{1,\dots,k+1} \ \, \phi_j(T) = \phi_j(T')
\end{multline} holds.

\end{theoe}

Note the case $k = 0$ is significant : it implies that $\phi_1(T)$ is the same edge for all spanning trees $T$.

\noindent \textbf{Back to the previous example}: Consider the order map $\phi$ defined by \eqref{exom} and set $T = \ens{a,c} $ and $T' = \ens{c,d}$. We have $\phi_1(T) = c$ and $c$ belongs to $T$ and $T'$. So for $k = 1$, the antecedent of \eqref{tc} holds. Thus the order map coincides one step further: \mbox{$\phi_2(T) = \phi_2(T')$} ($=d$). But for $k=2$, the antecedent of \eqref{tc} does not hold since $\phi_2(T) = d \notin T$ and $d \in T'$: no inferences can be made from this observation.

\begin{proof} \textbf{Left-to-right implication.} Let $\Delta$ be a decision tree and $\phi$ be the order map that sends each spanning tree $T$ onto the $(\Delta,T)$-ordering. Let us prove by induction on $k$ that \eqref{tc} is true for all spanning trees $T$ and $T'$. It holds for $k = 0$ since $\phi_1(T)$ is for all spanning trees $T$ the label of the root node of $\Delta$ (which is the first edge we visit in Algorithm \ref{type}). Assume now that 
$$T \cap \ens{\phi_1(T),\dots,\phi_{k+1}(T)} = T' \cap \ens{\phi_1(T),\dots,\phi_{k+1}(T)}.$$ By the induction hypothesis,  \eqref{tc} holds, hence for all $j \in \ens{1,\dots,k+1}$ we have $\phi_j(T) = \phi_j(T')$. Moreover, by  \eqref{ekdelta}, we have $\phi_{k+2}(T) = \Delta(d^T_1,\dots,d^T_{k+1})$, where $d^T_i = \ell$ when $\phi_i(T)$ is external in $T$ and $d^T_i = r$ when $\phi_i(T)$ is internal in $T$. (Recall that the type an edge $e_i$ in a spanning tree is equal to \bSe\, or \bL\, when $e$ is external, and is equal to \bSi\, or \bI\, when $e$ is internal.) Of course, it also holds if we replace $T$ by $T'$. But by assumption, for every $i \in \, \ens{1,\dots,k+1}$, the edge $\phi_i(T)$ is internal in $T$ if and only if $\phi(T')$ is internal in $T'$. Consequently, we have $\phi_{k+2}(T) = \phi_{k+2}(T')$.

\textbf{Right-to-left implication.} Let $\phi$  be an order map that satisfies \eqref{tc} for  all spanning trees $T$ and $T'$ and for every $k \in \ens{0,\dots,|E(G)|-1}$.
As indicated in Subsection \ref{subsec:dectree}, we can define a decision function instead of a decision tree in order to prove that $\phi$ is tree-compatible. 
%, since these objects have the same data. 

\textbf{(1) Inductive definition of the decision function.} Consider a sequence of directions $(d_1,\dots,d_{k-1})$ with $k \in \ens{1,\dots,|E(G)|}$. We suppose by induction that for $i \in \ens{1,\dots,k-1}$ the edge $\Delta_\phi(d_1,\dots,d_{i-1})$, denoted by $\eta_{i}$, is well defined. Given a spanning tree $T$ and  $k \in \ens{1,\dots,|E(G)|}$, we denote by $P(T,k)$ the property
\begin{equation*} 
\ens{ j \in \ens{1,\dots,k-1} \  |  \  \eta_j \in T } = \ens{ j \in \ens{1,\dots,k-1} \  |  \ d_j = r }.
\end{equation*}
Always by induction, we assume that for every $i \in \ens{1,\dots,k-1}$
  if there exists a spanning tree $T$ that satisfies $P(T,i)$, then $\eta_i$ equals $\phi_i(T)$.
  (The base case of the induction is embodied by the case $k=1$. Indeed, when $k=1$, the set $\ens{1,\dots,k-1}$ is empty : no induction hypothesis is assumed.) 

If there exists a spanning tree $T$ that satisfies $P(T,k)$, we define $\Delta_\phi(d_1,\dots,d_{k-1})$ as $\phi_k(T)$. Two points have to be checked :
\begin{enumerate}
\item[(a)] $\Delta_\phi(d_1,\dots,d_{k-1})$ is  different from $\eta_1$, $\dots$, $\eta_{k-1}$. For every $i \in \ens{1,\dots,k-1}$, the property $P(T,i)$ holds because $P(T,k)$ holds. So by the induction hypothesis, we have $\eta_i = \phi_i(T)$, which is well different from $\Delta_\phi(d_1,\dots,d_{k-1}) = \phi_k(T)$.

\item[(b)] The definition of $\Delta_\phi(d_1,\dots,d_k)$ does not depend on the chosen spanning tree. Suppose that there exist two spanning trees $T$ and $T'$ such that $P(T,k)$ and $P(T',k)$ hold. Then $\ens{ i \in \ens{1,\dots,k-1} \  |  \  \eta_i \in T } = \ens{ i \in \ens{1,\dots,k-1} \  |  \  \eta_i \in T'}$. By the induction hypothesis,e have $\eta_i = \phi_i(T)$ for all $i \in \ens{1,\dots,k-1}$, hence $T \cap \ens{\phi_1(T),\dots,\phi_{k-1}(T)} = T' \cap \ens{\phi_1(T),\dots,\phi_{k-1}(T)}$. By \eqref{tc}, we have $\phi_k(T) = \phi_k(T')$. 
\end{enumerate}

If there exists no spanning tree $T$ such that $P(T,k)$ is true, we arbitrarily define $\Delta_\phi(d_1,\dots,d_{k-1})$ as an edge different from $\eta_1$, $\dots$, $\eta_{k-1}$. Here there is no importance for the choice of the edge, since the corresponding branch of the decision tree will never  be  visited.

%\textbf{(2)}  For any spanning tree $T$, we claim that for $k \in \ens{1,\dots,|E(G)|}$, the edge $e_k$ (defined by Algorithm \ref{type}) corresponds to $\Delta_\phi(d_1,\dots,d_{k-1})$, where the direction $d_i$ is defined for $i \in \ens{1,\dots,k-1}$ as:
%\begin{equation} \label{defdi}
%d_i = \left\{\begin{array}{cl} \ell &   \textrm{if }e_i\textrm{ is external for }T, \\ r & \textrm{if }e_i\textrm{ is internal for }T. \end{array} \right.\end{equation} 
%Indeed, when the considered subgraph is a spanning tree, the type $t_i$ of the edge $e_i$ is equal to \bSe\, or \bL\, when $e_i$ is external, and is equal to \bSi\, or \bI\, when $e_i$ is internal. The point (2) is then proved by using relation \eqref{ekdelta}.

% The result for $k=0$ is obvious. If $k>0$, by  induction hypothesis, we only need to prove that $d_{k-1}$ corresponds to \eqref{defdi}. By Prop. \ref{intextact}, being external for a spanning tree is equivalent to having type \bSe\, or \bL. So if $e_{k-1}$ is external, it means that $n_k$ is the left child of $n_{k-1}$, or equivalently $d_{k-1} = \ell$. In the same way, if $e_i$ is internal for $T$, then $d_{k-1} = r$, which proves the point (2).

\textbf{(2) The $\boldsymbol{(\Delta_\phi,T)}$-ordering coincides with $\boldsymbol \phi$.} Let $k \in \ens{1,\dots,|E(G)|}$ and $T$ be a spanning tree. By  \eqref{ekdelta}, the edge $e_k$ (set by Algorithm \ref{type} with decision tree $\Delta_\phi$ and input $T$) equals $\Delta_\phi(d_1,\dots,d_{k-1})$, the direction $d_i$ being defined for $i \in \ens{1,\dots,k-1}$ as:
\begin{equation} \label{defdi}
d_i = \left\{\begin{array}{cl} \ell &   \textrm{if }e_i\textrm{ is external for }T, \\ r & \textrm{if }e_i\textrm{ is internal for }T. \end{array} \right.\end{equation}
Hence,
\begin{align*}
\ens{ i \,  \in \{1,\dots,k-1\} \  |  \  \Delta_\phi(d_1,\dots,d_{i-1}) \in T } & = \ens{i  \,  \in \{1,\dots,k-1\} \  |  \ e_i \in T } \\
\ & = \ens{i \,  \in \{1,\dots,k-1\} \  |  \ d_i = r }.
\end{align*}
So the spanning tree $T$ satisfies $P(T,k)$. By definition of $\Delta_\phi$, we have $$e_k = \Delta_\phi(d_1,\dots,d_{k-1}) = \phi_k(T),$$ which proves the tree-compatibility of $\phi$. \end{proof}

 % \noindent \textbf{Remark.}  Conversely, one can prove that for every decision tree $\Delta$, the order map that sends each spanning tree $T$ onto the corresponding ($\Delta$,$T$)-ordering satisfies \eqref{tc}. We did not mention it inside Theorem \ref{treecompatible} since the reverse is not very useful.

The following corollary will be helpful to prove that the activities from Section~\ref{sec:activity} p.\pageref{sec:activity} are Tutte-descriptive. Observe that its statement does not involve Algorithm \ref{type}.

\begin{core} \label{cor:act} Let $\phi$ be an order map. Assume that $\phi$ is tree-compatible, i.e. for all spanning trees $T$ and $T'$ and for $k \in \ens{0,\dots,|E(G)|-1}$ such that $T \cap \ens{\phi_1(T),\dots,\phi_k(T)} = T'~\cap~\ens{\phi_1(T),\dots,\phi_k(T)}$, we have $\phi_i(T)=\phi_i(T')$ for each $i \in \ens{1,\dots,k+1}$. 

The activity that maps any spanning tree $T$ onto the set of edges that are maximal for $\phi(T)$ in their fundamental cocycles/cycles is a $\Delta$-activity. Therefore it is Tutte-descriptive.
\end{core}
\begin{proof} By the definition of tree compatibility, there exists a decision tree $\Delta_\phi$ such that  the $(\Delta_\phi,T)$-ordering is  $\phi(T)$ for every spanning tree $T$. Furthermore, Proposition \ref{maximal} states that an edge is $\Delta_\phi$-active if and only if it is maximal in its fundamental cycle/cocyle for the $(\Delta_\phi,T)$-ordering, that is, $\phi(T)$. Then we conclude thanks to Theorem~\ref{charact}.
\end{proof}

%********************************************************************%
\chapter{Partition of the subgraph poset into intervals }
\label{s:partition}
\label{c:partition}
%********************************************************************%

Crapo discovered in \cite{crapo} that the ordering activity induces a natural partition of the poset of the subgraphs with some nice properties. Bernardi, Gessel and Sagan also defined a similar partition based on their respective notion of activities. In this section, we  prove the universality of this partition by extending it to the $\Delta$-activities.

%%%%%%%%%%%%%%%%%%%%%%%%%%%%%%%%%%%%%%%%%%%%%%%%%%%%
\section{Introduction on an example}
%%%%%%%%%%%%%%%%%%%%%%%%%%%%%%%%%%%%%%%%%%%%%%%%%%%%

We are going to introduce the different properties of this chapter throughout an example. Consider $G$ the graph and $\Delta$ the decision tree of Figure \ref{fig:ex} p.~\pageref{fig:ex}. The following table lists the types of edges for all subgraphs of $G$ (described by their sets of edges).

\begin{center}
\noindent \begin{tabular}{|c|c|c|c|c|}
\hline Subgraphs & Type of $a$ & Type of $b$ & Type of $c$ & Type of $d$ 
\\
\hline

$\emptyset$, $\{b\}$, $\{d\}$, $\boldsymbol{\{b,d\}}$  & \bSe & \bI & \bSe & \bI \\
\hline
$\{a\}$, $\boldsymbol{\{a,b\}}$, $\{a,d\}$, $\{a,b,d\}$  & \bSi & \bI & \bSe & \bL \\
\hline
$\{c\}$, $\boldsymbol{\{a,c\}}$  & \bI & \bSe & \bSi & \bSe \\
\hline
$\boldsymbol{\{b,c\}}$, $\{a,b,c\}$ & \bL & \bSi & \bSi & \bSe \\
\hline
$\boldsymbol{\{c,d\}}$, $\{a,c,d\}$, $\{b,c,d\}$, $\{a,b,c,d\}$ & \bL & \bL & \bSi & \bSi \\
\hline
\end{tabular}
\end{center}
\vspace{0.1cm}

We have gathered in each line the subgraphs that share the same partition of edges. We can observe that the set of subgraphs inside any line is a subgraph interval. More particularly, given $S$ and $S'$ two subgraphs taken from the same line, $\diffs S S'$ is only made of active edges (for $S$ \emph{and} for $S'$). This will be proved in Proposition \ref{samehistory}. Furthermore, in each line we can find exactly one spanning tree, indicated in bold. This means that a spanning tree is included in each part of the partition of subgraphs, as stated in Theorem \ref{partition}. Finally, observe that if we put a weight $x$ per edge of type \bI\, and a weight $y$ per edge of type \bL, then by summing over all the lines of the table we get
$$x^2 + x\,y + x + y + y^2,$$
which is exactly the Tutte polynomial of $G$.  This is a consequence of Proposition \ref{rouge}. All these properties can be observed on Figure  \ref{fig:part}.

\begin{figure}[h!]
\begin{center}
\includegraphics[width=\textwidth]{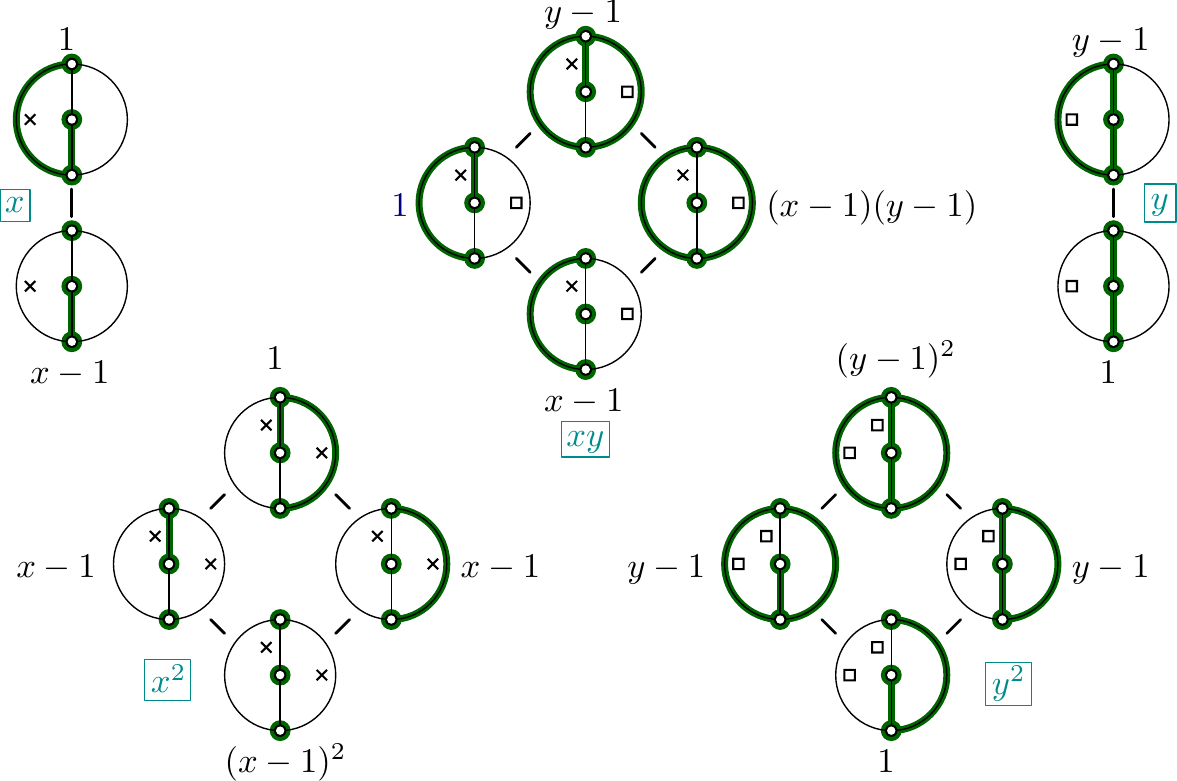}
\end{center}
\caption{Partition of the subgraphs with respect to the relation $\sim$, given the graph and decision tree from Figure \ref{fig:ex}. The crosses indicate the edges with type \bI\, and the squares the edges with type \bL. The contribution of each subgraph to the Tutte polynomial is also mentioned, as well as the global  contribution of each part of the partition.}
\label{fig:part}
\end{figure}

%%%%%%%%%%%%%%%%%%%%%%%%%%%%%%%%%%%%%%%%%%%%%%%%%%%%
\section{Equivalence relation}
%%%%%%%%%%%%%%%%%%%%%%%%%%%%%%%%%%%%%%%%%%%%%%%%%%%%

% In the previous subsection, we showed the characterization  of the Tutte polynomial with $\Delta$-activity (see Th. \ref{charact}) by induction. We are going to give an alternative proof, based on a combinatorial interpretation, which allows to have a better understanding about the $\Delta$-activity. In particular, we are going to see that there is a natural partition of the subgraphs with respect to the decision tree.

For the rest of the chapter, we fix $G$ a graph and $\Delta$ a decision tree. Given two subgraphs $S$ and $S'$ of $G$, we say that $S$ and $S'$ are \emph{equivalent} if they share the same history. In this case, we write $S \sim S'$. 
This is obviously an equivalence relation. Here are some other characterizations of this relation.

\begin{prop}
\label{samehistory}
Let $S$ and $S'$ be two subgraphs of $G$. Denote by $\Act(S)$ the set of active edges of $S$. The following properties are equivalent:
\begin{enumerate}
\item [(i)] $S \sim S'$,
\item [(ii)] the subgraphs $S$ and $S'$ induce the same partition of $E(G)$,
\item [(iii)] the subgraphs $S$ and $S'$ yield the same set of edges with type \bSe\, and the same set of edges with type \bSi,
\item [(iv)] $\diffs S {S'} \subseteq \Act(S)$,
\item [(v)] there exists $R \subseteq \Act(S)$ such that $S' = \diffs S R$.
\end{enumerate}

%For each $S \subseteq E(G)$ and each $R \subseteq \Act(S)$, each edge of $G$ has the same type in $S$ or in $S \, \triangle \, R$. In other terms, removing from $S$ or adding to $S$ an active edge does not change the types of edges, nor \emph{a fortiori} the set of active edges. 

\end{prop}
\begin{proof} \textbf{(i) $\boldsymbol{ \Rightarrow }$ (ii) $\boldsymbol{ \Rightarrow }$ (iii)}: Trivial. \\
\noindent \textbf{(iii) $\boldsymbol{ \Rightarrow }$ (iv)}: Let the set of edges with type \bSi\, be denoted by $J$. So we have $S = J \cup (S \cap \Act(S))$ and $S' = J \cup (S' \cap \Act(S'))$ since an edge of type \bSi\, is always internal and an edge of type \bSe\, always external. Thus $$\diffs S {S'} = \diffs  {(S \cap \Act(S))} {(S' \cap \Act(S'))} \subseteq \Act(S) \cup  \Act(S').$$
But an edge that has neither type \bSe\, nor \bSi\, is active, and \textit{vice versa}, so $\Act(S) = \Act(S')$ and then $\diffs S {S'} \subseteq \Act(S).$ \\
\textbf{(iv) $\boldsymbol{ \Rightarrow }$ (v)}: Set $R = \diffs S {S'}$.
Since the symmetric difference is associative, we have
$$\diffs S R = \diffs S {(\diffs S {S'})} = \diffs {(\diffs S S)} {S'} = S'.$$
\textbf{(v) $\boldsymbol{ \Rightarrow }$ (i)}: Let $e_1 \ua{t_1}  \cdots  \ua{t_{m-1}} e_m \ua{t_{m}}$ (resp. $e'_1 \ua{t'_1}  \cdots  \ua{t'_{m-1}} e'_m \ua{t'_{m}}$) be the history of $S$ (resp. $S'=\diffs S R$); we are going to prove by induction on $k$ that for each $1 \leq j \leq k$, we have $e_j = e'_j$ and $t_j = t'_j$. 
We assume that the inductive hypothesis is true at the step $k-1$. Then, by \eqref{ekdelta}, we have $$e_k = \Delta(d(t_1),\dots,d(t_{k-1})) = e'_k.$$
Let us consider the beginning of the $k$-th iteration.  The graph $H$ has undergone the same transformations whether the input was $S$ or $\diffs S R$. Indeed, the deletions or contractions we have performed are governed by the types of the visited edges, which are identical by the inductive hypothesis. 
If the edge $e_k$ is a loop (resp. an isthmus) in $H$, then by following the algorithm we get $t_k = t'_k = \bL\, $ (resp. $t_k = t'_k =\bI$). If $e_k$ is standard, then $e_k$ is not active for $S$ and so $e_k \notin R$. In this case, either  $e_k \notin S$ which implies $e_k \notin \diffs S R$ and $t_k=t'_k=\mSe$, or $e_k \in S$ which implies $e_k \in \diffs S R$ and $t_k=t'_k=\mSi$.
\end{proof}

The most interesting implication is probably $(ii) \Leftrightarrow (v)$. In particular, it means that removing from $S$ or adding to $S$ an active edge does not change the types of edges, nor \emph{a fortiori} the set of active edges. This can be also be rewritten in term of subgraph intervals.

\begin{core} \label{cor:inter}
For any subgraph $S$, the equivalence class of $S$ is exactly the subgraph interval \mbox{$[S \backslash \Act(S), S \cup \Act(S)]$}.
\end{core}

%\noindent \textbf{Remark:} It is also true that two subgraphs $S$ and $S'$ are equivalent when they yield the same set of edges with type \bSe\, \emph{or} the same set of edges with type \bSi.

% \textbf{Example:} The partition that corresponds to the graph and decision tree from Figure \ref{fig:ex} is depicted at Figure \ref{fig:part}. We have gathered all the subgraphs of a same class in a lattice. The infimum subgraph does not have any internal active edges, while the supremum subgraph includes all of them. \\

% Then the idea will be the following: The relation $\sim$ induces a partition of the subgraphs of $G$ in equivalent classes. We are going to show that each class contains exactly one spanning tree. Thus, by the previous proposition, every subgraph of $G$ can be obtained by adding/removing active edges from the spanning tree which is in the same equivalence class. 

%%%%%%%%%%%%%%%%%%%%%%%%%%%%%%%%%%%%%%%%%%%%%%%%%%%%%%%%%%%%%%
\section{Indexation of the intervals  by spanning trees}
%%%%%%%%%%%%%%%%%%%%%%%%%%%%%%%%%%%%%%%%%%%%%%%%%%%%%%%%%%%%%%

\begin{theoe} \label{partition}
Let $G$ be a graph and $\Delta$ a decision tree. Then the set of subgraphs can be partitioned into subgraph intervals indexed by the spanning trees:
\begin{equation} 2^{E(G)} = \biguplus_{T\textrm{ spanning tree of }G} [T \backslash \mathcal I(T), T \cup \mathcal E(T)], \label{eqpart} \end{equation}
where $\mathcal I(T)$ and $\mathcal E(T)$ are the sets of internal and external $\Delta$-active edges of the spanning tree $T$.
\end{theoe}

Note  that Corollary \ref{cor:inter} tells that each of the intervals $[T \backslash \mathcal I(T), T \cup \mathcal E(T)]$ constitutes an equivalence class for the relation $\sim$.

We  first prove that every subgraph of $G$ is equivalent to some spanning tree:

\begin{leme}
\label{reptree}
For every $S$ subgraph of $G$, the set of edges with type \bSi \,or \bI\, forms a spanning tree of $G$ equivalent to $S$.
\end{leme}
\begin{proof} Let us denote by $T$ the set of edges with type \bSi\, and \bI. We choose the version of Algorithm \ref{type} where edges of type \bL\, are deleted and edges of type \bI\, are contracted. Let us prove that $T$ has no cycle and is connected, which exactly means that $T$ is a spanning tree.

\textbf{1. $\boldsymbol T$ has no cycle.} Let us assume that $T$ has a cycle $C$. Consider $e_k$ the maximal edge of $C$ for the $(\Delta,S)$-ordering.   At the $k$-th iteration of the algorithm, every edge of $C$ is contracted except $e_k$, because $C$ is only made of edges of type \bSi\, or \bI. This implies that $e_k$ is a loop, which contradicts the fact that $e_k$ has type \bSi\, or \bI.

\textbf{2. $\boldsymbol T$ has only one connected component} (This is precisely the dual of the previous point.). Let us assume that $T$ has more than one connected component. This means that there exists a cocycle $D$ of $G$ only made of edges of type \bSe\, and \bL. Consider $e_k$ the maximal edge of $D$ for the $(\Delta,S)$-ordering. Then at the $k$-th iteration of the algorithm, every edge of $D$ has been deleted except $e_k$. This implies that $e_k$ is a isthmus, which contradicts the fact that $e_k$ has type \bSe\, or \bL.

\textbf{3. $\boldsymbol T$ is equivalent to $\boldsymbol S$.} We have $\diffs S T \subseteq \Act(S)$ because the edges of type \bSi\, (resp. \bSe) for $S$ are internal (resp. external) in both subgraphs. By the implication (iv)~$\Rightarrow$~(i) of Proposition \ref{samehistory}, this means that $S \sim T$. 
%To conclude the proof, let us show that there do not exist more than two spanning trees in the same class.
\end{proof}

\begin{proof}[Proof of Theorem \ref{partition}]
Since the equivalence classes for the relation $\sim$ partition $2^{E(G)}$, we only need to show that there exists a unique spanning tree inside each of these classes. The existence is proved by Lemma \ref{reptree}. Now let us show the uniqueness of the spanning tree.

Let $T$ and $T'$ be two spanning trees of $G$ such that $T \sim T'$. Thus they share the same partition of $E(G)$. But remember that edges of type \bSi\, are always internal and those of type \bSe\, always external. Furthermore,  Proposition \ref{intextact} tells that each edge of type \bI\, is internal and each edge of type \bL\, is external. Therefore $T = T'$.
\end{proof}
\section{Some descriptions of the Tutte polynomial}
%%%%%%%%%%%%%%%%%%%%%%%%%%%%%%%%%%%%%%%%%%%%%%%%%%%%%%%%%%%%%%

% To each $S \subseteq E(G)$, we can map a spanning tree $T_S$ by putting in $S$ all external edges of type \bI\, and by removing from $S$ all internal edges of type \bL. Conversely, we rebuild $S$ from $T_S$ by removing the edges we have added and putting the edges we have removed back on.

Let us fix a graph $G$ and a decision tree $\Delta$. 
The following lemma describes how the number of connected components and the cyclomatic number behave when we add/remove an active edge.

\begin{leme} 
\label{addremove}
Let $S$ be a subgraph of $G$ and $e$ an active edge for $S$.
\begin{enumerate}
\item[(a)] If $e$ is external and has type \bL, then
$$\cc(S \cup \{ e \})= \cc(S),  \quad
\cycl(S \cup \{ e \})= \cycl(S) + 1.$$
\item[(b)] If $e$ is internal and has type \bI, then
$$\cc(S \backslash  e )= \cc(S) + 1,  \quad
\cycl(S \backslash  e )= \cycl(S). $$
\end{enumerate}
\end{leme}

\begin{proof}
\textbf{(a) $\boldsymbol e$ is external and has type \bL.} Point (1) from Lemma \ref{cyclecocycle} ensures that there exists a path only made of edges of type \bSi\, (so this is a path in $S$) linking the endpoints of $e$. Therefore including $e$ in $S$ does not add a connected component, hence $\cc(S \cup \{ e \})= \cc(S)$. Moreover,
$$\cycl(S \cup \{ e \})= |S \cup \{ e \}| + \cc(S \cup \{ e \}) - |V(G)| =   |S| + 1 + \cc(S) - |V(G)| =  \cycl(S) + 1.$$

\textbf{(b) $\boldsymbol e$ is internal and has type \bI.} Point (2) from Lemma \ref{cyclecocycle} ensures that there exists a cocycle only made of $e$ and edges of type \bSe. In other terms, there exists no path with edges of $S \backslash e$ linking the endpoints of $e$. So removing $e$ from $S$ will increase the number of connected components: $\cc(S \backslash  e )= \cc(S) + 1.$ Moreover,
$$\cycl(S \backslash  e )= |S \backslash  e | + \cc(S \backslash  e ) - |V(G)|  =   |S| - 1 + \cc(S) + 1 - |V(G)| =  \cycl(S),$$
which ends the proof. \end{proof}

\begin{prop} \label{rouge}
Let $T$ be a spanning tree. We denote by $I_T$ the set of the subgraphs equivalent to $T$. Then
\begin{equation} 
\label{eq:weight}
\sum_{S \in I_T}(x-1)^{\cc(S)-1}(y-1)^{\cycl(S)} = x^{|\mathcal I(T)|} y^{|\mathcal E(T)|},
\end{equation}
where $\mathcal I(T)$ (resp. $\mathcal E(T)$) denotes the set of internal (resp. external) $\Delta$-active edges of $T$.
\end{prop}

\begin{proof}
By Corollary \ref{cor:inter} the interval $I_T$ is the set of subgraphs of the form  $(T \backslash R_i) \cup R_e$, where $R_i \subseteq \mathcal I(T)$ and $R_e \subseteq \mathcal E(T)$.
In other terms, each subgraph of $I_T$ can be obtained by deleting from $T$ some edges of type \bI\, one by one, then adding to $T$ some edges of type \bL. Thus, by repeated applications of Lemma \ref{addremove}, given $S=(T \backslash R_i) \cup R_e \in I_T$, the number $\cc(S)$ is equal to $\cc(T) + |R_i| = |R_i| + 1$ and $\cycl(S)$ is equal to $\cycl(T) + |R_e| = |R_e|$. The formula \eqref{eq:weight} then derives from the identity $(X+1)^{|A|}(Y+1)^{|B|}=\sum_{\substack{S_A \subseteq A \\ S_B \subseteq B}} X^{|S_A|}Y^ {|S_B|}.$
%
%
%Theorem \ref{partition} shows that every subgraph of $G$ can be bijectively seen as a subgraph of the form $T \backslash I \cup E$, where $I \subseteq T \cap \Act(T)$ and $I \subseteq \overline T \cap \Act(T)$. 
%So $T_G(x,y)$ is equal to:
%%$$\sum_{S \subseteq E(G)} (x-1)^{\cc(S)-1}(y-1)^{\cycl(S)} =  \sum_{T\textrm{ spanning tree}}nte(T)}} (x-1)^{ \cc(\diffs T {(R_i \cup R_e)})-1}(y-1)^{\cycl(\diffs T {(R_i \cup R_e)} )}.$$
%Given any triplet $(T,R_i,R_e)$, the subgraph $\diffs T {(R_i \cup R_e)}$ can be obtained by deleting from $T$ the edges of $R_i$ one by one, then adding in $T$ the edges of $R_e$. Thus, by successive uses of Lemma \ref{addremove}, the number $\cc((T \backslash R_i) \cup R_e )$ is equal to $\cc(T) + |R_i| = |R_i| + 1$ and $\cycl((T \backslash R_i) \cup R_e )$ is equal to $\cycl(T) + |R_e| = |R_e|.$ 
%
%So by the basic identity $$(X+1)^{|A|}(Y+1)^{|B|}=\sum_{\substack{S_A \subseteq A \\ S_B \subseteq B}} X^{|S_A|}Y^ {|S_B|},$$ we get 
%
%$$T_G(x,y) = \sum_ {T\textrm{ spanning tree}} \sum_{\substack{R_e \subseteq \Ext(T) \\ R_i \subseteq \Inte(T)}} (x-1)^{|R_i|}(y-1)^{|R_e|} = \sum_ {T\textrm{ spanning tree}} x^{|\Inte(T)|} y^{|\Ext(T)|},$$
%which is the expected formula. 
\end{proof}

\noindent \textbf{Remark.} Theorem \ref{charact} p.~\pageref{charact} can be easily  deduced from the previous proposition. Indeed, the formula \eqref{celuila} can be obtained by summing \eqref{eq:weight} over all spanning trees $T$. (Remember that the intervals $I_T$ form a partition of the subgraph set -- see Theorem \ref{partition}.)

We can adapt the previous reasoning to obtain more descriptions of the Tutte polynomial.

\begin{prop} Given any decision tree $\Delta$, the Tutte polynomial of $G$ admits the three following descriptions:
\begin{equation}
\label{eqfor}
T_G(x,y) = \sum_{F\textrm{ spanning forest of $G$}} (x-1)^{\cc(F)-1}y^{\ell(F)},
\end{equation}
\begin{equation}
\label{eqfive}
T_G(x,y) = \sum_{K\textrm{ connected subgraph of $G$}} x^{i(K)}(y-1)^{\cycl(K)},
\end{equation}
\begin{equation}
\label{eqsix}
T_G(x,y) = \sum_{S\textrm{ spanning subgraph of $G$}} \left(\frac x 2 \right)^{i(S)} \left(\frac y 2 \right)^{ \ell(S) },
\end{equation}
where $i(S)$ (resp. $\ell(S)$) denotes the number of edges with type \bI\, (resp. \bL) for a subgraph $S$.
\end{prop}

These three descriptions can be also found in \cite{gordon-traldi}, but in the specific case of ordering activity.

\begin{proof} \textbf{Formula  \eqref{eqfor}:} Let $T$ be a spanning tree and $I_T$ the set of the subgraphs equivalent to $T$. Looking back on the proof of Proposition \ref{rouge}, we see that the spanning forests of $I_T$ are of the form $T \backslash R_i$ with $R_i \subseteq \mathcal I(T)$. So every subgraph of $I_T$ can be written as $F \cup R_e$, where $F$ is a spanning forest of $I_T$ and $R_e \subseteq \mathcal E(T)$. By Proposition \ref{samehistory}, the set $\mathcal E(T)$, which consists of edges with type \bL\, for $T$, equals $\mathcal L(F)$ for any forest $F$ in $I_T$, where $\mathcal L(F)$ denotes the set of edges with type \bL\, for $F$. Hence, we have 
$$I_T = \biguplus_{F\textrm{ spanning forest of }I_T} \left[F,F \cup \mathcal L(F)\right].$$
As the intervals $I_T$ partition the set of subgraphs (see Theorem \ref{partition}), we deduce:
\begin{equation} \label{equite}
2^{E(G)} = \biguplus_{F\textrm{ spanning forest of }G} [F, F \cup \mathcal L(F)].
\end{equation}
A subgraph $S \in [F, F \cup \mathcal L(F)]$ is of the form  $F \cup R$, with $R \subseteq \mathcal L(F)$, and thanks to Lemma \ref{addremove} satisfies $\cc(S) = \cc(F)$ and $\cycl(S) = \cycl(F) + |R| = |R|$.
So we get
$$\sum_{S \in [F, F \cup \mathcal L(F)]} (x-1)^{\cc(S)-1}(y-1)^{\cycl(S)} = (x-1)^{\cc(F)-1} y^{|\mathcal L(F)|}.$$
The formula \eqref{eqfor} results from the previous equation and \eqref{equite}.

\noindent \textbf{Formula  \eqref{eqfive}:} Similar to the previous point, in a dual way.
 
\noindent \textbf{Formula  \eqref{eqsix}:} Let $T$ be a spanning tree. By Proposition \ref{samehistory}, each subgraph $S$ equivalent to $T$ satisfies $i(S)=i(T)$ and $\ell(S)=\ell(T)$. So the sum   $\sum_{S \sim T} \left(\frac x 2 \right)^{i(S)} \left(\frac y 2 \right)^{ \ell(S) }$ equals $\left(\frac x 2 \right)^{i(T)} \left(\frac y 2 \right)^{ \ell(T)}$ times the number of subgraphs equivalent to $T$, namely $2^{|\Act(T)|}=2^{i(T)+\ell(T)}$. So the previous sum is equal to $x^{i(T)}y^{\ell(T)}$. Summing over all spanning tree $T$, we obtain formula \eqref{eqsix}. \end{proof}

%********************************************************************%
\chapter{Back to the earlier activities}
\label{sec:spec}
\label{c:spec}
%********************************************************************%

It is time to show that each activity defined in Section \ref{sec:activity} p. \pageref{sec:activity} is a $\Delta$-activity. In this way, we prove in a different way that all these activities are Tutte-descriptive.
% Except for ordering activity, we will deduce the existence of adequate decision trees from Theorem \ref{treecompatible}. Keep in mind that the proof of this theorem allows us to algorithmically build the decision trees.
%, even we will not explicitly mention this
The key tools for this section are Theorem \ref{treecompatible} p. \pageref{treecompatible} and Corollary \ref{cor:act} p. \pageref{cor:act}. 
% Each time we use either one, we implicitly consider a decision tree, as their proofs suggest.   

%%%%%%%%%%%%%%%%%%%%%%%%%%%%%%%%%%%%%%%%%%%%%%%%%%%%%%%%%%%%%%%%%%%%%%%%%%
\section{Ordering activity}
%%%%%%%%%%%%%%%%%%%%%%%%%%%%%%%%%%%%%%%%%%%%%%%%%%%%%%%%%%%%%%%%%%%%%%%%%%

Let $G$ be a graph and consider a linear order $<_{ord}$:
$\epsilon_1 < \epsilon_2 < \dots < \epsilon_m,$
where $E(G) = \ens{\epsilon_1,\dots,\epsilon_m}$.
Define the decision function $\Delta_{ord}$ by setting
\begin{equation} \Delta_{ord}(d_1,\dots,d_k) =  \epsilon_{m-k} \end{equation}
for every sequence of directions $(d_1,\dots,d_k)$ with $k < m$.

\begin{prop}
The ordering activity is equal to the  $\Delta_\phi$-activity. Therefore it is Tutte-descriptive.
% for $<_{ord}$ if and only if it is $\Delta_{ord}$-active. 
%Therefore, the Tutte polynomial is equal to 
%\begin{equation} T_G(x,y) = \sum_{T\textrm{ spanning tree}} x^{\inte_{\textrm{ord}}(T)}\,y^{\ext_{\textrm{ord}}(T)}, \end{equation}
%where $\inte_{\textrm{ord}}(T)$ and $\ext_{\textrm{ord}}(T)$ respectively denote the number of internally and externally ordering-active edges of $T$.
\end{prop}

We recall that the ordering activities (also called Tutte's activities) have been introduced in Subsection~\ref{ss:ord} p.\pageref{ss:ord}.

\begin{proof} Fix a spanning tree $T$. It is straightforward to see from the definition of $\Delta_{ord}$ that $\epsilon_m < \dots < \epsilon_1$
for the $(\Delta_{ord},T)$-order. In other terms, the $(\Delta_{ord},T)$-order is inverse to $<_{ord}$. So by Proposition~\ref{maximal}, an edge is ordering-active if and only if it is $\Delta_{ord}$-active. We conclude thanks to Theorem~\ref{charact}, applied to the decision function $\Delta_{ord}$.
\end{proof}

\noindent \textbf{Remark.} Alternatively, we could have also used Corollary \ref{cor:act} with constant order map $T \mapsto\  >_{ord}$.

\noindent  \textbf{Example}. The decision tree associated to the graph from Figure \ref{fig:extut} p.\pageref{fig:extut} with the linear order $ a < b < c < d$  is depicted at Figure \ref{fig:dttut}.

\begin{figure}[h!]
\begin{center}
\includegraphics[scale=1]{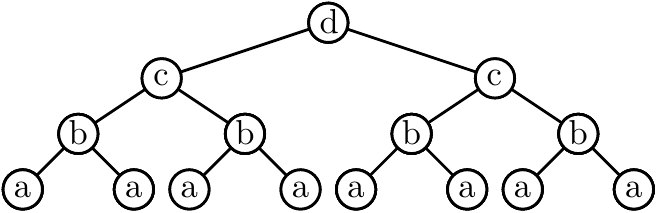}
\end{center}
\caption{The decision tree corresponding to the linear order $a < b < c < d$.}
\label{fig:dttut}
\end{figure}

%%%%%%%%%%%%%%%%%%%%%%%%%%%%%%%%%%%%%%%%%%%%%%%%%%%%%%%%%%%%%%%%%%%%%%%%%%
\section{Embedding activity}
%%%%%%%%%%%%%%%%%%%%%%%%%%%%%%%%%%%%%%%%%%%%%%%%%%%%%%%%%%%%%%%%%%%%%%%%%%

Let $G$ be a graph. We want to express the embedding activities (see Subsection \ref{ss:bernardi} for the definition) as $\Delta$-activities and thus give an alternative proof of Theorem 7 from \cite{bernardi-tutte}.

\begin{prop}
\label{prop:ber}
For any embedding $M_G$ of the graph $G$, the embedding activity is a  $\Delta$-activity and so a Tutte-descriptive activity. 
\end{prop}

\noindent \textbf{Example. }Let $M$ and $\Delta$  be the map and the decision tree of  Figure \ref{fig:dqber}. Consider $T = \ens{b,d}$. The $(M,T)$-ordering \footnote{Reminder : the $(M,T)$-order is the order of first visit during the tour of $T$ -- see Subsection \ref{ss:bernardi}.} is $a < b < c < d$. So the set of $(M,T)$-active edge is $\ens{a,b}$. The $(\Delta,T)$-ordering is $d < a < c < b$, so the set of $\Delta$-active edges for $T$ is as well $\ens{a,b}$. We can check that the two sets of active edges coincide (even if the two orderings do not look related).

\begin{figure}[h!]
\begin{center}\includegraphics[scale=1]{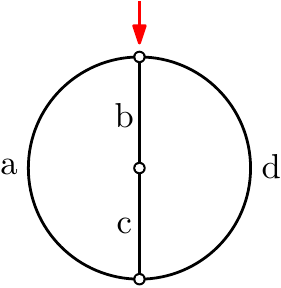} \hspace{2cm}
\includegraphics[scale=1.2]{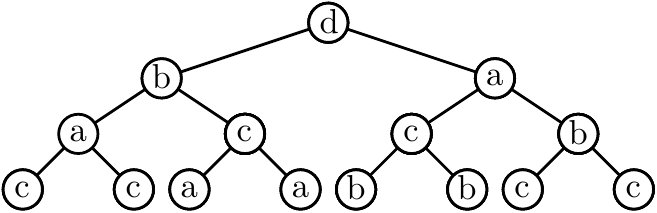}
\end{center}
\caption{A map and a decision tree for the embedding activity.}
\label{fig:dqber}
\end{figure}

 The proof is in two steps, embodied in two lemmas. The first lemma uses tree-compatibility via Corollary~\ref{cor:act}.

\begin{leme}
For each embedding $M_G$ of a graph $G$, the activity that maps any spanning tree $T$ onto the set of edges that are \emph{maximal} for the ($M_G$,$T$)-order in their fundamental cycles/cocycles is a $\Delta$-activity and so Tutte-descriptive.
\end{leme}

\begin{proof} It suffices to prove that the order map that sends a spanning tree $T$ of $G$ onto the $(M_G,T)$-order satisfies the hypothesis of Corollary \ref{cor:act}.

For any spanning tree $T$, we denote by $\phi_i(T)$ (resp. $h_i(T)$) the $i$-th smallest edge (resp. half-edge) for the $(M_G,T)$-ordering, namely the edge (resp. half-edge) which is visited in $i$-th position during the tour of the tree $T$.  We consider  $T$ and $T'$ two spanning trees of $G$ and $k$ an integer from $\ens{0,\dots,|E(G)|-1}$ such that \begin{equation} T \cap \ens{\phi_1(T),\dots,\phi_k(T)} = T' \cap \ens{\phi_1(T),\dots,\phi_k(T)}.
\label{bertc}
\end{equation} Moreover, let $\ell$ be the integer in $\ens{1,\dots,2|E(G)|}$ such that $h_\ell(T)$ is the smallest half-edge of $\phi_{k+1}(T)$. We want to show that $\phi_i(T) = \phi_i(T')$  for each $i \in \ens{1, \dots, k+1}$.

Let us prove by recurrence on $j \in \ens{1,\dots,\ell}$ that $h_j(T) = h_j(T')$. For $j=1$, the half-edges $h_1(T)$ and $h_1(T')$ are both equal to the root half-edge of $M_G$. Now assume that for some $j \in \ens{1,\dots,\ell-1}$ we have $h_j(T) = h_j(T')$. Let $\phi_i(T)$ be the edge that corresponds to $h_j(T)$. The edge $\phi_i(T)$ belongs to $\ens{\phi_1(T),\dots,\phi_k(T)}$  since we have $h_j(T) < h_{\ell}(T)$ for the $(M_G,T)$-ordering. So the equivalence 
$$\phi_i(T)\textrm{ internal for }T \ \Leftrightarrow \phi_i(T)\textrm{ internal for }T',$$
that results from \eqref{bertc},
holds. Thus, we have $$t(h_j(T);T) = t(h_j(T');T')$$
(let us recall that $t(.;T)$ denotes the motion function of the spanning tree $T$, see Subsection \ref{ss:bernardi}), that is 
$$h_{j+1}(T) = h_{j+1}(T'),$$
which proves the recurrence.

We have shown that whether we consider the tour of $T$ or the tour of $T'$, the first visited half-edges are $h_1(T),\dots,h_\ell(T)$ in this order. Consequently, the $k+1$ first visited edges are identical in both cases, namely $\phi_1(T),\dots,\phi_k(T),\phi_{k+1}(T)$ in this order. This means that for every $i \in \ens{1,\dots,k+1}$ we have $\phi_i(T) = \phi_i(T')$.
\end{proof}

Despite the similarities with embedding activities, the previous lemma is not sufficient to recover Bernardi's activity, for which an active edge is \emph{minimal} in its fundamental cycle/cocycle rather than \emph{maximal}. 

We could then have the idea to adapt the previous proof by considering the order map that sends each spanning tree $T$ onto the \emph{reversed} ($M_G$,$T$)-ordering. But a problem occurs: this order map is not tree-compatible\footnote{Indeed, when an order map $\phi$ is tree-compatible, the minimal edge for $\phi(T)$ is the same for all spanning trees $T$, which is obviously not the case here.}.

We could also have the seemingly desperate idea to "reverse" the map instead of the ($M_G$,$T$)-ordering. For example, consider the rooted maps $M$ and $M'$ equipped with the spanning tree $T$ of Figure \ref{fig:mirror}. The map $M'$ is the mirror map of $M$. The $(M,T)$-ordering for half-edges equals $a < b < c < b' < d < c' < a' < d'$ while the $(M',T)$-ordering equals $a > b' > c > b > d' > c' > a' > d$\footnote{If we formally remove the primes in these two orderings, we will observe a strange phenomenon: the two orderings are reverse! An explanation is implicitly given in the proof of Lemma \ref{miror}.}. The orderings on the edges are not reverse ($\ar a < \ar b < \ar c < \ar d$ for $M$ and $ \ar d < \ar a < \ar c < \ar b$ for $M'$) but the $(M,T)$-active edges correspond to the edges that are maximal in their fundamental cycle/cocycle for the $(M',T)$-ordering, namely $\ar a$ and $\ar b$. (For instance, take $\ar a$: its fundamental cycle is $\ens{\ar a,\ar d}$. We have $\ar a < \ar d$ for the $(M,T)$-order and $\ar d < \ar a$ for the $(M',T)$-order.) It turns out that this property is general. 

\begin{figure}[h!]
\begin{center}
\includegraphics[scale=1]{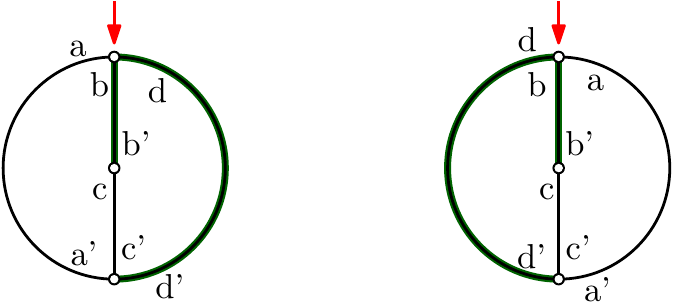}
\end{center}

\caption{A map $M$ with a spanning tree $T$ and its mirror map $M'$.}
\label{fig:mirror}
\end{figure}

\begin{leme} \label{miror}
Consider any embedding $M_G= (H,\sigma,\alpha)$ of the graph $G$ rooted on a half-edge $a$, and denote by $M^\#_G$ the mirror map of $M_G$, that is to say the map $(H,\sigma^{-1},\alpha)$ rooted on the half-edge $\sigma^{-1}(a)$.

For any spanning tree $T$, an internal (resp. external) edge is $(M_G,T)$-active if and only if it is maximal in its fundamental cocycle (resp. cycle) for the $(M^\#_G,T)$-ordering. 
\end{leme}

This lemma shows that if we change "minimal" by "maximal" in the definition of the embedding activity, we obtain a strictly equivalent notion (and maybe more natural).

\begin{proof} Fix $T$ a spanning tree. We denote by $t$ the motion function corresponding to $M_G$ and $T$, and by $t'$ the motion function corresponding to $M^\#_G$ and $T$.

\textbf{1.} Let us prove that an edge $\ar h$ is smaller than $\ar g$ for the $(M^\#_G,T)$-ordering if and only if we have $\max(g,g') < \max(h,h')$ for the  $(M_G,T)$-ordering. Let $\ar g$ be an edge and $m$ the number of edges. Since $t'$ is a cyclic permutation of the set of half-edges (see Lemma \ref{tourcyclique} p. \pageref{tourcyclique}), there exist $k$ and $\ell$ in $\ens{0,\dots,2m-1}$ such $g = t'^k(\sigma^{-1}(a))$ and $g' = t'^\ell(\sigma^{-1}(a))$.
Moreover, it is easy to see from the algebraic definition of the  motion function that $t' = \sigma^{-1} \circ t^{-1} \circ \sigma$ and that $\sigma^{-1}(t(\hat g)) \in \ens{\hat g,\alpha(\hat g)}$ for every half-edge $\hat g$. Hence, 
$$g=t'^k(\sigma^{-1}(a)) = \sigma^{-1}(t^{-k}(a)) = \sigma^{-1}(t^{2m-k}(a)) \in \ens{t^{2m-k-1}(a),\alpha(t^{2m-k-1}(a))}.$$
Similarly, we have $g' \in \ens{t^{2m-\ell-1}(a),\alpha(t^{2m-\ell-1}(a))}$. But $g$ and $g'$ belong to the same edge, so 
$$\ar g =  \ens{t^{2m-k-1}(a), \alpha(t^{2m-k-1}(a))} = \ens{t^{2m-\ell-1}(a), \alpha(t^{2m-\ell-1} (a))}$$
and since $k \neq \ell$ (we have $g \neq g'$), we deduce that 
$$ \ar g =  \ens{t^{2m-k-1}(a),t^{2m-\ell-1}(a)}.$$
Thus, the following two equalities hold:
$$\min(g,g')= \min \left(t'^k(\sigma^{-1}(a)),t'^\ell(\sigma^{-1}(a))\right) = t'^{\min(k,\ell)}(\sigma^{-1}(a)),$$  where the $\min$  concern the $(M^\#_G,T)$-ordering,
and
$$ \max(g,g')= \max \left(t^{2m-k-1}(a),t^{2m-\ell-1}(a)\right) = t^{2m - \min(k,\ell) - 1}(a), $$ where the $\max$  concern the $(M^\#_G,T)$-ordering,
Let $\ar h$ be an edge different from $\ar g$ with $h = t'^i(\sigma^{-1}(a))$ and $h' = t'^j(\sigma^{-1}(a))$. Using the above equalities, there is equivalence between the following statements:
$$\begin{array}{lcl}
\ & \ar h < \ar g &  \textrm{ for the }(M^\#_G,T)\textrm{-ordering}, \\
\Leftrightarrow \  & \min (h,h') < \min (g,g') &  \textrm{ for the }(M^\#_G,T)\textrm{-ordering}, \\
\Leftrightarrow \  & \min(i,j) < \min(k,\ell) &  \\
\Leftrightarrow \ & 2m - \min(k,\ell) - 1 < 2m - \min(i,j) - 1 &  \\
\Leftrightarrow \ & \max(g,g') < \max(h,h') &  \textrm{ for the }(M_G,T)\textrm{-ordering}.
\end{array}$$

\textbf{2.} Let $\ar g$ be an external (resp. internal) edge and $\ar h$ an edge in the fundamental cycle (resp. cocycle) of $\ar g$ with $g < g'$, $h < h'$ and $g < h$, where $<$ denotes from now on the $(M_G,T)$-ordering. Let us show that $g' < h'$.

Using Lemma 4 from \cite{bernardi-tutte}, it is not hard to see that for every spanning tree $T$, deleting an external edge or contracting an internal edge of $M_G$ does not change the $(M_G,T)$-ordering between the remaining half-edges.
This is why we can restrict the proof of this point to the case where $E(G) = \ens{\ar g,\ar h}$. As it must contain a cycle  with 2 edges or a cocycle with 2 edges, $G$ must consist of two vertices linked by $\ar g$ and $\ar h$. Since $\ar g < \ar h$, the root of the map is $g$. The only two possibilities for $(G,T)$ are depicted in Figure \ref{baseber}.
Thus, we must have in both cases $t = (g,h,g',h')$.  Therefore $g' < h'$.

\fig{[scale=1]{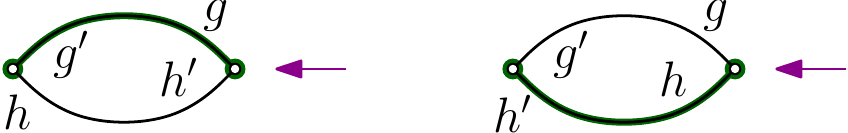}}{Two possibilities when $G$ is restricted to two edges}{baseber}

\textbf{3.} Fix   a $(M_G,T)$-active edge $\ar g$. By definition, $\ar g$ is minimal for the $(M_G,T)$-ordering in its fundamental cycle/cocycle $C$. Let $\ar h$ be any other edge in $C$. If we assume $g < g'$ and $h < h'$, then $g < h$. By point 2, we have $g' < h'$. So by point 1, $\ar g$ is greater than $\ar h$ for the $(M^\#_G,T)$-ordering. This being true for all $\ar h$, the edge $\ar g$ is maximal for the $(M^\#_G,T)$-ordering inside $C$.

\textbf{4.} Conversely, fix an edge $\ar g$ that is maximal for the $(M^\#_G,T)$-ordering in its fundamental cycle/cocycle $C$, with $g < g'$. For any other edge $\ar h$ in $C$ with $h < h'$, we have $g' < h'$ by point $1$. By contraposition of point 2, we have $g < h$. This means that inside its fundamental cycle/cocycle, $\ar g$ is minimal for the $(M_G,T)$-ordering, and thus $(M_G,T)$-active.\end{proof}

Finally, the combination of the two previous lemmas gives a proof of Proposition~\ref{prop:ber}, since $M^\#_G$ is an embedding of the graph $G$.

\section{DFS activity}
\label{ss:dfst}
%%%%%%%%%%%%%%%%%%%%%%%%%%%%%%%%%%%%%%%%%%%%%%%%%%%%%%%%%%%%%%%%%%%%%%%%%%

We end this chapter with the DFS activity defined in Subsection \ref{ss:dfs} p. \pageref{ss:dfs}. Let us recall that we now consider a graph $G$ without multiple edges.

\begin{prop}
\label{dfstut}
For any labelling of $V(G)$ with integers $1,2,\dots,|V(G)|$, the external DFS-activity restricted to spanning trees can be extended into a $\Delta$-activity and so a Tutte-descriptive activity.
\end{prop}

\noindent \textbf{Example. } We show in Figure \ref{fig:dtdfs} a graph with a decision tree $\Delta$ inducing the DFS activity. For instance, consider the spanning tree $T = \ens{a,c,e}$. The $(\Delta,T)$-ordering is $b < a < c < e < d$, so the only external $\Delta$-active edge is $d$. One can check that $d$ is also the only external DFS-active edge.

\begin{figure}[h!]
\begin{center}
\includegraphics[scale=1]{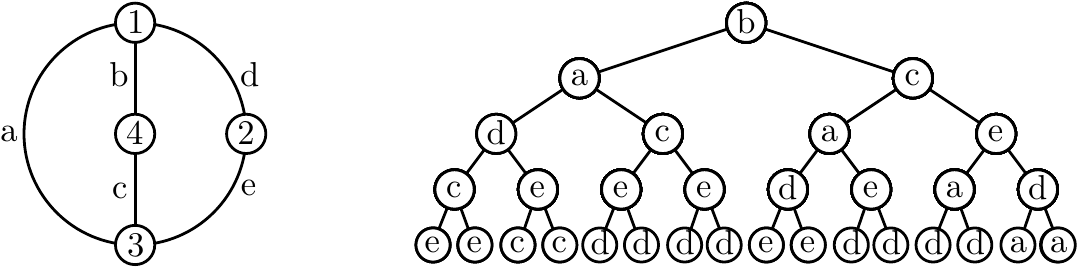}
\end{center}
\caption{A graph with no multiple edges and the decision tree corresponding to the DFS-activity.}
\label{fig:dtdfs}
\end{figure}

\begin{proof} The principle of the proof can be detailed as follows:
\begin{itemize}
\item we define an order map $\phi$ that is related to the visit order in the (greatest-neighbor) DFS, which is described by Algorithm \ref{DFS}
\item we prove that $\phi$ is tree-compatible by checking the hypotheses of Corollary \ref{cor:act},
\item we show that the external DFS-activity can be described in terms of maximality in $\phi$.
\end{itemize}

We would want to match an order map $\phi$ with the visit order in Algorithm \ref{DFS}. But this algorithm only considers the internal edges while we also need to order the external edges to define an order map. That is why we need to enrich Algorithm \ref{DFS} to take into account external edges. The result is Algorithm \ref{pmDFS}. 
\begin{algorithm}[h!]
\caption{The order map for DFS activity }

\label{pmDFS}
\begin{algorithmic}[5]
\Require {\color{Navy}$H$ spanning subgraph of $G$}.
\Ensure {\color{Navy}A total ordering of $E(G)$ given by the edges $\phi_1(H),\phi_2(H),\dots,\phi_{|E(G)|}(H)$ in this order.}
\State $\mathcal F(H) \leftarrow \emptyset$;
\State $j \leftarrow 0$; 
\While {there is a unvisited vertex} 
	\State  \textbf{mark} the least unvisited vertex of $G$ \textbf{as visited};
	\While {{\color{Navy}there is a visited vertex with unvisited incident edges}}
		\State \ \Comment{In this algorithm "incident" means "incident in $G$" (not in $H$).}
		\State $v \leftarrow$ the most recently visited such vertex;
		\While{{\color{Navy} $v$ has an unvisited incident edge}}
			\State $u \leftarrow$ the greatest neighbor of $v$ {\color{Navy}linked in $G$ by an unvisited edge};
			\State $j \leftarrow j + 1$;
			\State $\phi_j(H) \leftarrow \ens{u,v}$;
			\State {\color{Navy}\textbf{mark} $\phi_j(H)$ \textbf{as visited}};
			\If{{\color{Navy}$\phi_j(H)$ is internal \textbf{and} if $u$ is unvisited}} \label{lineuse}
				\State \textbf{mark} $u$ \textbf{as visited};
				\State $v \leftarrow u$; 
				\State \textbf{add} $\phi_j(H)$ in $\mathcal F(H)$;
			\EndIf
		\EndWhile
	\EndWhile
\EndWhile 
\State \Return {\color{Navy}$\phi_1(H) < \phi_2(H) < \dots < \phi_{|E(G)|}(H)$};
\end{algorithmic}
\end{algorithm}

\noindent \textbf{Informal description.} The input is a subgraph $H$ of $G$. As in Algorithm \ref{DFS}, we begin by the least vertex. We proceed to the DFS of the graph $\boldsymbol G$ that favors the largest neighbors with an extra rule: when we visit an external edge (for $H$), we do not go through it, we come back to the original vertex as if this edge had never existed. Moreover, the visited edges (internal and external) are marked so that we visit them only once each. The rest of the algorithm is exactly as Algorithm~\ref{DFS}. The output is the visit order of the edges, denoted  by $\phi(H)$, instead of the DFS forest. The changes between Algorithm \ref{DFS} and \ref{pmDFS} are indicated in navy\footnote{or in dark grey for those who read a black and white version of this thesis}.

\noindent \textbf{Remark.} In this proof, only connected subgraphs are important. We could have simplified Algorithm \ref{pmDFS} with a  restriction of the input to connected subgraphs, but it could have interfered with understanding. 

\noindent \textbf{Example.} A run of Algorithm \ref{pmDFS} is depicted in Figure \ref{dfsex}. The resulting order is \mbox{$b < a < c < e < d$}.

\fig{[width=\textwidth]{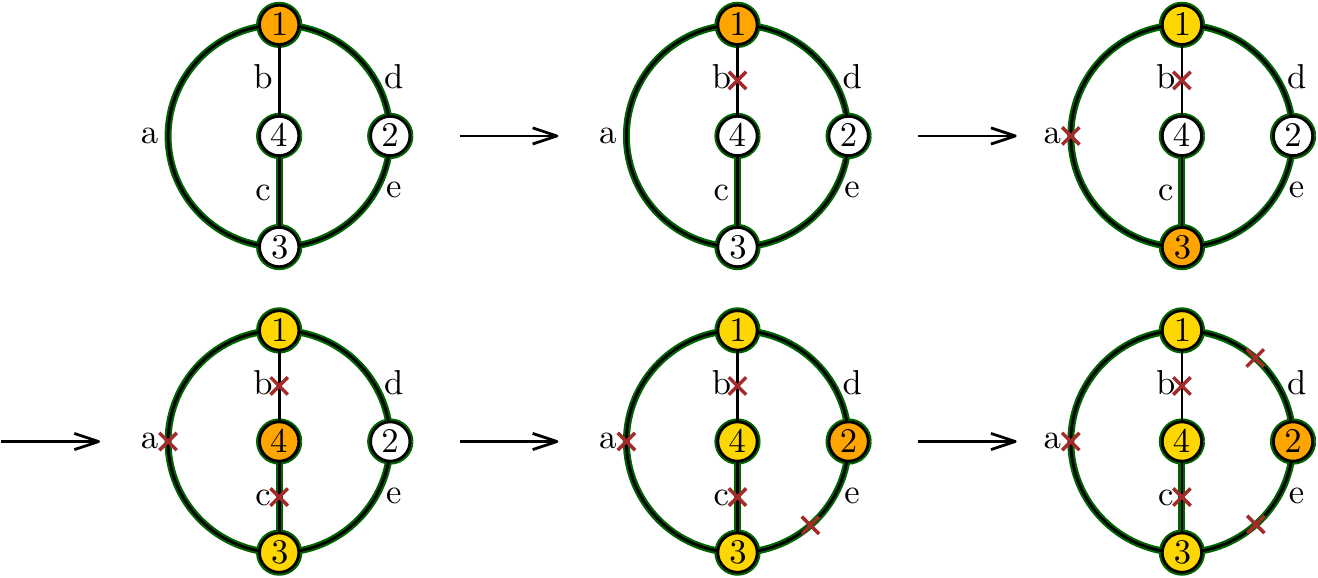}}{Illustration of a run of Algorithm \ref{pmDFS}.}{dfsex}

Let us denote by $\phi(H)$ the result of Algorithm \ref{pmDFS} for a  subgraph $H$. One can straightforwardly see that the restriction of $\phi$ to the spanning trees satisfies the hypotheses of Corollary \ref{cor:act}: indeed, the only influence of the input lies in Line \ref{lineuse}, where we test if the successive values of $\phi_j(T)$ belong to $T$ or not. 

Let $T$ be a spanning tree and $e$ an external edge. Let us prove that $e$ is DFS-active if and only if $e$ is maximal for $\phi(T)$ in its fundamental cycle $C$. If we manage to do so, we use Corollary \ref{cor:act} and the proof is ended.

a. Assume that $e$ is maximal in $C$. Since the only difference between $T$ and $T \cup e$ is $e$ (!), the executions in Algorithm \ref{pmDFS} with input $T$ and $T \cup e$ are strictly identical until the visit of $e$. In particular, at this moment, as every edge in $C$ other than $e$ has been visited and belongs the DFS forest of $T$ (because internal), both endpoints of $e$ are visited. 
Since Algorithm \ref{pmDFS} is an enriched version of Algorithm \ref{DFS}, the endpoints of $e$ are also visited just before the first visit of $e$ for Algorithm \ref{DFS} with input $T \cup e$. We never add to a DFS forest an edge whose both endpoints are marked, hence $e \notin \mathcal F(T \cup e)$.
 But $\mathcal F(T \cup e)$ is a spanning tree included in $T \cup e$. So we must have $\mathcal F(T \cup e)=T$, which means that $e$ is DFS-active.

% Given the ordering $\phi(T \cup e)$, consider the greatest edge $e'$ of the DFS forest of $T \cup e$ and smaller than $e$. As no edge of the DFS forest has been visited between $e$ and $e'$, the endpoints of $e$ are always marked just after the visit of $e'$ (in Algorithm \ref{pmDFS} with input $T \cup e$). So, by the property we described above, the endpoints of $e$ are also marked just after the visit of $e'$ in Algorithm \ref{DFS} with input $T \cup e$. 

%Assume now that $e$ belongs the DFS forest of $T \cup e$. Since $e' < e$ for $\phi(T \cup e)$, it cannot be visited before $e'$ in Algorithm \ref{DFS}. But it cannot be visited after $e'$ too, since both endpoints of $e$ has been visited and we never add to a DFS forest an edge whose both endpoints are marked. This results in a contradiction, hence $e \notin \mathcal F(T \cup e)$. But  $\mathcal F(T \cup e)$ is a spanning tree included in $T \cup e$. So we must have $\mathcal F(T \cup e)=T$, which means that $e$ is DFS-active. 

b. Observe that for any spanning tree $\hat T$, an external edge $\hat e$ is maximal for $\phi(\hat T)$ in its fundamental cycle in $\hat T$ if and only if $\hat e$ is maximal for $\phi(\hat T \cup \hat e)$ in its fundamental cycle in $\hat T$. Indeed, the executions in Algorithm \ref{pmDFS} with input $\hat T$ and $\hat T \cup \hat e$ are the same until the visit of $\hat e$. In particular, the set of edges visited before $\hat e$ are identical.

c. Assume that $e$ is not maximal in $C$ for $\phi(T)$. 
% Let us remark that for any spanning tree $\hat T$, an external edge $\hat e$ is maximal for $\phi(\hat T)$ in its fundamental cycle in $T$ if and only if $\hat e$ is maximal for $\phi(\hat T \cup \hat e)$ in its fundamental cycle. Indeed, the executions in Algorithm \ref{pmDFS} with input $\hat T$ and $\hat T \cup \hat e$ are the same until the visit of $\hat e$. In particular, the set of edges visited before $\hat e$ are identical. 
By point b, the edge $e$ is not maximal in $C$ for $\phi(T \cup e)$.
%, since the executions in Algorithm \ref{pmDFS} with input $T$ and $T \cup e$ are the same until the visit of $e$. 
Let $e'$ be the maximal edge in $C$ for $\phi(T \cup e)$ and denote by $T'$ the spanning tree $(T \cup e) \backslash e'$. By point b, $e'$ is maximal in $C$ for $\phi(T')$. Then, by point a, it means that $e'$ is DFS-active for $T'$. Hence $$\mathcal F(T \cup e) = \mathcal F(T' \cup e') = T' \neq T,$$
the edge $e$ is not DFS-active. 
\end{proof}

Let us conclude this section by some remarks. We have just proved that the restriction of the external DFS activity to spanning trees can be extended into a Tutte-descriptive activity. But in the light of the equations \eqref{DFStut} p.~\pageref{DFStut} and \eqref{eqfor} p.~\pageref{eqfor}, a question is looming : does there exist a decision tree $\Delta$ such that for all spanning forest (not necessarily spanning trees) the external DFS-active edges coincide with the edges with $\Delta$-type \bL? The answer is generally negative: consider for instance the graph  \begin{tabular}{c}\includegraphics[scale=0.8]{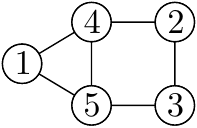}\end{tabular}. The only external DFS-active edge for the spanning forest $\ens{\ens{2,4},\ens{4,5},\ens{3,5}}$ is $\ens{2,3}$. But one can check by inspection that this edge is never DFS-active for any spanning tree. This cannot occur with $\Delta$-activities since every subgraph shares the same partition  into types as some spanning tree.

However, with some slight modifications of the definition of DFS activity, we can ensure that the correspondance between DFS-active edges and edges with type \bL\, is effective. For this, without proving it, we just have to change the input of Algorithm \ref{DFS} -- taking subgraphs of $G$ instead of general graphs -- and choose a more convenient vertex $v$ at Line \ref{linedfs}\footnote{Let us describe in a few words how to modify Line \ref{linedfs} from Algorithm \ref{DFS}. If it is the first iteration, meaning that no vertex has been visited yet, we choose $v$ as the least vertex. If some vertices have been visited, we consider $\Delta_\phi$ the decision tree induced by Corollary \ref{cor:act} and order map $\phi$, output of Algorithm \ref{pmDFS} (see proof of Proposition \ref{dfstut}). Then we run Algorithm \ref{type} with the same subgraph and with decision tree $\Delta_\phi$. We consider the edge of type \bI\, with a visited endpoint and an unvisited endpoint that was visited first. The new vertex $v$ is the unvisited endpoint of this edge.}. This new definition does not change the results of Gessel and Sagan about the DFS activity \cite{GesselSagan}, like Proposition \ref{lemGS} p. \pageref{lemGS}.

Furthermore, the external DFS activity can derive from another notion of $\Delta$-activity, named $\Delta$-forest activity. This will be the subject of Section \ref{ss:partial}. 

%There is a second possibility, can we describe variants of $\Delta$-activity such that:
%\begin{itemize}
%\item the departure set is the set of spanning forests,
%\item the image of each forest is only made of external edges,
%\item the image set of a forest can be different from every image set of spanning tree,
%\item the external DFS activity could derive 
%\end{itemize} would be external but not necessarily extendable into a Tutte-descriptive activity, and from which external DFS activity could derive ? Here again, the answer is yes and is detailed in Subsection \ref{ss:partial}.

\chapter{Blossoming activity}
\label{c:bloact}

We saw that all the activities that were known in the literature can be seen as a $\Delta$-activity. It is then natural to ask if there exist other $\Delta$-activities (necessarily Tutte-descriptive) that  have nice combinatorial descriptions. Following this idea, we derive from the bijection between maps and blossoming trees (see Section \ref{s:bb} p. \pageref{s:bb}) a new family of activities, called \emph{blossoming activities}.

\section{Blossoming activity and blossoming transformation}

% We first define it for internal edges only.

As for any embedding activity, we need beforehand to embed the graph $G$ and root it. We denote by $M$ the resulting  map. 
Given a spanning forest $F$, Algorithm \ref{prune} outputs a spanning tree denoted $\tau(F)$.  

\begin{algorithm}[h!]
\caption{Computing $\tau(F)$}
\label{prune}
\begin{algorithmic}
\Require $F$ spanning forest of $M$.
\Ensure $\tau(F)$ spanning tree of $M$.
\State $h \leftarrow$ root of $M$;
\State $M' \leftarrow M$;
\While {some edges are not visited} 
	\State $e \leftarrow$ edge that contains $h$;
	\State $h \leftarrow$ half-edge that immediately follows $h$ in $M'$;
	\Comment{i.e. $h \leftarrow \sigma \circ \alpha (h)$  }
	\If{$e$ is not an isthmus of $M'$ \textbf{and} $e \notin F$}
		\State $M' \leftarrow \delete {M'} {e};$
	\EndIf 
\EndWhile
\State $\tau(F) \leftarrow M'$
\State \Return $\tau(F)$;
\end{algorithmic}
\end{algorithm}

\noindent \textbf{Informal description.} Starting from the root, we turn counterclockwise around the map. \emph{After} we walk along an edge, we remove it whenever the deletion of this edge leaves the map connected (i.e. it is not an isthmus) \emph{and} it is external. We stop the algorithm when every edge has been visited.

\noindent \textbf{Example.} We consider the first embedded graph of Figure \ref{fig:exblo} with spanning forest $F = \sing {d}$. The run of Algorithm \ref{prune} with this input is illustrated on the same figure. We have $\tau(F) = \ens{c,d}$.

\begin{figure}[h!]
\begin{center}
\includegraphics[scale=0.9]{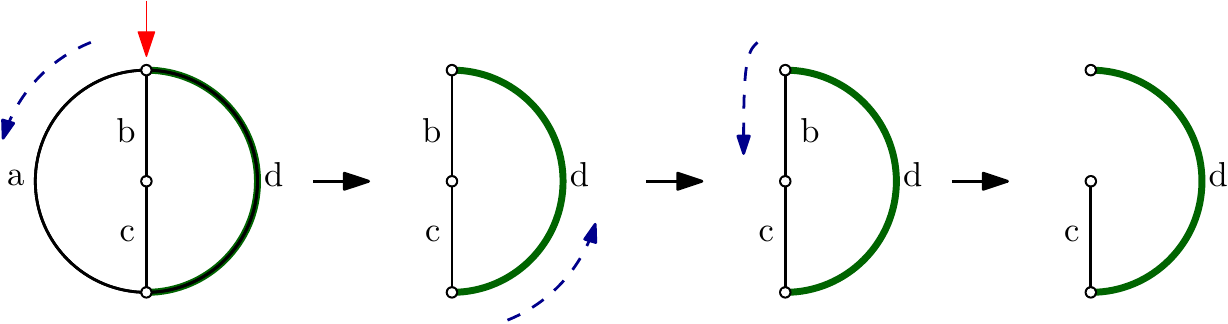}
\end{center}
\caption{Run of Algorithm \ref{prune} with spanning forest $F = \sing {d}$.}
\label{fig:exblo}
\end{figure}

Let us justify the termination of the algorithm.

\begin{prop}
For any spanning forest $F$, Algorithm \ref{prune} with input $F$ stops and outputs a spanning tree.
\end{prop}
\begin{proof} Consider a step in the run of Algorithm \ref{prune} where the map $M'$ is not a tree. Then some edges incident to the face\footnote{The faces of a combinatorial map $(H,\sigma,\alpha)$ are the cycles of $\sigma \circ \alpha$.} that contains the half-edge $h$ in $M'$ form a cycle. This cycle contains an external edge because $F$ is acyclic. So the face that contains $h$ includes an external non-isthmus edge. As $e$ takes for successive values in Algorithm \ref{prune} the edges of the face incident to $h$, the edge $e$ will be eventually both external and a non-isthmus edge. At this time, $e$ will be removed from $M'$. 

In other terms, while $M'$ is not a tree, an edge is deleted. Since $M'$ remains connected, $M'$ will end as a tree. Then, $M'$ has only one face and every edge has been visited by the algorithm.\end{proof}

\noindent \textbf{Remark 1.}  Algorithm \ref{prune}
%reminiscent of
is related to the  opening transformation $\Theta^*$ of a forested map into a \emph{blossoming tree} that we described in Section~\ref{s:bbforet} p.~\pageref{s:bbforet}. More precisely, when $M$ is planar,Algorithm \ref{prune} describes the  transformation of $(M,F)$ into an enriched T-tree in which the marked face is the external one (except that no \textit{bud} nor \textit{leaf}\footnote{"bourgeon" and "feuille" in French} is added).  Recall that an iteration in the transformation $\Theta^*$ consisted in cutting the external edges of the external face of $M$ that were not isthmuses. In the case where several edges were incident to the same two faces, we chose to cut the first edge we met when we walk around the map in counterclockwise order starting from the root. 
This explains the parallel between the two transformations.
However Algorithm \ref{prune} is more general than the transformation $\Theta^*$  since the map $M$ is not necessarily planar.

%If we do not take into account the \textit{buds} and \textit{leaves}, the blossoming tree that corresponds to the map $M$ is exactly $\tau(\emptyset)$. The constraint that Algorithm \ref{prune} adds compared with the classical transformation is the impossibility of deleting internal edges. In view of the connection between these two transformations, the blossoming activities can be well adapted to the situations where we study the Tutte polynomial on maps. 

% The principle of Algorithm \ref{prune} is to mime the transformation of a map into a \emph{blossoming} tree (see \cite{Sch97,bdg2002}) without involving any white or black leaves. More precisely, by starting from the root, we turn counter clockwise around the map and we remove the visited edges if the deletion does not break the connectivity of the map. We stop the algorithm when every edge has been visited. However, at the opposite of classic transformations, there is an extra constraint: internal edges cannot be deleted. 

%Thus, for any spanning forest $F$ of $M$, Algorithm \ref{prune} defines a spanning tree $\tau(F)$ such that $F \subseteq \tau(F)$.

Now let us define blossoming activity for internal edges.
Given a spanning tree $T$ of $G$, we say that an internal edge $e$ is \emph{blossoming-active} if
$$\tau(T \backslash e) = T.$$
The \emph{internal blossoming activity} is the function that maps a spanning tree $T$ onto the set of its internal blossoming-active edges. 

\noindent \textbf{Example.} For the first map of Figure \ref{fig:exblo} with spanning tree $T = \sing {c,d}$, the only internally blossoming-active edge is $c$. Indeed, we saw that $\tau(\ens{d})=T$ and  we observe that $\tau(\ens{c})= \ens{b,c}$.

 We will see with Proposition \ref{prop:blo} that the internal blossoming activity can be extended into a Tutte-descriptive activity, for any embedding of the graph.

% More generally, we are going to see that this function can also been extended in a complete edge activity.  \\

% If more details are needed, Algorithm \ref{goldfish} explicitly shows how to compute both internal and external blossoming-active edges.  \\

% \noindent \textbf{Remark 2.} The active edges are not characterized in terms of their fundamental cocycles, but such a description exists -- see Corollary \ref{corblo}. 

\noindent \textbf{Remark 2.} If we keep track of the buds and leaves of the transformation $\Theta^*$  in the form of charges $+1/-1$, we find an interesting new characterization of the internal blossoming activity when $M$ is planar. 
%In order to be more formal, we have to modify Algorithm \ref{prune} to assign charges d. 
A \textit{charge} is an element from $\ens{-1,+1}$. Given a spanning tree $T$ of $M$, we run Algorithm \ref{prune} and whenever we delete an edge $e$, we add a charge $-1$ to the departure vertex and a charge $+1$ to the arrival vertex. At the end of the algorithm, there only remains a rooted plane tree, namely $T$, where each vertex has several charges. Let $e$ be an internal edge. Deleting $e$ splits $T$ into two components. The one which does \textit{not} contain the root is the \textit{subtree} corresponding to $e$.
 
 \begin{prop} \label{p:newcar}
   Given any spanning tree, if an internal edge is blossoming-active, then the sum of charges in the corresponding subtree is $0$ or $1$. 
If $M$ is  planar, then the converse is true.  
  \end{prop}

\begin{figure}[h!]
\begin{center}
\includegraphics[scale=0.9]{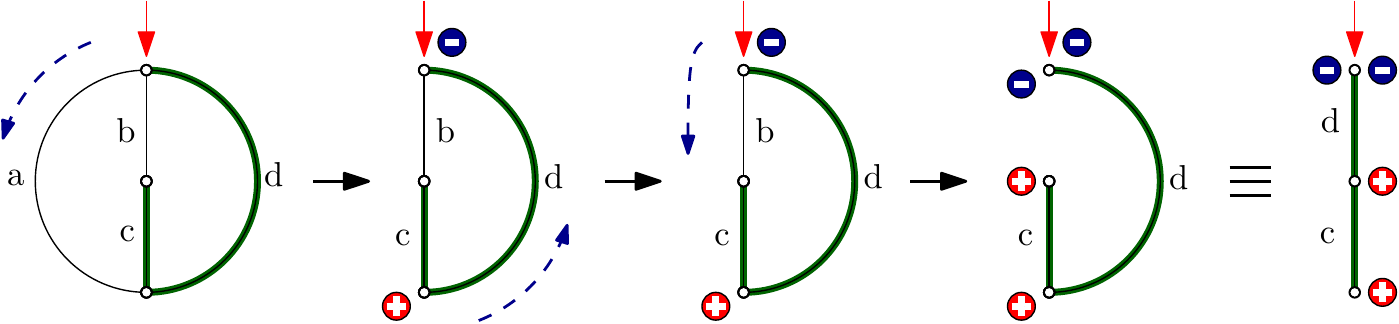}
\end{center}
\caption{Run of the version with charges of Algorithm \ref{prune} with spanning tree $T = \ens {c,d}$.}
\label{fig:excharge}
\end{figure}

We delay the proof of this proposition to the next section.

An example is shown in Figure \ref{fig:excharge}. Four charges have been distributed on the plane tree $T$.
% At the end, we obtain a rooted plane tree $T$ with two edges $c$ and $d$, the edge $c$ being the descendant of $d$. 
%Starting from the least deep, the vertices have respectively charge $-2$, $+1$, $+1$. 
If we delete $d$, the corresponding subtree has two vertices with one charge $+1$ each, so the charge of the subtree is $+2$. Hence the edge $d$ is not blossoming-active. On the contrary, if we delete $c$, the subtree has only one charge $+1$: the edge $c$ is blossoming-active. We have checked on this example Proposition \ref{p:newcar}: the only internal blossoming-active edge is $c$. 

When $M$ is not planar, the previous property is false. A counterexample is shown in Figure \ref{cexba}. The subtree corresponding to the only internal edge has charge $0$ but this edge is not blossoming active. (It will be deleted at the first step of Algorithm \ref{prune} if we remove it from the spanning tree.)

\begin{figure}[h!]
\begin{center}
\includegraphics[scale=0.9]{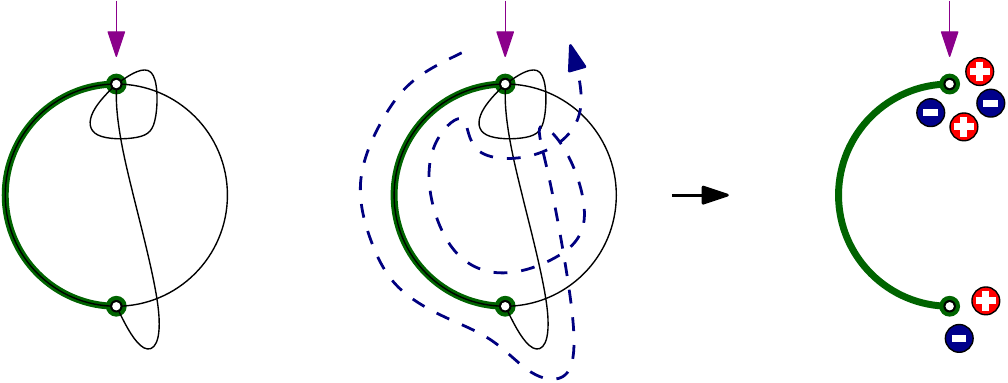}
\end{center}
\caption{A counterexample for Proposition \ref{p:newcar} with $M$ not planar.}
\label{cexba}
\end{figure}

Furthermore, we saw that Algorithm \ref{prune} transforms a planar map into an enriched T-tree in the exact same way as the transformation $\Theta$  when the external face is marked. But  an edge is \textit{flippable}\footnote{"basculable" in French} in an enriched  T-tree if and only if the sum of the charges in the corresponding subtree is $0$ or $1$. So we can deduce the following corollary.

\begin{core} Let $M$ be a planar map with external face $\chi$ and $T$ a spanning tree.  The \textit{flippable edges} of the enriched T-tree $\Theta^*(M,T,\chi)$ become, when we apply the map $\Theta$, the internal blossoming-active edges of $M$.
\end{core}

\section{Blossoming activity as $\boldsymbol \Delta$-activity}
%%%%%%%%%%%%%%%%%%%%%%%%%%%%%%%%%%%%%%%%%%%%%%%%%%%%%%%%%%%%%%%%%%%%%%%%%%

In this section we show that the blossoming activity is Tutte-descriptive by translating it in terms of $\Delta$-activity. We fix an embedding $M$ of $G$. We begin by a lemma (with a $\Delta$-activity flavour) establishing another characterization of internal blossoming-active edges.

%Now we are going to deal with the blossoming activity of Subsection \ref{ss:blo}. Let $M$ be a rooted map with underlying graph $G$.

\begin{leme}
For any spanning tree $T$ of $G$, an internal edge $e$ is blossoming-active if and only if $e$ is an isthmus of $M'$ when it is visited for the first time during the computation of $\tau(T)$ (cf. Algorithm \ref{prune}).
\end{leme}
\begin{proof} 
% First, let us remark that if an edge becomes an isthmus during a run of Algorithm \ref{prune}, then it will stay an isthmus for the rest of the run. 
%
Let us compare the executions of Algorithm \ref{prune} with inputs $T$ and $T \backslash e$. Before the first visit of $e$, the map $M'$ is the same. At this point, there are two possibilities for $e$.

If $e$ is not an isthmus in $M'$, then $e$ will be deleted when the input is $T \backslash e$. But when the input is $T$, the edge $e$ will be never deleted since it is internal. So $e \notin \tau(T \backslash  e)$ and $e \in \tau(T)$. Thus $\tau(T \backslash e) \neq \tau(T)$: the edge $e$ is not blossoming-active.

If $e$ is an isthmus in $M'$, then $e$ will not be deleted in both cases, and this for the rest of the run. Each other edge having the same status internal/external for the two inputs, we will have $\tau(T \backslash e) = \tau(T)$: the edge $e$ is blossoming-active.
\end{proof}

We can now prove that our new internal edge activity can be extended into a Tutte-descriptive activity.

\begin{prop} \label{prop:blo}
For any embedding of the graph $G$, the internal blossoming activity can be extended into a $\Delta$-activity and so a Tutte-descriptive activity.
\end{prop}

\textbf{Example:} Consider the map from Figure \ref{fig:exblo}. A suitable decision tree $\Delta$ is shown in Figure \ref{fig:dtblo}. For instance, consider the spanning tree $T = \ens{c,d}$. The $(\Delta,T)$-ordering is $a < d < b < c$ and the only internal active edge is $c$, as for the blossoming activity.

\begin{figure}[h!]
\begin{center}
\includegraphics[scale=1]{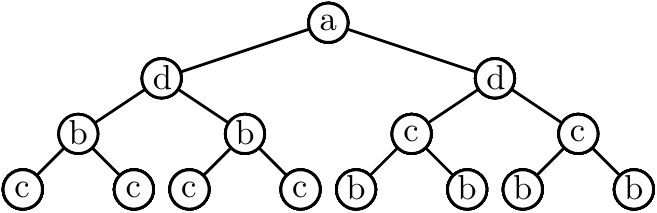}
\end{center}
\caption{A decision tree for the blossoming activity, for  the embedded graph from Figure \ref{fig:exblo}.}
\label{fig:dtblo}
\end{figure}

\begin{proof}
For any spanning tree $T$, let $\phi(T)$ denote the first visit order of the edges of $G$ in a run of Algorithm \ref{prune} with input $T$. For example, if we consider the first embedded graph from Figure \ref{fig:exblo}, then we map the spanning tree $\ens{c,d}$ onto the ordering $a < d < b < c$. Moreover, let $\phi_k(T)$ be the $k$-th smallest edge of $G$ for $\phi(T)$. We want to use Theorem~ \ref{treecompatible}.

%Now let us prove that the path mapping $\phi$ is tree-compatible. 
Consider $T$ and $T'$ two spanning trees of $G$ and $k \in \ens{0,\dots,|E(G)|-1}$ such that $$T \cap \ens{\phi_1(T),\dots,\phi_k(T)} = T' \cap \ens{\phi_1(T),\dots,\phi_k(T)}.$$
In Algorithm \ref{prune}, only the status (external, internal, isthmus) of $e$ in $M'$ at each iteration has an influence on the next values of $e$, $h$ and $M'$. But before the visit of $\phi_{k+1}(T)$, the statuses of $e$ are the same in $T$ and  in $T'$, since $T \cap \ens{\phi_1(T),\dots,\phi_k(T)} =$ \mbox{$T' \cap \ens{\phi_1(T),\dots,\phi_k(T)}$}.
%before the first visit of $\phi_{k+1}(T)$, we have visited the same edges, in the same order, whether the spanning tree is $T$ and $T'$. Indeed, there is no edge among the first ones which could distinguish the execution with input $T$ from the execution with input $T'$, since they have the same status (external, internal, isthmus) in both cases. 
So we have $\phi_i(T) = \phi_i(T')$ for every $i \in \ens{1,\dots,k+1}$.  

By Theorem \ref{treecompatible}, the order map $\phi$ is tree-compatible, meaning that we can construct a decision tree $\Delta_\phi$ such that the $(\Delta_\phi,T)$-ordering coincides with $\phi(T)$. In other terms, 
the edges are visited in the same order in Algorithm \ref{prune} and in Algorithm \ref{type} p.~\pageref{type} (with decision tree $\Delta_\phi$).

By Theorem \ref{charact}, the activity that sends a spanning tree onto the set of its $\Delta_\phi$-active edges is Tutte-descriptive. Let us show that it extends the internal blossoming activity. Algorithm~\ref{algint} p.~\pageref{algint}, when executed on a spanning tree $T$ with decision tree $\Delta_\phi$, outputs the set of edges that are isthmuses at the time of their first visit in Algorithm~\ref{prune}. By Proposition \ref{prop:algint}, this is the set of internal $\Delta_\phi$-active edges. But by the previous lemma, this coincides also with the set of internal blossoming-active edges.
\end{proof}

This proof allows us to give a natural definition of the complete blossoming activity: an external edge for a spanning tree $T$ is  \textit{blossoming-active} if it is $\Delta_\phi$-active, where $\Delta_\phi$ is any decision tree compatible with the order map $\phi$ defined in the previous proof. The $\Delta_\phi$-active edges are uniquely determined, although $\Delta_\phi$ is not. Indeed, we can give an intrinsic characterization of the (complete) blossoming activity, like the following one.

\begin{core} 
\label{corblo}
Given any spanning tree $T$, an edge is blossoming-active if and only if it has been visited last in its fundamental cycle/cocycle during the run of Algorithm \ref{prune} with input $T$.
\end{core}
\begin{proof} By definition of $\Delta_\phi$, the $(\Delta_\phi,T)$-ordering corresponds to the order of first visit during the run of Algorithm \ref{prune} with input $T$. We conclude using Proposition \ref{maximal}.
\end{proof}

Let us derive from this the proof of Proposition \ref{p:newcar}.

\begin{proof}[Proof of Proposition \ref{p:newcar}]
We use the same arguments  as in the point 4 of the proof of Theorem~ \ref{g1rs} p.\pageref{g1rs}. Let $T$ be a spanning tree and $e$ an edge of $T$. Denote by $E_2$ the set of vertices of the subtree corresponding to $e$ and by $E_1$ the complementary set of $E_2$ in $S(M)$. The fundamental cocycle $\gamma$ of $e$ is the set of edges in $M$ with one extremity in $E_1$ and the other in $E_2$. We use the variant of  Algorithm \ref{prune} where charges are assigned to the vertices of the map. Only the cut of the edges of $\gamma$ can modify the total charge of $E_2$. Indeed, the cut of another edge lets one charge $+1$ and one charge $-1$ either in $E_1$ or in $E_2$, so has not global effect on the charge of $E_2$.

Assume that $e$ is active. By Corollary \ref{corblo}, this means that $e$ is visited last among the edges of $\gamma$ during the run of Algorithm \ref{prune} with input $T$. Before the visit of $e$, the transitions from $E_1$ to $E_2$ and from $E_2$ to $E_1$ alternate and occur via external edges (because $e$ is the only internal edge of $\gamma$), so that 
the charges induced by two consecutive cuts compensate each other.
So the charge of $E_2$ only depends on the parity of the number of transitions between $E_1$ and $E_2$ before the first visit of $e$. If it is even, it equals $0$; if it is odd, it equals $1$. (The charge is $0$ at the start.)

\fig{[width=\textwidth]{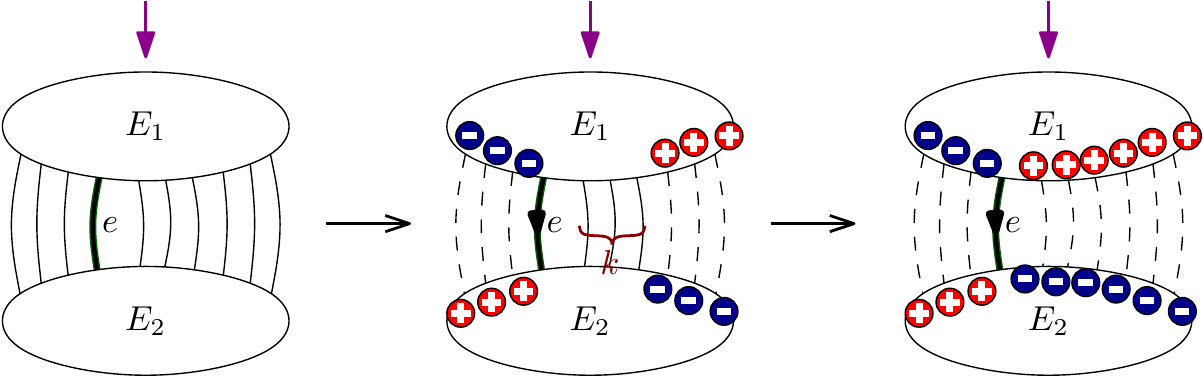}}{When the first transition through $e$ is proceeded from $E_1$ to $E_2$.}{transitions}

Assume now that $M$ is planar and $e$ is not active. By Corollary \ref{corblo}, this means that $e$ is  not visited last among the edges of $\gamma$. If the first transition through $e$ is proceeded from $E_1$ to $E_2$, then all the future transitions from $E_1$ to $E_2$ will be done through $e$. Why is this? Fix $c$ a corner of $M$ incident to the external face. Since an edge is only cut \textit{after} we walk along it, the edges we visit between two consecutive visits of $c$ are exactly the edges incident to the external face. But in a planar map, at most two edges inside a same cocycle can be incident to the external face. So, if $c$ is the corner incident to the external face that precedes the edge $e$, we visit at most two edges of $\gamma$ between two consecutive visits of $c$ and the first one must be $e$. So after the first visit of $c$, the transitions from $E_1$ to $E_2$ must be made through $e$. Let $k$ be the number of edges of $\gamma$ that we visit after $e$ (cf. Figure \ref{transitions}). (We recall that $e$ is not deleted since it is internal.) Just before the first visit of $e$, the charge of $E_2$ is $0$ because we have alternated transitions from $E_1$ to $E_2$ and transitions from $E_2$ to $E_1$. We have proved that the transition through an edge visited after $e$ must be made from $E_2$ to $E_1$, so each of these transitions decreases the charge of $E_2$ by $1$. In the end, the total charge of $E_2$ will be $-k$.

\fig{[width=\textwidth]{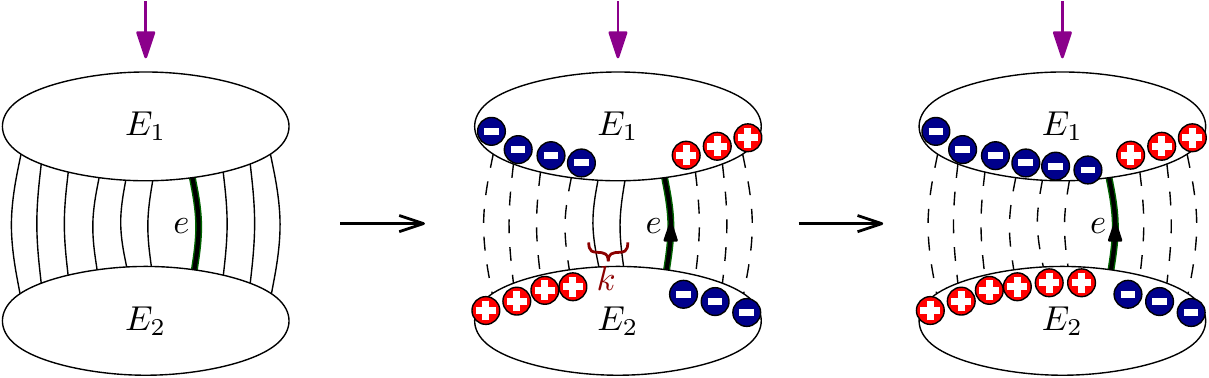}}{When the first transition through $e$ is proceeded from $E_2$ to $E_1$.}{transitions2}

The case where the first transition through $e$ is proceeded from $E_2$ to $E_1$ can be worked out in the same manner (see Figure \ref{transitions2}). The charge of $E_2$ is then $k+1$, where $k$ is the number of edges of $\gamma$ visited after $e$.
\end{proof}

Finally, let us give a description of the preimage of a spanning tree $T$ under $\tau$. (We recover a property similar to what we have already seen for the DFS activity.)
% Indeed, Proposition \ref{prop:tauf} will prove that for any spanning forest $F$ and spanning tree $T$ we have $\tau(F) = T$ if and only if $F \in [T \backslash \mathcal I (T), T]$, where $\mathcal I$ denotes the internal blossoming activity. 

\begin{prop} \label{prop:tauf}
For each spanning tree $T$ and each spanning forest $F$, we have
$\tau(F) = T$ if and only if $F \in [T \backslash \mathcal I (T), T]$  (see Subsection \ref{sss:interval} for the definition of an interval),
where $\mathcal I$ denotes the internal blossoming activity.
\end{prop}
\begin{proof} Given the spanning forest $F$, the edges we delete in Algorithm \ref{prune} are the external edges that are not isthmuses at the time of their first visit. They are precisely the edges of $\Delta_\phi$-type \bSe\, or \bL\, for $F$. Thus, $\tau(F)$ corresponds to the set of edges of type \bSi\, or \bI\, for $F$. So by Lemma \ref{reptree}, the spanning tree $\tau(F)$ is equivalent to $F$. By Corollary \ref{cor:inter}, we have then $F \in [\tau(F) \backslash \mathcal I (\tau(F)), \tau(F) \cup \mathcal E(\tau(F))]$, where $\mathcal E$ is the external blossoming activity. 
So, by Theorem \ref{partition}, we have $F \in [T \backslash \mathcal I (T), T \cup \mathcal E(T)]$ if and only if $T = \tau(F)$. But the restriction of the interval $[T \backslash \mathcal I (T), T \cup \mathcal E(T)]$ to the spanning forests of $G$ is $[T \backslash \mathcal I (T), T]$: indeed, Lemma \ref{addremove} states that a subgraph with an internal edge of type \bL\, has at least a cycle.
\end{proof}

\chapter{Final comments and prospects}
\label{s:com}
%%%%%%%%%%%%%%%%%%%%%%%%%%%%%%%%%%%%%%%%%%%%%%%%%%%%%%%%%%%%%%%%%%%%%%%%%%%%%%%%%%%%%%%
In this chapter, we make a few comments about $\Delta$-activity and mention some prospects.

\section{Generalization to matroids}

The generalization of the $\Delta$-activity to  matroids is rather immediate.
Indeed, as the notions of contraction, deletion, cycle, cocycle, isthmus, loop also exist in the world of matroids, the $\Delta$-activity can be defined for matroids \textit{verbatim} (except we no longer speak about "edges" but more generally about "elements"). Moreover, it is not difficult to check that all the results from Chapters  \ref{s:alg} and \ref{s:partition} hold in the same manner.
Let us add to this a property of duality, the proof of which is rather straightforward.

\begin{prop}
Let $M$ and $M^*$ be two dual matroids and $\Delta$ a decision tree. We define the decision tree $\Delta^\#$ as the mirror of $\Delta$\footnote{In term of decision functions, it means that $\Delta^\#(d_1,\dots,d_k) = \Delta(\overline d_1,\dots,\overline d_k)$, where $\overline \ell = r$ and $\overline r=\ell$.}. Given any subset $S$ of $M$, an edge is $\Delta$-active for $S$ in $M$ if and only if it is $\Delta^\#$-active for $S$ in $M^*$.
\end{prop}

\noindent \textbf{Remark. }We have not introduced the $\Delta$-activities directly on matroids because the notions of activities that we wanted to unify are more based on graphs than matroids. Moreover, I think that everyone who is familiar with matroids can generalize the notion of $\Delta$-activity without any difficulty. But the converse is not particularly true: those who are used to study the Tutte polynomial on graphs (or maps) could have been discouraged if this part was written in terms of matroids.

%%%%%%%%%%%%%%%%%%%%%%%%%%%%%%%%%%%%%%%%%%%%%%%%%%%%%%%%%%%%%%%%%%%%%%
\section{Forest activities}
\label{ss:partial}
%%%%%%%%%%%%%%%%%%%%%%%%%%%%%%%%%%%%%%%%%%%%%%%%%%%%%%%%%%%%%%%%%%%%%%
The example of DFS activity suggests that a broader notion of external activity can be found for spanning forests. This is the topic of this subsection. For the purpose of brevity, no proof will be given, but they can be easily adapted from the rest of this thesis. 

Let $G$ be a graph. A \textit{forest activity} is a function that maps a spanning forest $F$ onto a subset of $\mathcal G(F)$, where $\mathcal G(F)$ is the set of external edges $e$ such that $F \cup e$ has a cycle. Each such edge has a \textit{fundamental cycle}, which is the unique cycle included in $F \cup e$. A forest activity $\epsilon$ is \textit{Tutte-descriptive} if the Tutte polynomial of $G$ equals
\begin{equation}
T_G(x,y) = \sum_{F\textrm{ spanning forest of }G} (x-1)^{\cc(F)-1} y^{|\epsilon(F)|}.
\end{equation}
The external DFS activity is an example of Tutte-descriptive forest activity (see Equation ~\ref{DFStut} p.~\pageref{DFStut}).

Consider now a decision tree $\Delta$. Given any subgraph $S$, Algorithm \ref{ftype} outputs a subset $\epsilon(S)$ of edges, called the set of \textit{$\Delta$-forest active} edges of $S$.

\begin{algorithm}[h!]
\caption{Computing the set of $\Delta$-forest active edges.}
\label{ftype}
\begin{algorithmic}[5]
% \Algsetup{indent=2em}
\Require $S$ subgraph of $G$.
\Ensure A subset $\epsilon(S)$ of $E(G)$.
\State $m \leftarrow $ number of edges in $G$; 
\State $\epsilon(S) \leftarrow \emptyset$; 
\State $n \leftarrow$ root of $\Delta$; \State $H \leftarrow G$;
\For {$k$ from $1$ to $m$}
	\State $e_{k} \leftarrow$ label of $n$;
	\If{$e_k$ is not a loop in $H$ \textbf{and} $e_k \notin S$ }
			\State $n \leftarrow$ left child of $n$;
	\EndIf
	\If{$e_k$ is not a loop in $H$ \textbf{and} $e_k \in S$}
		\State {$H \leftarrow \contract H e$;} 			    
		\State $n \leftarrow$ right child of $n$;
	\EndIf 
	\If{$e_k$ is a loop in $H$}
			\State \textbf{add} $e_k$ in $\epsilon(S)$;
			\State $n \leftarrow$ left child of $n$;
	\EndIf 
\EndFor
\State \Return $\epsilon(S)$
\end{algorithmic}
\end{algorithm}

\noindent \textbf{Informal description.} We start from the edge that labels the root of $\Delta$. If this edge is external or a loop, we go to the left subtree of $\Delta$. If this edge is internal and not a loop, the edge is contracted and we go to the right subtree of $\Delta$. We repeat the process until the graph has no more edge.  An external edge is $\Delta$-forest active if it is a loop  when it is  visited. Figure \ref{schema} illustrates this description.

\fig{[scale=1.4]{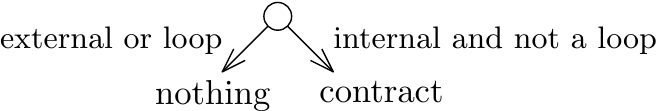}}{Diagram representing a step of Algorithm \ref{ftype}.}{schema2}

We can prove that the function that maps a spanning forest $F$ onto its set of $\Delta$-forest active edges is a Tutte-descriptive forest activity. We call it the \textit{$\Delta$-forest activity}. For any decision tree $\Delta$ the following properties are also true.
\begin{itemize}
	\item Let $F$ be a spanning forest. An edge $e$ is $\Delta$-forest active if and only if $e$ is an external edge such that $F \cup e$ has a cycle \emph{and} $e$ is maximal in its fundamental cycle for the ordering induced by the edges $e_k$ in Algorithm \ref{ftype}.
	\item The set of subgraphs can be partitioned into subgraph intervals indexed by the spanning forests:
$$ 2^{E(G)} = \biguplus_{F\textrm{ spanning forest of }G} [F, F \cup \epsilon(F)], $$
where $\epsilon$ denotes the $\Delta$-forest activity.
  \item There exists a decision tree $\Delta'$ such that the external DFS activity coincides with the $\Delta'$-forest activity.
\end{itemize}

We could also define a dual notion of internal activity, with similar properties, based on connected subgraphs.

\section{Strongly Tutte-descriptive activities}

It is illusory to want to characterize every Tutte-descriptive activity. Most of them do not have any nice structure -- for example, an isthmus or a loop could be non active. We thereby need to add a constraint in order to narrow the set of interesting activities.

An activity $\psi$ is said to be \emph{strongly Tutte-descriptive} if it is Tutte-descriptive and it induces a partition of the subgraphs:
\begin{equation} \label{eqpart2} 
2^{E(G)} = \biguplus_{T\textrm{ spanning tree of }G} [T \backslash \psi(T), T \cup \psi(T)] \end{equation}
(this equation is equivalent to \eqref{eqpart}).
We conjecture that the strongly Tutte-descriptive activities are precisely the $\Delta$-activities.
\begin{conjecture}
For any graph $G$, an activity $\psi$ is strongly Tutte-descriptive iff there exists a decision tree $\Delta$ such that $\psi$ equals the $\Delta$-activity.
\end{conjecture}

The right-to-left implication has been already proven in this manuscript through Theorem~\ref{charact} p. \pageref{charact} and Theorem \ref{partition} p. \pageref{partition}.  Furthermore, the above conjecture is equivalent to the following one.
\begin{conjecture}
Let $G$ be a graph with a standard edge and $\psi$ a strongly Tutte-descriptive activity. There exists an edge $e$ of $G$ that is active in no spanning tree.
\end{conjecture}

I will be very glad if a proof or a counter-example is found for this conjecture. Here is a proof of the equivalence between the two conjectures:
\begin{proof}
\textbf{Conjecture 1 implies Conjecture 2.}  Let us prove by induction on the total number of loops and isthmuses  of $G$ that for each decision tree $\Delta$, there exists an edge $e$ that is  $\Delta$-active in no spanning tree. Let $e$ be the label of the root node of $\Delta$. If $e$ is standard (which happens when there is no isthmus nor loop in $G$), then $e$ cannot have type \bL\, nor \bI. Consequently, $e$ cannot be $\Delta$-active for any spanning tree. If $e$ is a loop (resp. an isthmus), then $e$ is external (resp. internal) in any spanning tree. So it will be deleted (resp. contracted) in the first step in Algorithm \ref{type} and we go to the left part (resp. right part) of $\Delta$, denoted by $\Delta'$. We then use the induction hypothesis with graph $\delete G e$ (resp. $\contract G e$) and decision tree $\Delta'$. 

 \textbf{Conjecture 2 implies Conjecture 1.} Let $\psi$ be a strongly Tutte-descriptive activity for a graph $G$. 
  First we have to prove that loops and isthmuses are active for any spanning tree $T$. Let $e$ be a loop and set $S = T \cup e$. By definition of a strongly Tutte-descriptive activity, there exists a spanning tree $T'$ such that $S \in \, [T' \backslash \psi(T'),T' \cup \psi(T')]$. Since $e$ is a loop, we have $e \notin T'$. But $e \in S$, hence $e \in \psi(T')$. The spanning tree $T = S \backslash e$ must thereby belong to $[T' \backslash \psi(T'),T' \cup \psi(T')]$. But the union in \eqref{eqpart2} is disjoint, so $T=T'$. Thus $e \in \psi(T)$. We can similarly show that every isthmus belongs to $\psi(T)$.
 
    Now let us build by induction (on the number of edges) a decision tree $\Delta$ such that $\psi$ equals the $\Delta$-activity. If $G$ has standard edges, we label the root of $\Delta$ by an edge $e$ that is active in no spanning tree. (We have used Conjecture 2.) The edge $e$ is standard since it is not active for $\psi$. So $e$ cannot be $\Delta$-active since it labels the root node of $\Delta$.
Then we build a strongly Tutte-descriptive activity $\psi_c$ for $\contract G e$ by setting $\psi_c(T) = \psi(T \cup e)$ and a strongly Tutte-descriptive activity $\psi_d$ for $\delete G e$ by setting $\psi_d(T) = \psi(T)$. (The proof that these activities are indeed strongly Tutte-descriptive are left to the reader.) By induction, there exist two decision trees $\Delta_c$ et $\Delta_d$ that correspond to $\psi_c$ and $\psi_d$. The left subtree of $\Delta$ is taken to be $\Delta_d$ and the right one  is taken to be $\Delta_c$. This $\Delta$-activity is equal to $\psi$ since every spanning tree $T$ with $e$ external satisfies $\psi(T) = \psi_d(T)$ and every spanning tree $T$ with $e$ internal satisfies $\phi(T) = \psi_c(T \backslash e)$.

There remains the case (which includes the base case of the induction) where $G$ has only loops and isthmuses. However, as any loop and  any isthmus is active for $\psi$, any decision tree fits.         
\end{proof}

\selectlanguage{french}

\bibliographystyle{alpha} 

\bibliography{coloured}

\noindent \textbf{Note. }Les numéros en fin de chaque référence correspondent aux pages du mémoire dans lesquelles cette référence a été citée.

\newpage
\thispagestyle{empty}
\selectlanguage{french}

\noindent {\Large  \textbf{Résumé}.}

Cette thèse porte sur le polynôme de Tutte, étudié selon différents points de
vue.
Dans une première partie, nous nous intéressons à l'énumération des cartes
planaires munies d'une forêt couvrante, ici appelées \textit{cartes forestières}, avec un
poids $z$ par face et un poids $u$ par composante non racine de la forêt. De
manière équivalente, nous comptons selon le nombre de faces les cartes planaires
$C$ pondérées par $T_C(u+1,1)$, où $T_C$ désigne le polynôme de Tutte de $C$.
Nous commençons par une caractérisation purement combinatoire de la série
génératrice correspondante, notée $F(z,u)$. Nous en déduisons que $F(z,u)$ est
\textit{différentiellement algébrique} en $z$, c'est-à-dire que $F$ satisfait une
équation différentielle polynomiale selon $z$. Enfin, pour $u \geq -1$, nous
étudions  le comportement asymptotique du $n$-ième coefficient de $F(z,u)$. Nous
observons une transition de phase en $0$, avec notamment un régime très atypique
en $n^{-3}\ln^{-2} (n)$ pour $u \in [-1,0[$, témoignant d'une nouvelle \textit{classe
d'universalité} pour les cartes planaires.
 Dans une seconde partie, nous
proposons un cadre unificateur pour les différentes notions d'activités utilisées dans
la littérature pour décrire le polynôme de Tutte. La nouvelle notion d'activité
ainsi définie  est appelée \textit{$\Delta$-activité}. Elle regroupe toutes les notions
d'activité déjà connues et présente de belles propriétés, comme celle de Crapo
qui définit une partition  (adaptée à l'activité) du treillis des sous-graphes couvrants
en intervalles. Nous conjecturons en dernier lieu que toute activité qui décrit
le polynôme de Tutte et qui satisfait la propriété susmentionnée de Crapo peut
être définie en termes de $\Delta$-activités.

\noindent \textbf{Mots clés :}  combinatoire énumérative, polynôme de Tutte, cartes planaires,
forêts couvrantes, algébricité différentielle, comportement asymptotique,
activités.
\vfill

\selectlanguage{english}
\noindent {\Large \textbf{Abstract}.}

This thesis deals with the Tutte polynomial, studied from different points of
view. In the first part, we address the enumeration of planar maps equipped with
a spanning forest, here called \textit{forested maps}, with a weight $z$ per face and a
weight $u$ per non-root component  of the forest. Equivalently, we count (with
respect to the number of faces) the planar maps $C$ weighted by $T_C(u+1,1)$,
where $T_C$ is the Tutte polynomial of $C$. We begin by a purely combinatorial
characterization of the corresponding generating function, denoted
by $F(z,u)$. We deduce from this that $F(z,u)$ is \textit{differentially algebraic} in
$z$, that is, satisfies a polynomial differential equation in $z$. Finally, for
$u \geq -1$, we study the asymptotic behaviour of the $n$th coefficient of $F
(z,u)$. We observe a phase transition at $0$, with a very unusual
regime in $n^{-3}\ln^{-2} (n)$ for $u \in [-1,0[$, which testifies a new
\textit{universality class} for planar maps. 
In the second part, we propose a
framework  unifying the  notions of activity used in the literature to
describe the Tutte polynomial. The new notion of activity thereby defined is
called \textit{$\Delta$-activity}. It gathers all the notions of activities that were
already known and has nice properties, as Crapo's property that defines a partition of the lattice of the spanning subgraphs into intervals with respect to the activity. Lastly we
conjecture that every activity that describes the Tutte polynomial and that
satisfies Crapo's property can be defined in terms of $\Delta$-activity.

\noindent \textbf{Keywords:} enumerative combinatorics, Tutte polynomial, planar maps, spanning
forests, differentially algebraic, asymptotic behaviours, activities.

\end{document}